\documentclass[12pt,a4paper]{amsart}
\usepackage{amssymb}
\usepackage{amsfonts}
\usepackage{amsmath}
\usepackage[mathscr]{eucal}

\usepackage{color}

\newtheorem{thm}{Theorem}[section]
\newtheorem{prop}[thm]{Proposition}
\newtheorem{lem}[thm]{Lemma}
\newtheorem{cor}[thm]{Corollary}

\numberwithin{equation}{section}

\def\N{{\Bbb N}}
\def\Z{{\Bbb Z}}
\def\Q{{\Bbb Q}}
\def\R{{\Bbb R}}
\def\C{{\Bbb C}}
\def\A{{\Bbb A}}
\def\bT{{\Bbb T}}

\def\emp{\varnothing}

\def\fa{{\frak a}}
\def\fb{{\frak b}}
\def\fd{{\frak d}}

\def\ff{{\frak f}}

\def\fn{{\frak n}}
\def\fo{{\frak o}}
\def\fp{{\frak p}}

\def\ft{{\frak t}}

\def\fS{{\frak S}}

\def\fF{{\frak F}}

\def\fE{{\frak E}}

\def\fJ{{\frak J}}

\def\fc{{\frak c}}
\def\fW{{\frak W}}
\def\fX{{\frak X}}
\def\fC{{\frak C}}
\def\fQ{{\frak Q}}
\def\fP{{\frak P}}

\def\cO{\frak o}

\def\cB{{\mathcal B}}

\def\cE{{\mathcal E}}

\def\cQ{{\mathcal Q}}

\def\cW{{\mathcal W}}
\def\cU{{\mathcal U}}

\def\GL{{\operatorname {GL}}}

\def\SL{{\operatorname{SL}}}
\def\SO{{\operatorname{SO}}}

\def\PGL{{\operatorname{PGL}}}

\def\Re{{\operatorname {Re}}}
\def\Im{{\operatorname {Im}}}
\def\tr{{\operatorname{tr}}}
\def\nr{{\operatorname{N}}}
\def\Mat{{\operatorname{M}}}

\def\sgn{{\operatorname {sgn}}}
\def\Ad{{\operatorname{Ad}}} 
\def\vol{{\operatorname{vol}}}

\def\ch{{\cosh\,}}
\def\sh{{\sinh\,}}

\def\leq{\leqslant}
\def\geq{\geqslant}
\def\bsl{\backslash}
\def\le{\leq}
\def\ge{\geq}

\def\d {{{d}}}

\def\JJ{{\Bbb J}}
\def\bs{{\bold s}}
\def\bK{{\bold K}}

\def\1{{\bold 1}}

\renewcommand{\a}{\alpha}
\renewcommand{\b}{\beta}

\newcommand{\e}{\epsilon}

\renewcommand{\l}{\lambda}
\renewcommand{\k}{\kappa}

\newcommand{\s}{\sigma}

\newcommand{\ga}{{\mathfrak{a}}}

\newcommand{\gf}{{\mathfrak{f}}}

\newcommand{\gn}{{\mathfrak{n}}}
\newcommand{\go}{{\mathfrak{o}}}
\newcommand{\gp}{{\mathfrak{p}}}

\newcommand{\Acal}{{\mathcal A}}
\newcommand{\Bcal}{{\mathcal B}}

\newcommand{\Ecal}{{\mathcal E}}

\newcommand{\Gcal}{{\mathcal G}}

\newcommand{\Ical}{{\mathcal I}}

\newcommand{\Ncal}{{\mathcal N}}
\newcommand{\Ocal}{{\mathcal O}}
\newcommand{\Pcal}{{\mathcal P}}
\newcommand{\Qcal}{{\mathcal Q}}

\newcommand{\Scal}{{\mathcal S}}
\newcommand{\Tcal}{{\mathcal T}}
\newcommand{\Ucal}{{\mathcal U}}

\newcommand{\Wcal}{{\mathcal W}}

\renewcommand{\AA}{\mathbb{A}}

\newcommand{\CC}{\mathbb{C}}

\newcommand{\II}{\mathbb{I}}

\newcommand{\NN}{\mathbb{N}}
\newcommand{\PP}{\mathbb{P}}
\newcommand{\QQ}{\mathbb{Q}}
\newcommand{\RR}{\mathbb{R}}

\newcommand{\ZZ}{\mathbb{Z}}

\newcommand{\bfc}{{\mathbf c}}

\newcommand{\bfs}{{\mathbf s}}

\newcommand{\bfx}{{\mathbf x}}

\newcommand{\bfB}{{\mathbf B}}

\newcommand{\bfK}{{\mathbf K}}

\newcommand{\ord}{\operatorname{ord}}

\newcommand{\fin}{{\rm fin}}
\renewcommand{\Re}{\operatorname{Re}}

\newcommand{\Res}{\operatorname{Res}}

\def\tint{\textstyle{\int}}
\def\ro{\varrho}

\title{An explicit trace formula of Jacquet-Zagier type for Hilbert modular forms}

\author{Shingo Sugiyama}
\author{Masao Tsuzuki}

\begin{document}

	\maketitle
\begin{abstract}
	We give an exact formula of the average of adjoint $L$-functions of holomorphic Hilbert cusp forms with a fixed weight and a square-free level, which is a generalization of Zagier's formula known for the case of elliptic cusp forms on $\SL_2(\Z)$. As an application, we prove that the Satake parameters of Hilbert cusp forms with a fixed weight and with growing square-free levels are equidistributed in an ensemble constructed by values of the adjoint $L$-functions.
\end{abstract}
	\setcounter{tocdepth}{1}

\section{Introduction}

\subsection{Background and motivation} \label{Background and motivation}
In \cite{Zagier}, Zagier proposed an elegant way to compute the traces of Hecke operators on the space of elliptic cusp forms by means of the Rankin-Selberg method. This is a more direct way than the Selberg's one because it does not need a deliberate rearrangement for renormalization of divergent terms which are inevitably produced by truncation process. Later, a similar study was conducted for Maass forms on $\SL_2(\Z)$ in \cite{Zagier2}, and for general cusp forms on the adelization $\GL_2(\A_F)$ with an arbitrary number field $F$ in the work of Jacquet and Zagier \cite{JacquetZagier}.
The formula proved in \cite{JacquetZagier} can be viewed as an ``arithmetic deformation'' by a complex parameter $z$ of the usual Arthur-Selberg trace formula for $\GL(2)$ because the latter one is expected to be recovered from the former one as the residue at $z=1$. Although their general formula is less explicit than Zagier's one, it provides us with a different proof of holomorphicity of the symmetric square $L$-function of a cuspidal representation of $\GL_2(\A_F)$, which was first proved by Shimura \cite{Shimura} in a classical setting and later was generalized to an adelic setting by \cite{GelbartJacquet2}. The paper \cite{Zagier} also gave an application of the formula to the proof of the algebraicity of the critical values of the symmetric square $L$-functions for elliptic cusp forms, which  was independently obtained by Strum \cite{Strum} based on \cite{Shimura}.  Mizumoto \cite{Mizumoto} and Takase \cite{Takase} extended Zagier's method to Hilbert cusp forms under the assumption that the narrow class number of the base field is one. For application to special values, the explicit nature of their formula is crucial. 

In this paper, motivated by these works and intending further potential applications, we shall calculate Jacquet-Zagier's trace formula for holomorphic Hilbert cusp forms as explicitly as possible when the level is square-free without the assumption on the class number of the base field $F$. For a technical reason, we assume that the prime $2$ splits completely in $F$.
We remark that this assumption is mild enough
to include the interesting cases $F=\QQ$ and $F=\QQ(\sqrt{D})$ with $D>0$ and $D\equiv 1 (\text{mod } 8)$, and our formulas for Hilbert modular forms with large levels are new even for $F=\QQ$.
Since we use the matrix coefficients of discrete series representations at archimedean places, which are not compactly supported contrary to the test functions dealt in \cite{JacquetZagier}, we have to modify the convergence proof of the geometric side in a substantial way. Moreover, we completely calculate all local terms in the formula for a large class of test functions. As an application, we prove an equidistribution result of Satake parameters in the ensemble defined by the symmetric square $L$-functions $L(z,\pi;\Ad)$ of holomorphic Hilbert cusp forms $\pi$ of a fixed weight with the varying square-free levels, in such a way that a part of the famous Serre's equidistribution theorem (\cite[Th\'{e}or\`{e}me 1]{Serre}) is recovered from our formula by the specialization at $z=1$. The non-vanishing of the symmetric square $L$-function of an elliptic modular form at a point in the critical strip has been pursued by many authors (\cite{Khan}, \cite{Kohnen-Seng}, \cite{Blomer}). As a corollary to our asymptotic formula, we obtain infinitely many Hilbert cusp forms with a fixed weight of growing levels whose symmetric square $L$-functions are non-vanishing at a given point in the critical strip. 
Our method is not the Rankin-Selberg method in the accurate sence, because the Eisenstein series is not unfolded on the convergence region as was done in \cite{Zagier} and \cite{JacquetZagier}; actually, the same proof works if the Eisenstein series is replaced with a Maass cusp form. To illustrate the robustness of our method, we deduce Theorem~\ref{MAINTHM2} which is viewed as an adelic version of Gon's formula in an opposite setting to \cite{Gon}.

\subsection{Description of main results} \label{DMR}
Let us state our main result introducing notation used in this article on the way. Let $\NN$ be the set of all positive integers and we write $\NN_{0}$ for $\NN \cup \{0\}$. For any condition $P$, we put $\delta(P)=1$ if $P$ is true, and
$\delta(P)=0$ if $P$ is false, respectively. We set $\Gamma_\R(s)=\pi^{-s/2}\Gamma(s/2)$ and $\Gamma_{\C}(s)=2(2\pi)^{-s}\Gamma(s)$. All the fractional ideals appearing in this paper are supposed to be invertible. 

Let $F$ be a totally real number field,
$\cO$ the integer ring of $F$ and $\A$ the adele ring of $F$. The completed Dedekind zeta function of $F$ is denoted by $\zeta_F(s)$. (All $L$-functions in this article are supposed to be completed by appropriate gamma factors.)
The set of finite places and the set of archimedean places of $F$ are denoted by $\Sigma_\fin$ and $\Sigma_\infty$, respectively. Set $\Sigma_F=\Sigma_\fin \cup \Sigma_\infty$. The completion of $F$ at $v\in \Sigma_F$ is denoted by $F_v$ and the modulus of $F_v$ is by $|\,|_v$.
Then, $|\,|_\A = \prod_{v \in \Sigma_F}|\,|_v$ defines the idele norm on $\AA^\times$.
When $v\in \Sigma_\fin$, $\cO_v$, $\fp_v$ and $q_v$ denote the maximal order of $F_v$, the maximal ideal of $\cO_v$ and $\#(\go_v/\gp_v)$, respectively.
For a non-zero ideal $\ga \subset \go$, let $\nr(\ga)=\prod_{v \in \Sigma_\fin}q_v^{\ord_v(\ga)}$ denote the absolute norm of $\ga$ and $S(\ga)$ the set of $v \in \Sigma_\fin$ dividing $\ga$.
Let $D_F$ be the absolute discriminant of $F/\Q$ and $\psi=\otimes_{v\in \Sigma_F}\psi_v$ the character of $\A/F$ defined as $\psi_{\Q}\circ \tr_{F/\Q}$ where $\psi_\Q$ is the character of $\A_\Q/ \QQ$ such that $\psi_\Q(x)=e^{2\pi i x}$ for $x\in \R$.
Let $B$, $H$ and $N$ be $F$-subgroups of $G=\GL(2)$ defined symbolically as $B=\{\left[\begin{smallmatrix} * & * \\ 0 & * \end{smallmatrix}\right]\}$, $H=\{\left[\begin{smallmatrix} * & 0 \\ 0 & * \end{smallmatrix}\right]\}$ and $N=\{\left[\begin{smallmatrix} 1 & * \\ 0 & 1 \end{smallmatrix}\right]\}$; $B=HN$ is a Borel subgroup of $G$. Let $Z$ be the center of $G$. For $v\in \Sigma_F$, let $\bK_v$ denote $G(\cO_v)$ if $v\in \Sigma_\fin$ and ${\rm O}(2)$ if $v\in \Sigma_\infty$. Set $\bK=\prod_{v\in \Sigma_F}\bK_v$ viewed as a subgroup of the adelization $G_\A =\prod_{v \in \Sigma_F}' G_v$. Given a non-zero ideal $\fn\subset \cO$ and an even weight $l=(l_{v})_{v\in \Sigma_\infty} \in (2\N)^{\Sigma_{\infty}}$, let $\Pi_{\rm{cus}}(l,\fn)$ be the set of all those irreducible cuspidal automorphic representations $\pi \cong \otimes_v'\pi_v$ of $Z_\A\bsl G_\A$ such that $\pi_v$ is a discrete series representation of $Z_v\bsl G_v$ of weight $l_v$ for all $v\in \Sigma_\infty$ and $\pi_v$ has a non-zero vector invariant by the group $\bK_0(\fn\cO_v)=\{\left[\begin{smallmatrix} a & b \\ c & d \end{smallmatrix}\right]\in \bK_v|\,c\in \fn\cO_v\,\}$ for all $v\in \Sigma_\fin$. Let $v\in \Sigma_\fin$ for a while. Set $\fX_v=\C/4\pi i (\log q_v)^{-1}\Z$. For a $\bK_v$-spherical irreducible representation $\pi_v$ of $G_v$ with trivial central character, we define 
\begin{align}
	Q(\pi_v)=x_v(\pi)/(q_v^{1/2}+q_v^{-1/2}), \quad x_v(\pi)=a_v+a_v^{-1},\quad  a_v=q_v^{-\nu_v(\pi)/2}
	\label{Qpi}
\end{align}
with $(a_v,a_v^{-1})$ the Satake parameter of $\pi_v$; note that the exponent $\nu_v(\pi)\in \fX_v$ is determined only up to sign. Let $I(|\,|_v^{s})$ denote the normalized induced representation ${\rm Ind}_{B_v}^{G_v}(|\,|_v^{s}\boxtimes |\,|_v^{-s})$. Then $\pi_v \cong I(|\,|_v^{\nu_v(\pi)/2})$ and $Q(\pi_v)\in \R$ if $\pi_v$ is generic and unitarizable. For $\delta \in F_v^\times$, let $\varepsilon_\delta$ denote the real-valued character of $F_v^\times$ corresponding to the extension $F_v(\sqrt{\delta})/F_v$ by local class field theory; it depends only on the coset $\delta(F_v^{\times})^2$. For $z\in \C$ and $s\in \C$ (such that $\Re(s)>(|\Re(z)|-1)/2$), we define complex-valued functions $\Ocal_{v,\e}^{\delta,(z)}$ ($\e=0,1$) and $\Scal_{v}^{\delta,(z)}$ on $F_v^\times$ as 
{\allowdisplaybreaks
	\begin{align*}
		\Ocal_{v,\e}^{\delta,(z)}(a)&=\frac{\zeta_{F_v}(-z)}{L_{F_v}\left(\tfrac{-z+1}{2},\varepsilon_{\delta}\right)}\biggl(\frac{1+q_v^{\frac{z+1}{2}}}{1+q_v}\biggr)^{\e}
		|a|_v^{\frac{-z+1}{4}}
		+\frac{\zeta_{F_v}(z)}{L_{F_v}\left(\tfrac{z+1}{2},\varepsilon_{\delta}\right)}
		\biggl(\frac{1+q_v^{\frac{-z+1}{2}}}{1+q_v}\biggr)^{\e}
		|a|_v^{\frac{z+1}{4}}
\end{align*}}
and
{\allowdisplaybreaks\begin{align}
		\Scal_{v}^{\delta,(z)}(s;a)
		&=-q_v^{-\frac{s+1}{2}}\frac{\zeta_{F_v}\left(s+\frac{z+1}{2}\right)\zeta_{F_v}\left(s+\tfrac{-z+1}{2}\right)}{L_{F_v}(s+1,\varepsilon_{\delta})}
		|a|_v^{\frac{s+1}{2}}, \quad (|a|_v\leq 1), 
		\label{Scaldeltazsa-f1}
		\\
		\Scal_v^{\delta,(z)}(s;a)
		&=-q_v^{-\frac{s+1}{2}}
		\left\{\frac{\zeta_{F_v}(-z)\zeta_{F_v}\left(s+\tfrac{z+1}{2}\right)}{L_{F_v}\left(\tfrac{-z+1}{2},\varepsilon_{\delta}\right)}\left|a\right|_v^{\frac{-z+1}{4}}
		+\frac{\zeta_{F_v}(z)\zeta_{F_v}\left(s+\tfrac{-z+1}{2}\right)}{L_{F_v}\left(\tfrac{z+1}{2},\varepsilon_{\delta}\right)}\left|a\right|_v^{\frac{z+1}{4}}\right\}, \quad (|a|_v>1).
		\notag
\end{align}}We remark that, when viewed as meromorphic functions in $z$, the singularities of these functions at $q_v^{z/2}=1$ are removable. A computation reveals{\allowdisplaybreaks\begin{align}
		\Ocal_{v,0}^{\delta,(z)}(a)=1, \quad
		\Ocal_{v,1}^{\delta,(z)}(a)=\tfrac{1}{1+q_v}\begin{cases}
			2 \quad (\text{$v$ splits in $F(\sqrt{\delta})/F$}), \\
			1 \quad (\text{$v$ ramifies in $F(\sqrt{\delta})/F$}),\\
			0  \quad (\text{$v$ remains prime in $F(\sqrt{\delta})/F$})
		\end{cases}
		\quad \text{
			for all $a\in \cO_v^\times$.}
		\label{Ovanish} 
	\end{align}
}
For a finite subset $S\subset \Sigma_\fin$, a square-free ideal $\fn$ such that $S\cap S(\fn)=\emp$, an element $\Delta\in F^\times$, and a non-zero fractional ideal $\fa\subset F$, the product
\begin{align*}
&{\bf B}_{\fn}^{(z)}(\bfs|\Delta;\fa)=\prod_{v\in \Sigma_\fin-(S\cup S(\fn))}\Ocal_{0,v}^{\Delta,(z)}(a_v)\prod_{v\in S(\fn)}\Ocal_{1,v}^{\Delta,(z)}(a_v)\prod_{v\in S}
\Scal_v^{\Delta,(z)}(s_v,a_v), \qquad \\
&\quad (\ \bfs \in \fX_S, \quad \min_{v \in S}\Re(s_v)>(|\Re(z)|-1)/2\ )
\end{align*}
is well-defined due to \eqref{Ovanish}, where $(a_v)\in \A_{\fin}^\times$ is the idele corresponding to $\fa$ and $\fX_S=\prod_{v \in S}\fX_v$.
Let $v\in \Sigma_\infty$. We define a complex-valued function $\Ocal_{v}^{\pm ,(z)}$ on $F_v^\times$ as
{\allowdisplaybreaks
	\begin{align*}
		\Ocal_v^{+,(z)}(a)&=\frac{2\pi}{\Gamma(l_v)}\frac{\Gamma\left(l_v+\tfrac{z-1}{2} \right) \Gamma\left(l_v+\tfrac{-z-1}{2} \right)}{\Gamma_\R\left(\tfrac{1+z}{2}\right)\Gamma_\R\left(\tfrac{1-z}{2}\right)}\delta(|a|>1)(a^{2}-1)^{1/2}\fP^{1-l_v}_{\frac{z-1}{2}}(|a|), \quad \\
		\Ocal_v^{-,(z)}(a)&=\frac{\pi i }{\Gamma(l_v)} \Gamma\left(l_v+\tfrac{z-1}{2} \right) \Gamma\left(l_v+\tfrac{-z-1}{2} \right)\,{\rm sgn}(a)\, (1+a^{2})^{1/2}\{\fP^{1-l_v}_{\frac{z-1}{2}}(ia)-\fP^{1-l_v}_{\frac{z-1}{2}}(-ia)\},
\end{align*}}where $\fP_{\nu}^{\mu}(x)$ is the Legendre function of the 1st kind which is defined for points $x\in \C$ outside the interval $(-\infty,+1]$ of the real axis (\cite[\S 4.1]{MOS}).    

Suppose $\fn$ is square-free from now on. For $\pi\cong \otimes_{v}\pi_v\in \Pi_{\rm cus}(l,\fn)$ and a finite set $S\subset \Sigma_\fin$ disjoint from $S(\fn)$, set $\nu_S(\pi)=\{\nu_v(\pi)\}_{v\in S}$, an element of $\fX_S/\{\pm 1\}^{S}=\prod_{v\in S}(\fX_v/\{\pm 1\})$. We set 
\begin{align*}
	W^{(z)}_{\fn}(\pi)&=
	\nr(\fn\ff_{\pi}^{-1})^{(1-z)/2}
	\prod_{v\in S(\fn\ff_\pi^{-1})}
	\biggl\{1+\frac{Q(I_v(|\,|_v^{z/2}))-Q(\pi_v)^2}{1-Q(\pi_v)^2}\biggr\},
\end{align*}
where $\ff_{\pi}$ denotes the conductor of $\pi$. Since $-1<Q(\pi_v)<1$ by the unitarity of $\pi_v$, we have $W^{(z)}_\fn(\pi)\geq 0$ when $z$ is a non-negative real number. Let $L(s,\pi;\Ad)$ be the symmetric square $L$-function of $\pi\in \Pi_{\rm cus}(l,\fn)$, which on $\Re(s)>1$ is defined by the Euler product of $(1-a_v^2q_v^{-s})^{-1}(1-q_v^{-s})^{-1}(1-a_v^{-2}q_v^{-s})^{-1}$ over $v\in \Sigma_\fin-S(\ff_\pi)$ completed by the gamma factor $\Gamma_\R(s+1)\Gamma_{\C}(s+l_v-1)$ over $v\in \Sigma_\infty$ and by finitely many factors $(1-q_v^{-s-1})^{-1}$ over $v\in S(\ff_\pi)$; it is known to be entire on $\C$ and is identified with the standard $L$-function of an automorphic representation $\Ad(\pi)$ of $\GL(3,\A)$ (\cite{GelbartJacquet2}). We have the functional equation $L(s,\pi; \Ad)=D_F^{3(1/2-s)}\,\nr(\gf_{\pi})^{2(1/2-s)}\,L(1-s,\pi;\Ad)$. In particular, the sign of the functional equation is plus.
Let $\hat I_{\rm{cusp}}^0(\bfs, z)$ be a meromorphic function on $\fX_S\times \C$ defined as
{\allowdisplaybreaks \begin{align}\label{Icusp0sz}
		& \hat I_{\rm cusp}^0(\bfs, z) \\
		=&C(l,\fn)\sum_{\pi \in \Pi_{\rm{cus}}(l,\fn)}
		\frac{2^{-1}D_F^{z-1/2}\nr(\gn)^{(z-1)/2}W_{\gn}^{(z)}(\pi)}{\prod_{v\in S} \{(q_v^{(1+\nu_v(\pi))/2}+q_v^{(1-\nu_v(\pi))/2})-(q_v^{(1+s_v)/2}+q_v^{(1-s_v)/2})\}}
		\frac{L(\frac{z+1}{2},\pi;\Ad)}{L(1,\pi;\Ad)},\notag 
\end{align}
}where
\begin{align}
	C(l,\fn)=D_F^{-1}\,\{\prod_{v\in \Sigma_\infty}\tfrac{4\pi}{l_v-1}\}\,\{\prod_{v\in S(\fn)}(1+q_v)^{-1}\}.
	\label{SPEC-const}
\end{align}
For $\Delta \in F^\times-(F^\times)^2$, let $L(s,\varepsilon_{\Delta})$ be the completed $L$-function of the class field character $\varepsilon_{\Delta}$ of the quadratic extension $F(\sqrt{\Delta})/F$, $\fd_{\Delta}$ the relative discriminant of $F(\sqrt{\Delta})/F$, and $\ff_\Delta$ a fractional ideal uniquely determined by $\Delta\cO=\fd_{\Delta}\ff_{\Delta}^{2}$. 

\begin{thm}\label{thm0} Let $l=(l_v)_{v\in \Sigma_\infty}\in (2\N)^{\Sigma_\infty}$ with $l_v\geq 4$ for all $v\in \Sigma_\infty$ and $\fn$ a non-zero square-free ideal of $\cO$. Let $S$ be a finite subset of $\Sigma_\fin$ such that $S\cap S(\fn)=\emp$. We suppose $2$ splits completely in $F/\Q$ and $|2|_v=1$ for all $v\in S\cup S(\fn)$.
	Then, for $\bfs \in \fX_S$ and $z\in \CC$ such that
	$\min_{v \in S}\Re(s_v) > 2\min_{v\in \Sigma_{\infty}}l_v-1$ and $|\Re(z)|<\min_{v \in \Sigma_{\infty}}l_v-3$, we have the identity
	\begin{align*}
		\hat I^0_{\rm cusp}(\bfs,z)=&D_{F}^{\frac{z}{4}}\{\hat J_{\rm unip}^{0}(\bfs,z)+\hat J_{\rm unip}^{0}(\bfs,-z)\}+\hat J_{\rm hyp}^0(\bfs,z)+\hat J_{\rm ell}^0(\bfs,z),
	\end{align*}
	where the terms on the right-hand side are described as follows. The first term is defined as 
	\begin{align*}
		\hat J_{\rm unip}^0(\bfs,z)&=D_{F}^{-\frac{z+2}{4}}\zeta_F(-z)\,
		\prod_{v \in S}
		\frac{-q_v^{-{(s_v+1)}/{2}}}{1-q_v^{-s_v-(z+1)/{2}}}
		\, 
		\prod_{v \in S(\gn)}\frac{1+q_v^{\frac{z+1}{2}}}{1+q_v}\,
		\prod_{v \in \Sigma_{\infty}}2^{1-z}\pi^{\frac{3-z}{4}}
		\frac{\Gamma\left(l_v+\frac{z-1}{2}\right)}{\Gamma\left(\tfrac{z+1}{4}\right)\Gamma(l_v)}.
	\end{align*}
	The second is given by the absolutely convergent sum
	\begin{align*}
		\hat J_{\rm hyp}^0(\bfs,z)&=\tfrac{1}{2}D_F^{-1/2}\zeta_{F}\left(\tfrac{1-z}{2}\right)\,\sum_{a\in \cO(S)_{+}^\times-\{1\}} {\bf B}_{\fn}^{(z)}(\bfs|1;a(a-1)^{-2}\cO)\,\prod_{v\in \Sigma_\infty}\Ocal_{v}^{+,(z)}((a+1)/(a-1)),
	\end{align*}
	where $\cO(S)_{+}^\times$ is the totally positive units of the $S$-integer ring of $F$. The third term is given by the absolutely convergent sum  
	\begin{align*}
		\hat J_{\rm ell}^0(\bfs,z)&=\tfrac{1}{2}D_F^{\frac{z-1}{2}} \sum_{(t:n)_F} \nr(\fd_{\Delta})^{\frac{z+1}{4}}L\left(\tfrac{z+1}{2},\varepsilon_{\Delta}\right)\,{\bf B}_{\fn}^{(z)}(\bfs|\Delta;n\ff_{\Delta}^{-2})\,\prod_{v\in \Sigma_\infty}\Ocal_v^{\sgn(\Delta^{(v)}),(z)}(t|\Delta|_v^{-1/2}),
	\end{align*}
	where $(t:n)_F$ runs over different cosets $\{(ct,c^2n)\in F^2|\,c\in F^\times\}$ such that $\Delta=t^2-4n\in F^\times-(F^\times)^2$, $(t,n) \in \{(c_v t_v,c_v^2n_v)|\,c_v\in F_v^\times,\,t_v\in \cO_v,\,n_v\in \cO_v^\times\}$ for all $v\in \Sigma_\fin-S$, and
	$\ord_v(n\gf_\Delta^{-2})<0$ for all $v \in S(\fn)$ with $\varepsilon_{\Delta,v}$ being unramified and non-trivial.
\end{thm}

As a corollary of Theorem \ref{thm0}, we have a generalization of \cite{Zagier}, \cite{Mizumoto} and  \cite{Takase}.
Let $\Acal_v$ be the space of holomorphic functions on $\fX_v$ such that $\a(-s_v)=\a(s_v)$. Let $\d\mu_v(s_v)$ denote the holomorphic $1$-form $2^{-1}(\log q_v)(q_v^{(1+s_v)/2}-q_v^{(1-s_v)/2})\,\d s_v$ on $\fX_v$, and let $L_v(c)$ be the contour $y\mapsto c+iy\,(c, y\in \R,\,|y|\leq {2\pi}{(\log q_v)}^{-1})$. For $\a\in \Acal_v$, set 
$$
\hat \Scal_v^{\delta,(z)}(\a;a)=
\tfrac{1}{2\pi i}\int_{L_v(c)} \Scal_v^{\delta,(z)}(s;a)\,\a(s)\,\d\mu_v(s). 
$$
For a finite subset $S\subset \Sigma_\fin$, let $\a(\bfs)=\otimes_{v\in S}\a_v(s_v)$ be a pure tensor of $\Acal_S=\bigotimes_{v\in S}\Acal_v$, viewed as a function on $\fX_S$. For a square-free ideal $\fn$ such that $S\cap S(\fn)=\emp$, an element $\Delta\in F^\times$, and a non-zero fractional ideal $\fa\subset F$
and $z\in \CC$, set
$$
{\bf B}_{\fn}^{(z)}(\a|\Delta;\fa)=\prod_{v\in \Sigma_\fin-(S\cup S(\fn))}\Ocal_{0,v}^{\Delta,(z)}(a_v)\prod_{v\in S(\fn)}\Ocal_{1,v}^{\Delta,(z)}(a_v)\prod_{v\in S}
\hat\Scal_v^{\Delta,(z)}(\a_v,a_v),
$$
$$
\Upsilon^{(z)}(\a)=
\prod_{v \in S}
\tfrac{1}{2\pi i} \int_{L_v(c_v)}\frac{-q_v^{-{(s_v+1)}/{2}}}{1-q_v^{-s_v-(z+1)/{2}}}\a_v(s_v)\d\mu_v(s_v)
$$
and
\begin{align}
	\II^0_{\rm cusp}(\fn|\a,z)=&2^{-1}D_F^{z-1/2}\nr(\fn)^{(z-1)/2}\sum_{\pi \in \Pi_{\rm cus}(l,\fn)}W_{\fn}^{(z)}(\pi)\,\frac{L\left(\tfrac{z+1}{2},\pi;\Ad\right)}{L(1,\pi;\Ad)}\,\a(\nu_S(\pi)).
	\label{AverageAdL}
\end{align} 

\begin{cor}\label{MAINTHM}
	Let $l$, $\gn$, $S$  be as in Theorem \ref{thm0}.
	We suppose $2$ splits completely in $F/\Q$ and $|2|_v=1$ for all $v\in S\cup S(\fn)$. 
	Then for $\a\in \Acal_S$ and $z\in \CC$ such that $|\Re(z)|<\min_{v \in \Sigma_{\infty}}l_v-3$, we have the identity
	\begin{align*}
		(-1)^{\# S}C(l,\fn)\,\II^0_{\rm cusp}(\fn|\a,z)=&D_{F}^{\frac{z}{4}}\{\JJ_{\rm unip}^{0}(\fn|\a,z)+\JJ_{\rm unip}^{0}(\fn|\a,-z)\}+\JJ_{\rm hyp}^0(\fn|\a,z)+\JJ_{\rm ell}^0(\fn|\a,z),
	\end{align*}
	where $C(l,\gn)$ is the number defined as \eqref{SPEC-const} and the terms on the right-hand side are described as follows. The first term is defined as
	\begin{align*}
		\JJ_{\rm unip}^0(\fn|\a,z)&=D_{F}^{-\frac{z+2}{4}}\zeta_F(-z)\,\Upsilon^{(z)}(\a)\, 
		\prod_{v \in S(\gn)}\frac{1+q_v^{\frac{z+1}{2}}}{1+q_v}\,
		\prod_{v \in \Sigma_{\infty}}2^{1-z}\pi^{\frac{3-z}{4}}
		\frac{\Gamma\left(l_v+\frac{z-1}{2}\right)}{\Gamma\left(\tfrac{z+1}{4}\right)\Gamma(l_v)}.
	\end{align*}
	The second is given by the absolutely convergent sum
	\begin{align*}
		\JJ_{\rm hyp}^0(\fn|\a,z)&=\tfrac{1}{2}D_F^{-1/2}\zeta_{F}\left(\tfrac{1-z}{2}\right)\,\sum_{a\in \cO(S)_{+}^\times-\{1\}} {\bf B}_{\fn}^{(z)}(\a|1;a(a-1)^{-2}\cO)\,\prod_{v\in \Sigma_\infty}\Ocal_{v}^{+,(z)}((a+1)/(a-1)).
	\end{align*}
	The third term is given by the absolutely convergent sum  
	\begin{align*}
		\JJ_{\rm ell}^0(\fn|\a,z)&=\tfrac{1}{2}D_F^{\frac{z-1}{2}} \sum_{(t:n)_F} \nr(\fd_{\Delta})^{\frac{z+1}{4}}L\left(\tfrac{z+1}{2},\varepsilon_{\Delta}\right)\,{\bf B}_{\fn}^{(z)}(\a|\Delta;n\ff_{\Delta}^{-2})\,\prod_{v\in \Sigma_\infty}\Ocal_v^{\sgn(\Delta^{(v)}),(z)}(t|\Delta|_v^{-1/2}),
	\end{align*}
	where
	$\Delta=t^2-4n$ and $(t:n)_F$ runs over the same set as in Theorem \ref{thm0}.
\end{cor}
Corollary~\ref{MAINTHM} includes known formulas as a special case. Indeed, it is confirmed that \cite[Theorem 1]{Zagier}, \cite[(4.6) and (5.9)]{Mizumoto}, and \cite[Proposition 2]{Takase} are all recovered from Corollary~\ref{MAINTHM}. Moreover, we can deduce an explicit trace formula of Hecke operators in a more general setting than in these works. 

\subsubsection{A limit formula}
For a real number $z\in [0,1]$ and a square-free ideal $\fn\subset \cO$, let us define a discrete Radon measure on $\Omega_S=[-2,+2]^{S}$ as 
\begin{align*}
	\langle \Lambda_{l,\fn}^{(z)},f \rangle =\frac{1}{{\rm M}(\fn)^{\delta(z=0)}}\,\,\prod_{v\in S(\fn)}\frac{q_v^{(z-1)/2}}{1+q_v^{(z+1)/2}}\,\sum_{\pi \in \Pi_{\rm cus}(l,\fn)}
	W^{(z)}_{\fn}(\pi)
	\frac{L(\tfrac{z+1}{2},\pi;{\rm Ad})}{L(1,\pi;{\rm Ad})}f({\bf x}_S(\pi)), \quad f\in C(\Omega_S),
\end{align*}
where ${\bf x}_S(\pi)=\{x_v(\pi)\}_{v\in S}$, and ${\rm M}(\fn)=\sum_{v\in S(\fn)}{\log q_v}/(1+q_v^{-1/2})$. Note that $2^{-1}\log \nr(\fn)\leq {\rm M}(\fn) \leq \log \nr(\fn)$. 
The value $L(\tfrac{z+1}{2},\pi;\Ad)$ is non-negative for $z$ lying in $\Re(z)\geq 1$, the closure of the absolute convergence region of the Euler product. In some cases, it is more enlightening to work with the assumption
\begin{center}
	{\bf (P)} : $L(s,\pi,\Ad)\geq 0$ for any $s\in [0,1)$ and for all holomorphic cuspidal automorphic representations $\pi$ of $\PGL(2,\A)$ with weight $l$,
\end{center}
which ensures that the measure $\Lambda_{l,\fn}^{(z)}$ is non-negative for $z\in [0,1]$. The statement {\bf (P)} is not proved up to now, however it is highly expected to be true from the entireness of $L(s,\pi;\Ad)$ (\cite{Shimura}, \cite{GelbartJacquet2}) combined with the Riemann hypothesis of the $L$-function $L(s,\pi;\Ad)$.

For $z\in [0,1]$ and $v\in S$, we define a non-negative Radon measure on $[-2,2]$ as 
\begin{align*}
	\langle \lambda_v^{(z)}, f_v\rangle=\frac{1+q_v^{(z+1)/2}}{\pi}\int_{-2}^{2} f_v(x)\frac{(1-x^2/4)^{1/2}}{(q_v^{(1+z)/4}+q_v^{-(1+z)/4})^{2}-x^2}\,\d x, \quad f_v\in C([-2,2]).
\end{align*}
Note that the measure $\lambda_v^{(1)}$ coincides with the limit measure in Serre's theorem \cite[Th\'{e}or\`{e}me 1]{Serre}. 
%
For $z\in [0,1]$, set 
\begin{align*}
	C_l^{(z)}=2D_F^{3/2}\left\{2^{\frac{1-z}{2}}\pi^{-\frac{3z+1}{4}}\Gamma\left(\tfrac{z+3}{4}\right)\right\}^{[F:\Q]}\,\prod_{v\in \Sigma_\infty}\frac{\Gamma\left(l_v+\tfrac{z-1}{2} \right)}{4\pi \Gamma(l_v-1)}
\end{align*}
and $r(z)=\zeta_{F,\fin}(z+1)$ if $z>0$ and $r(0)={\rm Res}_{z=1}\zeta_{F,\fin}(z)$. 
Let $\Pcal(\Omega_S)$ be the space of complex-valued polynomial functions on $\Omega_S$; by the Stone-Weierstrass theorem, $\Pcal(\Omega_S)$ is a dense subspace of $C(\fX_S^0)$. Fix a set of prime ideals $\fa_{j}\,(1\leq j\leq h)$ of $\cO$  different from $\fp_v\,(v\in S)$ and mapped bijectively to the ideal class group of $F$; this is possible by Chebotarev's density theorem.

\begin{thm} \label{MVT-thm} Let $F$ be as in Corollary~\ref{MAINTHM}. Let $l=(l_v)_{v\in \Sigma_\infty}\in (2\N)^{\Sigma_\infty}$ be such that ${\underline l}=\min_{v\in \Sigma_\infty}l_v \geq 4$. Let $S$ be a finite set of non-dyadic finite places of $F$. 
	Let $\fn \subset \cO$ be a square-free ideal relatively prime to the ideals $2\cO$, $\fa_i\,(1\leq j\leq h)$ and $\fp_v\,(v\in S)$.
	\begin{itemize}
		\item[(1)] Let $0<\s<\underline{l}-3$. As $\nr(\fn)\rightarrow \infty$, 
		$$\sup_{z\in [0,\min(1,\s)]}\left|\langle \Lambda_{l,\fn}^{(z)},f\rangle-r(z)C_l^{(z)}\langle \otimes_{v\in S}\lambda_v^{(z)},f\rangle \right|\rightarrow 0\quad \text{for any $f\in \Pcal(\Omega_S)$}.
		$$
		Under the condition {\bf (P)}, the same limit formulas are true for all $f\in C(\Omega_S)$, i.e., the measure $\Lambda_{l,\fn}^{(z)}$ converges $*$-weakly to $r(z)\,C_l^{(z)}\,\otimes_{v\in S}\lambda_v^{(z)}$ on $\fX_S^0$.
		\item[(2)] Assume {\bf (P)}. Let $J_{v}=[t_v,t_v']\,(t_v<t_v',\,v\in S)$ be a family of closed subintervals of $[-2,2]$ and $0<\s<\underline{l}-3$.
		There exists $M>0$ such that for any prime ideal $\fn$ with $\nr(\fn)>M$ and relatively prime to $2\prod_{v \in S}(\gp_v\cap\go)\prod_{j=1}^{h}\fa_j$ and for any $z\in [0,\min(1,\s)]$, there exists $\pi \in \Pi_{cus}(l,\fn)$ with the following properties: $\ff_\pi=\fn$, $L(\frac{z+1}{2},\pi; \Ad)\not=0$, and $\bfx_S(\pi)\in J$. When $J=\Omega_S$, the same conclusion holds without {\bf(P)}.
		
	\end{itemize}
\end{thm}

\subsection{Organization of paper}
In \S 2, we construct a kernel function ${\mathbf\Phi}^l(\gn|\bfs;g,h)$ on $G_\A\times G_\A$ which represents the resolvent $\prod_{v\in S}\{\bT_v-(q_v^{(1-s_v)/2}+q_v^{(1+s_v)/2})\}^{-1}$ of the product of the shifted normalized Hecke operators at $v\in S$ depending on a set of parameters $\bfs=(s_v)_{v\in S}$ and acting on the space of Hilbert modular cusp forms on $G_\A$ of a fixed weight invariant by the open compact subgroup $\bK_0(\fn)$. In \S 3, we introduce the smoothed Eisenstein series $\Ecal_\b^{*}(g)$ on $G_\A$ depending on an entire function $\b(z)$ vertically of rapid decay. Our definition is similar to but a bit different from the usual definition of the wave-packet in that we start from the normalized Eisenstein series. By a constraint on $\b(z)$ posed to cancel the poles of the Eisenstein series, our $\Ecal_\b^*(g)$ is shown to be rapidly decreasing on a Siegel domain (Proposition~\ref{SmoothEis}). As in \cite{JacquetZagier}, we compute the inner-product of $\Ecal_\b^{*}$ and the diagonal restriction of ${\mathbf\Phi}^l(\gn|\bfs;g,g)$ on $Z_\A G_F \bsl G_\A$ in two ways obtaining its two expressions referred to the spectral side and the geometric side. The spectral side is constructed in \S 3 without any difficulty by the Rankin-Selberg integral; the main issue here is the calculation of local zeta integrals for old forms. The computation of the geometric side is harder to accomplish. Since our test function is not compactly supported, the argument in \cite{JacquetZagier} to deal with the unipotent and the hyperbolic terms by the Poisson summation formula is not applied as it is. By a counter shifting argument inspired by a similar analysis in \cite{Shintani}, we entirely circumvent the usage of the Poisson summation formula. 
The computations of unipotent terms and hyperbolic terms are done in \S \ref{singular terms}
and \S \ref{The $F$-hyperbolic term}, respectively:
We prove the absolute convergence of the hyperbolic term by estimating
local orbital integrals whose exact formulas are calculated in \S \ref{Localorbitalintegrals}. 
By the local multiplicity one theorem of the Waldspurger model (\cite[Proposition 9']{Waldspurger0}, \cite{Waldspurger01}, \cite{Waldspurger}, \cite[\S1]{BFF}), the computation of the elliptic term boils down to the determination of the Eisenstein periods along elliptic tori and the calculation of local orbital integrals. In \S 7, we recall the formula of Eisenstein period originally due to Hecke, providing a proof in a modern style, and establish the absolute convergence of the elliptic term, granting the formulas of local orbital integrals to be proved in \S 10.
The main results are proved in \S 8. In \S 9, we state Theorem~\ref{MAINTHM2} which should be regarded as a version of Gon's formula. Since its proof is almost identical to Corollary~\ref{MAINTHM} and is much easier than that in some aspect, only a brief indication for proof will be given there. We suggest the readers to move to \S~9 to see the statement of Theorem~\ref{MAINTHM2} after finishing this introduction. Theorem~\ref{MAINTHM2} has an application to non-vanishing $L$-values, which we refer to our forthcoming work \cite{SugiyamaTsuzuki2018}.

\subsection{Notation}
Throughout this paper,
we adapt Vinogradov's notation $\ll$.
For any complex-valued functions $f$ and $g$ on a set $X$,
we write $f(x)\ll g(x)$ if there exists a constant $C>0$ independent of $x \in X$
such that $|f(x)| \le C |g(x)|$ for all $x$.
We write $f(x)\asymp g(x)$ when both $f(x)\ll g(x)$ and $g(x)\ll f(x)$ hold.
If we emphasize dependence of the implied constant $C$ on some parameters $a,b,c,\ldots$,
we write $f(x) \ll_{a,b,c,\ldots} g(x)$.

\section{Construction of the kernel function}
The Haar measures we work with in this article are fixed in the following way. For $v\in \Sigma_F$, let $\d x_v$ be the additive Haar measure of $F_v$ such that $\int_{\cO_v}\d x_v=q_v^{-d_v/2}$ if $v\in \Sigma_\fin$ and $\int_{|x|_v<1}\d x_v=2$ if $v\in \Sigma_\infty$,
where $d_v$ is the exponent of the local different of $F_v$. Fix measures $\d^\times x_v$ on $F^\times_v$ and $\d^\times x $ on $\A^\times$ by $\d^\times x_v=\zeta_{F_v}(1)\,\d x_v/|x_v|_v$ and by $\d^\times x=\prod_{v\in \Sigma_F}\d^\times x_v$, respectively. We fix Haar measures on $Z_v\cong F_v^\times$ and $Z_\A\cong \A^\times$ accordingly. We endow $H_v$, $N_v$, and $\bK_v$ with measures $\d h_v$, $\d n_v$, and $\d k_v$, respectively by setting $\d h_v=\d^\times t_{1,v}\,\d^\times t_{2,v}$ if $h_v=\left[\begin{smallmatrix} t_{1,v} & 0 \\ 0 & t_{2,v} \end{smallmatrix}\right]$, $\d n_v=\d x_v$ if $n_v=\left[\begin{smallmatrix} 1 & x_v \\ 0 & 1 \end{smallmatrix}\right]$, and by requiring $\vol(\bK_v, d k_v)=1$. Then our Haar measure on $G_v$ is defined as $\d g_v=\d h_v\,\d n_v\,\d k_v$ by the Iwasawa decomposition $G_v=H_vN_v\bK_v$. Note that $\vol(\bfK_v, dg_v)=q_v^{-3d_v/2}$. On $\bK$, $N_\A$, $H_\A$, $Z_\A$ and $G_\A$, we always use the product measures of their factors.
Let ${\rm ch}_X$ denote the characteristic function of a set $X$. We omit its domain as we guess easily
from context.

\subsection{Convergence lemmas}
Let $\Ad:G\rightarrow \GL({\frak s}{\frak l}_2)$ be the adjoint representation of $G$ on ${\frak{sl}}_2$. For $g\in G$, let $(\Ad(g)_{ij})_{1\leq i,j\leq 3}$ be the representing matrix with respect to the $F$-basis $\{\left[\begin{smallmatrix} 1 & 0 \\ 0& -1 \end{smallmatrix}\right],\,\left[\begin{smallmatrix} 0 & 1 \\ 0 & 0 \end{smallmatrix}\right],\,\left[\begin{smallmatrix} 0 & 0 \\ 1 & 0 \end{smallmatrix}\right]\}$ of ${\frak{sl}}_2(F)$. For $g_v\in G_v$, set
\begin{align*}
	\|g_v\|_v=
	\begin{cases}
		&\max\{|\Ad(g_v)_{ij}|_v\,|\,1\leq i,j\leq 3\,\}, \quad (v\in \Sigma_\fin), \\
		&\{\sum_{ij}|\Ad(g_v)_{ij}|_v^2\}^{1/2}, \quad (v\in \Sigma_\infty).
	\end{cases}
\end{align*}
Then, $\|z_vk_vg_vk'_v\|_v=\|g_v\|_v$ for any $z_v\in Z_v$ and $k_v,k_v'\in \bK_v$. The norm of an adele point $g=(g_v)\in G_\A$ is defined by $\|g\|_\A=\prod_{v\in \Sigma_F} \|g_v\|_v$. We remark that the norm $\|g\|_\A$ is a $Z_\A$-invariant and bi-$\bK$-invariant function on $G_\A$ taking values in $[1,+\infty)$. It satisfies $\|gh\|_\A\leq \|g\|_\A\|h\|_\A$ for all $g,\,h\in G_\A$ and $\|g\|_\A \asymp y(g)$, $g\in \fS^{1}$, where $\fS^1$ is a Siegel domain of $G_\A^1=\{g\in G_\A|\,|\det g|_\A=1\,\}$ and $y:G_\A\rightarrow \R_+$ is defined as
$$y\left(\left[\begin{smallmatrix} a & b \\ 0 & d \end{smallmatrix}\right]k\right)=|a/d|_\A, \quad \left[\begin{smallmatrix} a & b \\ 0 & d \end{smallmatrix}\right]\in B_\A,\, k\in \bK.$$

\begin{lem} \label{L3}
	For $\s\in \R$ and $h\in G_\A$, set $\Xi_{\s}(h)=\sum_{\gamma\in Z_F\bsl G_F} \|\gamma h\|_\A^{-\sigma}$. If $\s>1$, then the series $\Xi_\s(h)$ converges absolutely and locally uniformly in $h\in G_\A$; moreover, the function $h\mapsto \Xi_\s(h)$ is bounded on $G_\A$. 
\end{lem}
\begin{proof}
	Set $\overline{G}=Z\bsl G$.
	Let $\cU$ be a compact neighborhood of the identity of ${\overline G}_\A=Z_\A\bsl G_\A$ such that ${\overline G}_F \cap \cU^{-1}\cU=\{1_2\}$. Then $\bigcup_{\gamma \in  {\overline G}_F} \cU \gamma h$ is a disjoint union for any $h\in G_\A$. Set $c_0=\sup_{g\in \cU}\|g\|_\A$. We have $\|g\gamma h\|_\A \leq \|g\|_\A \|\gamma h\|_\A \leq c_0\|\gamma h\|_\A$, and thus 
	$$
	\|\gamma h\|_{\AA}^{-\sigma}\leq c_0^{\s}\|g\gamma h\|_\A^{-\s}\quad (\sigma>0)
	$$
	for any $g\in \Ucal,\,h\in {\overline G}_\A$ and $\gamma \in {\overline G}_F$. From this, \begin{align*}
		c_0^{-\s}\vol(\cU)\,\Xi_\s(h)\leq & \sum_{\gamma \in {\overline G}_F}\tint_{\cU}\|g\gamma h\|_\A^{-\s}\,\d g 
		=\tint_{\bigcup_{\gamma \in {\overline G}_F}\cU\gamma h}\|g\|_\A^{-\s}\,\d g 
		\leq\tint_{{\overline G}_\A} \|g\|_\A^{-\s}\,\d g.
	\end{align*} 
	Thus it suffices to show that the function $g\mapsto \|g\|_\A^{-\s}$ on $Z_\A\bsl G_\A$ is integrable if $\s>1$. For $\sigma\in \R$, set $\fJ_v(\sigma)=\int_{Z_v\bsl G_v}\|g_v\|_v^{-\sigma}\,\d g_v.$ If $v\in \Sigma_\fin$ and $\s>0$, then by the Cartan decomposition $Z_v\bsl G_v =\bigcup_{n\in \N_0} \bK_v \left[\begin{smallmatrix} \varpi_v^{n} & 0 \\ 0 & 1 \end{smallmatrix}\right]\bK_v$, we have
	\begin{align*}
		\fJ_v(\s)&=\sum_{n=0}^{\infty} q_v^{-n\s}\vol\left(\bK_v \left[\begin{smallmatrix} \varpi_v^{n} & 0 \\ 0 & 1 \end{smallmatrix}\right]\bK_v\right)
		=1+\sum_{n=1}^{\infty} q_v^{-n\s}\,q_v^{n}(1+q_v^{-1})
		=(1+q_v^{-\s-1})/(1-q_v^{-\s}).
	\end{align*}
	From this, the convergence of $\prod_{v\in \Sigma_{\fin}} \fJ_v(\s)$ for $\s>1$ is evident. If $v\in \Sigma_\infty$ and $\s>1$, then by the Cartan decomposition $G_v=\bK_v H_v \bK_v$, we have
	\begin{align*}
		\fJ_v(\s)&=4\pi \tint_{0}^{\infty}\left\|\left[\begin{smallmatrix} e^{t} & 0 \\ 0 & e^{-t} \end{smallmatrix}\right]\right\|^{-\s}\,\sh 2t\,\d t
		=4\pi \tint_{0}^{\infty} (e^{4t}+e^{-4t}+1)^{-\s/2}\,\sh 2t \,\d t<+\infty. 
	\end{align*}
	This completes the proof of the pointwise convergence. It remains to confirm the local uniformity of the convergence. Let $U$ be a compact subset of ${\overline G}_\A$. Then, $\|\gamma\|_\A\leq \,(\sup_{h\in U}\|h^{-1}\|_\A)\|\gamma h\|_\A$ for any $\gamma\in {\overline G}_F$ and $h\in U$. From this, $\Xi_\s(h)$ $(h\in U)$ is majorized term-wisely by the convergent series $\Xi_\s(1_2)$ for $\s>1$.
\end{proof}

\begin{prop} \label{CONVERGENCE-L}
	Let $\varphi:G_\A\rightarrow \C$ be a function such that $\varphi(zg)=\varphi(g)$ for all $z\in Z_\A$ and $|\varphi(g)|\ll \|g\|_\A^{-m}$ on $G_\A$ with $m>1$. Then, the series
	$$
	K_\varphi(g,h)=\sum_{\gamma \in Z_F\bsl G_F}\varphi(g^{-1}\gamma h), \quad (g,h)\in G_\A\times G_\A
	$$
	converges absolutely and locally uniformly. Moreover, the following holds.
	\begin{itemize}
		\item[(i)] $\sup_{h\in G_\A} |K_\varphi(g,h)|\ll \|g\|_\A^{m}, \quad g\in G_\A$ with the implied constant independent of $g$. 
		\item[(ii)] For any $g\in G_\A$, the function $K_\varphi(g,-)$ belongs to $L^{q}(Z_\A G_F\bsl G_\A)$ for any $q>0$.
		
	\end{itemize}
\end{prop}
\begin{proof}
	Let $U$ be a compact subset of ${\overline G}_\A$. Then 
	$$
	|\varphi(g^{-1}\gamma h)|\ll \|g^{-1}\gamma h\|_\A^{-m}\asymp \|\gamma\|_\A^{-m} ,\quad (g,h)\in U\times U,\,\gamma\in {\overline G}_F.$$
	From this, the series $\sum_{\gamma \in \overline{G}_F}|\varphi(g^{-1}\gamma h)|$ $(g,h\in U)$ is dominated by $\Xi_m(1_2)$, which is convergent if $m>1$ by Lemma~\ref{L3}.  
	
	(i) From $\|\gamma h\|_\A\leq \|g\|_\A\|g^{-1}\gamma h\|_\A$, we have
	\begin{align*}
		\sum_{\gamma \in {\overline G}_F} |\varphi(g^{-1}\gamma h)| \ll \sum_{\gamma \in {\overline G}_F}\|g^{-1}\gamma h\|_\A^{-m} 
		\leq \|g\|_\A^{m}\,\Xi_{m}(h)
	\end{align*}
	for any $g,\,h\in {\overline G}_\A$. Since $\Xi_m(h)$ is bounded in $h\in {\overline G}_\A$, we are done. 
	
	(ii) From (i), for a fixed $g$, the function $K_\varphi(g,h)$ in $h$ is bounded. Since $Z_\A G_F \bsl G_\A$ is of finite volume, $K_{\varphi}(g,-)\in L^{q}(Z_\A G_F \bsl G_\A)$ for any $q>0$. 
\end{proof}

\subsection{Matrix coefficients of the discrete series} \label{MCDS}
We recall a basic material from \cite[\S 11, 14]{KnightlyLi} to fix notation. For $l\in 2\N$, let $(\delta_l,D_l)$ be the discrete series representation of $\PGL(2,\R)$ with minimal $\SO(2)$-type $\tau_{l}$, where 
$\tau_l$ is the character of $\SO(2)$ given by $\tau_l(\begin{smallmatrix}
\cos \theta & \sin \theta \\
-\sin \theta & \cos \theta
\end{smallmatrix})=e^{il\theta}$. The matrix coefficient of $D_l$ is defined by $$
\Phi^{l}(g)=(\delta_l(g)v_{l}|v_{l}), \qquad g\in \PGL(2,\R),
$$
where $v_{l}$ is a unit vector belonging to $\tau_{l}$. It satisfies the conditions :
\begin{itemize}
	\item[(a)] $R({W})\Phi^{l}(g)=0$ with $W=\frac{1}{2}\left(\begin{smallmatrix} 1 & -i \\ -i & -1 \end{smallmatrix}\right)$, 
	\item[(b)] $\Phi^{l}(k gk')=\tau_{l}(k')\,\tau_{l}(k)\,\Phi^{l}(g)$ for $k,k'\in \SO(2)$, 
	\item[(c)] $\Phi^{l}(1_2)=1$.
\end{itemize}
These conditions uniquely determine the $C^\infty$-function $\Phi^{l}$ as 
\begin{align}
	\Phi^{l}(g)=\delta(\det g>0)\,(4\det g)^{l/2}
	\{(a+d)-i(b-c)\}^{-l}
	\quad \text{for}\,g=\left[\begin{smallmatrix} a & b \\ c & d \end{smallmatrix}\right]. 
	\label{DSMexplicitformula}
\end{align}

Let $v\in \Sigma_\infty$. We denote by $\Phi_v^{l}$ the function $\Phi^{l}$ on $G_v\cong \GL(2,\R)$. Then, 
\begin{align}
	|\Phi_v^{l}(g)|\ll \|g_v\|_v^{-l/2}, \quad g\in G_v.
	\label{MCDSest}
\end{align}

\begin{lem} \label{MATCOEFF-L1}
	Let $l\in 2\N$. Let $v\in \Sigma_\infty$ and $\varphi:G_v\rightarrow \C$ a $C^\infty$-function such that $\varphi(zgk)=\tau_{l}(k)\varphi(g)$ for $(z,k)\in Z_v \times \bK_v^{0}$ and such that $R(W_v)\varphi=0$. Then we have
	the relation
	$$
	\int_{Z_v\bsl G_v}\Phi_v^{l}(h_v)\,\overline{\varphi( h_v)}\,\d h_v=
	\tfrac{4\pi}{l-1}\,\overline{\varphi(1_2)}.$$
	
\end{lem}
\begin{proof}
	Set $f_\varphi(h_v)=\vol(\bfK_v^{0})^{-1}\int_{\bK_v^{0}}\varphi(kh_v)\,\tau_{-l}(k)\,\d k$ for $h_v\in G_v$, where $\bK_v^{0}\cong \SO(2)$. Then $f_\varphi$ satisfies the same conditions (a) and (b) as $\Phi^{l}$. By the uniqueness of the solution to the differential equation (a), there is a constant $C$ such that $f_{\varphi}(h_v)=C\,{{\Phi_v^{l}(h_v)}}$ for all $h_v\in G_v$. By putting $h_v=1_2$, we have $C=f_{\varphi}(1_2)=\varphi(1_2)$. Substituting the equation $\overline{f_{\varphi}(h)} = \overline{ \varphi(1_2)}\, \overline{\Phi_v^{l}(h)}$, we have
	\begin{align*}
		\tint_{Z_v\bsl G_v}\Phi_v^{l}(h_v)\,\overline{\varphi(h_v)}\,\d h_v
		&=\tint_{Z_v\bsl G_v}\Phi_v^{l}(h_v)\,\overline{f_\varphi(h_v)}\,\d h_v=\overline{\varphi(1_2)}\,\int_{Z_v\bsl G_v}|\Phi_v^{l}(h_v)|^2\,\d h_v
		=\tfrac{4\pi}{l-1}\,\overline{\varphi(1_2)}
	\end{align*}
	as desired. We refer to \cite[Proposition 14.4]{KnightlyLi} for the last identity.
\end{proof}

\subsection{Green functions on $\GL(2)$ over non-Archimedean local fields} \label{Greenftn} 
Let $v\in \Sigma_\fin$. For $s\in \C$ such that $\Re(s)>1$, it is easy to show that there exists a unique function $\varphi:G_v\rightarrow \C$ with the properties:
\begin{itemize}
	\item[(a)] $\varphi(zkgk')=\varphi(g)$ for any $k,k'\in \bK_v$, $z\in Z_v$.
	\item[(b)] $[R(\bT_v-(q_v^{(1+s)/2}+q_v^{(1-s)/2})\, 1_{\bK_v})]\,\varphi={\rm ch}_{Z_v\bK_v}$.
	\item[(c)] $\varphi(g)=O(1)$, $g\in G_v$. 
\end{itemize}
Here $\bT_v=\vol(\bfK_v; dg_v)^{-1} {\rm ch}_{\bfK_v[\begin{smallmatrix}
	\varpi_v & 0 \\ 0 & 1
	\end{smallmatrix}]\bfK_v}$ and $1_{\bK_v}=\vol(\bK_v;\d g_v)^{-1}{\rm ch}_{\bK_v}$. We denote this function $\varphi$ by $\Phi_v(s;-)$. By the Cartan decomposition, it is explicitly written as  
\begin{align*}
	\Phi_v\left(s;\left[\begin{smallmatrix} \varpi_v^{n_1} & 0 \\ 0 & \varpi_v^{n_2} \end{smallmatrix}\right]\right)
	=(q_v^{-(s+1)/2}-q_v^{(s+1)/2})^{-1}\,q_v^{-\frac{(s+1)}{2}(n_1-n_2)}, \quad (n_1,n_2\in \Z,\,n_1\geq n_2). 
\end{align*}
We also have
\begin{align}
	\Phi_v(s;g)=(q_v^{-(s+1)/2}-q_v^{(s+1)/2})^{-1}\,\{|\det g|_v^{-1}\,\max(|a|_v,|b|_v,|c|_v,|d|_v)^2\}^{-(s+1)/2} \quad \text{for}\,
	g=\left[\begin{smallmatrix} a & b \\ c & d \end{smallmatrix}\right]. 
	\label{nonarchGreenftn}
\end{align}
Let $U$ be a compact set of the half plane $\Re(s)>1$. Then, from the formula above,
\begin{align}
	|\Phi_v(s;g_v)|\ll_{U} \|g_v\|_v^{-(\Re(s)+1)/2}, \quad g_v\in G_v,\,s\in U.
	\label{GreenFtnEst}
\end{align}

\begin{lem}\label{Greenequation}
	Let $\varphi : G_v \rightarrow \CC$ be a smooth function such that $\varphi(zgk)=\varphi(g)$
	for all $g \in G_v$, $z\in Z_v$ and $k \in \bfK_v$. Then we have
	\begin{align*}
		\int_{Z_v\bsl G_v} \Phi_v(s;g_v)\,[R(\bT_v-(q_v^{(1+s)/2}+q_v^{(1-s)/2})\,\1_{\bK_v})\,\varphi](g_v)\,\d g_v = \vol(Z_v\bsl Z_v\bfK_v)\varphi(1_2)
	\end{align*}
	as long as the integral of the left-hand side is absolutely convergent.
\end{lem}

\subsection{The kernel functions} \label{The kernel function} 
Let $\fn\subset \cO$ be a non-zero ideal and $l=(l_v)_{v\in \Sigma_\infty}$ an element of $(2\N)^{\Sigma_\infty}$ such that $l_v\geq 4$ for all $v\in \Sigma_\infty$. Let $S\subset \Sigma_\fin$ be a finite set. For any $\bs=(s_v)_{v\in S} \in \fX_S$, set
\begin{align*}
	\Phi^{l}(\fn|\bs;g)=\{\prod_{v\in \Sigma_\infty}\Phi_v^{l_v}(g_v)\}\,\{\prod_{v\in S}\Phi_v(s_v;g_v)\}\,\{\prod_{v\in \Sigma_\fin-S}\Phi_{\fn,v}(g_v)\}, \quad g=(g_v)\in G_\A,
\end{align*}
where $\Phi_{\fn,v}={\rm ch}_{Z_v\bK_0(\fn\cO_v)}$. Define a function on $G_\A\times G_\A$ by 
\begin{align}
	{\bf \Phi}^{l}(\fn|\bs;g,h)&=\sum_{\gamma\in Z_F\bsl G_F} \Phi^{l}(\fn|\bs;g^{-1}\gamma h), \quad g,h\in G_\A.
	\label{Kernel}
\end{align}

\begin{lem}\label{CONVERGE}
	Let $\cU$ be a compact subset of $\{\bfs\in \fX_S|\,\Re(s_v)>1\,(\forall v\in S)\,\}$. Then the series \eqref{Kernel} converges absolutely and 
	\begin{align*}
		\sum_{\gamma \in Z_F\bsl G_F} | \Phi^{l}(\fn|\bs;g^{-1}\gamma h)|\ll_{\cU} y(g)^{m/2}, \qquad (g,h)\in \fS^{1}\times G_\A,\, \bs\in \cU 
	\end{align*}
	with $m$ being the minimum of the set $\{l_v|\,v\in \Sigma_\infty\}\cup\{\Re(s_v)+1|\,\bs\in \cU\,\}$.  
\end{lem}
\begin{proof}
	This follows from \eqref{MCDSest}, \eqref{GreenFtnEst} and Proposition~\ref{CONVERGENCE-L}. 
\end{proof}
Let ${\mathcal A}_{\rm cusp}(l,\fn)$ denote the space of all those functions $\varphi\in C^{\infty}(Z_\A G_F\bsl G_\A)[\tau_l]^{\bK_0(\fn)}$ square integrable on $Z_\A G_F\bsl G_\A$ and $R(W_v)\,\varphi=0$ for all $v\in \Sigma_\infty$. Then ${\mathcal A}_{\rm cusp}(l,\gn)$ is contained in $L_{\rm cusp}^{2}(Z_\AA G_F\bsl G_\AA)$ by Wallach's criterion \cite[Theorem 4.3]{Wallach}. It is known that the space ${\mathcal A}_{\rm cusp}(l,\fn)$ is finite dimensional.

\begin{prop} \label{SPECTRALEXP}
	Suppose $\bs\in \fX_S=\prod_{v \in S}\CC/4\pi i(\log q_v)^{-1}\ZZ$ with $\Re(s_v)>1$ for all $v\in S$ and $l_v\geq 4$ for all $v\in \Sigma_\infty$. 
	\begin{itemize}
		\item[(1)]
		For any $g\in G_\A$, the function $h\mapsto {\bf \Phi}^{l}(\fn|\bs;g,h)$ belongs to the space ${\mathcal A}_{\rm cusp}(l,\fn)$. 
		\item[(2)] 
		Let $\cB(l,\fn)$ be an orthonormal basis of the space ${\mathcal A}_{\rm cusp}(l,\fn)$ endowed with the induced $L^2$-inner product from $L^2(Z_\A G_F\bsl G_\A)$. Let $(\nu_v(\varphi))_{v \in S} \in \fX_S$ for each $\varphi\in \cB(l,\fn)$ be a family such that
		\begin{align}\label{SPECTRALEXP-1}
			R(\bT_v)\,\varphi=(q_v^{(1+\nu_v(\varphi))/2}+q_v^{(1-\nu_v(\varphi))/2})\,\varphi
		\end{align}
		for all $v\in S$. Then, ${\bf \Phi}^{l}(\fn|\bfs;g,h)$ with $g,h\in G_\A$ equals 
		\begin{align}
			&C(l,\fn)\,
			\sum_{\varphi\in \cB(l,\fn)}{{\overline{\varphi(g)}}\,\varphi(h)}{\prod_{v\in S} \{(q_v^{(1+\nu_v(\varphi))/2}+q_v^{(1-\nu_v(\varphi))/2})-
				(q_v^{(1+s_v)/2}+q_v^{(1-s_v)/2})\}^{-1}},
			\label{SPECTRALEXP-f1}
		\end{align}
		where $C(l,\fn)$ is the constant \eqref{SPEC-const}. 
	\end{itemize}
\end{prop}
\begin{proof}
	The assertion (1) follows from Proposition~\ref{CONVERGENCE-L} (ii) together with the properties (a) and (b) of $\Phi_v^{l_v}$ in \S~\ref{MCDS}. 
	
	(2)	To simplify notation, we write ${\Phi}(h)$ in place of ${\Phi}^{l}(\fn|\bfs;h)$ in this proof. For any $\varphi \in \cB(l,\fn)$, from the well-known computation (\cite{GelbartJacquet}), we have 
	\begin{align*}
		\langle {\bf \Phi}^l(\gn|\bfs; g,\bullet)|\varphi\rangle_{L^2}
		=\tint_{Z_\A\bsl G_\A} {\Phi}(h)\,\bar \varphi(gh)\,\d h.
	\end{align*}
	Let us write a general element $h\in G_\A$ as $h=h_{\infty} h_{S} h^{S,\infty}$ with $h_\infty\in G_\infty$, $h_S\in G_S=\prod_{v\in S}G_v$ and $h^{S,\infty}\in \prod_{v\in \Sigma_\fin-S}G_v$. Suppose ${\Phi}(h)\not=0$. Then $h^{S,\infty}\in \prod_{v\in \Sigma_\fin-S}\bK_0(\fn\cO_v)$, and thus $\varphi(gh)=\varphi(g h_\infty h_S)$. Hence the $h^{S,\infty}$-integral yields the volume factor $$\prod_{v\in \Sigma_\fin-S}\vol(Z_v\bsl Z_v \bK_0(\fn\cO_v);\d h_v)
	=[\bK_\fin:\bK_0(\fn)]^{-1}\,\prod_{v\in \Sigma_\fin-S}\vol(Z_v\bsl Z_v\bK_v;\d h_v).
	$$ 
	The $h_\infty$-integral and $h_{S}$-integral are computed by Lemmas~\ref{MATCOEFF-L1} and \ref{Greenequation}, respectively. To complete the proof, we use $\vol(Z_v\bsl Z_v\bK_v)=q_v^{-d_v}$ for $v\in \Sigma_\fin$ and the relation $\nu_v(\bar\varphi)=\pm\nu_v(\varphi)$ in $\fX_v=\CC/4\pi i(\log q_v)^{-1}\ZZ$ for $v\in S$ which results from the self-adjointness of $\bT_v$. 
\end{proof}

\noindent
{\bf Remark} : Let $\pi$ be the cuspidal automorphic representation generated by
$\varphi \in \cB(l,\gn)$. Then, $\nu_v(\varphi)=\pm\nu_v(\pi)$ for any $v \in S$.
It is known to be purely imaginary by Blasius \cite{Blasius}.
We do not need this fact in this paper.

\section{Smoothed convolution product: The spectral side}

\subsection{Smoothed Eisenstein series} 

Let 
$$
E(z;g)=\sum_{\gamma \in B_F\bsl G_F} y(\gamma g)^{\frac{z+1}{2}}, \quad \Re(z)>1,\,g\in G_\A
$$
be the $\bK$-spherical Eisenstein series on $G_\A$. As a function in $z$, it has a meromorphic continuation to $\C$, holomorphic on $\Re(z)\geq 0$ away from the simple pole at $z=1$ and satisfying the functional equation $E^*(-z;g)=E^{*}(z;g)$, where
$$
E^{*}(z;g)=\Lambda_F(z+1)\,E(z;g),  \quad \Lambda_F(z)=D_F^{z/2}\zeta_F(z). 
$$
We have the functional equation $\Lambda_F(1-z)=\Lambda_F(z)$ and the Fourier expansion 
\begin{align}
	E^{*}(z;g)&=\Lambda_F(-z)\,y(g)^{(1+z)/2} +\Lambda_F(z)\,y(g)^{(1-z)/2}
	+\Lambda_F(-z)\,\sum_{a\in F^\times} 
	W_{\psi}\left(z; \left[\begin{smallmatrix} a & 0 \\ 0 & 1 \end{smallmatrix} \right] g\right), 
	\label{EisFexp}
\end{align}
where $W_\psi(z)$ is the global Whittaker function defined as 
$$
W_{\psi}(z;g)=\int_{N_\A} y\left(w_0\left[\begin{smallmatrix} 1 & x \\ 0 & 1 \end{smallmatrix}\right] g \right)^{(z+1)/2}\,\psi(-x)\d x, \quad g\in G_\A.
$$
We have the product formula $W_\psi(z;g)=\prod_{v\in \Sigma_F}W_v(z;g_v)$, where $W_v(z;g_v)$ is the $\bK_v$-invariant Whittaker function on $G_v$ determined by 
\begin{align}
	W_v\left(z;\left[\begin{smallmatrix} t & 0 \\ 0 & 1 \end{smallmatrix}\right]\right)=\zeta_{F_v}(z+1)^{-1}\delta(\varpi_v^{d_v}\,t\in \cO_v)\,
	|\varpi_v^{d_v}t|_v^{1/2}\,\tfrac{|\varpi_v^{d_v+1} t|_v^{z/2}-|\varpi_v^{d_v+1} t|_v^{-z/2}}{q_v^{-z/2}-q_v^{z/2}}, \quad t\in F_v^\times
	\label{nonarchLocWhitt}
\end{align}
if $v\in \Sigma_\fin$ and by 
\begin{align}
	W_v\left(z;\left[\begin{smallmatrix} t & 0 \\ 0 & 1 \end{smallmatrix}\right]\right)=\zeta_{F_v}(z+1)^{-1}\,2\,|t|_v^{1/2}\,K_{z/2}(2\pi |t|_v), \quad t\in F_v^\times
	\label{archLocWhitt}
\end{align}
if $v\in \Sigma_\infty$. We need the uniform estimate of the Eisenstein series:


\begin{lem}\label{Eisunifest}
	For any $0<\s_1<\s_2$, there exists $m>0$ such that,
	for any element $D$ of the universal enveloping algebra of $G(F_\infty)\cap G_\A^1$, it holds that 
	$$|(z-1)R(D)\,E^*(z;g)|\ll_D y(g)^{m},\quad \Re(z) \in [\s_1,\s_2],\,g\in \fS^{1}.
	$$
\end{lem}
\begin{proof} 
	This follows from the materials proved in \cite[\S 19]{Jacquet}. Here is a brief indication of the proof. Let $\Phi$ be a Schwartz-Bruhat function on $\A^2$ defined as $\otimes_{v\in \Sigma_F}\Phi^0_v$ with $\Phi^{0}_v={\rm ch}_{\cO_v\oplus \cO_v}$ for $v\in \Sigma_\fin$ and $\Phi_v^{0}(x,y)=e^{-\pi(x^2+y^2)}$ for $v\in \Sigma_\infty$. By the multiplication from the right, $G(F_\infty)$ acts on $\A^2$ and hence on the space of Schwartz-Bruhat functions; by the induced action of the universal enveloping algebra of $G(F_\infty)$, we can define the derivative $D\Phi$ of $\Phi$. Let $D$ be as in the lemma. Then $D_F^{z/2}E(g.D\Phi,|\,|_\A^{z/2},|\,|_\A^{-z/2})$ in \cite[\S 19]{Jacquet} coincides with our $R(D)E^{*}(z,g)$ up to a non-zero constant multiple; this is checked by a computation on their absolute-convergence region $\Re (z)>1$. The desired bound follows from a similar bound of $\theta^0(|\,|_\A^{(z+1)/2},g.\Phi)$ by the formula \cite[(19.4)]{Jacquet}.
\end{proof}

Let $\cB_1$ be the space of entire functions $\b(z)$ such that $\b(0)=\b(\pm1)=\b'(\pm1)=0$ and such that on any interval $[\s_1,\s_2]\subset \R$ and $N>0$, 
$$|\b(z)|\ll (1+|\Im (z)|)^{-N}, \quad \Re(z)\in [\s_1,\s_2].
$$ 
For $\beta\in \cB_1$ and $\s>0$, we set
\begin{align*}
	\cE_{\beta}^*(g)=\int_{L_\sigma} \beta(z)
	\,E^{*}(z;g)\,\d z, \quad g\in G_\A.
\end{align*}

\begin{prop} \label{SmoothEis}
	The contour integral $\cE_{\b}^*(g)$ converges absolutely and is independent of the choice of a contour $L_\s\,(\s>0)$. For any $N>0$, 
	\begin{align*}
		|\cE_{\b}^*(g)|\ll_{} y(g)^{-N}, \quad g\in \fS^{1}.
	\end{align*}
\end{prop}
\begin{proof} Due to $\b(1)=0$, the function $\b(z)E^{*}(z;g)$ is holomorphic on $\Re(z)>0$. Thus the first two assertions follow from Lemma~\ref{Eisunifest}. From Lemma~\ref{Eisunifest}, by means of \cite[Lemma I.2.10]{MW}, we can deduce the following estimate for the non-constant term of the Eisenstein series $E_{\rm NC}^*(z;g)=E^*(z;g)-\{\Lambda_F(-z)y(g)^{\frac{z+1}{2}}+\Lambda_F(z)y(g)^{\frac{-z+1}{2}}\}$:
	\begin{align}
		|E_{\rm NC}(z;g)|\ll_{n} y(g)^{-n}, \quad z\in \Tcal_\delta,\, g\in \fS^1,
		\label{SmoothEisEst2}
	\end{align}
	for arbitrary $n>0$. To argue, we write $\Ecal_\b^{*}(g)$ as a sum of the following three terms: 
	{\allowdisplaybreaks\begin{align*}
			I_{+}(g)&=\tint_{L_{3/2}}\b(z)\Lambda_{F}(-z)y(g)^{\frac{z+1}{2}}\d z, \quad
			I_{-}(g)=\tint_{L_{3/2}}\b(z)\Lambda_{F}(z)y(g)^{\frac{-z+1}{2}}\d z, \\
			I_{\rm NC}(g)& =\tint_{L_{3/2}}\b(z)\Lambda_F(-z)E_{\rm NC}(z;g)\,\d z.
	\end{align*}}By \eqref{SmoothEisEst2}, the integral $I_{\rm NC}(g)$ has the majorant $y(g)^{-N}$ on $\fS^1$ with an arbitrary large $N$. Due to $\b(0)=\b(\pm 1)=0$, the integrands of $I_{\pm}(g)$ are holomorphic on $\C$. Thus by shifting the contour, the integrals $I_{\pm}(g)$ are also shown to be bounded by $y(g)^{-N}$ on $\fS^1$. 
\end{proof}


\subsection{The smoothed convolution}
In this section, we fix $\bfs=(s_v)_{v \in S}\in \fX_{S}$ such that $\Re(s_v)>1$ for all $v\in S$, and consider 
\begin{align}
	\II(\bfs, \b)=\int_{Z_\A G_F\bsl G_\A} {\bf\Phi}^{l}(\fn|\bfs;g,g)\cE_\b^{*}(g)\,\d g,\quad \b\in \cB_1. 
	\label{IIb}
\end{align}
From Proposition~\ref{SmoothEis} and Lemma~\ref{CONVERGE}, the integral converges absolutely. 

Let $\hat I_{\rm{cusp}}(\bfs, z)$ be a meromorphic function on $\C$ defined as
{\allowdisplaybreaks \begin{align}
		\hat I_{\rm cusp}(\bfs, z) 
		=&C(l,\fn)\sum_{\pi \in \Pi_{\rm{cus}}(l,\fn)}
		\frac{\PP_{E^*(z)}(l,\fn;\pi)}{\prod_{v\in S} \{(q_v^{(1+\nu_v(\pi))/2}+q_v^{(1-\nu_v(\pi))/2})-(q_v^{(1+s_v)/2}+q_v^{(1-s_v)/2})\}},
		\label{Icusp0sz}
\end{align}}where $C(l,\fn)$ is the constant \eqref{SPEC-const} and 
\begin{align}
	\PP_{E^*(z)}(l,\fn;\pi)=\sum_{\varphi\in \cB_\pi(l,\fn)}\langle E^{*}(z;-)|\varphi\,\bar\varphi\rangle_{L^2} 
	\label{AverageInnprod}
\end{align}
with $\cB_\pi(l,\fn)$ being an orthonormal basis of $V_\pi[\tau_l]^{\bK_0(\fn)}$. Let $I\subset \R$ be an open interval. A meromorphic function $f(z)$ on the vertical strip $\Re(z)\in I$ which is holomorphic away from possible poles on the real axis is said to be vertically of moderate growth on the strip $\Re(z)\in I$ if for any $[\s_1,\s_2]\subset I$ there exists $N_1>0$ such that $|f(z)|\ll (1+|\Im (z)|)^{N_1}$ on any $\Tcal_\delta=\{z\in \C|\Re(z)\in [\s_1,\s_2],\,|\Im (z)|\geq \delta\}$. For example, Lemma~\ref{Eisunifest} shows that the functions $z\mapsto \langle E^*(z)|\varphi\,\bar\varphi\rangle_{L^2}$ with $\varphi\in \Bcal_\pi(l,\fn)$ are vertically of moderate growth on $\C$. Since $\Pi_{\rm{cus}}(l,\fn)$ and $\Bcal_\pi(l,\fn)$ are finite sets, we see that the function $\hat I_{\rm{cusp}}(\bfs, z)$ is also vertically of moderate growth on $\C$ and is holomorphic away from the possible simple poles at $z=0, \pm 1$. By the spectral expansion given in Proposition~\ref{SPECTRALEXP} (2), we have 
\begin{align}
	\II(\bfs, \b)=&\int_{L_\s}\b(z) \hat I_{\rm cusp}(\bfs, z)\,dz\quad (\s\in\R)
	\label{SpectralsideThm}
\end{align}
for all $\b\in \cB_1$. 

Let us calculate \eqref{AverageInnprod} when $\fn$ is square-free. For $\pi\in \Pi_{\rm cus}(l,\fn)$, the conductor $\ff_{\pi}$ of $\pi$ divides the ideal $\fn$. Recall the basis $\cB_{\pi}(l,\fn)=\{\|\varphi_{\pi, \rho}\|_{L^2}^{-1}\, \varphi_{\pi,\rho}\}_{\rho\in \Lambda_{\pi}(\fn)}$ constructed in \cite{Sugiyama1} (see also \cite{Tsuzuki2015}), where the index set $\Lambda_{\pi}(\gn)$ consists of all the mappings $\rho : \Sigma_F \rightarrow \{0,1\}$ such that $\rho(v)=0$ for all $v \in \Sigma_{F}-S(\gn\ff_{\pi}^{-1})$ and such that 
\begin{align*}
	\int_{F\bsl \AA}\varphi_{\pi,\rho}([\begin{smallmatrix}
		1&x\\0&1
	\end{smallmatrix}]g)\psi(-x)dx=\prod_{v \in \Sigma_{F}}\phi_{\rho(v),v}(g_v),
\end{align*}
where $\phi_{j,v}$ are $\psi_v$-Whittaker functions; for $v\in \Sigma_\infty$, $\phi_{0,v}$ is given as 
$$
\phi_{0, v}(\begin{smallmatrix}
t&0\\0&1
\end{smallmatrix})=\delta(t>0)\,2|t|_v^{l_v/2}\,e^{-2\pi t},\qquad t \in F_v^{\times},
$$
and for $v\in \Sigma_\fin$, $\phi_{0,v}$ and $\phi_{1,v}$ will be recalled in the proof of Lemma~\ref{Rankin integral of new form}. By the unfolding procedure, the Rankin-Selberg integral
$$\langle E(2s-1;-)| \varphi_{\pi,\rho}\,\overline{\varphi_{\pi,\rho}}\rangle_{L^2}=\int_{Z_{\AA}G_F\bsl G_{\AA}}E(2s-1;g) \varphi_{\pi,\rho}(g)\overline{\varphi_{\pi,\rho}(g)}dg$$
for the cusp forms $\varphi_{\pi,\rho}$ is shown to be decomposed into the product of $Z_v(s,\phi_{\rho(v),v},\phi_{\rho(v),v})$ over all $v \in \Sigma_F$, where
\begin{align}Z_v(s,\phi,\phi')=
	\int_{F_v^{\times}}\int_{\bfK_v}\phi([\begin{smallmatrix}
		t & 0 \\ 0 & 1
	\end{smallmatrix}]k)\overline{\phi'([\begin{smallmatrix}
			t & 0 \\ 0 & 1
		\end{smallmatrix}]k)}|t|_v^{s-1}dkd^{\times}t, \quad \phi, \phi' \in {\Wcal}(\pi_v,\psi_v)
	\label{LocalZeta}
\end{align}
is the local zeta integral. Recall the quantity $Q(\pi_v)$ defined by \eqref{Qpi}. For computations over a finite place $v$, we use the decomposition
\begin{align}
	\bK_v=\bK_0(\fp_v)\cup \bigcup_{\xi \in \cO_v/\fp_v}\left[\begin{smallmatrix} 1 & \xi \\ 0  & 1 \end{smallmatrix}\right] w_0 \bK_0(\fp_v) \quad \text{with $w_0=\left[\begin{smallmatrix}0 & -1 \\ 1 & 0 \end{smallmatrix}\right]$.}
	\label{KBruhat}
\end{align}

\begin{lem}\label{Rankin integral of new form}
	For any $v \in \Sigma_\infty$, we have
	$$Z(s,\phi_{0, v},\phi_{0, v}) = 2^{1-l_v}\frac{\Gamma_\RR(s)}{\Gamma_\RR(2s)}L(s, \pi_v;\Ad).$$
	For $v\in \Sigma_\fin$ with $c(\pi_v)=0$, we have
	{\allowdisplaybreaks\begin{align}
			Z(s, \phi_{0, v}, \phi_{0,v})& = q_v^{d_v(s-3/2)}\frac{\zeta_{F,v}(s)}{\zeta_{F,v}(2s)}L(s,\pi_v;\Ad), 
			\label{RINF-f1}
			\\
			Z(s, \phi_{1,v},\phi_{1,v})&= \{\frac{q_v^{1-s}+q_v^s}{1+q_v}-Q(\pi_v)^2\}Z(s,\phi_{0,v},\phi_{0,v}),
			\label{RINF-f2}
	\end{align}} where $Q(\pi_v)$ is the number defined in \eqref{Qpi}.
	For $v\in \Sigma_\fin$ with $c(\pi_v)=1$, we have 
	\begin{align}
		Z(s, \phi_{0,v}, \phi_{0,v}) = \vol(\bfK_0(\gp_v))q_v^s \times q_v^{d_v(s-3/2)} \frac{\zeta_{F,v}(s)}{\zeta_{F,v}(2s)}L(s,\pi_v; \Ad).
	\end{align}
\end{lem}
\begin{proof} If $v\in \Sigma_\infty$, the integral \eqref{LocalZeta} for $\phi_{0,v}$ is computed in the proof of \cite[Lemma 6.4]{SugiyamaTsuzuki}. (Note that $\phi_{0,v} (v \in \Sigma_{\infty})$ in \cite{SugiyamaTsuzuki} coincides with $\pi_v(\begin{smallmatrix}-1&0\\0&1\end{smallmatrix})\phi_{0,v}$ according to our notation).
	In the rest of the proof, let $v\in \Sigma_\fin$. Suppose $c(\pi_v)=0$. Then, $\pi_v$ is a unitarizable spherical representation of $\PGL(2,F_v)$. Thus $V_{\pi_v}^{\bK_v}=\C\phi_{0,v}$ and $V_{\pi_v}^{\bK_0(\fp_v)}=\C\phi_{0,v}+\C\phi_{1,v}$ with $
	\phi_{1,v}=\pi_v\left(\left[\begin{smallmatrix} \varpi_v^{-1} & 0 \\ 0 & 1 \end{smallmatrix} \right]\right)\phi_{0,v}-Q(\pi_v)\,\phi_{0,v}$. By \cite[(2.30)]{Tsuzuki2015}, we have the formula \eqref{RINF-f1}. From the formula of $\phi_{1,v}$, 
	\allowdisplaybreaks{\begin{align*}
			Z(s,\phi_{1,v}, \phi_{1,v}) = & Z(s,\pi_v(\left[\begin{smallmatrix}
				\varpi_v^{-1}&0\\0&1
			\end{smallmatrix}\right])\phi_{0,v}, \pi_v(\left[\begin{smallmatrix}
				\varpi_v^{-1}&0\\0&1
			\end{smallmatrix}\right])\phi_{0,v})
			-Q(\pi_v)Z(s, \phi_{0,v}, \pi_v(\left[\begin{smallmatrix}
				\varpi_v^{-1}&0\\0&1
			\end{smallmatrix}\right])\phi_{0,v})\\
			&-Q(\pi_v)Z(s,\pi_v(\left[\begin{smallmatrix}
				\varpi_v^{-1}&0\\0&1
			\end{smallmatrix}\right])\phi_{0,v}, \phi_{0,v})
			+Q(\pi_v)^2 Z(s, \phi_{0,v}, \phi_{0,v}).
		\end{align*}
	}As for the first term, by \eqref{KBruhat}, 
	\allowdisplaybreaks{\begin{align*}
			&Z(s,\pi_v(\left[\begin{smallmatrix}
				\varpi_v^{-1}&0\\0&1
			\end{smallmatrix}\right])\phi_{0,v}, \pi_v(\left[\begin{smallmatrix}
				\varpi_v^{-1}&0\\0&1
			\end{smallmatrix}\right])\phi_{0,v})\\
			=&\tint_{F_v^{\times}}\tint_{\bfK_{0}(\gp_v)}\phi_{0,v}([\begin{smallmatrix}
				t & 0 \\ 0 & 1
			\end{smallmatrix}]k[\begin{smallmatrix}
				\varpi_v^{-1}&0\\0&1
			\end{smallmatrix}])\overline{\phi_{0,v}([\begin{smallmatrix}
					t & 0 \\ 0 & 1
				\end{smallmatrix}]k[\begin{smallmatrix}
					\varpi_v^{-1}&0\\0&1
				\end{smallmatrix}])}|t|_{v}^{s-1}d^{\times}tdk\\
			& +\sum_{\xi\in\go_v/\gp_v}\tint_{F_v^{\times}}\tint_{\bfK_{0}(\gp_v)}\phi_{0,v}([\begin{smallmatrix}
				t & 0 \\ 0 & 1
			\end{smallmatrix}][\begin{smallmatrix}
				1&\xi\\0&1
			\end{smallmatrix}]w_0k
			[\begin{smallmatrix}
				\varpi_v^{-1}&0\\0&1
			\end{smallmatrix}])\overline{\phi_{0,v}([\begin{smallmatrix}
					t & 0 \\ 0 & 1
				\end{smallmatrix}][\begin{smallmatrix}
					1&\xi\\0&1
				\end{smallmatrix}]w_0k[\begin{smallmatrix}
					\varpi_v^{-1}&0\\0&1
				\end{smallmatrix}])}|t|_{v}^{s-1}d^{\times}tdk.
	\end{align*}}Since $\Ad(\left[\begin{smallmatrix} \varpi_v & 0 \\ 0 & 1 \end{smallmatrix}\right])\bK_0(\fp_v)\subset \bK_v$, the first integral of the right-hand side is seen to be equal to $\vol(\bfK_{0}(\gp_v)) q_{v}^{1-s}Z(s,\phi_{0,v},\phi_{0,v})$. Similarly, by the equation $\phi_{0,v}([\begin{smallmatrix}t & 0 \\ 0 & 1 \end{smallmatrix}][\begin{smallmatrix}
	1&\xi\\0&1 \end{smallmatrix}] g)=\psi_v(t\xi) \phi_{0,v}(
	[\begin{smallmatrix}t & 0 \\ 0 & 1\end{smallmatrix}]g)$ and by a simple variable change, each of $\xi$-terms is computed to be  $\vol(\bfK_0(\gp_v))q_v^{s-1}Z(s,\phi_{0,v},\phi_{0,v}).$ Thus, we have
	\begin{align*}
		&Z(s,\pi_v(\left[\begin{smallmatrix}
			\varpi_v^{-1}&0\\0&1
		\end{smallmatrix}\right])\phi_{0,v}, \pi_v(\left[\begin{smallmatrix}
			\varpi_v^{-1}&0\\0&1
		\end{smallmatrix}\right])\phi_{0,v}) = (1+q_v)^{-1}(q_v^{1-s}+q_{v}^s)Z(s,\phi_{0,v},\phi_{0,v}).
	\end{align*}
	Similarly, by the relation $\phi(\left[\begin{smallmatrix}
	t&0\\0&1
	\end{smallmatrix}\right])=\overline{\phi(\left[\begin{smallmatrix}
		t&0\\0&1
		\end{smallmatrix}\right])}$ which follows from the unitarity of $\pi_v$ and by \eqref{nonarchLocWhitt}, the integral $Z(s,\pi_v(\left[\begin{smallmatrix}\varpi_v^{-1}&0\\0&1
	\end{smallmatrix}\right])\phi_{0,v}, \phi_{0,v})$ is computed as 
	$$\vol(\bfK_0(\gp_v))(q_v^{1-s}+q_v)\times 
	q_v^{-1/2}(\a_v+\a_v^{-1})q_v^{d_v(s-3/2)}L(s,\pi_v;\Ad).
	$$
	Hence, we obtain
	$$
	Z(s,\pi_v(\begin{smallmatrix}
	\varpi_v^{-1}&0\\0&1
	\end{smallmatrix})\phi_{0,v}, \phi_{0,v}) = (1+q_v)^{-1}q_v^{1/2}(\a_v+\a_v^{-1})Z(s, \phi_{0,v},\phi_{0,v}).
	$$
	In a similar fashion, the equality $Z(s,\phi_{0,v}, \pi_v(\left[\begin{smallmatrix}
	\varpi_v^{-1}&0\\0&1
	\end{smallmatrix}\right])\phi_{0,v}) = Z(s,\pi_v(\left[\begin{smallmatrix}
	\varpi_v^{-1}&0\\0&1
	\end{smallmatrix}\right])\phi_{0,v}, \phi_{0,v})$ holds. All in all, we have the formula \eqref{RINF-f2}.
	
	For the last formula (i.e., $c(\pi_v)=1$), we refer to the proof of \cite[Lemma 2.14]{Tsuzuki2015}.
\end{proof} 
By using Lemma \ref{Rankin integral of new form} with \cite[Lemmas 2.4 and 2.14]{Tsuzuki2015} and \cite[Lemma 6.4]{SugiyamaTsuzuki} to compute the summands of \eqref{AverageInnprod}, we obtain 
{\allowdisplaybreaks 
	\begin{align*}& 
		\PP_{E^*(z)}(l,\gn;\pi)= \frac{D_{F}^{z-1/2}}{2\nr(\gf_{\pi})^{1/2-z/2}}\{\prod_{v \in S(\gn\ff_\pi^{-1})}
		\biggl(1+\frac{Q(I_v(|\,|_v^{z/2}))-Q(\pi_v)^2}{1-Q(\pi_v)^2}\biggr)\}
		\,\frac{\zeta_F\left(\tfrac{z+1}{2}\right)L\left(\tfrac{z+1}{2}, \pi; \Ad\right)}{L(1,\pi;\Ad)}.
\end{align*}}

\section{An overview of the geometric side} \label{overview of geom}

\subsection{Basic assumption} \label{BasicAssumption}
From this section on, until otherwise stated, we let $\fn$ be a non-zero ideal of $\cO$, $l=(l_v)_{v\in \Sigma_\infty}$ an element of $(2\N)^{\Sigma_\infty}$, and $S$ a finite set of $\Sigma_\fin$. We keep the following assumptions on $(\fn,S,l)$:
\begin{itemize}
	\item[(i)] The ideal $\fn\subset \cO$ is square-free.
	\item[(ii)] The two sets $S$, $S(\fn)$ are mutually disjoint.
	\item[(iii)] The weight $l=(l_v)_{v\in \Sigma_\infty}$ is large in the sense that ${\underline l}:=\min_{v\in \Sigma_\infty}l_v\geq 4$. 
\end{itemize}
Let $\Sigma_{\rm{dyadic}}=\{v\in \Sigma_\fin|\,|2|_v<1\,\}$ be the set of all the dyadic places in $\Sigma_\fin$. After \S~\ref{EFLOInt}, we further suppose 
\begin{itemize}
	\item[(iv)] $S\cup S(\fn)$ is disjoint from $\Sigma_{\rm dyadic}$. 
	\item[(v)] The place $2$ of $\Q$ splits completely in $F/\Q$. 
\end{itemize}
In this paper, $l$ and $S$ are fixed. Until \S \ref{MVT}, the ideal $\fn$ is also fixed. To simplify notation, $\Phi^{l}(\fn|\bfs,g)$ and ${\bf \Phi}^{l}(\fn|\bfs,g,h)$ are abbreviated to $\Phi(\bfs;g)$ and ${\bf \Phi}(\bfs;g,h)$, respectively. 

\subsection{An overview}
From the classification of conjugacy classes of $Z_F\bsl G_F$, we have
\begin{align}
	{\bf \Phi}(\bfs;g,g)=J_{\rm{id}}(\bfs; g)+J_{\rm{ell}}(\bfs; g)+J_{\rm{hyp}}(\bfs; g)+J_{\rm{unip}}(\bfs; g),
	\label{kernelftndec}
\end{align}
where the four terms on the right-hand side are defined as follows (\cite{GelbartJacquet}, \cite{JacquetZagier}). 
\begin{align}
	J_{\rm{id}}(\bfs; g)&={\Phi}(\bfs;1_2) \quad (\text{identity term}),\\
	J_{\rm{unip}}(\bfs; g)&=\sum_{\xi \in Z_F N_F\bsl G_F} {\Phi}(\bfs;g^{-1}\xi^{-1} \left[\begin{smallmatrix} 1 & 1 \\ 0 & 1 \end{smallmatrix}\right]\xi g) \quad (\text{the unipotent term}), 
	\label{Junip-g}
	\\
	J_{\rm{hyp}}(\bfs; g)&=\frac{1}{2}\sum_{\xi\in H_F\bsl G_F}\sum_{a\in F^{\times}-\{1\}}{\Phi}(\bfs;g^{-1}\xi^{-1} \left[\begin{smallmatrix} a & 0 \\ 0 & 1 \end{smallmatrix}\right]\xi g) \quad (\text{the $F$-hyperbolic term}), 
	\label{Jhyp-g}
	\\
	J_{\rm{ell}}(\bfs; g)&=\frac{1}{2}\sum_{E\in \fE}\sum_{\xi \in T_E \bsl G_{F}}\sum_{\gamma \in Z_F\bsl (T_E-Z_F)
	}{\Phi}(\bfs;g^{-1}\xi^{-1}\gamma \xi g) \quad (\text{the $F$-elliptic term}),
	\label{Jell-g}
\end{align}
where $\fE$ is the set of all quadratic division $F$-subalgebra $E \subset\Mat_2(F)$ and $T_E=E^\times$ viewed as an $F$-elliptic torus of $\GL(2)$. Then, from Lemma~\ref{CONVERGE} and Proposition~\ref{SmoothEis}, for each type of conjugacy classes $\natural \in \{\rm id, unip, hyp, ell\}$, we know that the integral 
\begin{align}
	\JJ_{\natural}(\bfs, \b)=\int_{Z_\A G_F\bsl G_\A}{\Ecal}_\b^{*}(g)\,J_{\natural}(\bfs; g)\,\d g
	\label{JJnatural-b}
\end{align}
is absolutely convergent for $\b\in \cB_1$. In the succeeding three sections, by computing this integral quite explicitly, we shall show the following:

\begin{thm}\label{hatJ*}
	For any $\bfs\in \fX_S$ lying on some domain $\min_{v\in S}\Re(s_v)>\s_0$, there exists a meromorphic function $\hat J_{\natural}(\bfs; z)$ in $z$ on a vertical strip $\Re(z) \in (\s_1,\s_2)$ containing $0<\Re(z)<1$, which is holomorphic away from $z=\pm 1,0$, is vertically of moderate growth, and satisfies  
	\begin{align}
		\JJ_\natural(\bfs, \b)=\int_{L_\s}\hat J_{\natural}(\bfs, z)\,\b(z)\,\d z, \quad \b\in \cB_1, \quad \s\in (\s_1,\s_2).
		\label{hatJ*-f1}
	\end{align}
\end{thm}

Fix $\bfs\in \fX_S$ with $\min_{v\in S}\Re(s_v)>\s_0$. From \eqref{SpectralsideThm} and \eqref{kernelftndec},
$$
\int_{L_\s}\hat I_{\rm{cusp}}(\bfs,z)\b(z)\,\d z=\int_{L_\s}\{\hat J_{\rm id}(\bfs,z)+\hat J_{\rm unip}(\bfs,z)+\hat J_{\rm hyp}(\bfs,z)+\hat J_{\rm ell}(\bfs,z)\}\b(z)\d z
$$
for all $\b \in \cB_1$ and $\s \in (\s_1,\s_2)$. By applying Lemma~\ref{Fequal0} below, noting \eqref{SpectralsideThm} and \eqref{hatJ*}, we obtain the trace formula in the form $\hat I_{\rm cusp}(\bfs,z)=\sum_{\natural}\hat J_\natural(\bfs,z)$ as a meromorphic function on $\Re(z) \in (\s_1,\s_2)$.

\begin{lem} \label{Fequal0}
	Let $F(z)$ be a meromorphic function on a strip $\Re(z)\in (\s_1,\s_2)$, which is vertically of moderate growth and is holomorphic away from possible simple poles at $z=0,\pm 1$. Assume that there exists $\s \in (\s_1,\s_2)$, $\s\not=0,\pm 1$ such that $\int_{L_\s}F(z)\b(z)dz=0$ for all $\b \in \Bcal_1$. Then we have $F(z)=0$ identically on the strip $\Re(z)\in (\s_1,\s_2)$.
\end{lem}
\begin{proof}  Set $F_1(z)=z(z^2-1)^2F(z)$. The function $A(t)=F_1(\s+it)e^{\s^2+2i\s t-t^2/2}$ belongs to $L^2(\R)$ due to the assumption that $F(z)$ is vertically of moderate growth and to the presence of the factor $e^{-t^2/2}$. For $P(z)\in \C[z]$, by assumption, we see
	$$\tint_{-\infty}^{\infty}A(t) P(\s+it)e^{-t^2/2}dt = \tint_{L_\s}F(z)\times z(z^2-1)^2e^{z^2}P(z)\,dz=0.$$
	Thus $A=0$ in $L^2(\R)$. Since $A$ is continuous, we have the pointwise equality $F_1(\s+it)=0$ for all $t\in \R$. By the holomorphicity, $F_1(z)=0$ identically. \end{proof}

\section{The singular terms}\label{singular terms}
\subsection{The identity term} 
We shall see $\JJ_{\rm{id}}(\bfs, \b)=0$. Let $\s>1$. Since $\sum_{\gamma\in B_F\bsl G_F}y(\gamma g)^{\frac{\s+1}{2}}$ is convergent, 
{\allowdisplaybreaks\begin{align*}
		\tint_{Z_\A G_F\bsl G_\A} \cE_\b^{*}(g)\,\d g
		&=\tint_{Z_\A G_F\bsl G_\A} \sum_{\gamma \in B_F\bsl G_F} \left( 
		\tint_{L_\s} y(\gamma g)^{\frac{z+1}{2}}\,\Lambda_F(z+1)\,\b(z)\,\d z\right)\,\d g
		\\
		&=\tint_{Z_\A B_F\bsl G_\A} \left(\int_{L_\s} y(g)^{\frac{z+1}{2}}\,\Lambda_F(z+1)\,\b(z)\,\d z\right)\,\d g
		=S_0+S_\infty,
\end{align*}}where $S_0$ and $S_\infty$ denote the integrals over the subdomains $y(g)<1$ and $y(g)\geq 1$, respectively. By writing $g\in Z_\A\bsl G_\A$ as
$\left[\begin{smallmatrix} 1 & x \\ 0 & 1 \end{smallmatrix} \right]\left[\begin{smallmatrix} \underline{t}\,u & 0 \\ 0 & 1 \end{smallmatrix}\right] k$ with $x\in \A,\,t>0,\,u\in \A^1, \,k \in \bK$, and by changing the order of integrals, we have
\begin{align*}
	S_0&=\tint_{L_\s} \left(\int_{F\bsl \A}\d x\tint_{F^\times \bsl \A^1}\d^{1}u \tint_{0}^{1} t^{\frac{z-1}{2}}\d^\times t\right)\Lambda_F(z+1)\b(z)\,\d z
	=\vol(F^\times\bsl \A^1)\,\tint_{L_\s}\tfrac{2\Lambda_F(z+1)\b(z)}{z-1}\,\d z.
\end{align*}
As for $S_\infty$, shifting the contour $L_\s$ to $L_{-\s}$ and then in the same way as above, we have
\begin{align*}
	S_\infty&=\tint_{\substack{Z_\A B_F \bsl G_\A \\ y(g)>1}} 
	\left(\tint_{L_{-\s}}y(g)^{\frac{z-1}{2}}\,\Lambda_F(z+1)\b(z)\,\d z\right)\,\d g=\vol(F^\times \bsl \A^1)\tint_{L_{-\s}}\tfrac{-2\Lambda_F(z+1)\b(z)}{z-1}\,\d z.\end{align*}
Thus, by the residue theorem, we see that $\vol(F^\times \bsl \A^1)^{-1}(S_0+S_\infty)$ equals ${\rm{Res}}_{z=1}\tfrac{2\Lambda_F(z+1)\b(z)}{z-1}=2D_F^{1/2}\zeta_F(2)\b(1)=0$. By $\JJ_{\rm{id}}(\bfs,\b)=\Phi(\bfs;1_2)(S_0+S_\infty)$, we obtain $\JJ_{\rm{id}}(\bfs,\b)=0$. If we set $\hat J_{\rm{id}}(\bfs, z)=0$, then this shows that Theorem~\ref{hatJ*} is valid for the identity term. 

\subsection{The unipotent term}
By noting $\vol(N_F \bsl N_\AA)=1$, from \eqref{Junip-g} and \eqref{JJnatural-b}, we have
\allowdisplaybreaks{\begin{align*}
		\JJ_{\rm unip}(\bfs, \b) = \int_{Z_\AA N_\AA \bsl G_\AA}\Phi(\bfs;g^{-1} [\begin{smallmatrix}
			1 & 1 \\ 0 & 1
		\end{smallmatrix}] g)\Ecal_{\b, \circ}^{*}(g)dg.
	\end{align*}
}Here $\Ecal_{\b, \circ}^{*}(g)=\int_{N_F\bsl N_\AA}\Ecal_{\b}^{*}([\begin{smallmatrix}
1 & x \\ 0 & 1
\end{smallmatrix}]g)dx$ is the constant term which is computed as 
$$\Ecal_{\b,\circ}^{*}(g)=\int_{L_\s}\b(z)\{\Lambda_{F}(-z)y(g)^{(1+z)/2}+\Lambda_F(z)y(g)^{(1-z)/2}\}dz.$$
By substituting this to the formula of $\JJ_{\rm{unip}}(\bfs, \b)$ and by exchanging the order of integrals formally, we encounter the integral 
\begin{align}
	U(\bfs;w) = \int_{Z_\AA N_\AA \bsl G_\AA}\Phi(\bfs; g^{-1} [\begin{smallmatrix}
		1 & 1 \\ 0 & 1
	\end{smallmatrix}] g)y(g)^{w}dg.
	\label{UlSgn}
\end{align}
In order to analyze this, we consider the local integrals
$$U_v(w)= \int_{Z_vN_v\bsl G_v}\Phi_v(g^{-1} [\begin{smallmatrix}
1 & 1 \\ 0 & 1
\end{smallmatrix}] g)y(g)^{w}dg=\int_{F_v^{\times}}\int_{\bfK_v}\Phi_v(k^{-1}[\begin{smallmatrix}
1 & t \\ 0 & 1
\end{smallmatrix}]k)|t|_v^{1-w}d^{\times}tdk$$
for any $v \in \Sigma_F$, where $\Phi_v(g_v)$ denotes the $v$-th factor of $\Phi(\bfs;g)$.

\begin{lem} \label{UnipLocal-L1}
	\begin{itemize}
		\item[(i)] For any $v \in \Sigma_{\infty}$ and $w \in \CC$ such that $1-l_v<\Re(w)<1$,
		$$U_v(w)= \Gamma_\RR(1-w)\times2^{2-2w}\pi^{1-w/2}
		{\Gamma(l_v+w-1)}{\Gamma(w/2)^{-1}\Gamma(l_v)^{-1}}.$$
		\item[(ii)] For any $v \in S$ and $(s, w)\in \CC^2$ such that $\Re(s)>1$ and $-\Re(s)<\Re(w)<1$, we have
		$$U_v(w) = q_v^{-d_v/2}\times(-q_v^{-(s+1)/2})(1- q_v^{w-1})^{-1}(1-q_v^{-s-w})^{-1}.$$
		\item[(iii)] For any $v \in S(\gn)$ and $w \in \CC$ such that $\Re(w)<1$, we have
		$$U_v(w)=q_v^{-d_v/2}(1+q_v^{w})(1+q_v)^{-1}(1-q_v^{w-1})^{-1}.$$
		\item[(iv)] For any $v \in \Sigma_{\fin}-(S\cup S(\gn))$ and $w\in\CC$ such that $\Re(w)<1$, we have
		$$U_v(w) = {q_v^{-d_v/2}}({1-q_v^{w-1}})^{-1}.$$
	\end{itemize}
\end{lem}
\begin{proof} 
	(i) By \eqref{DSMexplicitformula} and by using \cite[3.194, 3]{Gradshteyn}, we obtain
	\begin{align*}
		U_v(w) = & \tint_{F_v^{\times}}\left({1-it/2}\right)^{-l_v}|t|^{-w}dt
		= \tint_{0}^{\infty}\left({1-it/2}\right)^{-l_v} t^{-w}dt +\int_{0}^{\infty}\left({1+it/2}\right)^{-l_v} t^{-w}dt \\
		= & (-i/2)^{w-1}B(-w+1,l_v+w-1) + (i/2)^{w-1}B(-w+1,l_v+w-1) \\
		= & 2^{2-w}\sin\left(\tfrac{\pi}{2}w\right)B(-w+1,l_v+w-1),
	\end{align*}
	where $B(x,y)=\frac{\Gamma(x)\Gamma(y)}{\Gamma(x+y)}$ is the beta function. By using the formulas $\sin(\pi w/2)=\pi\Gamma(w/2)^{-1}\Gamma(1-w/2)^{-1}$ and $\Gamma(1-w)=2^{-w}\pi^{-1/2}\Gamma(1/2-w/2)\Gamma(1-w/2)$, we are done. 
	
	(ii) From \eqref{nonarchGreenftn}, $U_v(w) = ({q_v^{-(s+1)/2}-q_v^{(s+1)/2}})^{-1} \tint_{F_v^{\times}}\max(1,|t|_v)^{-(s+1)}|t|_v^{1-w} d^{\times}t$. We compute the integral by dividing the integral domain to $|t|_v>1$ and $|t|_v\leq 1$ to obtain the desired formula.
	
	(iii) Let $t\in F_v^{\times}$ and $k=[\begin{smallmatrix}
	k_{11}&k_{12}\\k_{21}&k_{22}
	\end{smallmatrix}] \in \bfK_v$.
	Then, $k^{-1} [\begin{smallmatrix}
	1 & t \\ 0 & 1
	\end{smallmatrix}] k \in Z_v\bfK_0(\gp_v)$ if and only if $t \in \go_v^{\times}, k_{21}\in \gp_v$ or $t \in \gp_{v}-\{0\}, k_{21} \in \go_v$.
	Thus, $U_v(w)$ is computed as the sum of the integral $\int_{t \in \go_v^{\times}}\int_{\bfK_{0}(\gp_v)} \d k\, |t|_v^{1-w}d^{\times}t=q_v^{-d_v/2}\vol(\bfK_0(\gp_v))$ and the integral $\int_{t \in \gp_v-\{0\}}\int_{\bfK_v} |t|_v^{1-w} d^{\times}tdk=q_v^{-d_v/2}\frac{q_v^{w-1}}{1-q_v^{w-1}}$. 
	
	(iv) Since $U_v(w)=\int_{\go_v-\{0\}}|t|_v^{1-w}d^{\times}t$, we have the formula immediately.
\end{proof}
Suppose that $\min_{v\in S}\Re(s_v)>1$, $-\min_{v\in S}\Re(s_v)<\Re(w)<0$ and 
$1-{\underline l}<\Re(w)<0$ are satisfied. Then from Lemma~\ref{UnipLocal-L1} we see that the integral \eqref{UlSgn} converges absolutely and have the formula
\begin{align*}
	U(\bfs;w)= & D_F^{-1/2}\zeta_F(1-w)
	\prod_{v \in \Sigma_{\infty}}2^{2-2w}\pi^{1-w/2}\frac{\Gamma(l_v+w-1)}{\Gamma(w/2)\Gamma(l_v)}\,
	\prod_{v \in S}
	\frac{-q_v^{-(s_v+1)/2}}{1-q_v^{-s_v-w}}
	\prod_{v \in S(\gn)}\frac{1+q_v^{w}}{1+q_v},
\end{align*}
which gives a meromorphic continuation of $w\mapsto U(\bfs;w)$ to $\CC$ for a fixed $\bfs$. By changing the order of integrals, we obtain
\begin{align*}
	\JJ_{\rm unip}(\bfs, \b) = & \int_{Z_\AA N_\AA\bsl G_\AA} \Phi^l(\bfs|\gn ;g^{-1}[\begin{smallmatrix}
		1 & 1 \\ 0 & 1
	\end{smallmatrix}]g) \Ecal_{\b,\circ}^{*}(g)dg \\
	=& \int_{L_{-\s}}\b(z)\Lambda_F(-z)U_{l,S,\gn}(\bfs;(1+z)/2)dz
	+\int_{L_{\s}}\b(z)\Lambda_F(z)U_{l,S,\gn}(\bfs;(1-z)/2)dz \\
	=& \int_{L_\s}\b(z)\hat J_{\rm unip}(\bfs, z)dz,
\end{align*}
where we set 
$$\hat J_{\rm unip}(\bfs,z)=D_{F}^{\frac{z}{4}}\,\zeta_{F}\left(\tfrac{1+z}{2}\right)\{\hat J_{\rm unip}^{0}(\bfs,z)+\hat J_{\rm unip}^{0}(\bfs,-z)\}
$$ 
with 
\begin{align}
	\hat J_{\rm unip}^{0}(\bfs, z)=D_{F}^{\frac{z-2}{4}}\Lambda_F(-z) 
	\prod_{v \in \Sigma_{\infty}}2^{1-z}\pi^{\frac{3-z}{4}}
	\frac{\Gamma\left(l_v+\frac{z-1}{2}\right)}{\Gamma\left(\tfrac{z+1}{4}\right)\Gamma(l_v)}\,
	\prod_{v \in S}
	\frac{-q_v^{-\frac{s_v+1}{2}}}{1-q_v^{-s_v-\frac{z+1}{2}}}
	\prod_{v \in S(\gn)}\frac{1+q_v^{\frac{z+1}{2}}}{1+q_v}.
	\label{Unipotent0Term}
\end{align}
Then we obtain Theorem~\ref{hatJ*} for the unipotent term with $\s_0=-\s_1=\s_2=2{\underline l}-1$.

\section{The $F$-hyperbolic term} \label{The $F$-hyperbolic term}

In this section, we study $\JJ_{\rm hyp}(\bfs, \b)$ to show Theorem~\ref{hatJ*} for the $F$-hyperbolic term.  
\subsection{Spherical functions}
First, we recall the explicit formula of $\bK_v$-invariant spherical functions on $H_v\bsl G_v$. Let $(w,z)\in \C^{2}$ and $v\in \Sigma_F$. By the Iwasawa decomposition $G_v=H_v N_v\bK_v$, we have a well-defined smooth function $\varphi_v^{(w,z)}:G_v\rightarrow \C$ such that 
\begin{align}
	&\varphi_v^{(w,z)}\left(\left[\begin{smallmatrix} t_1 & 0 \\ 0 & t_2\end{smallmatrix}\right]\,\left[\begin{smallmatrix} 1 & x \\ 0 & 1 \end{smallmatrix}\right]\,k\right)=|t_1/t_2|_v^{-w}\,\{A_v(w,z) \,h_v^{(w,z)}(x) +A_v(w,-z)\,h_v^{(w,-z)}(x)\}, \label{Hsphericalfnt} \\
	& \qquad (t_1,\,t_2\in F_v^{\times},\,x\in F_v,\,k\in \bK_v),\notag
\end{align}
where 
\begin{align*}
	h_v^{(w,z)}(x)&=
	\begin{cases} \max(1,|x|_v)^{-(z+2w+1)/2}, \quad &(v\in \Sigma_\fin), \\
		(1+x^2)^{-(z+2w+1)/4}{}_2F_1\left(\tfrac{z+2w+1}{4},\tfrac{z-2w+1}{4};\tfrac{z+2}{2};\tfrac{1}{x^2+1}\right), \quad &(v\in \Sigma_\infty),
	\end{cases}\\
	A_v(w,z)&={\zeta_{F_v}(1)\,\zeta_{F_v}(-z)}{\zeta_{F_v}((-z+2w+1)/2)^{-1}\,\zeta_{F_v}((-z-2w+1)/2)^{-1}}.
\end{align*}
The simple poles of the factor $\zeta_{F_v}(-z)$ of $A_v(w,z)$ may cause singularities of $\varphi_v^{(w,z)}$, but they are removable due to the obvious functional equation $\varphi_v^{(w,z)}(g)=\varphi_{v}^{(w,-z)}(g)$. For any $v \in \Sigma_\fin$, we collect several easily proved formulas for later use: 
\begin{align}
	&A_v(0,z)+A_v(0,-z)=1, \quad
	{\zeta_{F_v}(-z)}{\zeta_{F_v}\left(\tfrac{-z+1}{2}\right)^{-1}}
	+{\zeta_{F_v}(z)}{\zeta_{F_v}\left(\tfrac{z+1}{2}\right)^{-1}}=1, 
	\label{FFz1}
	\\
	&A_v(0,z)\,q_v^{\frac{z}{2}}{\zeta_{F_v}\left(\tfrac{1-z}{2}\right)}{\zeta_{F_v}\left(\tfrac{1+z}{2}\right)^{-1}}+A_v(0,-z)
	q_v^{\frac{-z}{2}}{\zeta_{F_v}\left(\tfrac{1+z}{2}\right)}{\zeta_{F_v}\left(\tfrac{1-z}{2}\right)^{-1}}
	=0. 
	\label{FAQ1}
\end{align}

\begin{lem} \label{FhypL1}
	The function $\varphi_v^{(w,z)}$ is the unique complex-valued smooth function on $G_v$ such that $\varphi_v^{(w,z)}(1_2)=1$ which satisfies the following conditions:  
	{\allowdisplaybreaks
		\begin{align}
			\varphi_v^{(w,z)}\left(\left[\begin{smallmatrix} t_1 & 0 \\ 0 & t_2 \end{smallmatrix}\right]gk\right)&=|t_1/t_2|_v^{-w}\varphi_v^{(w,z)} (g), \quad t_1,\,t_2 \in F_v^{\times},\,g\in G_v,\,k\in \bK_v 
			\label{Fhyp-1}, \\
			R(\bT_v)\varphi^{(w,z)}_v&=q_v^{1/2}(q_v^{z/2}+q_v^{-z/2})\,\varphi_v^{(w,z)} \quad\text{if $v\in \Sigma_\fin$}, \label{Fhyp-4}\\
			R(\Omega_v)\varphi^{(w,z)}_v& =2^{-1}(z^2-1)\,\varphi_v^{(w,z)}\quad \text{if $v\in \Sigma_\infty$}, \label{Fhyp-2} 
	\end{align}}
	where $\Omega_v$ for $v \in \Sigma_{\infty}$ is the Casimir operator of $G_v$.
\end{lem}
\begin{proof}
	Let $v\in \Sigma_\fin$. In the same way as \cite[Lemma 5.2]{Tsuzuki2015}, for $a(m)=\varphi_v^{(w,z)}\left(\left[\begin{smallmatrix} 1 & \varpi^{-m}_v \\ 0 & 1 \end{smallmatrix} \right]\right)$ $(m\in \N_0$), we deduce a recurrence relation{\allowdisplaybreaks 
		\begin{align*}
			&q_v^{1+w}a(m+1)+q_v^{-w}a(m-1)=q_v^{1/2}(q_v^{z/2}+q_v^{-z/2})a(m), \quad (m>0),\\
			&(q_v-1)q_v^{w}a(1)+(q_v^{w}+q_v^{-w})a(0)=q_v^{1/2}(q_v^{z/2}+q_v^{-z/2})a(0),
	\end{align*}}and $a(0)=1$ from \eqref{Fhyp-4} and $\varphi^{(w,z)}(1_2)=1$, which can be solved uniquely by $a(m)=A_v(w,z) \,h_v^{(w,z)}(\varpi_v^{-m}) +A_v(w,-z)\,h_v^{(w,-z)}(\varpi_v^{-m})$. By the Iwasawa decomposition and \eqref{Fhyp-1}, we are done. Let $v\in \Sigma_\infty$. From \cite[Proposition 4.3]{Hirano} for $f(r)=\varphi_v^{(w,z)}\left(\left[\begin{smallmatrix} \ch r & \sh r \\ \sh r & \ch r \end{smallmatrix} \right]\right)$, we have the differential equation
	$$
	\left(\tfrac{1}{2}\tfrac{\d^2}{\d r^2}+\tfrac{\sh 2r}{\ch 2r}\tfrac{\d}{\d r}+\tfrac{2w^2}{\ch^2 2r}\right)f(r)=\tfrac{z^2-1}{2}f(r)
	$$ 
	and $f(0)=1$ from \eqref{Fhyp-2} and $\varphi^{(w,z)}(1_2)=1$, which has the unique $C^\infty$ solution 
	$$f(r)=(\ch ^2 2r)^{-(z+1)/4}{}_2F_1\left(\tfrac{z+2w+1}{4}, \tfrac{z-2w+1}{4};\tfrac{1}{2};\tfrac{\sh^2 2r}{\ch^2 2r}\right), \quad r\in \R. 
	$$
	By the Gauss connection formula (on the last line of \cite[p.47]{MOS}), it turns out that $f(r)=(1+x^2)^{w/2}\{A_v(w,z)\,h_v^{(w,z)}(x)+A_v(w,-z)\,h_v^{(w,-z)}(x)\}$ with $1+x^2=\ch^2 2r$. By \cite[Lemma 3.1]{Tsuzuki2015}, we are done.  
\end{proof}

\subsection{Unfolding and contour shifting} \label{UnConSh}
From Lemma \ref{CONVERGE} and Proposition \ref{SmoothEis}, 
\begin{align*}
	&\int_{Z_\A G_F\bsl G_\A} \d g\,\sum_{\xi\in H_F \bsl G_F}\sum_{a\in F^\times-\{1\}} \left|\Phi\left(\bfs;g^{-1}\xi^{-1}\left[\begin{smallmatrix} a & 0 \\ 0 & 1 \end{smallmatrix} \right]\xi g\right)\,\cE^{*}_{\b}(g)\right| \\
	\leq &\int_{Z_\A G_F\bsl G_\A} \d g\,\sum_{\gamma \in Z_F\bsl G_F} \left|\Phi\left(\bfs; g^{-1}\gamma g\right)\,\cE^{*}_{\b}(g)\right|<+\infty,
\end{align*}
which justifies the applications of Fubini's theorem in the following computation:
{\allowdisplaybreaks
	\begin{align}
		\JJ_{{\rm{hyp}}}(\bfs, \b)&=\tfrac{1}{2}\int_{Z_\A H_F\bsl G_\A} \d g\,\sum_{a\in F^\times-\{1\}}\Phi\left(\bfs;g^{-1}\left[\begin{smallmatrix} a & 0 \\ 0 & 1 \end{smallmatrix} \right] g\right)\,\cE^{*}_{\b}(g)
		\notag
		\\
		&=\tfrac{1}{2}\sum_{a\in F^\times-\{1\}} \int_{\A}\d x \int_{\bK} \d k\,
		\Phi\left(\bfs;k^{-1}\left[\begin{smallmatrix} a & (a-1)x \\ 0 & 1 \end{smallmatrix} \right] k\right)\, 
		P_\b\left(0;\left[\begin{smallmatrix} 1 & x \\ 0 & 1 \end{smallmatrix} \right]\right),
		\label{Firststep-1}
	\end{align}
}where for $w\in \C$ and $g\in G_\A$, we set 
\begin{align}
	P_\b(w;g)=\int_{F^\times\bsl \A^\times}\cE_\b^{*}
	\left(\left[\begin{smallmatrix} t & 0 \\ 0 & 1 \end{smallmatrix} \right]g
	\right)\,|t|_\AA^{w}\,\d^\times t, 
	\label{RegPerEis}
\end{align}
which is shown to be absolutely convergent by Proposition~\ref{SmoothEis} for all $w\in \C$. By \eqref{EisFexp}, the value $\cE_\b^{*}\left(\left[\begin{smallmatrix} t & 0 \\ 0 & 1 \end{smallmatrix}\right] \left[\begin{smallmatrix} 1 & x \\ 0 & 1 \end{smallmatrix}\right] \right)$ $(t\in \A^\times,\,x\in \A)$ is expressed as the sum of the following three contour integrals: 
\begin{align}
	C_{\b}^{\pm}(t)&=\tint_{L_{\s}}\b(z)\Lambda_F(\mp z)|t|_\A^{(\pm z+1)/2}\,\d z,  \label{SmEisConst}
	\\ 
	\cW_{\b}(t,x)&=\tint_{L_{\s}}\b(z)
	\Lambda_F(-z)\,\sum_{a\in F^\times} 
	W_{\psi}\left(z; \left[\begin{smallmatrix} at & 0 \\ 0 & 1 \end{smallmatrix} \right]\left[\begin{smallmatrix} 1 & x \\ 0 & 1 \end{smallmatrix}\right] \right)\,\d z, 
	\label{SmEisNC}
\end{align}
where $1<\s$. For any $N_1>0$, by shifting the contour $L_\s$ to $L_{\s_1}$ ($\s_1=-2N_1-1$), we have the inequality
\begin{align*}
	|C_{\b}^{+}(t)|\leq \tint_{L_{\s_1}}|\b(z)\Lambda_F(-z)|\,|\d z| \times |t|_\A^{-N_1}, 
\end{align*} 
which yields the bound $C_{\b}^{+}(t)\ll_{N_1} |t|_\A^{-N_1}$ for $t\in \A^\times$. Combining this with the bound $C_{\b}^{+}(t)\ll_{\e} |t|_\A^{N_1}$, which is immediate from \eqref{SmEisConst} with $\s=2N_1-1$, we have 
$$
C_{\b}^{+}(t)\ll_{N_1,\e}\min(|t|_\A^{-N_1},|t|_\A^{N_1}), \quad t\in \A^\times.  
$$
From this, the integral 
$$
\fC^{+}_\b(w):=\tint_{F^\times\bsl  \A^\times} C_{\b}^{+}(t)\,|t|_\A^{w}\,\d^\times t =\vol(F^\times\bsl \A^1)\,\tint_{0}^{\infty} C_{\b}^{+}(\underline{t})\,t^{w}\,\d^\times t $$
is seen to be absolutely convergent for all $w\in \C$. In the same way as \cite[Lemma 7.6]{Tsuzuki2015}, using the residue theorem for $\s, \s_1 \in \RR$ such that $\s_1<-2\Re(w)-1<\s$,
we have
{\allowdisplaybreaks
	\begin{align*}
		\tint_{0}^{\infty} C_{\b}^{+}(\underline{t})\,t^{w}\,\d^\times t
		&=\tint_{0}^{1} \d^\times t\, \tint_{L_{\s}}\b(z)\Lambda_F(-z)\,t^{w+(z+1)/2}\,\d z
		+\tint_{1}^{\infty} \d^\times t\, \tint_{L_{\s_1}}\b(z)\Lambda_F(-z)t^{w+(z+1)/2}\,\d z
		\notag
		\\
		&=\tint_{L_{\s}}\b(z)\tfrac{\Lambda_F(-z)\d z}{w+(z+1)/2}+\tint_{L_{\s_1}}
		\b(z)\tfrac{-\Lambda_F(-z)\d z}{w+(z+1)/2} 	\\
		& =2\pi i {\rm{Res}}_{z=-2w-1}\left(\b(z)\Lambda_F(-z)\tfrac{1}{w+(z+1)/2}\right) \\
		& = 4\pi i \b(-2w-1)\Lambda_F(1+2w). 
		\notag
\end{align*}
}In a similar manner as above, we have the estimation
$$
C_\b^{-}(t)\ll_{N_1, \e} \min(|t|_\A^{-N_1},|t|_\A^{N_1}), \quad t\in \A^\times
$$
for any $N_1>0$. When we shift the contour to the left to obtain the majorant $|t|_\A^{N_1}$, we note that the singularity at $z=0,1$ of $\Lambda_F(z)$ is canceled with the zeros of $\beta(z)$. Hence the integral 
$$
\fC^{-}_\b(w):=\tint_{F^\times\bsl  \A^\times} C_{\b}^{-}(t)\,|t|_\A^{w}\,\d^\times t
$$
is absolutely convergent for all $w\in \C$, and is evaluated
as $-4\pi i\,\vol(F^\times\bsl\AA^1) \b(2w+1) \Lambda_F(1+2w).$
Consequently, $\fC_{\b}^{-}(w)+\fC_{\b}^{+}(w) = 4\pi i\, \vol(F^\times\bsl\AA^1)(\b(-2w-1) -\b(2w+1))\Lambda_F(1+2w)$. The absolute convergence of the integral
\begin{align*}
	\fW_\b(w;x)=\tint_{F^\times \bsl \A^\times} \cW_\b(t,x)\,|t|_\A^{w}\,\d^\times t\end{align*}
for $\Re(w)>(\Re(z)+1)/2$ is confirmed by the inequality 
\begin{align*}
	\tint_{F^\times \bsl \A^\times} |\cW_\b(t,x)|\,|t|_\A^{\Re(w)}\,\d^\times t
	\leq \tint_{L_\s} \left|\b(z)\Lambda_F(-z)\right|\ \left( \int_{\A^\times}|W_\psi \left(z; \left[\begin{smallmatrix} t & 0 \\ 0 & 1 \end{smallmatrix} \right]\left[\begin{smallmatrix} 1 & x \\ 0 & 1 \end{smallmatrix}\right] \right)|\,|t|_\A^{\Re(w)}\d^\times t\right)\,|\d z| 
\end{align*}
combined with the bound 
\begin{align*}
	\tint_{\A^\times}|W_\psi\left(z; \left[\begin{smallmatrix} t & 0 \\ 0 & 1 \end{smallmatrix} \right]\left[\begin{smallmatrix} 1 & x \\ 0 & 1 \end{smallmatrix}\right] \right) ||t|_\A^{\Re(w)}\,\d^\times t=O(1), \quad z\in L_{\s},\,
	\Re(w)>(\s+1)/2, 
\end{align*}
which follows from \eqref{nonarchLocWhitt} and \eqref{archLocWhitt}. By Lemma~\ref{FhypL1}, on the region $\Re(w)>(|\Re(z)|+1)/2$, for all $g=(g_v)\in G_\A$, we have
{\allowdisplaybreaks\begin{align*}
		\int_{\A^\times}W_\psi\left(z; \left[\begin{smallmatrix} t & 0 \\ 0 & 1 \end{smallmatrix}\right] g \right)\,|t|_\A^{w}\,\d^\times t=D_F^{-z/2+w-1/2}\frac{\zeta_F\left(w+\frac{z+1}{2}\right)\zeta_F\left(w+\frac{-z+1}{2}\right)}{\zeta_F(z+1)}\,\prod_{v\in \Sigma_F}\varphi_v^{(w,z)}\left(g_v\right).
\end{align*}}

\begin{lem} \label{FhypL4}
	For all $w \in \CC$ such that $0 \le \Re(w) \le 1$ and $x=(x_v)\in \A$, we have 
	\begin{align*}
		P_\b\left(w;\left[\begin{smallmatrix} 1 & x \\ 0 & 1 \end{smallmatrix}\right] \right)
		= & \int_{L_{\s'}} \b(z)\,D_F^{w}\, \zeta_F\left(w+\tfrac{z+1}{2}\right)\zeta_F\left(w+\tfrac{-z+1}{2}\right)\,\prod_{v\in \Sigma_F}\varphi_v^{(w,z)}\left(\left[\begin{smallmatrix} 1 & x_v \\ 0 & 1 \end{smallmatrix}\right] \right)\,\d z\\
		&+ 4\pi i \,\Res_{s=1}\zeta_F(s)\,\b(2w-1)\zeta_F(2w)\prod_{v \in \Sigma_F}\varphi_v^{(w,2w-1)}(\left(\left[\begin{smallmatrix} 1 & x_v \\ 0 & 1 \end{smallmatrix}\right] \right))\\
		& + 4\pi i\, \vol(F^\times\bsl\AA^1)\Lambda_F(1+2w)(\b(-2w-1) -\b(2w+1)), \quad
		(\s'>1).
	\end{align*}
\end{lem}
\begin{proof}
	Let $F(z,w)$ denote the integrand of the integral on the right-hand side.
	Suppose $\s>1$ and $\Re(w)>\frac{\s+1}{2}$.
	By taking $\s'>1$ so that $\s'>\s$ and $\frac{\s'-1}{2}<\Re(w)<\frac{\s'+1}{2}$,
	Cauchy's integral theorem yields
	\begin{align*}
		\int_{L_\s} F(z,w)dz=\int_{L_{\s'}}F(z,w) \d z -2\pi i \Res_{z=2w-1}F(z,w).
	\end{align*}
	Hence, we have the expression of $P_\b\left(w;\left[\begin{smallmatrix} 1 & x \\ 0 & 1 \end{smallmatrix}\right] \right)$ in the assertion
	for $\frac{\s'-1}{2}<\Re(w)<\frac{\s'+1}{2}$.
	We remark that the first term $\int_{L_{\s'}}F(z,w) \d z$ is holomorphic on $\frac{\s'-1}{2}<\Re(w)<\frac{\s'+1}{2}$
	by $\Re(w+\frac{z+1}{2})>1$ and $0<\Re(z+\frac{-z+1}{2})<1$, and that
	the second and the third terms are entire due to $\b(0)=\b(\pm1)=\b'(\pm 1)=0$.
\end{proof}

By shifting the contour $L_{\s'}$ to $L_\s$ ($-1<\s<1$) after substituting $z=0$,
from Lemma~\ref{FhypL4}, we have the expression
\begin{align}
	P_\b\left(0;\left[\begin{smallmatrix} 1 & x \\ 0 & 1 \end{smallmatrix}\right] \right)
	&=\int_{L_{\s}} \b(z)\,\zeta_F\left(\tfrac{z+1}{2}\right)\zeta_F\left(\tfrac{-z+1}{2}\right)\,\prod_{v\in \Sigma_F}\varphi_v^{(0,z)}\left(\left[\begin{smallmatrix} 1 & x_v \\ 0 & 1 \end{smallmatrix}\right] \right)\,\d z, \quad (-1<\s<1).
	\label{FhypL5}
\end{align}

\subsection{The orbital integrals} \label{HypSecondStep}
By substituting the expression \eqref{FhypL5} for the last formula of \eqref{Firststep-1} and formally changing the order of the integral and the summation, the integral
\begin{align}
	\fF^{(z)}(a)=\int_{\A} \left(\int_{\bK}\Phi\left(\bfs,k^{-1}\left[\begin{smallmatrix} a & (a-1)x \\ 0 & 1 \end{smallmatrix} \right] k\right)\,\d k \right)\,\prod_{v\in \Sigma_F}\varphi_v^{(0,z)} \left(\left[\begin{smallmatrix} 1 & x_v \\ 0 & 1 \end{smallmatrix}\right] \right)\,\d x, \quad a\in F-\{0,1\}
	\label{fFzaAdele}
\end{align}
emerges naturally.
In this subsection, we prove the following proposition, which is sufficient to legitimatize the formal computation explained above.  
\begin{prop} \label{FHypP1}
	Suppose $\min_{v \in S} \Re(s_v)>1$.
	\begin{itemize}
		\item[(1)] Let $a\in F-\{0,1\}$. The integral \eqref{fFzaAdele} converges absolutely for $\Re(z)<2{\underline l}-1$. On that region, we have the product formula $\fF^{(z)}(a)=\prod_{v\in \Sigma_F}\fF_v^{(z)}(a)$ with 
		$$
		\fF_v^{(z)}(a)=\int_{F_v} \left(\int_{\bK_v}\Phi_v\left(k_v^{-1}\left[\begin{smallmatrix} a & (a-1)x_v \\ 0 & 1 \end{smallmatrix} \right] k_v\right)\,\d k_v \right)\,\varphi_v^{(0,z)} \left(\left[\begin{smallmatrix} 1 & x_v \\ 0 & 1 \end{smallmatrix}\right] \right)\,\d x_v,  \quad v\in \Sigma_F,
		$$
		where $\Phi_v(g_v)$ is the $v$-th factor of $\Phi(\bfs;g)$. We have $\fF^{(z)}(a)=0$ unless $a \in F_\infty^{+}$, where $F_\infty^{+}=\{a_\infty\in F_\infty|\,a_v>0\,(v\in \Sigma_\infty)\,\}$. 
		\item[(2)] Set 
		$$
		\s_0=\min(\{2l_v-1|\,v\in \Sigma_\infty\}\cup \{2\Re(s_v)-1|\,v\in S\,\}).
		$$
		For any $\s\in (0,\s_0)$ and $\e\in (0,1)$, there exists a positive number $C_\e$ independent of $z$ and $\bfs$ such that 
		\begin{align*}
			|\fF^{(z)}(a)|\leq C_\e\,\{\prod_{v\in \Sigma_\fin}\phi_{v}(a)\}\,\{\prod_{v\in \Sigma_\infty}(1+|a|_v)^{-\frac{l_v}{2}+\frac{\s+1}{2}}\}
		\end{align*}
		for all $a\in F_\infty^{+}\cap(F^\times-\{1\})$ and $|\Re(z)|\leq \s$, where
		\begin{align*}
			\phi_v(a)=
			\begin{cases}
				\delta(a\in \cO_v^\times), \quad &(v\in \Sigma_\fin-(S\cup S(\fn))), \\ 
				{4}{(1+q_v)^{-1}}\delta(a\in \cO_v^\times)(1+q_v^{1/2})^{\delta(a\in 1+\fp_v)}\,\{{\rm{ord}}_v(a-1)+1\}, \quad &(v\in S(\fn)), \\ 
				\{\max(0,{\rm{ord}}_v(a-1))+1\}\,\max(1,|a-1|_v)^{\frac{|\Re(z)|-\Re(s_v)}{2}+\e}, \quad&(v\in S).
			\end{cases}
		\end{align*}
	\end{itemize}
\end{prop}
For the proof, we use the following explicit formulas of the local integrals $\fF_v^{(z)}(a)$,
which are proved later in \S \ref{Localorbitalintegrals}.

\begin{thm}
	\label{explicit hyp}
	Let $v \in \Sigma_F $ and $a\in F_v^{\times}-\{1\}$.
	\begin{enumerate}
		\item[(1)] Suppose $v \in \Sigma_\infty$. Then,
		on the region $|\Re(z)|<2l_v-1$, we have
		\begin{align*}
			\fF_v^{(z)}(a)=&
			\frac{4\pi}{\Gamma(l_v)} 
			\,
			\frac{\Gamma\left(l_v+\tfrac{z-1}{2}\right)
				\Gamma\left(l_v+\tfrac{-z-1}{2}\right)}
			{\Gamma_\R\left(\tfrac{1+z}{2}\right)\Gamma_\R\left(\tfrac{1-z}{2}\right)}
			\,\delta(a>0)\, |a|_v^{1/2}|a-1|_v^{-1}\fP_{\frac{z-1}{2}}^{1-l_v}\left(\l_v(a)^{-1/2}\right)\end{align*}
		with $\l_v(a)=|(a-1)/(a+1)|_v^2$.
		Here $\fP_{\nu}^{\mu}(x)$ is the Legendre function of the 1st kind which is defined for points $x\in \C$ outside the interval $(-\infty,+1]$ of the real axis (\cite[\S 4.1]{MOS}).
		
		\item[(2)] Let $v\in \Sigma_\fin-(S\cup S(\fn))$. Then, we have
		\begin{align*}
			\fF_v^{(z)}(a)=\delta(a\in \cO_v^\times)\, q_v^{-d_v/2}\,\Ocal_{v,0}^{1,(z)}((a-1)^{-2}a).
		\end{align*}
		In particular, $\fF_v^{(z)}(a)=q_v^{-d_v/2}$ if $|a|_v=|a-1|_v=1$. 
		
		\item[(3)] Suppose $v \in S(\gn)$. Then, we have
		\begin{align*}
			\fF_v^{(z)}(a)&=\delta(a\in \cO_v^\times)\,
			q_v^{-d_v/2}\,\Ocal_{v,1}^{1,(z)}((a-1)^{-2}a).
		\end{align*}
		
		\item[(4)] Suppose $v \in S$. Then,
		for $2\Re(s_v)+1>|\Re(z)|$ we have
		\begin{align*}
			\fF_v^{(z)}(s_v;a)=&q_v^{-\frac{d_v}{2}} \Scal_{v}^{1,(z)}((a-1)^{-2}a).
		\end{align*}
	\end{enumerate}
\end{thm}

\begin{lem}\label{LocalHyp-almostall-esti}
	We have the inequality
	$$|\fF_v^{(z)}(a)|\leq 3\,|{\rm{ord}}_v(a-1)|\,|a-1|_v^{-\frac{|\Re(z)|+1}{2}}
	$$
	for all $a\in 1+\fp_v$ and $z\in \C$. 
\end{lem}
\begin{proof}We suppose $a\in 1+\gp_v$ and set $t=q_v^{z/2}$ and $\a=\ord_v(a-1)\in \N$. Then, 
	by Theorem \ref{explicit hyp} (2), $\fF_v^{(z)}(a)$ is the product of $q_v^{-d_v/2}|a-1|_v^{-1/2}$ and
	\begin{align*}
		f(t)&:=\tfrac{1-q^{-1/2}t}{1-t^{2}}t^{-\a}+\tfrac{1-q_v^{-1/2}t^{-1}}{1-t^{-2}}t^{\a} =t^{\a}\left\{1+{g(t)}\,t^{-2\a}(1-tq_v^{-1/2})\right\}
	\end{align*}
	with $g(t)=\sum_{j=0}^{\a-1}t^{2j}$. From the second expression, $f(t)$ is seen to be holomorphic away from $t=0$. Let $|t|\geq 1$; then $|g(t)||t|^{-2\a}\leq \a|t|^{-2}$ and 
	\begin{align*}
		{|f(t)|}{|t|^{-\a}} &\leq 1+g(t)\,|t|^{-2\a}(1+|t|q_v^{-1/2})
		\leq 1+\a(|t|^{-2}+q_v^{-1/2}|t|^{-1})\leq 3\a.  
	\end{align*} 
	Thus $|f(t)|\leq 3\a\,|t|^{\a}$ for $|t|\geq 1$. From this we have $|f(t)|\leq 3\a\,|t|^{-\a}$ by the obvious functional equation $f(t^{-1})=f(t)$.
\end{proof}

\begin{cor} \label{FHypL7.5}
	Let $a \in F^\times-\{1\}$. Given $\e>0$, we have a constant $C_\e>0$ such that  
	\begin{align*}
		\prod_{v\in \Sigma_\fin -(S\cup S(\fn))} |\fF_v^{(z)}(a)|
		\leq \delta(a\in \cO(S\cup S(\fn))^\times )\,C_\e\,\prod_{v\in \Sigma_\fin-(S\cup S(\fn))}|a-1|_v^{-\frac{|\Re(z)|+1}{2}-\e}
	\end{align*}
	for $a\in F^\times -\{1\}$ and $z\in \C$, where $\cO(S\cup S(\fn))$ is the ring of $S\cup S(\fn)$-integers in $F$. 
\end{cor}
\begin{proof}
	Choose $A_\e>1$ such that $3x\leq A_\e\,2^{\e x}$ for all $x\ge 1$. We fix $q(\e)>1$ such that $A_\e\,2^{\e}\leq q^{\e}$ for all $q>q(\e)$, and set $P(\e)=\{v\in \Sigma_\fin |\,q_v\leq q(\e)\}$. Then $3|{\rm{ord}}_v(a-1)|\leq A_\e\,|a-1|_v^{-\e}$ for all $a\in 1+\fp_v$ with $v\in P(\e)$, and $3|{\rm{ord}}_v(a-1)|\leq |a-1|_v^{-\e}$ for all $a\in 1+\fp_v$ with $v\in \Sigma_\fin -P(\epsilon)$. Setting $C_\e=A_\e^{\# P(\epsilon)}$, we are done by Lemma~\ref{LocalHyp-almostall-esti}. \end{proof}

\begin{lem} \label{Sfn-Est}
	We have the estimate
	\begin{align*}
		|\fF_v^{(z)}(a)|\leq 4(1+q_v)^{-1}\,\delta(a\in \cO_v^\times)\,(1+q_v^{1/2})^{\delta(a\in 1+\fp_v)}\,({\rm{ord}}_v(a-1)+1)\,|a-1|_v^{-\frac{|\Re(z)|+1}{2}}\end{align*}
	for $v\in S(\fn)$, $a\in F_v^\times-\{1\}$ and $z\in \C$. 
\end{lem}
\begin{proof}
	Let $v\in S(\fn)$, $a\in 1+\fp_v$. Set $t=q_v^{z/2}$, $\a={\rm{ord}}_v(a-1)\ge 1$ and\begin{align*}
		f(t)&=\tfrac{1-q_v^{-1/2}t}{1-t^2}(1+q_v^{1/2}t)t^{-\a}+\tfrac{1-q_v^{-1/2}t^{-1}}{1-t^{-2}}(1+q_v^{1/2}t^{-1})t^{\a}\\
		&=t^{\a}\{1+t^{-2}+t^{-1}(q_v^{1/2}-q_v^{-1/2})+t^{-2\a}g(t)(1-q_v^{-1/2}t)(1+q_v^{1/2}t)\}
	\end{align*}
	with $g(t)=\sum_{j=0}^{\a-2}t^{2j}$. 
	Then Theorem~\ref{explicit hyp} (3) shows that $\fF_v^{(z)}(a)$ is the product of $\frac{q_v^{-d_v/2}}{1+q_v}|a-1|_v^{-1/2}$ and $f(t)$. If $|t|\geq 1$, we have $|g(t)||t|^{-2\a}\leq |t|^{-4}\sum_{j=0}^{\a-2}|t|^{-2j}\leq \a|t|^{-4}$ and 
	\begin{align*}
		|f(t)|&\leq |t|^{\a}\{1+|t|^{-2}+|t|^{-1}(q_v^{1/2}+q_v^{-1/2})
		+|t|^{-2\a}|g(t)|(1+q_v^{-1/2}|t|)(1+q_v^{1/2}|t|)\}
		\\
		&\leq |t|^{\a}\{1+1+(1+q_v^{1/2})+2\a(1+1)(1+q_v^{1/2})\} 
		\leq 4(\a+1)(1+q_v^{1/2})\,|t|^{\a}. 
	\end{align*}
	Hence $|f(t)|\leq 4(\a+1)(1+q_v^{1/2})|t|^{\a}$ for $|t|\geq 1$. By the obvious functional equation $f(t)=f(t^{-1})$, we also have $|f(t)|\leq 4(\a+1)(1+q_v^{1/2})|t|^{-\a}$ for $|t|\leq 1$. Thus 
	$$
	|\fF_v^{(z)}(a)|\leq q_v^{-d_v/2}(1+q_v)^{-1}\,4(\,{\rm{ord}}_v(a-1)+1)(1+q_v^{1/2})|a-1|_v^{-\frac{|\Re(z)|+1}{2}}.
	$$
\end{proof}

\begin{lem}\label{lem of esti hyp-S} 
	Let $v \in S$. On the region $\Re(s_v)>1$ and $|\Re(z)|<2\Re(s_v)-1$,  
	\begin{align*}
		&|\fF_v^{(z)}(s_v; a)|\leq 16 q_v^{-\frac{\Re(s_v)+1}{2}} (1+{\rm{ord}}_v(a-1))|a-1|_v^{-\frac{|\Re(z)|+1}{2}} \quad \text{for $|a-1|_v<1$}, \\
		&|\fF_v^{(z)}(s_v; a)|\leq 16q_v^{-\frac{\Re(s_v)+1}{2}}|a-1|_v^{-(\Re(s_v)+1)}|a|_v^{\frac{\Re(s_v)+1}{2}} \quad \text{for $|a-1|_v\geq 1$}.
	\end{align*} 
\end{lem}
\begin{proof}
	Put $s=s_v$. Suppose $|a-1|_v<1$ and let $\nu={\rm{ord}}_v(a-1)$ and $t=q_v^{z/2}$.
	From Theorem \ref{explicit hyp} (4), a computation shows the identity $\fF_v^{(z)}(a)=q_v^{-\frac{d_v+s+1}{2}}|a-1|_v^{-1/2}\,f(t)$ with 
	\begin{align*}
		f(t)&=\frac{t^{\nu}\{
			t^{-2\nu}g(t)(1-q_v^{-1/2}t(1+q_v^{-s})+t^2q_v^{-s-1})+1+t^{-2}-t^{-1}q_v^{-1/2}(1+q_v^{-s})\}}{(1-q_v^{-s-1/2}t)(1-q_v^{-s-1/2}t^{-1})},
	\end{align*}
	where $g(t)=\sum_{j=0}^{\nu-2}t^{2j}$. Suppose $|t|\geq 1$. Then $|g(t)|<|t|^{2\nu-2}\nu$. From $\Re(s)>1$ and $|\Re(z)|<2\Re(s)-1$, we have $|q_v^{-s}|< q_v^{-1}<1$ and $|q_v^{-s-1/2}t^{\pm 1}|<q_v^{-1}$. By these,
	\begin{align*}
		|f(t)|&\leq {|t|^\nu\bigl\{|t|^{-2}\nu(1+q_v^{-1/2}|t|(1+|q_v^{-s}|)+|t|^2|q_v^{-s}|)+1+|t|^{-2}+|t|^{-1}(1+|q_v^{-s}|)q_v^{-1/2}\bigr\}}{(1-q_v^{-1})^{-2}}
		\\
		&\leq {|t|^{\nu}(4\nu+4)}{(1-q_v^{-1})^{-2}}
		\leq {4|t|^\nu(\nu+1)}{(1/2)^{-2}} = 16|t|^{\nu}(\nu+1).
	\end{align*}
	Thus $|f(t)|\leq 16(\nu+1)|t|^{\nu}$ for $|t|\geq 1$. Since $f(t)=f(t^{-1})$, we then have $|f(t)|\leq 16(\nu+1)|t|^{-\nu}$ for $|t|\leq 1$. 
	
	Suppose $|a-1|_v\geq 1$ and set $t=q_v^{z/2}$. Then
	\begin{align*}
		|\fF^{(z)}(s;a)|&=q_v^{-\frac{d_v+\Re(s)+1}{2}} \frac{1+|q_v^{-s-1}|}{|(1-q_v^{-s-1/2}t)(1-q_v^{-s-1/2}t^{-1})|}\,|a-1|_v^{-\Re(s)-1}|a|_v^{\frac{\Re(s)+1}{2}}
		\\
		&\leq  q_v^{-\frac{\Re(s)+1}{2}}\frac{1+q_v^{-2}}{(1-q_v^{-1})^2}
		|a-1|_v^{-\Re(s)-1}|a|_v^{\frac{\Re(s)+1}{2}}
		\leq 5 q_v^{-\frac{\Re(s)+1}{2}} |a-1|_v^{-\Re(s)-1}|a|_v^{\frac{\Re(s)+1}{2}}.\end{align*}
\end{proof}

\begin{cor} \label{S-Est}
	On the region $|\Re(z)|<2\Re(s_v)-1$, $\Re(s_v)>1$, we have
	\begin{align*}
		|\fF^{(z)}_v(a)|\leq 16 \{\max(0,{\rm{ord}}_v(a-1))+1\}\,\max(1,|a-1|_v)^{\frac{|\Re(z)|-\Re(s_v)}{2}}\,|a-1|_v^{-\frac{|\Re(z)|+1}{2}}
	\end{align*}
	for all $a\in F_v^\times -\{1\}$. 
\end{cor}
\begin{proof}
	If $|a-1|_v\geq 1$, then $|a|_v\leq |a-1|_v+1\leq 2 |a-1|_v$. By this remark, we can deduce the inequality from those in Lemma \ref{lem of esti hyp-S}.
\end{proof}

\begin{lem}\label{FHypL9} 
	Let $v\in \Sigma_\infty$. 
	On the region $|\Re(z)|<2l_v-1$, the integral $\fF_v^{(z)}(a)$ converges absolutely
	and is a holomorphic function in $z$ for any fixed $a\in F-\{0,1\}$. We have
	\begin{align*}
		|\fF_v^{(z)}(a)|\ll \delta(a_v>0)\,|a|_v^{l_v/2} \,|a-1|_v^{-l_v}, \quad a\in F-\{0,1\}
	\end{align*}
	uniformly for $z$ lying in a compact set of $|\Re(z)|<2l_v-1$. 
\end{lem}
\begin{proof} Recall $\Phi_v=\Phi_v^{l_v}$ is given by \eqref{DSMexplicitformula}. By the Cartan decomposition, we have 
	$$
	\Phi_v^{l_v}\left(k^{-1}\left[\begin{smallmatrix} a & (a-1)x \\ 0 & 1 \end{smallmatrix}\right]k\right)=\delta(a>0)\,2^{l_v}a^{l_v/2}\,\{(a+1)-i(a-1)x\}^{-l_v}
	$$
	for any $k\in \bK_v$, $a\in \R^\times$ and $x\in \R$. From now on, we suppose $a>0$ and $a\neq1$. We have $\fF_v^{(z)}(a)=\delta(a>0) 2^{l_v}a^{l_v/2}(I^{+}(z,a)+I^{-}(z,a))$, where
	\begin{align*}
		I^{\varepsilon}(z,a)=\tint_{J^\varepsilon}\{(a+1)-i(a-1)x\}^{-l_v}
		\varphi_v^{(0,z)}\left(\left[\begin{smallmatrix} 1 & x\\ 0 & 1 \end{smallmatrix}\right]\right)
		\,\d x
	\end{align*}
	with $J^+=[-1,1]$, $J^-=\R-J^+$. We examine the convergence of these integrals separately. By the obvious estimate $|(a+1)-i(a-1)x|=\{(a+1)^2+x^2(a-1)^2\}^{1/2}\geq |a+1|$, 
	\begin{align*}
		|I^{+}(z,a)|&\leq \tint_{-1}^{1}|(a+1)-i(a-1)x|^{-l_v}|
		\varphi_v^{(0,z)}\left(\left[\begin{smallmatrix} 1 & x\\ 0 & 1 \end{smallmatrix}\right]\right)|\,\d x
		\leq |a+1|^{-l_v} \tint_{-1}^{1}|\varphi_v^{(0,z)}\left(\left[\begin{smallmatrix} 1 & x\\ 0 & 1 \end{smallmatrix}\right]\right)|\,\d x
	\end{align*}
	Since the integral is convergent for all $z\in \C$, we have the estimate $|I^{+}(z,a)|\ll |a+1|^{-l_v}$ compact uniformly in $z$. In the same way, by the estimate $|(a+1)-i(a-1)x|\geq |x||a-1|$, 
	\begin{align*}
		|I^{-}(z,a)|\leq |a-1|^{-l_v} \tint_{|x|>1} |x|^{-l_v}|\varphi_v^{(0,z)}\left(\left[\begin{smallmatrix} 1 & x\\ 0 & 1 \end{smallmatrix}\right]\right)|\,\d x.
	\end{align*}
	From the defining formula \eqref{Hsphericalfnt}, we easily have the estimate $\varphi_v^{(0,z)}\left(\left[\begin{smallmatrix} 1 & x\\ 0 & 1 \end{smallmatrix}\right]\right)=O((1+x^2)^{(|\Re(z)|-1)/4})$ as $x \rightarrow \infty$. From this the integral in the majorant of $|I^{-}(z,a)|$ is seen to be convergent for $|\Re(z)|<2l_v-1$. 
\end{proof}

\begin{cor} \label{FHypL11} 
	Let $v\in \Sigma_\infty$ and $0<\s<2l_v-1$. We have the estimate
	$$
	|\fF_v^{(z)}(a)|\ll \delta(a_v>0)\,|a-1|_v^{-\frac{|\Re(z)|+1}{2}}\,(1+|a|_v)^{-\frac{l_v}{2}+\frac{|\Re(z)|+1}{2}}, \quad a\in F-\{0,1\}
	$$
	uniformly for $z$ such that $|\Re(z)|\leq \s$.  
\end{cor}
\begin{proof} 
	Let $U$ be a neighborhood of $1$ in $\R_{+}$ such that $|a-1|_v<1$ for all $a \in U$. From Lemma~\ref{FHypL9}, we have the estimate $|\fF^{(z)}(a)|\ll \delta(a_v>0)\,(1+|a|_v)^{-l_v/2}$ on $a\in \R-U$ uniformly in $|\Re(z)|\leq \s$. To have a bound on $U$, we use Theorem~\ref{explicit hyp} (1). By  Stirling's formula, $\frac{\Gamma\left(l_v+\tfrac{z-1}{2} \right) \Gamma\left(l_v+\tfrac{-z-1}{2} \right)}{\Gamma_\R\left(\tfrac{1+z}{2}\right)\Gamma_\R\left(\tfrac{1-z}{2}\right)}$ is vertically bounded on $|\Re(z)|\leq \s$. By the formula \cite[line 12, p.184]{MOS},
	\begin{align*}
		\fP_{\frac{z-1}{2}}^{1-l_v}(\l_v(a)^{-1/2})
		&=\tfrac{1}{\sqrt{\pi}\Gamma(l_v-1/2)}|a|_v^{\frac{l_v-1}{2}}|a-1|_v^{\frac{-z+1}{2}}\,\tint_{0}^{\pi}\{2\sqrt{a_v}\cos \theta+a_v+1\}^{\frac{z+1}{2}-l_v}(\sin^2\theta)^{l_v-1}\,\d\theta.
	\end{align*}
	For varying $a\in U$ and $|\Re(z)|\leq \s$, the last integral is bounded by a constant. Thus, $|\fP_{\frac{z-1}{2}}^{1-l_v}(\l_v(a)^{-1/2})|=O(|a-1|_v^{\frac{-|\Re(z)|+1}{2}})$ for $a\in U$ and $|\Re(z)|\leq \s$. 
\end{proof}

Now, Proposition~\ref{FHypP1} follows from Corollaries~\ref{FHypL7.5}, \ref{FHypL11}
and \ref{S-Est}
and Lemma \ref{Sfn-Est}. We use the product formula $\prod_{v\in \Sigma_F}|a-1|_v=1$ for $a\in F^\times-\{1\}$ to eliminate the factors $|a-1|_v^{-\frac{|\Re(z)|+1}{2}}$ $(v\in \Sigma_F)$ in the majorant.

\subsection{Proof of absolute convergence} \label{PAbsConv}

\begin{prop} \label{FHypP2}
	Suppose $\Re(s_v)>2$ for all $v\in S$. Let $0<\s<l_v-3$. The series $\sum_{a\in F^\times-\{1\}}\fF^{(z)}(a)
	$ converges absolutely and uniformly for $|\Re(z)|\leq \s$ and locally uniformly in $\bfs$ defining a holomorphic function of $z$ on the region 
	$$|\Re(z)|<\min\,(\{l_v-3|v\in \Sigma_\infty\}\cup \{\Re(s_v)-2|\,v\in S\,\}).
	$$
\end{prop}
\begin{proof}
	Let $S_0\subset \Sigma_\fin$ be such that $\Phi_v = {\rm ch}_{Z_v\bK_v}$ and $d_v=0$ for all $v\in \Sigma_\fin-S_0$. Then we may take $\phi_v={\rm ch}_{\cO_v^\times}$ in Proposition~\ref{FHypP1}, by which the proof is reduced to showing the convergence of 
	$\sum_{a\in F^\times-\{1\}}f(a)$ with $f$ being a function on the adeles defined as 
	\begin{align*}
		f(a)=\{\prod_{v\in \Sigma_\fin}\phi_{v}(a_v)\}\,\,\{\prod_{v\in \Sigma_\infty}(1+|a_v|_v)^{-l_v/2+(\s+1)/2}\},\quad  a\in \A-\{0,1\}. 
	\end{align*}
	Let $\cO(S_0)$ be the ring of $S_0$-integers in $F$, i.e., $\cO(S_0)=F\cap \prod_{v\in \Sigma_\fin-S_0}\cO_v\times \prod_{v\in S_0\cup\Sigma_\infty}F_v$. 
	Set $f_v(x_v)=\phi_v(x_v)$ for $v\in S_0$ and $f_v(x_v)=(1+|x_v|_v)^{-l_v/2+(\s+1)/2}$ if $v\in \Sigma_\infty$. Let $\Ncal_v\subset \varpi_v\cO_v$ be a compact neighborhood of $0$ in $F_v$ such that $\phi_v$ is constant on cosets $x+\Ncal_v$ if $v\in S_0$. Then $f_v(x+y)=f_v(x)\,(x\in F_v,\,y\in \Ncal_v)$ if $v\in S_0$ and $f_v(x)\ll f_v(x+y)\,(x\in F_v,\,y\in \Ncal_v)$ if $v\in \Sigma_\infty$. Hence if we set $\Ncal_{S_0,\infty}=\prod_{v\in S_0\cup \Sigma_\infty}\Ncal_v$, then $f(x)\ll f(x+y)\,(x\in F_{S_0,\infty},\,y\in \Ncal_{S_0,\infty})$. Since $\cO(S_0)$ is a discrete subgroup of $F_{S_0,\infty}=\prod_{v\in S_0\cup\Sigma_\infty}F_v$, by choosing $\Ncal_{S_0,\infty}$ small enough, we have
	\begin{align*}
		\sum_{a\in F^\times-\{1\}}f(a)
		\ll \sum_{a\in \cO(S_0)-\{0,1\}}\int_{\Ncal_{S_0,\infty}}f_{}(a+y)\,\d y 
		\leq \int_{F_{S_0,\infty}}f_{}(x)\,\d x.
	\end{align*}
	Since $l_v\geq 4$ for all $v\in \Sigma_\infty$, the last integral is easily seen to be convergent.  
\end{proof}

We have
\begin{align*}
	\JJ_{\rm{hyp}}(\bfs, \b)= & \tfrac{1}{2}\sum_{a\in F^\times-\{1\}} 
	\int_{\A}\d x \int_{\bK}\d k \, \Phi\left(\bfs;k^{-1}\left[\begin{smallmatrix} a & (a-1)x \\ 0 & 1 \end{smallmatrix} \right] k \right)\\
	&\times \left\{\int_{L_\s} \b(z)\,\zeta_F\left(\tfrac{z+1}{2}\right)\zeta_F\left(\tfrac{-z+1}{2}\right)\,\prod_{v\in \Sigma_F}\varphi_v^{(0,z)}\left(\left[\begin{smallmatrix} 1 & x_v \\ 0 & 1 \end{smallmatrix}\right] \right)\,\d z\right\}
\end{align*}
with $-1<\s<1$. By the estimate of Proposition~\ref{FHypP2}, we can change the order of the integrals to have the formula 
\begin{align*}
	\JJ_{\rm{hyp}}(\bfs, \b)&=\tfrac{1}{2}\,
	\int_{L_\s} \b(z)\,\zeta_F\left(\tfrac{z+1}{2}\right)\zeta_F\left(\tfrac{-z+1}{2}\right)\,\{\sum_{a\in F^\times-\{1\}}\fF^{(z)}(a)\} \,\d z.
\end{align*}
Define
\begin{align}
	\hat J_{\rm hyp}(\bfs,z)=\tfrac{1}{2}\,\zeta_F\left(\tfrac{z+1}{2}\right)\zeta_F\left(\tfrac{-z+1}{2}\right)\,\{\sum_{a\in F^\times-\{1\}}\fF^{(z)}(a)\}.
	\label{Jhypbfsz}
\end{align}
Then we obtain a precise form of Theorem~\ref{hatJ*} for the $F$-hyperbolic term.

\begin{prop}\label{Conclusion of hyp}
	Suppose $\Re(s_v)>2$ for all $v\in S$. 
	The function $\hat J_{\rm hyp}(\bfs,z)$ is holomorphic away from $z=\pm 1$ and vertically of moderate growth on $|\Re(z)|<\min\,(\{l_v-3|\,v\in \Sigma_\infty\}\cup \{\Re(s_v)-2|\,v\in S\,\})$. For any contour $L_\s$ $(|\s|<1)$, we have the formula \eqref{hatJ*-f1} with $\natural = {\rm hyp}$. 
\end{prop}

\section{The $F$-elliptic term} \label{The $F$-elliptic term}

In this section, we study $\JJ_{\rm ell}(\bfs; \b)$ to show Theorem~\ref{hatJ*} for the $F$-elliptic term. From this section on, we suppose that $2$ splits completely in the extension $F/\Q$; thus $F_v\cong \Q_2$ for all $v\in \Sigma_{\rm dyadic}$. 

\subsection{Parametrization of elliptic elements}
Set $Q_F=\{(t,n)\in F^{2}|\,t^2-4n\not=0\,\}$. Let us say that two elements $(t,n)$ and $(t',n')$ from $Q_F$ are $F$-equivalent if there exists $c\in F^\times$ such that $(t',n')=(ct,c^2n)$. The $F$-equivalence class of a pair $(t,n)\in Q_F$ is denoted by $(t:n)_{F}$. The quotient set of $Q_F$ by the $F$-equivalence relation is denoted by $\cQ_F$, i.e., 
$$\cQ_F=\{(t:n)_{F}|\,t,\,n\in F,\,t^{2}-4n\not=0\}.$$

Let $\cQ_{F}^{\rm Irr}$ be the set of $(t:n)_F\in \cQ_F$ such that the polynomial $X^2-tX+n$ is $F$-irreducible. For $\tilde\gamma\in \cQ_{F}^{\rm Irr}$, fix its representative $(t,n)\in Q_F$ once and for all and set $\gamma=\left[\begin{smallmatrix} t/2 & 1 \\ \Delta/4 & t/2 \end{smallmatrix}\right]\in G_F$ and $\Delta=t^2-4n$. The $G_F$-conjugacy class with characteristic polynomial $t^2-tX+n$ is represented by the element $\gamma$. For a quadratic extension $E\cong F[X]/(X^2-\Delta/4)$ of $F$ with a prescribed square root $\sqrt{\Delta}\in E^\times$, let $\iota_{\Delta}:E\hookrightarrow \Mat_2(F)$ be the $F$-algebra embedding 
$$
\iota_{\Delta}(a+b\sqrt{\Delta}/2)=\left[\begin{smallmatrix} a & b \\ \Delta b/4 & a \end{smallmatrix} \right], \quad a,b\in F. 
$$
Then the centralizer of $\gamma$ in $G_F$ is $G_{\gamma,F}=\iota_{\Delta}(E^\times)$. 

For any place $v\in \Sigma_F$, we can write the image of $4^{-1}\Delta$ in $F_v$ as $4^{-1}\Delta=\Delta_v^{0}\,m_v^2$ with $m_v\in F_v^\times,\,\Delta_v^{0}\in F_v^\times/(F_v^\times)^2$; we suppose (a) $v\in \Sigma_\fin$, $\Delta_v^{0}$ belongs to $(\fp_v-\fp_v^2)\cup \{1\}\cup (\cO_v^\times -(\cO_v^\times)^2)$, or (b) $v\in \Sigma_\infty$ and $\Delta_v^{0}\in \{+1,-1\}$. We fix such a factorization of $\Delta/4$ in $F_v^\times$. Since $2$ is assumed to be completely split in $F/\Q$, we have $\cO_v^\times/(\cO_v^\times)^2\cong \Z_2^\times/(1+8\Z_2)=\{\pm 5, \pm 1\}$ thus may suppose $\Delta_v^{0}\in \{\pm 5,\pm 1, \pm 10,\pm 2\}$ for all $v\in \Sigma_{\rm dyadic}$. Let $\fd_{E/F}$ denote the relative discriminant of $E/F$.  From this, it is easily seen that $\fd_{E/F}\cO_v=\Delta_v^{0}\cO_v$ if $v\in \Sigma_\fin-\Sigma_{\rm dyadic}$ or $v\in \Sigma_{\rm dyadic}$, $\Delta_v^{0}\in \{5,1\}$, and $\fd_{E/F}\cO_v=4\Delta_v^{0}\cO_v$ for $v\in \Sigma_{\rm dyadic}$, $\Delta_v^{0}\in\{-5,-1, \pm10,\pm2\}$. Let $\varepsilon_{\Delta}$ be idele class character associated with the quadratic extension $E=F(\sqrt{\Delta})$ by class field theory. Then $\zeta_{E}(z)=L(z,\varepsilon_{\Delta})\zeta_{F}(z)$ and the conductor of $\varepsilon_{\Delta}$ is $\fd_{E/F}$.

We have the direct sum decomposition $E_v=F_v+\sqrt{\Delta_v^{0}}F_v$ as $F_v$-vector spaces, which determines an $F_v$-embedding $\iota_{\Delta_v^0}:E_v\hookrightarrow \Mat_2(F_v)$ as 
\begin{align*}
	&\iota_{\Delta_v^0}(a+b\sqrt{\Delta_v^0})=\left[\begin{smallmatrix} a & b \\ b\Delta_v^{0} & a \end{smallmatrix} \right], \quad a,b \in F_v \quad \text{
		if $\Delta_v^0\neq1$,}\\
	&\iota_{\Delta_v^0}(a+b\sqrt{\Delta_v^0})=\left[\begin{smallmatrix} a+b & 0 \\ 0 & a-b \end{smallmatrix} \right], \quad a,b \in F_v \quad \text{
		if $\Delta_v^0=1$.}
\end{align*} 
Set ${\frak T}_{\Delta,v}=\iota_{\Delta_v^{0}}(E_v^\times)$. For $v\in \Sigma_\fin$, we set ${\frak T}_{\Delta,v}^{+}=Z_v\bK_v\cap {\frak T}_{\Delta,v}$.

\begin{lem} \label{Hensel}
	Let $v \in \Sigma_\fin-\Sigma_{\rm dyadic}$. If $\tau \in \cO_v^\times-(\cO_v^\times)^2$, then $a^2-\tau \in \cO_v^\times$ for all $a\in \cO_v$. 
\end{lem}
\begin{proof} If $a^2-\tau \in \fp_v$, then $a\in \cO_v^\times$ and $\tau\equiv a^2 \pmod{\fp_v}$. Since $\tau\not\in (\cO_v^\times)^2$, this is impossible by Hensel's lemma. 
\end{proof}

\begin{lem} \label{fTDeltaL}
	Let $v\in \Sigma_F$. 
	\begin{itemize}
		\item[(i)] If $v\in \Sigma_\fin-\Sigma_{\rm dyadic}$ and $\Delta_v^{0}\in \cO_v^\times-(\cO_v^\times)^2$, then ${\frak T}_{\Delta,v}^{+}={\frak T}_{\Delta,v}$.\item[(ii)] If $v\in \Sigma_\fin$ and $\Delta_v^{0}\in \fp_v-\fp_v^2$, then 
		${\frak T}_{\Delta,v}={\frak T}_{\Delta,v}^{+} \cup \left[\begin{smallmatrix} 0 & 1 \\ \Delta_{v}^{0} & 0 \end{smallmatrix}\right]{\frak T}_{\Delta,v}^{+}$. 
		\item[(iii)] If $v\in \Sigma_{\rm dyadic}$ and $\Delta_v^{0}\in \{-1,-5\}$, then ${\frak T}_{\Delta,v}={\frak T}_{\Delta,v}^{+} \cup \left[\begin{smallmatrix} 1 & 1 \\ \Delta_{v}^{0} & 1 \end{smallmatrix}\right]{\frak T}_{\Delta,v}^{+}$.
		\item[(iv)] If $v\in \Sigma_{\rm dyadic}$ and $\Delta_v^{0} =5$, then ${\frak T}_{\Delta,v}={\frak T}_{\Delta, v}^+ \sqcup [\begin{smallmatrix}1 & 1 \\ 5 & 1
		\end{smallmatrix}] {\frak T}_{\Delta, v}^+ \sqcup [\begin{smallmatrix}
		3 & 1 \\ 5 & 3
		\end{smallmatrix}] {\frak T}_{\Delta, v}^+.$
		\item[(v)] If $v\in \Sigma_\infty$, $\Delta_v^{0}=-1$, then ${\frak T}_{\Delta,v}=Z_v\bK_v^{0}$. 
	\end{itemize}
\end{lem}
\begin{proof} (i) For $(x, y) \in F_v^2-\{(0,0)\}$, we have $[\begin{smallmatrix} x& y \\ \Delta_v^0 y&x \end{smallmatrix}]=	[\begin{smallmatrix} x& 0 \\ 0 &x \end{smallmatrix}][\begin{smallmatrix} 1& y/x \\ \Delta_v^{0} y/x&1 \end{smallmatrix}]$.
	Then, $1-\Delta_v^{0} y^2/x^2 \in \go_v^\times$ if $|x|_v>|y|_v$. If $|x|_v=|y|_v$, by putting $x=\varpi_v^{m}u_x$ and $y=\varpi_v^{m}u_y$ with $u_x, u_y \in \go_v^{\times}$ and $m\in\ZZ$, we have $[\begin{smallmatrix} x&\Delta_v^0 y \\ y&x \end{smallmatrix}]=[\begin{smallmatrix} \varpi_v^m& 0 \\ 0 &\varpi_v^m \end{smallmatrix}][\begin{smallmatrix} u_x& u_y \\ \Delta_v^{0} u_y & u_x \end{smallmatrix}]$, and  $u_x^2-\Delta_v^{0}\, u_y^2 \in \cO_v^\times$ by Lemmma~\ref{Hensel}. Thus we are done. (ii) and (iii) follow from the observation that an element $\iota_{\Delta_v^{0}}(a+b\sqrt{\Delta_v^{0}})$ belongs to ${\frak T}_{\Delta,v}^{+}$ if and only if ${\rm{ord}}_v(a^2-b^2\Delta_v^{0})\in 2\Z$. (iv) Set $M_j = \{[\begin{smallmatrix}
	u & 1 \\ \Delta_v^0 & u
	\end{smallmatrix}] \in {\rm M}_2(\ZZ_2) \ | \ u \in j+4\ZZ_2\}$ for $j=1,3$.
	Then $Z_vM_v$ with $M_v=\bK_v\cap {\frak T}_{\Delta,v}$ is written as
	$Z_vM_v=Z_vM_1 \sqcup Z_vM_3 =  [\begin{smallmatrix}
	1 & 1 \\ 5 & 1
	\end{smallmatrix}] {\frak T}_{\Delta, v}^+ \sqcup [\begin{smallmatrix}
	3 & 1 \\ 5 & 3
	\end{smallmatrix}] {\frak T}_{\Delta, v}^+.$
	By ${\frak T}_{\Delta,v} = {\frak T}_{\Delta,v}^+ \sqcup Z_vM_v$, we are done. (v) is confirmed by direct computation.  
\end{proof}

For $v\in \Sigma_F$, set 
$$
R_{\Delta,v}=\left[\begin{smallmatrix} m_v & 0 \\ 0 & 1 \end{smallmatrix} \right] \quad \text{if $\Delta_v^{0}\not=1$, and}\,\, R_{\Delta, v}=\left[\begin{smallmatrix} 1 & 1 \\ 1 & -1 \end{smallmatrix} \right]
\left[\begin{smallmatrix} m_v & 0 \\ 0 & 1 \end{smallmatrix} \right] \quad \text{if $\Delta_v^0=1$}.
$$
Then we have the relation 
\begin{align}
	R_{\Delta,v}^{-1} {\frak T}_{\Delta,v} 
	R_{\Delta,v}
	=G_{\gamma,v}.
	\label{fTGgammaLoc}
\end{align}
Since $\Delta \in \cO_v^\times$ for almost all $v$, the system $R_{\Delta}=\{R_{\Delta,v}\}_{v\in \Sigma_F}$ belongs to $G_\A$. 
Define
\begin{align*}
	{\frak T}_\Delta=\{(h_v)\in \prod_{v\in \Sigma_F}{\frak T}_{\Delta,v}|\,\text{$h_v\in \bK_v\cap {\frak T}_{\Delta,v}$ for almost all $v\in \Sigma_\fin$}\,\}.\end{align*}
If we view ${\frak T}_{\Delta}$ as a closed subgroup of $G_\A$, then 
\begin{align}
	R_\Delta^{-1}\,
	{\frak T}_\Delta 
	R_\Delta = G_{\gamma,\A}.
	\label{fTGgamma}
\end{align}
For each place $v$, we fix a Haar measure on $E_v^\times$ by   
\begin{align*}
	\d^\times\tau_v=\zeta_{E_v}(1){|x_v^2-4^{-1}\Delta y_v^2|_v^{-1}}\d x_v\,\d y_v \quad \text{with $\tau_v=x_v+2^{-1}\sqrt{\Delta}y_v\,(x_v,y_v\in F_v)$,}
\end{align*} 
and transfer this to $G_{\gamma,v}$ by $\iota_{\Delta}: E_v^\times \cong G_{\gamma,v}$. We transfer the Haar measure $\d^\times \tau=\otimes_{v}\d^\times \tau_v$ on $\A_E^\times$ to $G_{\gamma,\A}$ to define a Haar measure on $G_{\gamma,\A}$. We use the relations \eqref{fTGgammaLoc} and \eqref{fTGgamma} to define Haar measures on groups ${\frak T}_{\Delta,v}$ and ${\frak T}_{\Delta}$.

\begin{lem} \label{VolZfTd}
	If $v\in \Sigma_\fin$, then $\vol(Z_v\bsl {\frak T}_{\Delta,v}^+)=|m_v|_v^{-1}q^{-d_v/2}_v$ unless $v\in \Sigma_{\rm dyadic}$, $\Delta_v^{0}=5$, in which case $\vol(Z_v\bsl {\frak T}_{\Delta, v}^+) = \tfrac{2}{3}\,|m_v|_v^{-1}2^{-d_v/2}.$
	If $v\in \Sigma_\infty$, then $\vol(Z_v\bsl {\frak T}_{\Delta,v})=|m_v|_v^{-1}$. \end{lem}
\begin{proof}
	Let $v\in \Sigma_\fin$. Then, $\frak{T}_{\Delta,v}^+ = Z_v M_v$ with $M_v=\bK_v\cap {\frak T}_{\Delta,v}$, and ${\rm ch}_{Z_vM_v}(t)=\vol(Z_v\cap M_v)^{-1}\int_{z\in Z_v}{\rm ch}_{M_v}(zt)\,\d z$, as in the proof of \cite[Lemma 7.39]{KnightlyLi}. We have 
	\begin{align*}
		\vol(Z_v\bsl \frak{T}_{\Delta,v}^{+}) = \vol(Z_v\cap M_v)^{-1}\tint_{{\frak T}_{\Delta,v}}{\rm ch}_{M_v}(\tau)\,\d^\times \tau
		=q_v^{d_v/2}{\textstyle\iint}_{\substack{a,b\in \cO_v \\ a^2-\Delta_v^0m_v^2b^2\in \cO_v^\times}} \zeta_{E_v}(1)\tfrac{\d a \d b}{|a^2-\Delta^{0}_v(m_vb)^2|_v}.
	\end{align*}
	This equals $|m_v|_v^{-1}q_v^{-d_v/2}$ unless $|2|_v<1$ and $\tau=5$. If $|2|_v<1$ and $\tau=5$, then $E_v=\Q_2(\sqrt{5})$ and $\zeta_{E_v}(1)=(1-4^{-1})^{-1}=4/3$; thus $\vol(Z_v\bsl {\frak T}_{\Delta,v}^{+})=(4/3)q_v^{d_v/2}|m_v|_v^{-1}\,\iint_{U}\d a\d b$ with $U=\{(a,b)\in \Z_2|\,a^2-5b^2\in \Z_2^\times\}$. Since $U=\Z_2\times \Z_2-\{(2\Z_2)\times (2\Z_2) \sqcup \Z_2^\times \times \Z^\times_2\}$, we have $\vol(U)=1-(1/2)^2-(1/2)^2=1/2$. Hence $\vol(Z_v\bsl {\frak T}_{\Delta,v}^{+})=q_v^{-d_v/2}|m_v|_v^{-1}2/3$ as desired. Let $v\in \Sigma_\infty$ and $\Delta_v^{0}=-1$. Then ${\frak T}_{\Delta,v}=\{\left[\begin{smallmatrix} a & -b \\ b & a \end{smallmatrix}\right]|\,a,b\in \R,\,a^2+b^2\not=0\}$ is a direct product of $\R_{+}1_2$ and $\SO(2)$ and $\d^\times \tau=|m_v|_v^{-1}\Gamma_\C(1)\tfrac{\d r}{r}\,\d \theta$, where $\d r$ is the Lebesgue measure on $\R$ and $\int_{\SO(2)}\d \theta=2\pi$. Thus
	$\vol(Z_v\bsl {\frak T}_{\Delta,v}) = 2^{-1}\vol(\RR_{+} 1_2\bsl {\frak T}_{\Delta,v})=|m_v|_v^{-1}$. \end{proof}

Using the relation \eqref{fTGgamma}, we compute
{\allowdisplaybreaks
	\begin{align}
		\JJ_{\rm ell}(\bfs, \b)&=\tfrac{1}{2}\sum_{\tilde \gamma\in \cQ_{F}^{\rm Irr}} \int_{Z_{\AA}G_F\bsl G_{\AA}}\sum_{\xi \in G_{\gamma,F}\bsl G_F}\Phi(\bfs;g^{-1}\xi^{-1}\gamma\xi g) \cE_\b^*(g)dg \notag\\
		&=\tfrac{1}{2}\sum_{\tilde \gamma\in \cQ_{F}^{\rm Irr}} \int_{G_{\gamma, \A}\bsl G_\A}\Phi(\bfs;g^{-1}\gamma g)\,\{\int_{Z_\A G_{\gamma,F}\bsl G_{\gamma,\A}}\cE^{*}_\b(hg)\,\d h\}\,\d g
		\notag
		\\
		&=\tfrac{1}{2}\sum_{\tilde \gamma\in \cQ_{F}^{\rm Irr}} \int_{{\frak T}_{\Delta} \bsl G_\A}\Phi(\bfs;g^{-1}R_\Delta
		\gamma
		R_\Delta^{-1} g)\,(\cE_\b^{*})^{\Delta}(g)\,\d g \notag\\
		& =\tfrac{1}{2}\sum_{\tilde \gamma\in \cQ_{F}^{\rm Irr}} \int_{{\frak T}_{\Delta} \bsl G_\A}\Phi(\bfs;g^{-1} \hat \gamma g)\,(\cE_\b^{*})^{\Delta}(g)\,\d g,
		\label{10.2-f1}
\end{align}}where $\hat \gamma=(\hat \gamma_v)_{v\in \Sigma_F}$ denotes the element of $G_\A$ such that $ \hat \gamma_v=
\left[\begin{smallmatrix} \frac{t}{2m_v} & 1 \\ \Delta_v^{0} & \frac{t}{2m_v} \end{smallmatrix} \right]$ if $\Delta_v^{0}\not=1$ and $\hat \gamma_v=\left[\begin{smallmatrix} \frac{t}{2m_v}+1 & 0 \\ 0 & \frac{t}{2m_v}-1 \end{smallmatrix}\right]$ if $\Delta_v^0=1$, and 
\begin{align}
	(\cE_\b^*)^{\Delta}(g)&=\int_{Z_\A G_{\gamma,F}\bsl G_{\gamma,\A}}
	\cE_\b^{*} \left(h R_\Delta^{-1}g\right)\,\d h
	=\int_{\A^\times E^\times\bsl \A_E^\times}
	\cE_\b^{*} \left(\iota_{\Delta}(\tau)\,R_\Delta^{-1}
	g\right)\,\d^\times \tau
	. 
	\label{EllEisPer}
\end{align}

\subsection{Periods of Eisenstein series along elliptic tori} \label{PEAET}
We shall calculate the integral \eqref{EllEisPer}, which is absolutely convergent due to the compactness of $\A^{\times}E^\times \bsl \A_E^\times$. To attain this, let us recall the multiplicity one property of the Waldspurger model of the principal series $I(|\,|_v^{z/2})\,(z\in \C)$ and the explicit formula of associated spherical function. 

\begin{lem} \label{sphericalftn1value}
	Suppose $\Delta_v^0\not=1$. Let $f_{0,v} \in I(|\,|_v^{z/2})$ be a $\bK_v$-invariant vector determined by $f_{0,v}(1_2)=1$, and 
	set
	$$\varphi_{0,v}(g)=\tint_{Z_v \bsl {\frak T}_{\Delta,v}} f_{0,v}(t g)d t.$$
	Then we have
	$$\varphi_{0,v}(1_2) =|m_v|_v^{-1}\times \begin{cases} 1 &(v\in \Sigma_\infty), \\
	q_v^{-d_v/2} & (v\in \Sigma_\fin-\Sigma_{\rm dyadic}\,\text{and}\,\Delta_v^0 \in \go_v^\times-(\go_{v}^\times)^2), \\
	q_v^{-d_v/2}3^{-1}2(1+2^{-z}) & (v\in \Sigma_{\rm dyadic},\, \Delta_v^0 =5), \\
	q_v^{-d_v/2}(1+q_v^{(z+1)/2}) & (v\in \Sigma_\fin,\,\Delta_v^0\in \fp_v-\fp_v^2),\\
	q_v^{-d_v/2}(1+2^{-(z+1)/2}) & (v\in \Sigma_{\rm dyadic},\,\Delta_v^0 \in \{-5,-1\}).
	\end{cases}$$
\end{lem}
\begin{proof} We have the values of $\varphi_{0,v}(1_2)$ by Lemmas~\ref{fTDeltaL}
	and \ref{VolZfTd} easily. \end{proof}

\begin{lem}\label{MFoneSphftn}
	Let $v\in \Sigma_F$. Then there exists a unique smooth function $\varphi_v^{\Delta,(z)}:G_v\rightarrow \C$ such that $\varphi_v^{\Delta,(z)}(1_2)=1$ with the properties: 
	\begin{itemize}
		\item[(i)] $\varphi_v^{\Delta,(z)}(hgk)=\varphi_v^{\Delta,(z)}(g)$ for all $h\in {\frak T}_{\Delta,v}$ and $k\in \bK_v$. 
		\item[(ii)] $R(\bT_v)\varphi_v^{\Delta,(z)}=(q_v^{\frac{1+z}{2}}+q_v^{\frac{1-z}{2}})\,\varphi_v^{\Delta,(z)}$ if $v \in \Sigma_{\fin}$.
		\item[(ii)'] $R(\Omega_v)\varphi_v^{\Delta, (z)}=\frac{z^2-1}{2}\varphi_v^{\Delta,(z)}$ if $v \in \Sigma_{\infty}$.
	\end{itemize}
\end{lem}
\begin{proof}
	If $\Delta_v^{0}=1$, then ${\frak T}_{\Delta,v}=H_v$. Thus, the assertion follows from Lemma~\ref{FhypL1}. If $\Delta_v^{0}\not=1$, then the uniqueness of $\varphi_v^{\Delta,(z)}$ follows from \cite{Waldspurger} (\cite[Proposition 9']{Waldspurger0}). From Lemma~\ref{sphericalftn1value}, $\varphi_{0,v}(1_2)\not=0$. Then the function $\varphi_v^{\Delta,(z)}(g)=\varphi_{0,v}(1_2)^{-1}\varphi_{0,v}(g)$ meets all the requirements.  
\end{proof}

By exchanging the order of integrals,  
\begin{align}
	(\cE_\b^*)^{\Delta}(g)=\int_{L_\s}\b(z)\Lambda_F(z+1)\,E^{\Delta}(z;g)\,\d z
	\label{10.2-f3}
\end{align}
with 
\begin{align*}
	E^\Delta(z;g)=\int_{\A^\times E^\times \bsl \A_E^\times} E\left(z;\iota_{\Delta}(\tau) R_\Delta^{-1}g\right)\,\d^\times \tau, \quad z\in \C.  
\end{align*}
From Lemma~\ref{MFoneSphftn}, this is decomposed into a product as 
\begin{align}
	E^\Delta(z;g)=E^{\Delta}(z;1_2)\,\prod_{v\in \Sigma_F}\varphi_v^{\Delta,(z)}(g_v), \quad g\in G_\A. 
	\label{10.2-f2}
\end{align}
Let $\Re(z)>1$ and set $T=\iota_{\Delta}(E^\times)$. Since $BT = G$, by noting $T \cap B = Z$, we have
\begin{align}
	E^\Delta(z;1_2) =& \int_{Z_{\AA}\bsl T_{\AA}} y(h R_\Delta^{-1})^{(z+1)/2}dh = 
	\prod_{v \in \Sigma_F}Y_v(z) \quad \text{with} \quad
	Y_v(z)=\int_{Z_v\bsl T_v}y(h_v R_{\Delta,v}^{-1})^{(z+1)/2}\,dh_v.
	\label{EisPeriod-f0}
\end{align}

\begin{lem} \label{EisPeriod-L1}
	Let $D=\{v\in \Sigma_{\rm dyadic}|\,\Delta_v^{0}=5\,\}$ and $\Sigma_{\rm split}^{\Delta}=\{v\in \Sigma_\fin|\,\Delta_v^{0}=1\}$. Then
	\begin{align*}
		Y_v(z)=|m_v|_v^{-1}d(v)|\Delta_v^{0} m_v|_v^{-(z+1)/2}\zeta_{E_v}\left(\tfrac{z+1}{2}\right){\zeta_{F_v}(z+1)^{-1}}
		\times \begin{cases}
			|2|_v^{(z-1)/2} \quad (v\in \Sigma_{\rm split}^{\Delta}\cup \Sigma_\infty), \\
			1 \quad (v\in \Sigma_\fin-D\cup \Sigma_{\rm split}^{\Delta}), \\
			\tfrac{2}{3}(1+2^{-z}) \quad (v\in D),
		\end{cases}
	\end{align*}
	where we set $d(v)=q_v^{-d_v/2}$ for $v\in \Sigma_\fin$ and $d(v)=1$ for $v\in \Sigma_\infty$. 
\end{lem}
\begin{proof}
	Let us recall the formula
	$$
	y(g_v)^{(z+1)/2} = {d(v)^{-1}}{\zeta_{F,v}(z+1)^{-1}}|\det(g_v)|_v^{(z+1)/2} \tint_{F_v^{\times}}|a|_v^{z+1} \Phi_{v}(a[0,1]g)d^{\times}a,
	$$
	where $\Phi_v(x,y) = e^{-\pi(x^2+y^2)}$ if $v \in \Sigma_{\infty}$ and $\Phi_v(x,y)={\rm ch}_{\go_v\times \go_{v}}(x,y)$ if $v \in \Sigma_\fin$. Noting $\det (\iota_{\Delta}(\tau)) = \nr_{E/F}(\tau)$, we have
	\begin{align*}
		&Y_v(z)={d(v)^{-1}|\det(R_{\Delta, v}^{-1})|_v^{(z+1)/2}}{\zeta_{F_v}(z+1)^{-1}}\,Y_v^{0}(z)
	\end{align*}
	with 
	$$
	Y_v^{0}(z)=\tint_{F_v^\times \bsl E_v^\times}|\nr_{E_v/F_v}(\tau)|_v^{\frac{z+1}{2}}\{\int_{F_v^\times}|a|_v^{z+1}\Phi_v\biggl(a[0,1] \left[\begin{smallmatrix} x & y \\ 4^{-1}\Delta y & x \end{smallmatrix}\right] R_{\Delta,v}^{-1}\biggr)\d^\times a\}\,\d ^\times \tau,
	$$ where $\tau=x+y\sqrt{\Delta}/2$ and  $\d^\times\tau=\zeta_{E_v}(1)|x^2-4^{-1}\Delta y^2|_v^{-1}\,\d x\,\d y$. 
	
	(i) Suppose $v\in \Sigma_\fin$ and $\Delta_v^{0}=1$.  Then we have
	\begin{align*}
		Y_v^{0}(z)= & |m_v|_v^{-1}
		\zeta_{E_v}(1)\tint_{F_v^\times \times F_v^\times}|4XY|_{v}^{(z+1)/2}
		\Phi_v([X, Y]) \,|2XY|_v^{-1}\,\d X\,\d Y=|m_v|_v^{-1}\,|4|_v^{\frac{z}{2}}q_v^{-d_v} \zeta_{E_v}\left(\tfrac{z+1}{2} \right), 
	\end{align*}
	by the variable change $X=\frac{a(m_v y+x)}{2}$, $Y=\frac{a(m_vy-x)}{2}$ and by the relation $\Phi_v([X,Y])=\delta(X\in \cO_v)\delta(Y\in \cO_v)$.
	
	(ii) Suppose $v\in \Sigma_\fin-T\cup D$. Then we easily have
	\begin{align*}
		Y_v^{0}(z) = &|m_v|_v^{-1}
		\tint_{E_v^{\times}}|\nr_{E_v/F_v}(\tau_0)|_{v}^{(z+1)/2}
		\Phi_v([\Delta_v^0 y_0,x]) \,\d^\times \tau_0,
	\end{align*}
	with $y_0=m_v y$ and $\tau_0=x+\sqrt{\Delta_v^{0}}\, y_0$. From $\Phi_v([\Delta_v^{0}\, y_0,x])=\delta(x\in \cO_v)\delta(y_0\in (\Delta_v^{0})^{-1}\cO_v)$, 
	\begin{align*}
		Y_v^0(z)=&|m_v|_v^{-1}
		\tint_{x\in \cO_v-\{0\}} \int_{y_0\in \cO_v/\Delta_v^{0}-\{0\}} 
		|x^2-\Delta_v^{0}\, y_0^2|_{v}^{\frac{z+1}{2}-1}
		\zeta_{E_v}(1)\,\d x\,\d y_0
		\\
		= & |m_v|_v^{-1}|\Delta_v^{0}|_v^{-(z+1)/2}
		{\textstyle\iint}_{\cO_v^{2}-\{(0,0)\}}|y^2-\Delta_v^{0}\, x^2|_v^{\frac{z-1}{2}}\zeta_{E_v}(1)\d x\,\d y \\
		= & |m_v|_v^{-1} |{\Delta_v^{0}}|_v^{-\frac{z+1}{2}}
		\,\vol(\go_{E_v}^{\times}, d^\times\tau )\zeta_{E_v}\left(\tfrac{z+1}{2}\right)\end{align*}
	by making the variable change $y=\Delta_v^{0}y_0$ and noting $\go_{E_v} = \go_{v}+\sqrt{\Delta_v^0}\go_{v}$.
	
	(iii) Let $v\in D$. By the same computation as in (ii), we have 
	$$
	Y_v^0(z)=|m_v|^{-1}\tint_{\go_v-\{0\}}\tint_{\go_v-\{0\}}|y^2-\Delta_v^0 x^2|_v^{(z-1)/2}\zeta_{E_v}(1)dxdy.
	$$
	We decompose the set $(\cO_v-\{0\})\times (\cO_v-\{0\})$ into the disjoint union of $D_1=(\fp_v-\{0\}) \times (\cO_v-\{0\})$, $D_2=\cO_v^\times \times (\fp_v-\{0\})$ and $D_3=\cO_v^\times \times \cO_v^\times$ and write $|m_v|_v\,Y_v(z)=I_1+I_2+I_3$, where $I_i=\iint_{D_i} |y^2-\Delta_v^0 x^2|_v^{(z-1)/2} \zeta_{E_v}(1)dxdy$ with $i=1,2,3$. We have
	\begin{align*}
		I_1 =& \left(\tint_{\gp_v-\{0\}}\tint_{\go_v^\times} +\tint_{\gp_v-\{0\}}\int_{\gp_v-\{0\}}\right) |y^2-\Delta_v^0 x^2|_v^{(z-1)/2} \zeta_{E_v}(1)dxdy \\
		=&  (1+q_v)^{-1}q_v^{-d_v}  + q_v^{-z-1} |m_v|_v\,Y_v^{0}(z),
	\end{align*}
	and $I_2 = (1+q_v)^{-1}q_v^{-d_v}$ easily.  Since $v\in \Sigma_{\rm dyadic}$, we have that $u\in \cO_v^\times$ is a square if and only if $u\pmod{4\fp_v}$ is a square in $\cO_v^\times/1+4\fp_v$. In particular, $\xi^2-\Delta_v^0 \notin 4\gp_v$ for $\xi\in \cO_v^\times$. By this remark, we compute 
	\begin{align*}
		I_3 &= \tint_{\go_v^\times}\tint_{\go_v^\times}|x^2y^2-\Delta_v^0 x^2|_v^{(z-1)/2} \zeta_{E_v}(1)dxdy
		= (1+q_v^{-1})^{-1}q_v^{-d_v/2} \tint_{\go_v^\times}|y^2-\Delta_v^0|_v^{(z-1)/2} dy\\
		&= (1+q_v^{-1})^{-1}q_v^{-d_v/2} \sum_{\xi\in \go_v^\times / 1+4\gp_v} \tint_{\go_v}|\{\xi(1+4\varpi_vu)\}^2 - \Delta_v^0 |_v^{(z-1)/2}
		|4\varpi_v|_v du \\
		&= (1+q_v^{-1})^{-1}q_v^{-d_v}|4\varpi_v|_v \sum_{\xi\in \go_v^\times / 1+4\gp_v} |\xi^2 -\Delta_v^0|_v^{(z-1)/2}= (1+q_v^{-1})^{-1}q_v^{-d_v}|4|_v q_v^{-1}\sum_{j=0}^{\ord_v(4)}u(j)q_v^{-j(z-1)/2},
	\end{align*}
	where $u(j)$ is the number of all $\xi \in \go_v^{\times}/(1+4\gp_v)$ such that $\ord_v(\xi^2-\Delta_v^0)=j$. Hence 
	$$|m_v|_vY_v^{0}(z)=\zeta_{E_v}(\tfrac{z+1}{2}) \{2(1+q_v)^{-1}+(1+q_v^{-1})^{-1}|4|_vq_v^{-1}\sum_{j=0}^{\ord_v(4)}u(j)q_v^{-j(z-1)/2}\}=\zeta_{E_v}(\tfrac{z+1}{2}) 
	\tfrac{2}{3}(1+2^{-z})
	$$
	by using $q_v=2$, $\ord_v(4)=2$, $u(0)=u(1)=0$ and $u(2)=4$. We omit the detail of the computation for archimedean cases (iii) and (iv), which are elementary.\end{proof}

The following formula is originally due to Hecke (\cite[Chap. II, \S 3]{Siegel}).

\begin{prop} \label{EisPeriod}
	
	We have
	\begin{align*}
		E^\Delta(z;1_2) = &\{\prod_{v\in \Sigma_F}|m_v|_v^{-1}\}\,\{\prod_{\substack{v \in \Sigma_{F}\\
				\Delta_v^{0}=1}}|2|_v^{-1}\}\, 
		D_F^{-1/2}\,\nr(
		\fd_{E/F})^{\frac{z+1}{4}}\, \frac{\zeta_{F}((z+1)/2)L((z+1)/2,\varepsilon_{\Delta})}{\zeta_{F}(z+1)} \\
		& \times \prod_{\substack {v\in \Sigma_\infty \\ \Delta^{0}_v=-1}} 2^{-1} \,\prod_{\substack{v \in \Sigma_{\rm{dyadic}} \\ \Delta_v^{0}=5}}{3^{-1}}2^{\frac{z+1}{2}+1}(1+2^{-z}),
	\end{align*}
	where $E=F(\sqrt{\Delta})$ and $\fd_{E/F}$ denotes the relative discriminant of $E/F$. 
\end{prop}
\begin{proof} 
	From Lemma~\ref{EisPeriod-L1} and \eqref{EisPeriod-f0}, we have that $E^{\Delta}(z;1_2)$ with $\Re(z)>1$ is the product of $\zeta_E(\frac{z+1}{2})\zeta_{F}(1+z)^{-1}$ and finite factors
	\begin{align}
		\prod_{v\in \Sigma_F}|m_v|_v^{-1}\times D_F^{-1/2} \times \prod_{v\in \Sigma_F}|m_v\Delta_v^0|_v^{-\frac{z+1}{2}}\times \prod_{v\in \Sigma_{\rm split}^{\Delta}\cup \Sigma_{\infty}}|2|_v^{\frac{z-1}{2}} \times \prod_{v\in D}\tfrac{2}{3}(1+2^{-z}).
		\label{EisPeriod-f1}
	\end{align}
	By $1=\prod_{v\in \Sigma_F}|4^{-1}\Delta|_v=\prod_{v\in \Sigma_F}|\Delta_v^{0}|_v|m_v|_v^{2}=\prod_{v\in \Sigma_F}|\Delta_v^{0}|_v \prod_{v\in \Sigma_F}|m_v|_v^2$, 
	$$\prod_{v\in \Sigma_F}
	|\Delta_v^0 m_v|_v^{-(z+1)/2}=
	\prod_{v\in \Sigma_F}|\Delta_v^{0}|_v^{-\frac{z+1}{4}}=\nr({\frak D}_{\Delta})^{\frac{z+1}{4}},
	$$
	where ${\frak D}_{\Delta}=F\cap \prod_{v\in \Sigma_\fin}\Delta_v^0\cO_v$. Using the relation $\nr({\frak D}_\Delta)=\nr(\fd_{E/F})\prod_{v\in \Sigma_{\rm split}^{\Delta} \cup D \cup \Sigma_\infty}|2|_v^{-2}$, we see that the factor \eqref{EisPeriod-f1} is simplified in the desired form.  
\end{proof}

\begin{lem} \label{EisPeriodEst}
	For $M>1$ and $\e>0$,  
	\begin{align*}
		|(z^2-1)\Lambda_F(z+1)\,E^{\Delta}(z;1_2)|\ll_{M, \e} \{\prod_{v\in \Sigma_F}|m_v|_v^{-1}\}\,
		\nr({\frak D}_{\Delta})^{\frac{1+\e}{4}+\frac{\ro(z)}{4}}
	\end{align*}
	uniformly in $\Delta\in F^\times-(F^\times)^2$ and $z\in \C$ such that $\Re(z)\in[-M, M]$, where $\ro(z)=\max(|\Re(z)|,1)$. 
\end{lem}
\begin{proof} The function $(z^2-1)\zeta_F(\frac{z+1}{2})$ is holomorphic and vertically bounded on $\C$.
	By two estimates $|L(\frac{z+1}{2},\varepsilon_{\Delta})|\ll_{M,\e} \nr({\frak D}_{\Delta})^{\e}$ for $\Re(z) \in [1, M]$
	and $|L(\frac{z+1}{2},\varepsilon_{\Delta})|\ll_{M,\e} \nr({\frak D}_{\Delta})^{-\frac{\Re(z)}{2}+\e}$
	for $\Re(z)\in [-M,-1]$,
	a well-known argument by the Phragmen-Lindel\"{o}f principle yields a bound $|L(\tfrac{z+1}{2},\varepsilon_{\Delta})|\ll_{\e}\{\nr({\frak D}_{\Delta})\}^{\frac{1}{4}+\e-\frac{\Re(z)}{4}}$ uniformly valid for $\Re(z)\in [-1,1]$ and for any non-square $\Delta$. From this and Proposition~\ref{EisPeriod}, we have the desired bound easily. 
\end{proof}

\subsection{Explicit formulas of local orbital integrals} \label{EFLOInt}
For $v\in \Sigma_F$, let $\Phi_v(g_v)$ be the $v$-component of the test function $\Phi(g)=\Phi^{l}(\fn|\bfs,g)$ on $G_\A$. For $z\in \C$, set\begin{align}
	\fE_v^{(z)}(\hat \gamma_v)=\int_{{\frak T}_{\Delta,v}\bsl G_v}\Phi_v\left(g^{-1} \hat\gamma_v g\right)\,\varphi_v^{\Delta,(z)}(g)\,\d g 
	\label{EllipticorbitalIntegral}
\end{align}
with $\hat\gamma_v=\left[\begin{smallmatrix} {\frac{t}{2m_v}} & 1 \\ \Delta_v^{0} & {\frac{t}{2m_v}} \end{smallmatrix} \right]$ if $\Delta_v^{0}\not=1$ and $\hat\gamma_v=\left[\begin{smallmatrix} {\frac{t}{2m_v}+1} & 0 \\ 0 & {\frac{t}{2m_v}-1} \end{smallmatrix} \right]
$ if $\Delta_v^{0}=1$. 

\begin{thm} \label{OrbIntUnifEx}
	\begin{itemize}
		\item[(1)] Suppose $v\in \Sigma_\infty$. Then for $|\Re(z)|<2l_v-1$, we have 
		\begin{align*}
			\fE_v^{(z)}(\hat \gamma_v) = 2|m_v|_v \times \Ocal_v^{{\rm{sgn}}(\Delta_v^{0}),(z)}\left(\tfrac{t}{2m_v}\right).
		\end{align*}
		\item[(2)] Suppose $v\in \Sigma_\fin-(S\cup S(\fn)\cup \Sigma_{\rm dyadic})$ or $v\in \Sigma_{\rm dyadic}$ with $\Delta_v^{0}\neq 5$. Then we have 
		\begin{align*}
			\fE_v^{(z)}(\hat \gamma_v)=q_v^{-d_v/2}|2m_v|_v\,\Ocal_{v,0}^{\Delta_v^{0},(z)}\left(\tfrac{n}{4m_v^2}\right)
			\,\delta\left(\tfrac{t}{2m_v}\not\in \cO_v^{\times}\,{\text{or}}\, \tfrac{n}{4m_v^2}\not\in \fp_v\right)
		\end{align*}
		if $\Delta_v^0=1$, and otherwise
		\begin{align*}
			\fE_v^{(z)}(\hat \gamma_v)=q_v^{-d_v/2}|m_v|_v
			\,\Ocal_{v,0}^{\Delta_v^{0},(z)}\left(\tfrac{n}{m_v^2}\right)
			\begin{cases} 1 \quad &(v\not\in \Sigma_{\rm dyadic},\,\Delta_v^{0}\in \cO_v^\times-(\cO_v^\times)^2, ), \\
				\delta\left(\tfrac{t}{2m_v}\not\in \cO_v^{\times}\right) \quad &(v\in \Sigma_{\rm dyadic},\,\Delta_v^{0}\in \{-1,-5\}), \\
				\delta\left(\tfrac{t}{2m_v} \not\in \fp_v\right) \quad &(\Delta_v^{0}\in \fp_v-\fp_v^2)
				.
			\end{cases} 
		\end{align*}
		\item[(3)] Suppose $v\in \Sigma_{\rm dyadic}$ with $\Delta_v^{0}=5$. Then 
		\begin{align*}
			\fE_v^{(z)}(\hat \gamma_v)&=q_v^{-d_v/2}|2m_v|_v\,2^{\frac{-z-1}{2}}\,3(1+2^{-z})^{-1}\,\Ocal_v\left(\tfrac{n}{4m_v^2}\right).
		\end{align*}
		\item[(4)] Suppose $v\in S(\fn)$. Then 
		\begin{align*}
			\fE_{v}^{(z)}(\hat\gamma_v)&=q_v^{-d_v/2}|m_v|_v\,\Ocal_{v,1}^{\Delta_v^0,(z)}\left(\tfrac{n}{m_v^2}\right)\times \begin{cases} \delta\left(\tfrac{t}{2m_v} \not\in \cO_v^{\times}\, \text{or}\, \tfrac{n}{m_v^2}\not \in
				\fp_v\right) \quad (\Delta_v^{0}=1), \\
				\delta\left(\tfrac{n}{m_v^2} \not\in \cO_v\right) \quad (\Delta_v^{0} \in \cO_v^\times-(\cO_v^\times)^2), \\
				\delta\left(\tfrac{n}{m_v^2}\not\in \fp_v\right) \quad (\Delta_v^{0}\in \fp_v-\fp_v^2).
			\end{cases} 
		\end{align*}
		\item[(5)] Suppose $v\in S$. Then for $\Re(s_v)>(|\Re(z)|-1)/2$, we have 
		\begin{align*}
			\fE_v^{(z)}(\hat \gamma_v)=q_v^{-d_v/2}|m_v|_v\,\Scal_v^{\Delta_v^{0}, (z)}\left(s_v;\tfrac{n}{m_v^2}\right). 
		\end{align*}
	\end{itemize}
\end{thm}
\begin{proof} The case $\Delta_v^{0}\not=1$ will be treated in \S~\ref{Localorbitalintegrals}. Suppose $\Delta_v^{0}=1$. Since $\tilde\gamma \in \Qcal_F^{\rm{Irr}}$, we have $n\not=0$, which in turn implies $t\not=2m_v$. Recall $\mathfrak{T}_{\Delta,v}=H_v$. Set $a =\frac{t}{2m_v}\in F_v-\{1\}$ and $b=\frac{a+1}{a-1}$. Then, by adjusting the measures, we have
	{\allowdisplaybreaks\begin{align*}
			\fE_v^{(z)}(\hat \gamma_v) = &|2m_v|_v \tint_{H_v \bsl G_v}\Phi_v\left(g^{-1} \left[\begin{smallmatrix} b & 0 \\ 0 & 1 \end{smallmatrix} \right] g\right)\,\varphi_v^{(0,z)}(g)\,\d g \\
			=& |2m_v|_v\tint_{F_v}\left(\int_{\bfK_v}\Phi_v\left(k^{-1} \left[\begin{smallmatrix} b & (b-1)x \\ 0 & 1 \end{smallmatrix} \right] k\right)\,\d k\, \right)\varphi_v^{(0, z)}(\left[\begin{smallmatrix} 1 & x \\ 0 & 1 \end{smallmatrix} \right] )\,\d x
			=|2m_v|_v\,\fF_v^{(z)}(b),
		\end{align*}
	}where $\fF_v^{(z)}(b)$ is the integral treated in \S~\ref{HypSecondStep}. The case $\Delta_v^{0}=+1$ of (1) follows from Theorem~\ref{explicit hyp} (1). The first case of (2) follows from Theorem~\ref{explicit hyp} (2) if one notes the relation $\delta(\tfrac{a+1}{a-1} \not \in \cO_v^\times)=\delta(a\in \cO_v^\times,a^2-1\in4\fp_v)$ by the identity $|\tfrac{a+1}{a-1}-1|_v^{2}=|\tfrac{4}{a^2-1}|_v$ for $a$ with $|a-1|_v=|a+1|_v$, which is seen to be valid as $|\tfrac{a+1}{a-1}-1|^{2}=\tfrac{|4|_v}{|a-1|_v^2}=\tfrac{|4|_v}{|a-1|_v|a+1|_v}=\tfrac{|4|_v}{|a^2-1|_v}$. The first case of (4) follows from Theorem~\ref{explicit hyp} (3) similarly. The case $\Delta_v^{0}=1$ in (5) follows from Theorem~\ref{explicit hyp} (4) by $|\tfrac{a+1}{a-1}-1|_v=\frac{|2|_v}{|a-1|_v}=|a|_v^{-1}=|a^2-1|^{-1/2}$ if $|a|_v>1$ and by $|\frac{a+1}{a-1}-1|_v^{-2}|\tfrac{a+1}{a-1}|_v=|a^2-1|_v$ if $|a|_v\leq 1$. 
\end{proof}

\begin{lem} \label{OrbIntUnifExL1}
	Let $v\in \Sigma_\fin$. 
	Then 
	\begin{align*}
		&|\tfrac{t}{2m_v}|_v=|\Delta_{v}^{0}\tfrac{t^2}{t^2-4n}|_v^{1/2}, \quad 
		|\tfrac{n}{m_v^2}|_v=|\Delta_v^{0}\tfrac{4n}{t^2-4n}|_v=|(\tfrac{t}{2m_v})^2-\Delta_v^{0}|_v.
	\end{align*}
	If $|\tfrac{t}{2m_v}|_v<1$, then $|\tfrac{n}{m_v^2}|_v=|\Delta_v^{0}|_v$. If $|\tfrac{t}{2m_v}|_v=1$, then 
	$$|\tfrac{n}{m_v^2}|_v=\begin{cases}|t|_v^{-2}|4n|_v \quad (\Delta_v^{0}=1,\, |t|_v^{2}>|4n|_v), \\  
	1 \quad (\Delta_v^{0}\not=1\,\text{or}\, \Delta_v^{0}=1,\,|t|_v^2\leq |4n|_v). 
	\end{cases}
	$$
	except when $\Delta_v^0\in \cO_v^\times-(\cO_v^\times)^2$ and $v$ is dyadic. If $|\tfrac{t}{2m_v}|_v>1$, then $|\tfrac{n}{m_v^2}|_v=|\frac{4n\Delta_{v}^{0}}{t^2-4n}|_v$. 
\end{lem}
\begin{proof}
	This follows from the relation $t^2-4n=\Delta_v^0(2m_v)^{2}$. 
\end{proof}

For $v\in \Sigma_\fin$ and $m\in \N$, we set $U_v(m)=1+\fp_v^{m}$. 

\begin{lem} \label{OrbIntUnifExL2}
	Let $v\in \Sigma_\fin$.  
	\begin{itemize}
		\item[(1)] Suppose $v\not\in \Sigma_{\rm dyadic}$, $\Delta_v^{0}\in \cO_v^{\times}-(\cO_v^\times)^2$. Then 
		\begin{align*}
			|\tfrac{t}{2m_v}|_v<1 \, &\Longleftrightarrow \, |t|_v^2<|4n|_v, \\
			|\tfrac{t}{2m_v}|_v=1 \, &\Longleftrightarrow \, t\not=0,\, \tfrac{4n}{t^2} \in \cO_v^\times-U_v(1), \\
			|\tfrac{t}{2m_v}|_v>1 \, &\Longleftrightarrow \, t\not=0,\,\tfrac{4n}{t^2}\in U_v(1).
		\end{align*}
		\item[(2)] Suppose $\Delta_v^{0}\in \fp_v-\fp_v^2$. Then $t\not=0,\,\tfrac{4n}{t^2}\in \cO_v^\times-U_v(1)$ never happens. We have 
		\begin{align*}
			|\tfrac{t}{2m_v}|_v<1 \, &\Longleftrightarrow \, |t|_v^2<|4n|_v, \\
			|\tfrac{t}{2m_v}|_v=1 \, &\Longleftrightarrow \, t\not=0,\, \tfrac{4n}{t^2} \in U_v(1)-U_v(2), \\
			|\tfrac{t}{2m_v}|_v>1 \, &\Longleftrightarrow \, t\not=0,\,\tfrac{4n}{t^2}\in U_v(2).
		\end{align*}
		\item[(3)] Suppose $\Delta_v^{0}=1$. Then 
		\begin{align*}
			|\tfrac{t}{2m_v}|_v<1 \, &\Longleftrightarrow \, |t|_v^2<|4n|_v, \\
			|\tfrac{t}{2m_v}|_v=1 \, &\Longleftrightarrow \, |t|_v^2>|4n|_v\quad\text{or}\quad t\not=0,\, \tfrac{4n}{t^2} \in \cO_v^\times-U_v(1), \\
			|\tfrac{t}{2m_v}|_v>1 \, &\Longleftrightarrow \, t\not=0,\,\tfrac{4n}{t^2}\in U_v(1).
		\end{align*}
		\item[(4)]  Suppose $F_v=\Q_2$ and $\Delta_v^{0}=5$. Then 
		\begin{align*}
			|\tfrac{t}{2m_v}|_v<1 \, &\Longleftrightarrow \, |t|_v^2<|4n|_v, \\
			|\tfrac{t}{2m_v}|_v=1 \, &\Longleftrightarrow \, t\not=0,\, \tfrac{n}{t^2} \in \cO_v^\times, \\
			|\tfrac{t}{2m_v}|_v>1 \, &\Longleftrightarrow \, t\not=0,\,\tfrac{4n}{t^2}\in \cO_v^\times.
		\end{align*}
		Suppose $F_v=\Q_2$ and $\Delta_v^0\in \{-1,-5\}$. Then the above equivalences hold if we replace the $n$ in the second equivalence with $2n$.
	\end{itemize}
\end{lem}
\begin{proof}
	Since $\Delta\not\in (F^\times)^2$, we have $n\not=0$; thus $|t|_v^2\geq |4n|_v$ implies $tn\not=0$. We prove the case (1). From $t^2-4n=\Delta_v^{0}(2m_v)^2$, we have $|(\tfrac{t}{2m_v})^2-\Delta_v^0|_v=|\tfrac{n}{m_v^2}|_v$ and $|2m_v|_v^2=|t^2-4n|_v$. Assume $|\frac{t}{2m_v}|_v\leq 1$. Then $|\frac{n}{m_v^2}|_v=|(\frac{t}{2m_v})^2-\Delta_v^0|_v=1$ by Lemma~\ref{Hensel}. Hence $|4n|_v=|2m_v|_v^2=|t^2-4n|_v$, which yields $|t|_v^2\leq |4n|_v$. If $|t|_v^2=|4n|_v$, then $|t|_v^2\leq|2m_v|_v^2=|t^2-4n|_v$ gives us $|1-\frac{4n}{t^2}|_v=1$, or equivalently $\tfrac{4n}{t^2} \in \cO_v^\times-U_v(1)$, which in turns implies $|\tfrac{t}{2m_v}|_v=1$ as $|4m_v^2|_v=|t|_v^2|1-\tfrac{4n}{t^2}|_v=|t|_v^2$. The condition $|t|_v^2<|4n|_v$ implies $|t|_v^2<|t^2-4n|_v=|4m_v^2|_v$. This completes the first two cases of (1).  
	
	Assume $|\frac{t}{2m_v}|_v>1$. Then $|t|_v^2>|4m_v^2|_v=|t^2-4n|_v$, which yields $|t|_v^2\leq |4n|_v$. From this remark, we obtain the third case of (1) from the first two cases. 
	
	The remaining assertions are proved similarly.
\end{proof}

The following proposition combined with Theorem~\ref{OrbIntUnifEx} gives us a convenient description of the local orbital integrals $\fE_v^{(z)}(\hat \gamma_v)$ in terms of $(t:n)_F$. 

\begin{prop} \label{OrbIntUnifExL4}
	Let $(t:n)_F\in \Qcal^{\rm Irr}_F$ and $v\in \Sigma_\fin$. Write $t^2-4n=(2m_v)^2\Delta_v^{0}$ with $m_v\in F_v$, $\Delta_v^0\in (\fp_v-\fp_v^2)\cup (\cO_v^\times-(\cO_v^\times)^2)\cup \{1\}$. Then 
	\begin{itemize}
		\item[(1)] Suppose $\ord_v(t^2-4n) \in 2\Z$ and $t^2-4n\not\in (F_v^\times)^2$. If $v\not\in \Sigma_{\rm dyadic}$, then $|t|_v^2\leq |4n|_v$ and
		\begin{align*}
			|\tfrac{n}{m_v^2}|_v=
			\begin{cases}
				|\tfrac{4n}{t^2-4n}|_v \quad &(t\not=0,\tfrac{4n}{t^2} \in U_v(1)), \\
				1 \quad &(\text{otherwise}).
			\end{cases} 
		\end{align*}
		If $F_v\cong \Q_2$ and $\Delta_v^0=5$, then $|t|_v^2\leq |n|_v$ and 
		\begin{align*}
			|\tfrac{n}{m_v^2}|_v=\begin{cases}
				1 \quad &(|t|_v^2<|4n|_v \quad \text{or} \quad t\not=0,\tfrac{n}{t^2}\in \cO_v^\times), \\
				|\tfrac{4n}{t^2-4n}|_v \quad &(t\not=0,\,\tfrac{4n}{t^2} \in \cO_v^\times).
			\end{cases}
		\end{align*}
		If $F_v\cong \Q_2$ and $\Delta_v^{0}\in \{-1,-5\}$, then $|t|^2_v< |n|_v$ and\begin{align*}
			|\tfrac{n}{m_v^2}|_v=\begin{cases}
				1 \quad &(|t|_v^2<|4n|_v),\\
				2^{-1} \quad &(t\not=0,\,\tfrac{2n}{t^2}\in \cO_v^\times), \\
				|\tfrac{4n}{t^2-4n}|_v \quad &(t\not=0,\,\tfrac{4n}{t^2} \in \cO_v^\times).
			\end{cases}
		\end{align*}
		\item[(2)] If $\ord_v(t^2-4n)\in 1+2\Z$. Then $|t|_v^2\leq |4n|_v$, and 
		\begin{align*}
			|\tfrac{n}{m_v^2}|_v=\begin{cases} q_v^{-1} \quad &(|t|_v^2<|4n|_v)
				, \\
				1 \quad &(t\not=0,\tfrac{4n}{t^2}\in U_v(1)-U_v(2)), \\
				q_v^{-1}|\tfrac{4n}{t^2-4n}|_v \quad &(t\not=0,\tfrac{4n}{t^2} \in U_v(2)).
			\end{cases}
		\end{align*}
		\item[(3)] If $t^2-4n\in (F_v^\times)^2$. Then 
		\begin{align*}
			|\tfrac{n}{m_v^2}|_v=\begin{cases} 1 \quad &(|t|_v^2<|4n|_v \quad \text{or}\quad t\not=0,\tfrac{4n}{t^2}\in \cO_v^\times-U_v(1)), \\
				|t|_v^{-2}|4n|_v \quad &(|t|_v^2>|4n|_v), \\
				|\tfrac{4n}{t^2-4n}|_v \quad &(t\not=0, \tfrac{4n}{t^2}\in U_v(1)). 
			\end{cases}
		\end{align*}
	\end{itemize}
\end{prop}
\begin{proof}
	This follows from Lemma~\ref{OrbIntUnifExL1}, \ref{OrbIntUnifExL2} immediately. 
\end{proof}

\subsection{The absolute convergence of the $F$-elliptic term} \label{ABCelliptic}
In this subsection, we prove the following. 

\begin{thm} \label{ABCTheorem}
	For any sufficiently small $\e>0$, we have
	\begin{align}
		\sum_{\gamma=(t:n)_F\in \Qcal_{F}^{\rm{Irr}}}|(z^2-1)\Lambda_F(z+1)E^{\Delta}(z;1_2)|\prod_{v\in \Sigma_F}|\fE_v^{(z)}(\hat\gamma_v)|
		<+\infty
		\label{ABCf0}
	\end{align}
	uniformly in $(z,\bfs)\in \C\times \fX_S$ such that $\Re(z) \in [-2\underline{l}+3+\e, \underline{l}-1/2-\e]$, $\min_{v\in S}\Re(s_v)>2{\underline l}+d_F-1$.
\end{thm}

To prove this, we need estimates of local orbital integrals, which are given by the following lemmas. 

\begin{lem}\label{ABC-L1234}
	\begin{itemize}
		\item[(1)] We have $\Ocal_{v,0}^{\delta,(z)}(a)=1$ if $a\in \cO_v^\times$, and 
		\begin{align*}
			|\Ocal_{v,0}^{\delta,(z)}(a)|\leq 3|\ord_v(a)|\,|a|_v^{\frac{|\Re(z)|+1}{4}}, \quad a\in F_v-\cO_v.
		\end{align*}
		\item[(2)] 
		We have $|\Ocal_{v,1}^{\delta,(z)}(a)|\le 2(1+q_v)^{-1}$ if $a\in \cO_v^\times$, and
		\begin{align*}
			|\Ocal_{v,1}^{\delta,(z)}(a)|\leq \tfrac{4(1+q_v^{1/2})}{1+q_v}(1+|\ord_v(a)|)|a|_v^{\frac{|\Re(z)|+1}{4}}, \quad a\in F_v-\cO_v.
		\end{align*}
		\item[(3)]  Suppose $\Re(s)>1$ and $\Re(s)>(|\Re(z)|+1)/2$. Then for any $\e>0$, we have
		\begin{align*}
			|\Scal_v^{\delta,(z)}(s;a)|\leq 16\{1+\max(0,-\ord_v(a))\}\max(1,|a|_v^{-1})^{\frac{-\Re(s)+|\Re(z)|}{4}+\e}|a|_v^{\frac{|\Re(z)|+1}{4}+\e}\end{align*}
		for any $a\in F_v^\times$ with $\ord_v(a)\in \NN_0\cup(-2\NN)$. 
		\item[(4)] 
		Let $v\in \Sigma_\infty$. For any $0<\s<2l_v-1$,  
		\begin{align*}
			|\Ocal_v^{+,(z)}(a)|\ll \delta(|a|_v>1)\,\min(1,|a-1|_v)^{l_v/2}|a|_v^{\frac{|\Re(z)|+1}{2}}, \quad a\in F_v^\times
		\end{align*}
		uniformly for $z\in \C$ on the strip $|\Re(z)|\le \s$. 
		\item[(5)] 
		Let $v\in \Sigma_\infty$. For any subinterval $(\s_1,\s_2)\subset (-2l_v+1, l_v-1/2)$, we have    
		\begin{align*}
			|\Ocal_v^{-,(z)}(a)| \ll (1+|a|_v)^{\frac{1-2l_v}{4}}, \quad a\in F_v^\times 
		\end{align*}
		uniformly for $z\in \C$ on the strip $\s_1\le \Re(z) \le \s_2$. 
	\end{itemize}
\end{lem}
\begin{proof} 
	When $\delta \in (F_v^\times)^2$, the estimates in (1), (2) and (3) follow from the proof of Lemma~\ref{LocalHyp-almostall-esti}, Lemma~\ref{Sfn-Est} and Corollary~\ref{S-Est}, respectively. The proofs are modified to be applied to the case $\delta\in F^\times_v-(F_v^\times)^2$. The estimate (4) is given by Corollary~\ref{FHypL11}. From \S \ref{EllOrbAcase}, to prove (5) it suffices to estimate
	$\int_{0}^{\infty}\frac{y^{l_v+\frac{z-1}{2}}}{(y^2+1\pm 2ayi)^{l_v-1/2}}d^\times y$,
	which is majorized by the sum of
	$$\int_{0}^{1}\frac{y^{l_v+\frac{\Re(z)-1}{2}}}{(y^2+1)^{l_v-1/2}} d^\times y\ll 1, \qquad \int_{1}^{\infty}\frac{y^{l_v+\frac{\Re(z)-1}{2}}}{(4|a|_v y^3)^{l_v/2-1/4}}d^\times y
	\ll  |a|_v^{1/4-l_v/2}.
	$$
\end{proof}

To be precise, let $a\mapsto a^{(v)}$ denote the natural embedding $F\hookrightarrow F_v$. As usual, $\ord_v(a^{(v)})$ and $|a^{(v)}|_v$ for $a\in F^\times$ will be abbreviated to $\ord_v(a)$ and $|a|_v$, respectively.

\begin{lem} \label{levelcondition}
	Let $\gamma=(t:n)_{F}\in \Qcal_F^{\rm Irr}$ be such that $t,n\in \cO$ and  
	\begin{itemize}
		\item[(a)] For all $v\in S(\fn)$, $4n\cO_v+t\cO_v=\cO_v$, 
		\item[(b)] $\prod_{v\in S(\fn)}\fE_v^{(z)}(\hat\gamma_v)\not=0$. 
	\end{itemize}
	Then $4n$ is relatively prime to $\fn$. If we define a divisor $\fn_2$ of $\fn$ by $S(\fn_2)=\{v\in S(\fn)|\Delta_v^0=1,\,|t|_v=|n|_v=|m_v|_v=1\}$ and set $\fn_1=\fn\fn_2^{-1}$, $\fc=t\cO+\fn$ and $\fc'=\fn_1\fc^{-1}$, then $\fn_1=\fc\fc'$, $((t^2-4n)^{(v)})_{v\in S(\fc)}\in \prod_{v\in S(\fc)}(\cO_v^\times)^2$ and $(t^2-4n)\cO\subset (\fc')^2$. 
	
	If $n$ further satisfies the condition
	\begin{itemize}
		\item[(c)] $n \pmod{\fp_v}$ is not a square residue for some $v\in S(\fn)$, 
	\end{itemize}
	then $t\not=0$. 
\end{lem}
\begin{proof}
	Let $v\in S(\fc)$. Then $t^{(v)}\in \fp_v$. From (a), $t^{(v)}$ and $(t^2-4n)^{(v)}$ are relatively prime in $\cO_v$. Hence $(t^2-4n)^{(v)} \in \cO_v^\times$. From the relation $(t^2-4n)^{(v)}=4m^2_v\Delta_v^{0}$ with $\ord_{v}(\Delta_v^{0})\in \{0,1\}$, we conclude $m_v \in \cO_v^\times$ and $\Delta_v^{0}\in \cO_v^\times$. If $\Delta_v^{0}\in \cO_v^\times-(\cO_v^\times)^2$, from Theorem~\ref{OrbIntUnifEx} (3) and \eqref{Ovanish}, the non-vanishing of $\fE_v^{(z)}(\hat\gamma_v)$ implies $|\tfrac{n}{m_v^2}|_v>1$. From Lemma~\ref{OrbIntUnifExL4} (1), we have $|4n|_v=|t|_v^2$, which causes a contradiction when combined with $|t|_v<1$ and $|t^2-4n|_v=1$. Hence we must have $\Delta_v^{0}=1$ and $(t^2-4n)^{(v)}=(2m_v)^{2} \in (\cO_v^\times)^2$ as desired. Let $v\in S(\fc')$. Then $t^{(v)}\in \cO_v^\times$. Hence $|t|_v=1$ and $|4n|_v\leq 1$. From Theorem~\ref{OrbIntUnifEx}, Lemma~\ref{OrbIntUnifExL2} and Proposition~\ref{OrbIntUnifExL4}, the non-vanishing of $\fE_{v}^{(z)}(\hat \gamma_v)$ implies that $\tfrac{4n}{t^2}\in U_v(1)$ if $\ord_v(t^2-4n)\in 2\Z$ and $\tfrac{4n}{t^2}\in U_v(2)$ if $\ord_v(t^2-4n)\in 1+2\Z$. Therefore, $(t^2-4n)^{(v)}\in \fp_v^{2}$ and $|4n|_v=|t|_v^2=1$ for all $v\in S(\fc')$ as desired. Suppose $t=0$. Then $t^2-4n=-4n$ is prime to $\fn$ as shown above. Hence for all $v\in S(\fn)$, we have $m_v,\, \Delta_v^0\in \cO_v^\times$ and $|\tfrac{n}{m_v^2}|_v=1$. From (c), there is $v\in S(\fn)$ such that $\Delta_v^{0}\in \cO_v^\times-(\cO_v^\times)^2$. Hence from Theorem~\ref{OrbIntUnifEx} (4), we obtain $\fE_{v}^{(z)}(\hat \gamma_v)=0$ a contradiction. This shows $t\not=0$. \end{proof}

Let $\cQ_F^{S}$ be the set of $(t:n)_{F}\in \cQ_F$ such that $c_vt\in \cO_v$ and $c^{2}_vn\in \cO_v^\times$ for all $v\in \Sigma_\fin-S$ with an idele $c=(c_v)\in (\A_\fin^{S})^{\times}$ without $S$-component. 

\begin{lem} \label{ABC-L6}
	Let $\Phi=\otimes_v\Phi_v$ be our test function on $G_\A$ (which depends $l$, $\fn$ and $S$). Let $\gamma\in G_F$ be an element whose minimal polynomial is $X^2-tX+n$ with $(t,n)\in Q_F$. If $\Phi(g^{-1}\gamma g)\not=0$ with $g\in G_\A$, then $(t:n)_{F}\in \Qcal_F^{S}$.  
\end{lem}
\begin{proof}
	Let $v\in \Sigma_\fin-S$. Then $\Phi_v={\rm ch}_{Z_v\bK_0(\fn\cO_v)}$. Hence, from $\Phi(g^{-1}\gamma g)\not=0$, we see that there exist $g_v\in G_v$, $c_v\in F_v^{\times}$ and $k_v\in \bK_0(\fn\cO_v)$ such that $
	g_v^{-1} \gamma g_v=c_vk_v$. Hence $n=\det(\gamma)=c_v^2\,\det(k_v)$ and $t=\tr(\gamma)=c_v\,\tr(k_v)$. 
	Since $\det k_v\in \cO_v^\times$ and $\tr (k_v)\in \cO_v$ and since $n\in \cO_v^{\times}$, $t\in \cO_v$ for almost all $v$'s, we have $c=(c_v)\in (\A_\fin^S)^\times$ and $(t:n)_{F}\in \cQ^{S}_F$. 
\end{proof}

Let $h$ be the class number of $F$ and $\fa_{i}\,(1\leq i \leq h)$ a complete set of representatives of the ideal class group of $F$ such that $\fa_j \subset \cO $ are prime ideals different from $\fp_v\cap \go\,(v\in S)$. For $\nu=(\nu_v)_{v\in S} \in \Z^{S}$, set 
$\fp_S^{\nu}=\prod_{v\in S}(\fp_v\cap\go)^{\nu_v}$ and 
$$
E_i(S,\nu)=\left\{n\in F^\times|\,n\cO=\fa_i^2\fp_S^\nu\,\right\}, \quad (1\leq i \leq h).
$$
Let $I(S,\nu)$ be the set of indices $1\leq i\leq h$ such that $E_i(S,\nu)\not=\emp$. For each $i\in I(S,\nu)$, by fixing an element $n_{i,\nu}\in E_i(S,\nu)$ once and for all, we obtain a bijection $\cO^\times \rightarrow E_{i}(S,\nu)$ by sending $\varepsilon \in \cO^\times$ to $\varepsilon n_{i,\nu}\in E_i(S,\nu)$. By Dirichlet's unit theorem, the quotient group $\cO^\times/(\cO^\times)^2$ is finite. 

In the following lemma, we set $\ord_{v}(0)=+\infty$ for convention. 
\begin{lem} \label{ABC-L7}
	For any $\gamma\in \Qcal_{F}^{S}$, there exist $\nu\in \N_0^{S}$, $t\in \fa_{i}$, $i\in I(S,\nu)$ and $\varepsilon \in \cO^\times/(\cO^\times)^2$ such that $\gamma=(t:\varepsilon n_{i,\nu})_{F}$ and $\min(2\ord_v(t),\nu_v)\in \{0,1\}$ for all $v\in S$.  
\end{lem}
\begin{proof}
	To argue, we fix a representative $(t',n')\in F^2$ of $\gamma$. From definition, there exists a finite idele $c=(c_v)\in \A_\fin^\times$ with $c_v=1\,(v\in S)$ such that $c_vt'\in \cO_v$ and $c_v^2n' \in \cO_v^\times$ for all $v\in \Sigma_\fin-S$. Choose $(a_v)_{v\in S}\in \prod_{v\in S}F_v^\times$ such that the ideal $a_v^2\{(t')^2\cO_v+n'\cO_v\}$ is $\fp_v$ or $\cO_v$ for all $v\in S$. Let $\fc$ be a fractional ideal of $F$ defined by the idele $ca$, i.e., $\fc=F\cap (F_\infty \prod_{v\in \Sigma_\fin-S}c_v\cO_v\prod_{v\in S}a_v\cO_v)$. There exists a unique $i\in \{1,\dots,h\}$ and an element $b\in F^\times$ such that $\fc=(b)\fa_{i}^{-1}$. We set $t''=b t'$ and $n''=b^2n'$. Then $\gamma=(t':n')_{F}=(t'':n'')_{F}$, $t'' \in \fa_{i}$ and $n''\in E_{i}(S,\nu)$ with some $\nu\in \N_0^{S}$ such that $\min(2\ord_v(t''),\nu_v)\in \{0,1\}$ for all $v\in S$. We further write $n''=\varepsilon u^{2} n_{i,\nu}$ with $\varepsilon \in \cO^\times/(\cO^\times)^2$, $u\in \cO^\times$, and set $t=u^{-1}t''$. Then $\gamma=(t'':n'')_{F}=(t:\varepsilon n_{i,\nu})_{F}$ and $t\in \fa_{i}$ with $\min(2\ord_{v}(t),\nu_v)=0,1$ for all $v\in S$ as desired. 
\end{proof}

For $(t:n)_{F}\in \Qcal_F^{\rm{Irr}}$ and a place $v\in \Sigma_{\fin}$, by $(t^2-4n)^{(v)}=4\Delta_v^{0}m_v^{2}$ as before, we set
$\fn_v(t,n)=\frac{n^{(v)}}{m_v^2}$, which is an element of $F_v$ determined only up to proportionality constants from $\cO_v^\times$. We remark that $\fE_{v}^{(z)}(\hat \gamma_v)=0$ for $v \in \Sigma_\fin-S$ unless $|\fn_v(t,n)|_v\geq 1$ from Lemma~\ref{OrbIntUnifExL2} and Theorem~\ref{OrbIntUnifEx}.
For the arguments below to work, we need to make the following assumption \eqref{ASSUMPTION} so that any representative $(t,\varepsilon n_{i,\nu})$ as in Lemma~\ref{ABC-L7} satisfies the condition (a) in Lemma~\ref{levelcondition}: 
\begin{align}
	\text{The ideal $\fn$ is relatively prime to all of the ideals $\fa_j\,(1\leq j\leq h)$. }
	\label{ASSUMPTION}
\end{align}

\smallskip
\noindent
We fix prime ideals $\ga_i (1\le i \le h)$ satisfying \eqref{ASSUMPTION} for $\gn$ invoking Chebotarev's density theorem.
From Lemmas~\ref{ABC-L6} and \ref{ABC-L7}, we have a majorant of the series \eqref{ABCf0} by replacing the summation range with the set of all those pairs $(t,n)$ with $t\in \fa_i$, $n=\varepsilon n_{i,\nu}$ with $\varepsilon\in \cO^\times/(\cO^\times)^2$, $1\leq i\leq h$ and $\nu \in \N_0^S$. Let $\e>0$ be a small number. In the following all estimations are uniform in $z\in \C$ such that $\Re(z)\in [-2\underline{l}+3+6\e, \underline{l}-1/2-\e]$. By the evaluations of local orbital integrals $\fE_{v}^{(z)}(\hat\gamma_v)$ obtained in this section and by Lemmas~\ref{EisPeriodEst} and \ref{levelcondition}, we see that \eqref{ABCf0} is majorized by the sum of $\Xi^{(z,\bfs)}(\fc, \fn_1, i,n)$ over all tuples $(\fc,\fn_1, i, n)$ of an integral ideal $\fn_1$ dividing $\fn$, an integral ideal $\fc$ dividing $\fn_1$,
$1\leq i\leq h$, and an element $n\in F^\times$ which is of the form $n=\varepsilon n_{i,\nu}$ with $\varepsilon\in \cO^\times/(\cO^\times)^2$, $i\in I(S,\nu)$ and $\nu\in \N_0^{S}$, where $\Xi^{(z,\bfs)}(\fc,\fn_1,i,n)$ is defined to be
\begin{align*}
	&\sum_{t\in X(\fc,\fn_1,i,n)}
	\nr({\frak D}_{t^2-4n})^{\frac{1+\e}{4}+\frac{\ro(z)}{4}}
	\prod_{v\in \Sigma_\fin-(S\cup S(\fn))} |\Ocal_{v,0}^{\Delta_v^{0},(z)}(\fn_v(t,n))|
	\prod_{v\in S(\fn)} |\Ocal_{v,1}^{\Delta_v^{0},(z)}(\fn_v(t,n))|
	\\
	&\times \prod_{v\in S} |\Scal_{v}^{\Delta_v^{0},(z)}(s_v;\fn_v(t,n))|
	\prod_{\substack{v\in \Sigma_\infty \\ (t^2-4n)^{(v)}>0}} 
	\left|\Ocal_v^{+,(z)}\left(\tfrac{t^{(v)}}{|t^2-4n|_v^{1/2}}\right)\right|
	\prod_{\substack{v\in \Sigma_\infty \\ (t^2-4n)^{(v)}<0}} 
	\left|\Ocal_v^{-,(z)}\left(\tfrac{t^{(v)}}{|t^2-4n|_v^{1/2}}\right)\right|
\end{align*}
with $X(\fc,\fn_1,i,n)$ being the set of all $t\in \fa_i\fc-R(n)$ such that $t^{(v)} \in \cO_v^\times\,(\forall v\in S(\fn\fn_1^{-1}))$, $n^{(v)} \in \cO_v^\times \,(\forall v\in S(\fn\fn_1^{-1}))$, and such that $t^2-4n \in (\fc^{-1}\fn_1)^2$, $(t^2-4n)^{(v)} \in (\cO_v^{\times})^2\,(\forall v\in S(\fc\fn\fn_1^{-1}))$. Here $\bfs=(s_v)_{v\in S}\in \fX_S$ and we remark that $\Delta_v^{0}$ depends on $t,n$ through the relation $(t^2-4n)^{(v)}=4\Delta_v^0 m_v^2$. For $n=\varepsilon n_{i,\nu}$ as above, the subset $R(n)\subset \cO$ is defined as follows: $R(n)=\{0\}$ if $n^{(v)}\mod \fp_v$ is not a square residue for some $v\in S(\fn)$, and $R(n)=\emp$ otherwise. This comes from Lemma~\ref{levelcondition}.

From Lemmas~\ref{OrbIntUnifExL1} and \ref{OrbIntUnifExL2}, if $v\in \Sigma_\fin-\Sigma_{\rm dyadic}$, $t\in \cO_v^\times$, $4n\in \cO_v^\times,\,\tfrac{4n}{t^2}-1\in \cO_v^\times$, then $|\ft_v(t,n)|_v= 1$ and $|\fn_v(t,n)|_v=1$. In a way similar to Corollary~\ref{FHypL7.5}, from Lemma~\ref{ABC-L1234}, we have  
$$
\prod_{v\in \Sigma_\fin-(S\cup S(\fn))}|\Ocal_{v,0}^{\Delta_v^0, (z)}(\fn_v(t,n))|\leq C \prod_{v\in \Sigma_\fin-(S\cup S(\fn))}|\fn_v(t,n)|_v^{\frac{|\Re(z)|+1}{4}+\frac{\e}{2}}
$$
with a constant $C>0$, uniformly in $z\in \C$ and $(t,n)$ as above. 

Suppose $R(n)=\{0\}$ for a while. Applying Lemma~\ref{ABC-L1234}, we see that $\Xi^{(z,\bfs)}(\fc,\fn_1,i,n)$ is majorized by the product of 
\begin{align*}
	&\sum_{t\in X(\fc,\fn_1,i,n)}
	\nr({\frak D}_{t^2-4n})^{\frac{1+\e}{4}+\frac{\ro(z)}{4}}
	\times \prod_{v\in \Sigma_\fin-(S\cup S(\fn))}|\fn_v(t,n)|_v^{\frac{|\Re(z)|+1}{4}+\frac{\e}{2}} 
	\prod_{v\in S(\fn)} 
	|\fn_v(t,n)|_v^{\frac{|\Re(z)|+1}{4}+\frac{\e}{2}} \\
	&\times \prod_{v\in S}
	\{1+\max(0,-\ord_v(\fn_v(t,n)))\}\max(1,|\fn_v(t,n)|_v^{-1})^{\frac{-\Re(s_v)+|\Re(z)|}{4}+\frac{\e}{2}}|\fn_v(t,n)|_v^{\frac{|\Re(z)|+1}{4}+\frac{\e}{2}}
	\\
	&\times \prod_{\substack{v\in \Sigma_\infty \\ (t^2-4n)^{(v)}>0}} 
	\delta(|t|_v^2>|t^2-4n|_v)\,\left(\left|\frac{t^{(v)}}{\sqrt{(t^{(v)})^2-4n^{(v)}}}
	-1\right|_v^{+} \right)^{l_v/2} \left(\frac{|t|_v}{|t^2-4n|_v^{1/2}}\right)^{\frac{|\Re(z)|+1}{2}}\\
&	\times\prod_{\substack{v\in \Sigma_\infty \\ (t^2-4n)^{(v)}<0}} 
	\left(1+\frac{|t|_v}{|t^2-4n|_v^{1/2}} \right)^{\frac{1-2l_v}{4}} \times
	\prod_{v\in S(\fn\fn_1^{-1})}\frac{2}{q_v+1}\prod_{v\in S(\fn_1)}\frac{4(1+q_v^{1/2})}{1+q_v},
\end{align*}
where we set $|x|_v^{+}=\min(1,|x|_v)$ for $x\in F_v^\times$. Since $\nr({\frak D}_{t^2-4n})=\prod_{v\in \Sigma_\fin}|\Delta_v^{0}|_v^{-1}=\prod_{v\in \Sigma_F}|m_v|_v^{2}$ (see the last part of the proof of Proposition~\ref{EisPeriod}), it is easy to confirm the identity 
\begin{align*}
	\nr({\frak D}_{t^2-4n})^{\frac{|\Re(z)|+1}{4}+\frac{\e}{2}}\,
	\prod_{v\in \Sigma_\fin}|\fn_v(t,n)|_v^{\frac{|\Re(z)|+1}{4}+\frac{\e}{2}}=\prod_{v\in \Sigma_\infty}|n|_v^{-\frac{|\Re(z)|+1}{4}-\frac{\e}{2}}|4^{-1}(t^2-4n)|_v^{\frac{|\Re(z)|+1}{4}+\frac{\e}{2}}
\end{align*}
by the product formula and by the relation $|m_v|_v=|4^{-1}(t^2-4n)|_v^{1/2}$ for $v\in \Sigma_\infty$. Due to this, the above quantity is further majorized by 
\begin{align*}
	& \prod_{v\in S(\fn\fn_1^{-1})}\frac{2}{q_v+1}\,\prod_{v\in S(\fn_1)} \frac{4(1+q_v^{1/2})}{1+q_v}
	\sum_{t\in X(\fc,\fn_1,i,n)}
	\nr({\frak D}_{t^2-4n})^{\frac{\ro(z)-|\Re(z)|}{4}-\frac{\e}{4}}\\
	&\times
	\prod_{v\in S}
	\{1+\max(0,-\ord_v(\fn_v(t,n)))\}\max(1,|\fn_v(t,n)|_v^{-1})^{\frac{-\Re(s_v)+|\Re(z)|}{4}+\frac{\e}{2}} \\
&\times \prod_{\substack{v\in \Sigma_\infty \\ (t^2-4n)^{(v)}>0}} 
	\delta(n^{(v)}>0)\,\left(\left|\frac{t^{(v)}}{\sqrt{(t^{(v)})^2-4n^{(v)}}}
	-1\right|_v^{+}\right)^{l_v/2}\\
	& \times \prod_{v\in \Sigma_\infty} 
	\left(\frac{|t|_v^2}{|n|_v}\right)^{\frac{|\Re(z)|+1}{4}+\frac{\e}{2}} |t^2-4n|_v^{\frac{\e}{2}}
	\prod_{\substack{v\in \Sigma_\infty \\ (t^2-4n)^{(v)}<0}}
	\left(1+\frac{|t|_v}{|t^2-4n|_v^{1/2}}\right)^{\frac{1-2l_v}{4}}
	\left(\frac{|t|_v}{|t^2-4n|_v^{1/2}}\right)^{-\frac{1+|\Re(z)|}{2}-\e}.
\end{align*}
Let us examine the $S$-factor. If $\nu_v=0$, then $(4n)^{(v)}\in \cO_v^\times$ and $m_v \in \go_v$. Hence $|\fn_v(t,n)|_v=|m_v^2|_v^{-1}\ge 1$.
If $\nu_v>0$, then $(4n)^{(v)} \in \fp_v$.
When $t^{(v)}\in \cO_v^\times$, we have $|t|_v^{2}>|4n|_v$ and $t^2-4n \in \cO_v^\times$ is a square residue modulo $\fp_v$. Since $v$ is non-dyadic, this implies $(t^2-4n)^{(v)}$ is a square in $\cO_v^\times$, or equivalently $\Delta_v^0=1$ and $2m_v\in \cO_v^\times$. When $t^{(v)}\in \gp_v$, we have $\nu_v=1$ by $\min(2\ord(t), \nu_v) \in \{0,1\}$. Thus $\ord_v((2m_v)^{2}\Delta_v^0) = \ord_v(t^2-4n)=1$, and hence
$2m_v\in\go^\times$.
Thus we have $|\fn_v(t,n)|_v=|n|_v|m_v^2|_v^{-1}=|n|_v=q_v^{-\nu_v}$ if $\nu_v>0$. Therefore, ,  
\begin{align*}
	&\prod_{v\in S}\{1+\max(0,-\ord_v(\fn_v(t,n)))\}\max(1,|\fn_v(t,n)|_v^{-1})^{\frac{-\Re(s_v)+|\Re(z)|}{4}+\frac{\e}{2}} \\
	\le & 
	\prod_{v\in S}(1+\ord_v(m_v^2))^{\delta(\nu_v=0)}q_v^{\nu_v(\frac{-\Re(s_v)+|\Re(z)|}{4}+\frac{\e}{2})}
	\ll \prod_{v\in S}|m_v^2|_v^{-\frac{\e}{4}}q_v^{\nu_v(\frac{-\Re(s_v)+|\Re(z)|}{4}+\frac{\e}{2})}.
\end{align*}
Here, by noting $1=|\Delta/4|_\A=\prod_{v \in \Sigma_F}|m_v^2|_v \prod_{v \in \Sigma_\fin}|\Delta_v^0|_v$, the factor $\prod_{v \in S}(1+\ord_v(m_v^2))$ is bounded by 
\begin{align*}
	\prod_{v \in S}|m_v^2|_v^{-\frac{\e}{4}} = & \prod_{v \in \Sigma_{F}}|m_v^2|_v^{-\frac{\e}{4}} \times \prod_{v\in \Sigma_\infty}|m_v^2|_v^{\frac{\e}{4}} \times \prod_{v \in \Sigma_\fin-S}|m_v^2|_v^{\frac{\e}{4}} \\
	\le& \prod_{v \in \Sigma_{\fin}}|\Delta_v^0|_v^{\frac{\e}{4}} \times \prod_{v \in \Sigma_\infty}|m_v^2|_v^{\frac{\e}{4}} \times \prod_{v \in \Sigma_\fin-S}|4^{-1} (\Delta_v^0)^{-1}|_v^{\frac{\e}{4}} \\
	= & \prod_{v \in S}|\Delta_v^0|_v^{\frac{\e}{4}} \times \prod_{v \in \Sigma_\infty}|m_v^2|_v^{\frac{\e}{4}} \times \prod_{v \in \Sigma_\fin-S}|4^{-1}|_v^{\frac{\e}{4}}\\
	\ll_S & \prod_{v \in \Sigma_\infty}|t^2-4n|_v^{\frac{\e}{4}}.
\end{align*}
By this and by the majorization $\nr({\frak D}_{t^2-4n})\ll \prod_{v\in \Sigma_\infty}|t^2-4n|_v$, we have that $\Xi^{(z,\bfs)}(\fc,\fn_1,i,n)$ is majorized by the product of 
\begin{align*}
	& \prod_{v\in S(\fn\fn_1^{-1})}\frac{2}{1+q_v}\,\prod_{v\in S(\fn_1)} \frac{4(1+q_v^{1/2})}{1+q_v}
	\,\prod_{v\in S} q_v^{\nu_v(\frac{-\Re(s_v)+|\Re(z)|}{4}+\frac{\e}{2})}\end{align*}
and 
\begin{align}
	&\sum_{\substack{t\in \fa_i\fc-\{0\} \\ t^2-4n \in (\fn_1\fc^{-1})^2}}
	\prod_{\substack{v\in \Sigma_\infty \\ (t^2-4n)^{(v)}>0}} 
	\delta(n^{(v)}>0)\,\left(\left|\tfrac{t^{(v)}}{\sqrt{(t^{(v)})^2-4n^{(v)}}}
	-1\right|_v^{+}\right)^{l_v/2}
	\label{KKKK}
	\\
	&\prod_{v\in \Sigma_\infty}|t^2-4n|_v^{\frac{\ro(z)-|\Re(z)|}{4}+\frac{\e}{2}}\,\left(\tfrac{|t|_v^2}{|n|_v}\right)^{\frac{|\Re(z)|+1}{4}+\frac{\e}{2}}\times \prod_{\substack{v\in \Sigma_\infty \\ (t^2-4n)^{(v)}<0}}
	\left(1+\tfrac{|t|_v}{|t^2-4n|_v^{1/2}}\right)^{\frac{1-2l_v}{4}}
	\left(\tfrac{|t|_v}{|t^2-4n|_v^{1/2}}\right)^{-\frac{1+|\Re(z)|}{2}-\e}.
	\notag
\end{align}
By the embedding $\iota_\infty: t\mapsto (t^{(v)})_{v\in \Sigma_\infty}$, any fractional ideal of $F$ is viewed as a $\Z$-lattice of $[F:\Q]$-dimensional real vector space $F_\infty=\prod_{v\in \Sigma_\infty}F_v$. Let us examine the condition $t^2-4n\in (\fn_1\fc^{-1})^2$. Set $d_F=[F:\Q]$. We use the Euclidean norm $\|\xi\|=(\sum_{v\in \Sigma_\infty}x_v^2)^{1/2}$ on $F_\infty$ for estimations. 

\begin{lem} \label{iotaThyoukaL}
	For any $(t:n)_{F}\in \Qcal_{F}^{\rm Irr}$ with $t,\,n\in \cO$, $tn\not=0$ and a non-zero ideal $\fa \subset \cO$ such that $t^2-4n \in \fa$, it holds that $
	\|(\sqrt{|n|_v}^{-1}t^{(v)})_{v\in \Sigma_\infty} \|^{2} \geq d_F^{1/2}(\nr(n^{-1}\fa)^{1/d_F}-4).$
\end{lem}
\begin{proof}
	From $t^2-4n\in \fa$, there is an integral ideal $\fb$ such that $(t^2-4n)=\fa\fb$. By taking the norm, we have $|\nr(t^2-4n)|=\nr(\fa)\,\nr(\fb)\geq \nr(\fa)$ on one hand. On the other hand, by the geometric-arithmetic mean inequality, 
	\begin{align*}
		|\nr({t^2}{n^{-1}}-4)|=\prod_{v\in \Sigma_\infty}|t^2n^{-1}-4|_v
		\leq \biggl\{\sum_{v\in \Sigma_\infty}\tfrac{|t^2n^{-1}-4|_v^{2}}{d_F}\biggr\}^{d_F/2}=d_F^{-d_F/2}\|\iota_\infty(t^2n^{-1}-4)\|^{d_F}.
	\end{align*}
	Thus $\|\iota_\infty(t^2n^{-1}-4)\| \geq \sqrt{d_F}\,|\nr(n)|^{-1/d_{F}}\nr(\fa)^{1/d_F}.$ Set $\xi=(\sqrt{|n|_v}^{-1}t^{(v)})_{v\in \Sigma_\infty}$. Then from the inequality $(\sum_{v}x_v^4)^{1/2} \leq \sum_v x_v^2$,  
	\begin{align*}
		\|\xi\|^2+4d_F^{1/2} 
		&\geq  \|\iota_\infty(t^2n^{-1})\|+\|\iota_\infty(4)\|
		\geq \|\iota_\infty(t^2n^{-1}-4)\|\geq \sqrt{d_F}\,|\nr(n)|^{-1/d_F}\nr(\fa)^{1/d_F}.  
	\end{align*}
\end{proof}
Set 
$$
f(x)= 
\prod_{v\in \Sigma_\infty}|x_v^2-4|^{\frac{\ro(z)-|\Re(z)|}{4}+\frac{\e}{2}}\,
\prod_{v\in \Sigma_\infty} 
\begin{cases}
\biggl(\left|\tfrac{x_v}{\sqrt{x_v^2-4}}
-1\right|_v^{+}\biggr)^{l_v/2}|x_v|^{\frac{1+|\Re(z)|}{2}+\e} \quad &(|x_v|>2), \\
\left(|x_v^2-4|_v^{1/2}+{|x_v|}\right)^{\frac{1-2l_v}{4}}
|x_v^2-4|_v^{\frac{2l_v+2|\Re(z)|+1+4\e}{8}}.
\quad &(|x_v|\leq 2)
\end{cases}
$$
for $x=(x_v)\in \R^{d_F}$.
Set $\xi=(\sqrt{|n|_v}^{-1} t^{(v)})_{v\in \Sigma_\infty}$. From Lemma~\ref{iotaThyoukaL}, the series \eqref{KKKK} is bounded by 
\begin{align}
	&\prod_{v\in \Sigma_\infty} |n|_v^{\frac{\ro(z)-|\Re(z)|}{4}+\frac{\e}{2}}\sum_{\xi \in L_i(\fc,n)-\{0\}}\delta\biggl(\|\xi\|\geq \sqrt{d_F^{1/4}\,\max(0,\nr(n^{-1}\fn_1^2\fc^{-2})^{1/d_F}-4)}
	\biggr)\,f(\xi)
	\label{KKKKK}
\end{align}
where $L_i(\fc,n)$ denotes the $\Z$-lattice $\{(N_vx_v)_{v\in \Sigma_\infty}|\,x\in \iota_\infty(\fa_i\fc)\}$ in $\R^{d_F}$. The function $f(x)$ is continuous on $\R^{d_F}$ satisfying the majorization $|f(x)|\ll \prod_{v\in \Sigma_\infty}(1+|x_v|)^{-l_v+\frac{1+\ro(z)}{2}+2\e}$. 

\begin{lem} \label{SugTsudLA}
	Let $L=(L_j)_{j=1}^{d_F}\in (\R_{+})^{d_F}$ and define a function on $\R^{d_F}$ as $f_L(x)=\prod_{j=1}^{d_F}(1+|x_j|)^{-L_j}$. Suppose ${\underline L}>1$. Then we have a constant $C>0$ such that for any $A\geq 0$ and for any pair of $\Z$-lattices $\Lambda\subset \Lambda_0$ of full rank in $\R^{d_F}$, we have
	\begin{align*}
		\sum_{\xi \in \Lambda-\{0\}}\delta(\|\xi\|\geq A)\,f_{L}(\xi) \leq C \, r(\Lambda_0)^{-d_F}(1+r(\Lambda_0))^{d_F\bar L}\,\max(A, r(\Lambda))^{1-{\underline L}}\end{align*}
	with ${\underline L}=\min\{L_j|\,1\leq j\leq d_F\}$, $\bar L=\max\{L_j|\,1\leq j \leq d_F\}$ and $r(\Lambda)=2^{-1}\inf_{\xi\in \Lambda-\{0\}}\|\xi\|$. 
\end{lem}
\begin{proof}
	This follows from the proof of \cite[Theorem A.1]{SugiyamaTsuzuki2}. 
\end{proof}
Set $\Lambda_n=\{(\sqrt{|n|_v}^{-1} x_v)|\,x\in \iota_\infty(\cO)\}$. Since $L_i(\fc,n)$ are all contained in $\Lambda_n$, from Lemma~\ref{SugTsudLA}, the series \eqref{KKKKK} is absolutely convergent if ${\underline l} \geq 4$ and $\e>0$ is small enough, and majorized by
\begin{align*}
	\prod_{v\in \Sigma_\infty} |n|_v^{\frac{\ro(z)-|\Re(z)|}{4}+\frac{\e}{2}}
	\times \delta (n)\, \max\biggl(\sqrt{d_F^{1/4} \max(0,\nr(n^{-1}\fn_1^2\fc^{-2})^{1/d_F}-4)}
	,\, r(L_i(\fc,n))\biggr)^{1-{\underline L}(z)}
\end{align*}
with ${\underline L}(z)={\underline l}-\frac{1+\ro(z)}{2}-2\e\ge 1+\e$ and $\delta(n)
=r(\Lambda_n)^{-d_F}(1+r(\Lambda_n))^{d_F{\overline L}(z)}$ with $\bar L(z)=\overline{l}-\frac{1+\ro(z)}{2}-2\e$.
By a slight modification of the proof of \cite[Lemma A.7]{SugiyamaTsuzuki2}, we have $|\nr(n)|^{-1/2d_F}\leq d_F^{1/2}r(\Lambda_n)$ and $(|\nr(n)|^{-1/2}\nr(\fa_i\fc))^{1/d_F}\leq r(L_i(\fc,n))$. From Minkowski's convex body theorem, we have an upper-bound $r(\Lambda_n) \ll |\nr(n)|^{-1/2d_F}$. Hence 
$$\delta(n)\ll |\nr(n)|^{1/2}(1+|\nr(n)|^{-1/2d_F})^{d_F\overline{L}(z)} \ll |\nr(n)|^{1/2}\quad \text{for $n\in \cO$}.$$
As a consequence, we obtain a majorant of the series \eqref{KKKKK} as
\begin{align*}
	|\nr(n)|^{\frac{\ro(z)-|\Re(z)|}{4}+\frac{\e}{2}}\, |\nr(n)|^{1/2} \, \max\biggl(\sqrt{d_F^{1/4}\max(0,\nr(n^{-1}\fn_1^2\fc^{-2})^{1/d_F}-4)}, |\nr (n)|^{-1/2d_F}\nr(\fa_i\fc)^{1/d_F}\biggr)^{1-{\underline L}(z)}.\end{align*}
From the argument so far, we obtain the half of the following majorization when $R(n)=\{0\}$. 

\begin{lem} \label{ABC-L8}
	Let $\e>0$ be small enough. Let $\fn$ be a square-free integral ideal satisfying the assumption \eqref{ASSUMPTION}, $\fc$ a divisor of $\fn$, $\varepsilon\in \cO^\times/(\cO^\times)^2$, $i\in I(S,\nu)$ and $\nu \in \N_0^{S}$. Then there exists a constant $C>0$ independent of $(\fn,\fc,i,\nu)$ such that the following inequality holds uniformly in $(z,\bfs)\in \C\times \fX_S$ with $\Re(z)\in [-2\underline{l}+3+6\e, \underline{l}-1/2-\e]$,
	$\min_{v\in S}\Re(s_v)>\max(1, \tfrac{|\Re(z)|+1}{2})$, where ${\underline L}(z)={\underline l}-\frac{1+\ro(z)}{2}-2\e$ and $\ro(z)=\max(1,|\Re(z)|)$ : 
	{\allowdisplaybreaks \begin{align*}
			&|\Xi^{(z,\bfs)}(\fc,\fn_1, i,\varepsilon n_{i,\nu})\\
			\leq &\,  C\,
			\prod_{v\in S(\fn\fn_1^{-1})}\frac{2}{q_v+1} \prod_{v\in S(\fn_1)} \frac{4(1+q_v^{1/2})}{1+q_v}
			\,\prod_{v\in S} q_v^{\nu_v(\frac{-\Re(s_v)+\ro(z)}{4}+\e)} |\nr(n_{i,\nu})|^{1/2+({\underline L}(z)-1)/2d_F}
			\\
			&\times 
			\max\biggl(\sqrt{d_F^{1/4}\max\bigl(\,0,\nr(\fn_1^2\fc^{-2})^{1/d_F}-4|\nr(n_{i,\nu})|^{1/d_F}\,\,\bigr)}, \,
			\nr(\fa_i\fc)^{1/d_F}\biggr)^{1-{\underline L}(z)} 
			\\
			&+\delta(R(\varepsilon n_{i,\nu})=\emp)\,\delta((\nu_v)_{v \in S} \in \{0,1\}^S)\,\prod_{v\in S(\fn)} \frac{2}{1+q_v} \prod_{v \in S} \max(1,q_v^{\frac{-\Re(s_v)+|\Re(z)|}{4}+\e}).
	\end{align*}} 
\end{lem}
\begin{proof} The case $R(\varepsilon n_{i,\nu})=\{0\}$ is settled before the lemma. Suppose $R(\varepsilon n_{i,\nu})=\emp$. In this case we have to estimate the extra terms with $t=0$ in $\Xi^{(z,\bfs)}(\fc,\fn_1,i,n)$. Set $n=\varepsilon n_{i,\nu}$. By $t=0$, (from Lemma~\ref{ABC-L7}) we have $\nu_v\in \{0,1\}\,(\forall v\in S)$. Combining this with the relation $(-4n)^{(v)}=m_v^2\Delta_v^0$, we have $m_v\in \cO_v^\times$ for all $v\in S$. Thus the $S$-factor is bounded by $\prod_{v \in S} 16\max(1,q_v^{\frac{-\Re(s_v)+|\Re(z)|}{4}+\e})$ by Lemma~\ref{ABC-L1234} (3). The equality $R(\varepsilon n_{i,\nu})=\emp$ means that $n^{(v)}\pmod {\fp_v}$ is a square residue at all $v\in S(\fn)$. Hence $n^{(v)}\in (\cO_v^\times)^2$ for all $v\in S(\fn)$.
	Thus from the case of $\Delta_v^0=1$ in Lemma \ref{ABC-L1234} (2), we see $\prod_{v\in S(\fn)}|\fE_v^{(z)}(\hat \gamma_v)|\le \prod_{v\in S(\fn)}\tfrac{2}{1+q_v}.$ The archimedean component to be accounted for is the product of $\Ocal_v^{-,(z)}(0)=2^{l_v-1}(-1)^{l_v/2}\sqrt{\pi}\Gamma(\frac{l_v}{2}+\frac{z-1}{4})\Gamma(\frac{l_v}{2}+\frac{-z-1}{4})\Gamma(l_v)^{-1}$, which is bounded on the vertical strip $\Re(z)\in[-2\underline{l}+3+6\e, \underline{l}-1/2-\e]$. The conditions on $\Re(z)$ and $\min_{v\in S}\Re(s_v)$ come from ${\underline L}(z)>1+\e$ and Lemma~\ref{ABC-L1234} (3), (5). 
\end{proof}

By Lemma~\ref{ABC-L8}, to complete the proof of Theorem~\ref{ABCTheorem} it suffices to note the multi-series
$$
\sum_{\nu \in \N_0^{S}} \prod_{v\in S} q_v^{\nu_v(\frac{-\Re(s_v)+\ro(z)}{4}+\e)}|\nr(n_{i,\nu})|^{1/2+({\underline L}(z)-1)/2d_F}
$$
converges absolutely if $\Re(s_v) > \ro(z)+2d_F^{-1}({\underline L}(z)-1)+2+4\e$ for all $v\in S$. Note $\ro(z)+2d_F^{-1}({\underline L}(z)-1)+4\e+2<\ro(z)+2({\underline L}(z)-1)+4\e+2=2{\underline l}-1$ due to ${\underline L}(z)>1$.

\subsection{The conclusion}
From \eqref{10.2-f3} and \eqref{10.2-f2}, 
\begin{align*}
	(\Ecal_\b^{*})^\Delta(g)= \int_{L_\s}\b(z)\,\Lambda_F(z+1)E^{\Delta}(z;1_2) \prod_{v\in \Sigma_F}\varphi_v^{\Delta,(z)}(g_v)\,\d z, \quad g=(g_v)\in G_\A. 
\end{align*}
Substituting this to \eqref{10.2-f1} and exchanging the order of integrals, we obtain 
{\allowdisplaybreaks\begin{align*}
		\JJ_{\rm{ell}}(\bfs, \b)&=\tfrac{1}{2}\sum_{\tilde \gamma \in \Qcal_{F}^{\rm{Irr}}} \int_{L_\s}\b(z)\,\Lambda_F(z+1)E^{\Delta}(z;1_2)\,\{\int_{{\frak T}_{\Delta}\bsl G_\A} \Phi(\bfs;g^{-1}\hat \gamma g)\,\prod_{v\in \Sigma_F}\varphi^{\Delta,(z)}(g_v)\,\d g \}\,\d z
		\\
		&=\tfrac{1}{2}\sum_{\tilde \gamma \in \Qcal_{F}^{\rm{Irr}} } \int_{L_\s}\b(z)\,\Lambda_F(z+1)E^{\Delta}(z;1_2)\,\{\prod_{v\in \Sigma_F} \fE_v^{(z)}(\hat \gamma_v)\} \, \d z. 
\end{align*}}We need to legitimatize the order exchange of integral above. By Fubini's theorem and Lemma~\ref{EisPeriodEst}, it suffices to show the following. 

\begin{lem}
	For a fixed $\tilde \gamma=(t:n)_{F}$, there exist $N>0$ and a finite set $T_0\subset \Sigma_F$ containing $\Sigma_\infty$ such that 
	\begin{align}
		\prod_{v\in T} \int_{{\frak T}_{\Delta,v}\bsl G_v}|\Phi_v(g_v^{-1} \hat\gamma_v g_v)|\,|\varphi_v^{\Delta,(z)}(g_v)|\,\d g_v =O((1+|z|)^{N}) 
		\label{10.6-f1}
	\end{align}
	uniformly for all finite sets $T\subset \Sigma_F$ containing $T_0$ and for all $z$ in the strip $-2\underline{l}+3+6\e\le |\Re(z)|\leq \underline{l}-1/2-\e$ with small $\e>0$. 
\end{lem}
\begin{proof}
	For simplicity we argue assuming $t\not=0$. (The case $t=0$ is easier.) Let $T_0$ be the union of $\Sigma_\infty \cup \Sigma_{\rm{dyadic}} \cup S\cup S(\fn D_F)$ and the set of $v\in \Sigma_\fin$ such that $|t|_v=|4n|_v=|t^2-4n|_v=1$ does not hold. If $v\in \Sigma_\fin-T_0$ is such that $\Delta_v^{0}=1$, then $a=\frac{t}{2m_v}$, $b=\frac{a+1}{a-1}$ satisfies $|a-1|_v=|a+1|_v=1$ and $|b-1|_v^2=\frac{|4|_v}{|a^2-1|_v}=1$. Thus as we have seen in \S~\ref{EFLOInt},  
	{\allowdisplaybreaks\begin{align*}
			\tint_{{\frak T}_{\Delta,v}\bsl G_v}|\Phi_v(g_v^{-1} \hat \gamma_v g_v)||\varphi^{\Delta,(z)}(g_v)|\,\d g_v 
			=&
			\tint_{F_v}\left(\int_{\bfK_v}|\Phi_v\left(k^{-1} \left[\begin{smallmatrix} b & (b-1)x \\ 0 & 1 \end{smallmatrix} \right] k\right)|\,\d k\, \right)|\varphi_v^{(0,z)}(\left[\begin{smallmatrix} 1 & x \\ 0 & 1 \end{smallmatrix} \right] )|\,\d x\\
			=&\tint_{x\in \cO_v} \left|\varphi_v^{(0,z)}\left(\left[\begin{smallmatrix} 1 & x \\ 0 & 1 \end{smallmatrix} \right] \right)\right|\,\d x=1. 
	\end{align*}}If $v\in \Sigma_\fin-T_0$ is such that $\Delta_v^{0}\in \cO_v^\times-(\cO_v^\times)^2$, then since $|\fn_v(t,n)|_v=|\frac{t^{(v)}}{2m_v}|_v=1$, from the proof given in \S~\ref{EllOrbIntL1}, we also have $\int_{{\frak T}_{\Delta,v}\bsl G_v}|\Phi_v(g_v^{-1} \hat \gamma_v g_v)||\varphi^{\Delta,(z)}(g_v)|\,\d g_v=1$. Thus the left-hand side of \eqref{10.6-f1} is independent of $T$ containing $T_0$. At places in $T_0$, the absolute convergence of the integral and their necessary polynomial bound are shown or are easily deducible from the arguments in \S~\ref{HypSecondStep} and in \S~\ref{Localorbitalintegrals}. 
\end{proof}

Set 
\begin{align}
	\hat J_{\rm{ell}}(\bfs, z)=\tfrac{1}{2}\sum_{\tilde \gamma \in \Qcal_{F}^{\rm{Irr}}\cap\Qcal_F^{S}}
	\Lambda_F(z+1)E^{\Delta}(z;1_2)\,\{\prod_{v\in \Sigma_F} \fE_v^{(z)}(\hat \gamma_v)\}. 
	\label{JJell}
\end{align}

Now we obtain Theorem~\ref{hatJ*} for $F$-elliptic terms in the following form.

\begin{thm}\label{Conclusion of ell}
	For $(z,\bfs)\in \C\times \fX_S$ satisfying $z\neq \pm 1$, $\Re(z) \in (-2\underline{l}+3, \underline{l}-1/2)$, $\min_{v\in S}\Re(s_v)>2{\underline l}-1$, the series \eqref{JJell} converges absolutely and locally uniformly in $(z,\bfs)$. For a fixed $\bfs \in \fX_S$ such that $\min_{v\in S}\Re(s_v)> 2{\underline l}-1$, the function $z\mapsto \hat J_{\rm{ell}}(\bfs, z)$ is holomorphic away from $z=\pm 1$ and is vertically of moderate growth on the strip $\Re(z) \in [-2{\underline l}+3+\e, {\underline l}-1/2-\e]$ for any sufficiently small $\e>0$. We have the formula \eqref{hatJ*-f1} with $\natural={\rm ell}$. 
\end{thm}

\section{Proof of the main theorems} \label{MVT}

For any holomorphic function $f(\bfs)$ on $\fX_S$ and for any $\bfc=(c_v)_{v\in S}\in \R^{S}$, the multi-dimensional contour integral $\int_{L(\bfc)}f(\bfs)\d\mu_S(\bfs)$ is defined as the iteration of the one-dimensional contour integrals $\int_{L_v(c_v)} \d \mu_{v(j)}(s_{v(j)})$ $(j=1,\dots,k)$ in any ordering $S=\{v(1),\dots,v(k)\}$.

In this section, we let the square-free ideal $\fn$ vary in such a way that $\fn$ is prime to $2\fp_{S}\prod_{j=1}^{h}\fa_j$, where $\fa_j$ are the ideals fixed in \S~{7}
and we set $\gp_S=\prod_{v \in S}(\gp_v\cap \go)$.

\subsection{Proof of Theorem \ref{thm0} and Corollary \ref{MAINTHM}}

For $\natural\in \{{\rm unip},{\rm hyp},{\rm ell}\}$,
we define $\hat J_{\natural}^{0}(\bfs,z)$ by \eqref{Unipotent0Term} for $\natural={\rm unip}$
and $\hat J_{\natural}^{0}(\bfs,z)=\zeta_{F}\left(\tfrac{z+1}{2}\right)^{-1}\hat J_{\natural}(\bfs,z)$ for $\natural\in \{{\rm hyp},{\rm ell}\}$.
Set $\hat I_{\rm cusp}^0(\bfs,z)=\zeta_F\left(\tfrac{z+1}{2}\right)^{-1}\hat I_{\rm cusp}(\bfs,z)$.
Recall the function $\hat I_{\rm cusp}(\bfs,z)$ defined by \eqref{Icusp0sz}
and the explicit formulas of $\hat J_{\natural}^{0}(\bfs,z)$ stated in \S \ref{singular terms},
\S \ref{The $F$-hyperbolic term} and \S \ref{The $F$-elliptic term}.
From \S \ref{overview of geom}, we have the identity
$$\hat I_{\rm cusp}^{0}(\bfs,z)=D_{F}^{\frac{z}{4}}\{\hat J^{0}_{\rm unip}(\bfs,z)+\hat J^{0}_{\rm unip}(\bfs,-z)\}+\hat J_{\rm hyp}^{0}(\bfs,z)+\hat J_{\rm ell}^{0}(\bfs,z)
$$
for $|\Re(z)|< \underline{l}-3$ and $\bfs\in \fX_S$ with $\min_{v\in S}\Re(s_v)> 2\underline{l}-1$
(cf.\ Proposition \ref{Conclusion of hyp} and Theorem \ref{Conclusion of ell}).
To obtain the formula of $\hat J_{\rm ell}^{0}(\bfs,z)$ in Theorem~\ref{thm0}, we examine local conditions on $(t,n)$ posed by various $\delta$-factors in Theorem~\ref{OrbIntUnifEx}. From Lemma~\ref{OrbIntUnifExL1}, when $|\Delta_v^0|_v=1$ we have that $|\tfrac{t}{2m_v}|_v\not=1$ if and only if $|\tfrac{n}{m_v^2}|_v \geq 1$; hence $\tfrac{t}{2m_v}\not\in \cO_v^\times$ implies $|\tfrac{n}{4m_v^2}|_v\geq |4|_v^{-1}\geq 1$. Thus the $\delta$-symbol in the first formula of Theorem~\ref{OrbIntUnifEx} (2) is simplified to $\delta(\tfrac{n}{4m_v^2}\not\in \fp_v)$. By a similar argument, the local conditions posed by $\delta$-symbols in Theorem~\ref{OrbIntUnifEx} are reduced to the following: 
\begin{itemize}
	\item If $\Delta_v^0=1$, then $\tfrac{n}{4m_v^2} \not\in \fp_v$, 
	\item If $\Delta_v^0 \in \cO_v^\times-(\cO_v^\times)^2$, $v\in S(\fn)$, then $\tfrac{n}{m_v^2} \not\in \cO_v$, 
	\item If $\Delta_v^0 \in \fp_v-\fp_v^2$, or $v\in \Sigma_{\rm dyadic}$, $\Delta_v^0\in \{-1,-5\}$, then $\tfrac{n}{m_v^2} \not\in \fp_v$.  
\end{itemize}
We remark that $\ff_{\Delta}\cO_v$ coincides with $m_v\cO_v$ unless $\Delta_v^0=1$ or $v\in \Sigma_{\rm dyadic}, \Delta_v^0=5$, in which case it equals $2m_v\cO_v$. By this, we can write the condition above as follows :
\begin{itemize}
	\item $v\in S(\fn)$ and $\varepsilon_{\Delta,v}$ is unramified and non-trivial, then $\ord_v(n\ff_{\Delta}^{-2})<0$, 
	\item $\varepsilon_{\Delta,v}=1$ or $v\in S(\fd_{F(\sqrt{\Delta})/F})$, then
	$\ord_v(n\ff_{\Delta}^{-2}) \le 0$. 
\end{itemize}
The second condition follows from $(t:n)_{F}\in \Qcal_F^{S}$. Indeed, from the relations $c_vt\in \cO_v$, $c^2_vn\in \cO_v^\times$ and $(c_vt)^2-4(c_vn)=(2m_vc_v)^2 \Delta_v^0$, we see $m_v c_v\in \cO_v$. Hence $\ord_v({n}\ff_{\Delta}^{-2})=\ord_v(\frac{c_v^2n}{c_v^2\ff_{\Delta}^2}) \leq 0$. This completes the proof of Theorem~\ref{thm0}.

Take any $\a \in \Acal_S$. For $\natural\in \{{\rm unip},{\rm hyp},{\rm ell}\}$, set
\begin{align*}
	\JJ_{\natural}^0(\fn|\a,z)=\left(\tfrac{1}{2\pi i } \right)^{\# S} \int_{L(\bfc)}\hat J_{\natural}^{0}(\bfs,z)\,\a(\bfs)\,\d\mu_{S}(\bfs)
\end{align*}
with $\bfc=(c_v)\in \R^{S}$, $c_v\gg 1$.
We note that the ideal $\fn$ is implicit in the definition of $\hat J_{\natural}^{0}(\bfs,z)$.
Set
$$
\II_{\rm cusp}^{0}(\fn|\a,z)=(-1)^{\# S}C(l,\fn)^{-1}\,\left(\tfrac{1}{2\pi i }\right)^{\# S} \int_{L(\bfc)} \hat I_{\rm cusp}^{0}(\bfs,z)\,\a(\bfs)\,\d\mu_S(\bfs)
$$
for $\a(\bfs)$ as above.
By taking a multi-dimensional contour integral, we obtain the identity
\begin{align}
	(-1)^{\# S}C(l,\fn)\,
	\II^{0}_{\rm cusp}(\fn|\a,z)=D_{F}^{\frac{z}{4}}\{\JJ^{0}_{\rm unip}(\fn|\a,z)+\JJ_{\rm unip}^{0}(\fn|\a,-z)\}+\JJ_{\rm hyp}^{0}(\fn|\a,z)+\JJ_{\rm ell}^{0}(\fn|\a,z).
	\label{JZtraceformula}
\end{align}
from Theorem \ref{thm0}.
By \eqref{Icusp0sz} and the formula $$\tfrac{1}{2\pi i} \tint_{L_v(c)} \{q_v^{(1+\nu)/2}+q_v^{(1-\nu)/2}-q_v^{(1+s)/2}-q_v^{(1-s)/2}\}^{-1}\a(s)\d\mu_v(s)=-\a(\nu),$$
we see that the left-hand side of \eqref{JZtraceformula} coincides with \eqref{AverageAdL} multiplied by $(-1)^{\# S}C(l,\fn)$.
This completes the proof of Corollary \ref{MAINTHM}.

\subsection{Proof of Theorem \ref{MVT-thm}}
Let $z\in (0,1]$. Then the factor $\prod_{v\in S(\fn)}(1+q_v^{(z+1)/2})/(1+q_v)$ of $\JJ_{\rm unip}^{0}(\fn|\a,z)$ is bounded in absolute value by $\prod_{v\in S(\fn)}q_v^{(z-1)/2}=\nr(\fn)^{(z-1)/2}$ from below. The value of $\JJ_{\rm unip}^{0}(\fn|\a,z)+\JJ_{\rm unip}^{0}(\fn|\a,-z)$ at $z=0$ is of the form$$\left(\prod_{v \in S(\fn)}\frac{1+q_v^{\frac{1}{2}}}{1+q_v}\right)\times\biggl(C_0+C_1\sum_{v\in S(\fn)}\frac{\log q_v}{1+q_v^{-1/2}}\biggr)
$$
with some quantities $C_0,C_1$ independent of $\fn$. This is bounded from below by $\nr(\fn)^{-1/2}\log \nr(\fn)$. Hence to prove Theorem~\ref{MVT-thm}, it suffices to show the hyperbolic and the elliptic terms are bounded from above by $O(\nr(\fn)^{-1/2-\eta})$ with some $\eta>0$ uniformly on the strip $|\Re (z)| \le 1$.

Let $v\in S$ and $s_v$ a complex variable. Let $\Acal_v^0$ be the space of all Laurent polynomials in $\zeta=q_v^{-s_v/2}$ which is invariant by the substitution $\zeta\rightarrow \zeta^{-1}$. Set $\s_n(\zeta)=\zeta^{n}+\zeta^{-n}$ for $n\in \N_0$. Then as is well known, the elements $\s_n(\zeta)\,(n\in \N_0)$ form a $\C$-basis of the space $\Acal_v^{0}$. For $m\in \N_0$, let $\Acal_v^0[m]$ be the $\C$-linear span of functions $\s_n(0\leq n \leq m)$. For an integral ideal $\fa$ with the prime decomposition $\fa=\prod_{v\in S}(\fp_v\cap \go)^{n_v}$ with $n_v\in \N$, set $\Acal(\fa)=\bigotimes_{v\in S(\fa)}\Acal_v^0[n_v]$. 

\begin{lem} \label{SfactorvanishingL}
	Let $m\in \N_0$. We have $\hat \Scal_v^{\delta,(z)}(\a;a)=0$ for all $\a\in \Acal_v^{0}[m]$ if $\ord_v(a)>m$. 
\end{lem}
\begin{proof}
	Let $e:=\ord_v(a)>m$. Then from \eqref{Scaldeltazsa-f1}, we have that $\hat \Scal_v^{\delta,(z)}(\s_n,a)$ equals
	\begin{align*}
		q_v^{-e/2}\frac{1}{2\pi i }{\textstyle{\oint}} _{|\zeta|=q^{-c/2}}\frac{(1-\varepsilon_\delta(\varpi_v)\,q_v^{-1}\zeta^2)(1-\zeta^2)\zeta^{e-n}(1+\zeta^{2n})}{(1-q^{-(z+1)/2}\zeta^2)(1-q_v^{(z-1)/2}\zeta^2)}\frac{\d \zeta}{\zeta},
	\end{align*}
	which is zero by Cauchy's theorem if $e-n-1\geq 0$. 
\end{proof}

Let us take a test function $\a \in \Acal(\fa)$.

\begin{lem}\label{MVT-L4}
	Let $\s\in (0,1]$. Suppose $l$ satisfies ${\underline l}>\s+3$. Set $\delta=1$ if ${\underline l}>d_F+3+\s$ and $\delta=1/2+(\underline{l}-3-\s)/(2d_F)$ otherwise. Then $1/2<\delta\le 1$, and for any small $\e>0$,  
	\begin{align}\label{MVT-L4-f0}
		|\JJ_{\rm hyp}^0(\fn|\a,z)| \ll& \nr(\fn)^{-\delta+\e}
	\end{align}
	uniformly for $z$ on the strip $|\Re(z)|\leq \s$.
\end{lem}
\begin{proof}
	From \eqref{Jhypbfsz} and Theorem \ref{explicit hyp} (2) and (3), 
	\begin{align}
		\JJ_{\rm hyp}^0(\fn|\a,z)&=\tfrac{1}{2}\zeta_{F,\fin}\left(\tfrac{1-z}{2}\right)\sum_{a\in \cO(S)^{\times} -\{1\}}\prod_{v\in \Sigma_\fin-S}\fF_v^{(z)}(a)\times \prod_{v\in \Sigma_\infty} \tilde \fF_v^{(z)}(a) 
		\label{MVT-L4-0}
		\\
		&\quad \times \prod_{v\in S}\tfrac{1}{2\pi i} \tint_{L_v(c_v)} \fF_v^{(z)}(s_v,a)\,\alpha_{v}(s_v)\,\d\mu_{v}(s_v),
		\label{MVT-L4-1}
	\end{align}
	where $S=S(\fa)$ and $\cO(S)$ denotes the $S$-integer ring of $F$, and we set $\tilde \fF_v^{(z)}(a)=\Gamma_{\R}(\tfrac{1-z}{2})\fF_v^{(z)}(a)$ for $v\in \Sigma_\infty$. By Theorem~\ref{explicit hyp} (4) and Lemma~\ref{SfactorvanishingL}, the $S$-factor \eqref{MVT-L4-1} vanishes unless $a^{(v)} \in \fp_{v}^{-n_v}$ for all $v\in S$. Hence we may suppose that the summation in \eqref{MVT-L4-0} is over the set $a\in \fa^{-1}-\{0,1\}$. From Corollary~\ref{S-Est}, the $v$-factor in \eqref{MVT-L4-1} for $a^{(v)}\in \fp_{v}^{-n_v}$ is majorized by $O(|1-a|_v^{-\frac{|\Re(z)|+1}{2}-\e})$ uniformly in $z$. Hence as in the proof of Proposition~\ref{FHypP2}, we have that for any small $\e>0$, 
	\begin{align*}
		\prod_{v\in \Sigma_\fin-S}|\fF_v^{(z)}(a)|\prod_{v\in \Sigma_\infty}|\tilde \fF_v^{(z)}(a)|\,\prod_{v\in S}\left|\tfrac{1}{2\pi i} \tint_{L_v(c_v)} \fF_v^{(z)}(s_v,a)\,\alpha_{v}(s_v)\,\d\mu_{v}(s_v)\right| \ll f(a),\quad a\in \fa^{-1}-\{0,1\}
	\end{align*}
	uniformly on $|\Re(z)|\leq \s$ with $f(a)=\prod_{v\in S(\fn)\cup \Sigma_\infty}f_v(a^{(v)})$, where $f_v(a^{(v)})=\delta(a^{(v)}>0)(1+|a|_v)^{-\frac{l_v}{2}+\frac{\s+1}{2}+\e}$ if $v\in \Sigma_\infty$ and $f_v(a^{(v)})=\frac{4}{q_v+1}(q_v^{1/2}+1)^{\delta(a^{(v)}\in1+\fp_v)}$ if $v\in S(\fn)$. The sum $\sum_{a\in \fa^{-1}-\{0,1\}}f(a)$ is majorized by 
	\begin{align*}
		\sum_{\fc|\fn}\sum_{\substack{ a\in \fa^{-1}-\{0,1\} \\ a-1 \in \fc, a-1\not\in \fn\fc^{-1}}} f(a) \ll 
		\sum_{\fc|\fn}\prod_{v\in S(\fn\fc^{-1})}\tfrac{4}{q_v+1}\prod_{v\in S(\fc)}\tfrac{4(q_v^{1/2}+1)}{q_v+1}\,\biggl\{\sum_{x\in \fc\fa^{-1}-\{0\}} f_\infty(x+1)\biggr\}.  
	\end{align*}
	Let $T(\fn)$ denote the majorant. Since $\iota_\infty(\fc\fa^{-1}) \subset \iota_\infty(\fa^{-1})$ are $\Z$-lattices of full rank in $\R^{\Sigma_\infty}$, we apply Lemma~\ref{SugTsudLA} to estimate the sum $\sum_{x\in \fc\fa^{-1}-\{0\}}f_\infty(x+1)$ by $\nr(\fa)\times \nr(\fc\fa^{-1})^{\{1-{\underline l}/2+(\s+1)/2+\e\}/d_F}$ uniformly in the ideals $\fc$ and $\fa$. Therefore,  
	\begin{align*}
		T(\fn)&\ll \sum_{\fc|\fn}\nr(\fc^{-1}\fn)^{-1+\e/2}
		\times \nr(\fc)^{-1/2+\e/2}
		\times \nr(\fc)^{(1-{\underline l}/2+(\s+1)/2+\e)/d_F}\\
		&\ll \nr(\fn)^{-1+\e/2} \sum_{\fc|\fn}\nr(\fc)^{1/2+(1-{\underline l}/2+(\s+1)/2+\e)/d_F}. 
	\end{align*}
	Suppose ${\underline l}>d_F+3+\s$. Then we can choose $\e>0$ so that $1/2+(1-{\underline l}/2+(\s+1)/2+\e)/d_F \leq 0$. Thus, 
	\begin{align*}
		T(\fn)\ll \nr(\fn)^{-1+\e/2}\times \sum_{\fc|\fn}1\ll \nr(\fn)^{-1+\e}. 
	\end{align*}
	Suppose ${\underline l}\leq d_F+3+\s$. Then $1/2+(1-{\underline l}/2+(\s+1)/2+\e)/d_F>0$. Thus,  
	\begin{align*}
		T(\fn)&\ll \nr(\fn)^{-1+\e/2} \times \nr(\fn)^{1/2+(1-{\underline l}/2+(\s+1)/2+\e)/d_F}\sum_{\fc|\fn}1
		\ll \nr(\fn)^{-1/2+\e+(1-{\underline l}/2+(\s+1)/2+\e)/d_F}.
	\end{align*}
	This completes the proof. 
\end{proof}

\begin{lem}\label{MVT-L5}
	
	Set $\delta=1$ if ${\underline l}>d_F+2$ and $\delta=1/2+({\underline l}-2)/(2d_F)$ if $d_F+2\geq {\underline l}\geq 4$. Then $1/2<\delta \le 1$, and for any small $\e>0$,
	\begin{align*}
		|\JJ_{\rm ell}^{0}(\fn|\a,z)|\ll \nr(\fn)^{-\delta+\e}+\nr(\fn)^{-1+\e}
	\end{align*}
	uniformly in $z$ on the strip $|\Re(z)|\le 1$. 
\end{lem}
\begin{proof} In this proof, we set $d=d_F$. 
	From \eqref{JJell} and Proposition~\ref{EisPeriod}, by changing the order of the integral and the summation, 
	\begin{align}
		\JJ_{\rm ell}^0(\fn|\a,z)=\, &\frac{1}{2}\sum_{\gamma=(t:n)_{F}\in \Qcal_{F}^{\rm Irr}\cap \Qcal_F^{S}} D_F^{z/2}L\left(\tfrac{z+1}{2},\varepsilon_{t^2-4n}\right)\nr({\frak D}_{t^2-4n})^{\frac{z+1}{4}}{\bf D}_\gamma(z)\{\prod_{v\in \Sigma_F-S}\tilde \fE_v^{(z)}(\hat \gamma_v)\} \\
		& \times \prod_{v\in S}\tfrac{1}{2\pi i }\tint_{L(c_v)} \tilde \fE_{v}^{(z)}(s_v;\hat\gamma_v)\,\a_{v}(s_v)\,\d\mu_{v}(s_v).
		\label{MVT-L5-0}
	\end{align}
	Here ${\bf D}_\gamma(z)$ is the dyadic factor of the formula of $E^{\Delta}(z;1_2)$ in Proposition~\ref{EisPeriod}, and $\tilde \fE_v^{(z)}(\hat \gamma_v)$ equals $|2m_v|_v^{-1}\fE_v^{(z)}(\hat \gamma_v)$ or $|m_v|_v^{-1}\fE_v^{(z)}(\hat \gamma_v)$ according to ``$\Delta_v^{0}=1$ or $v \in \Sigma_{\infty}$'' or not, respectively. We choose the representatives $(t,\varepsilon n_{i,\nu})$ of $\Qcal_F^{\rm Irr}\cap \Qcal_{F}^{S}$ as in Lemma~\ref{ABC-L7}. For $v\in S$ such that $|t|_v^2>|n_{i,\nu}|_v$, we have $|t|_v=1$ from $\min(2\ord_v(t),\nu_v)\in \{0,1\}$; since $v$ is non-dyadic and $t^2-4\varepsilon n_{i,\nu}$ is a square residue modulo $\fp_v$, we have $t^2-4\varepsilon n_{i,\nu}\in (F_v^\times)^2$ by Lemma \ref{Hensel}. Then Proposition~\ref{OrbIntUnifExL4} implies $|\tfrac{n}{m_v^2}|_v=|n_{i,\nu}|_v=q_v^{-\nu_v}$. From Theorem~\ref{OrbIntUnifEx} and Lemma~\ref{SfactorvanishingL},
	\begin{align}
		\tfrac{1}{2\pi i }\tint_{L(c_v)} \tilde \fE_{v}^{(z)}(s_v;\hat\gamma_v)\,\sigma_{k}(s_v)\,\d\mu_{v}(s_v)=0
		\label{MVT-L5-1}
	\end{align}
	unless $|\frac{n}{m_v^2}|_v \geq q_v^{-k}$ for all $v\in S$. Thus we may suppose $\nu_v\le n_{v}(=\ord_{v}(\fa))$ for all $v\in S$ such that $|n_{i,\nu}|_v\leq |t|_v^2$. Combined with the constraint $\min(2\ord_{v}(t),\nu_v)\in \{0,1\}$ $(v\in S)$, this implies the vanishing of the $S$-factor \eqref{MVT-L5-0} except for finitely many $n=\varepsilon n_{i,\nu}$. Thus $|\JJ_{\rm ell}^{0}(\fn | \a,z)|$ is majorized by the sum of all $\Xi^{(z,\bfc)}(\fc,\fn_1,i,\varepsilon n_{i,\nu})$ with $\nu_v\le \ord_v(\fa)$ for all $v\in S$. Then we invoke the estimate Lemma~\ref{ABC-L8} for each of these to obtain the majorization $|\JJ_{\rm ell}^0(\fn|\a,z)|\ll_\e \sum_{i}\Sigma(\fn,N_i)+\nr(\fn)^{-1+\e}$ uniformly in $z$ with $\Re(z)\in [-2\underline{l}+3+6\e,\underline{l}-1/2-\e]$ and
	\begin{align*}
		\Sigma(\fn,N_i)&=
		\sum_{\fn_1|\fn} \sum_{\fc|\fn_1} \nr(\fn)^{-1+\e}\nr(\fn_1)^{1/2+\e/2}
		\max\biggl(\sqrt{\max(0,\nr(\fn_1^2\fc^{-2})^{1/d}-4N_i^{1/d})},\,\nr(\fc)^{1/d}\biggr)^{1-L},
	\end{align*}
	where $L={\underline L}(z)={\underline l}-\frac{1+\ro(z)}{2}-2\e$ and $N_i=|\nr(n_{i,\nu})|$. We divide the sum $\Sigma(\fn, N)$ to two parts $\Sigma_{I}$ and $\Sigma_{II}$ according as $\nr(\fn_1^2\fc^{-2})^{1/d}\geq 8 N^{1/d}$ or not. Set $M=8^{d/2}N^{1/2}$. Then $\Sigma_{I}$ is over those pairs of ideals $(\fn_1,\fc)$ such that $\fn_1|\fn$, $\fc|\fn_1$ and $\nr(\fn_1)\geq M\,\nr(\fc)$. As such, 
	\begin{align*}
		&\max\biggl(\sqrt{\max(0,\nr(\fn_1^2\fc^{-2})^{1/d}-4N^{1/d})},\,\nr(\fc)^{1/d}\biggr)
		\\
		\geq & \max(\sqrt{\nr(\fn_1^2\fc^{-2})^{1/d}-\tfrac{1}{2}\nr(\fn_1^2\fc^{-2})^{1/d} },\,\nr(\fc)^{1/d})
		\geq 2^{-1/2} \max(\nr(\fn_1\fc^{-1})^{1/d}
		,\,\nr(\fc)^{1/d})
		\\
		\geq & 2^{-1/2} \{{\nr(\fn_1\fc^{-1})^{1/d}+\nr(\fc)^{1/d}}\}/{2}
		\geq2^{-1/2}\left(\nr(\fn_1\fc^{-1})^{1/d} \times \nr(\fc)^{1/d}\right)^{1/2}
		=2^{-1/2} \nr(\fn_1)^{1/(2d)}.
	\end{align*}
	Since $1-L$ is non-positive, 
	\begin{align*}
		\Sigma_{I}&\leq \nr(\fn)^{-1+\e} \sum_{\fn_1|\fn}\nr(\fn_1)^{1/2+\e/2}\sum_{\fc|\fn_1}(2^{-1/2}\nr(\fn_1)^{1/(2d)})^{1-L}
		\\
		&\ll \nr(\fn)^{-1+\e} \sum_{\fn_1|\fn} \sum_{\fc|\fn_1}\nr(\fn_1)^{1/2+\e/2+(1-L)/(2d)}
		\ll \nr(\fn)^{-1+\e}\sum_{\fn_1|\fn}\nr(\fn_1)^{\e+1/2+(1-L)/(2d)},
	\end{align*}
	where we use the estimate $\sum_{\fc|\fn_1}1\ll \nr(\fn_1)^{\e/2}$. 
	
	The sum $\Sigma_{II}$ is over all pairs of ideals $(\fn_1,\fc)$ such that $\fn_1|\fn$, $\fc|\fn_1$, $M^{-1}\,\nr(\fn_1)\leq \nr(\fc)\leq \nr(\fn_1)$. By $\max(A,B)\geq B$ and by noting $1-L\leq 0$, we have trivially
	{\allowdisplaybreaks\begin{align*}
			\Sigma_{II}&\leq 
			\sum_{\fn_1|\fn}\sum_{\fc|\fn_1}\nr(\fn)^{-1+\e}\nr(\fn_1)^{1/2+\e/2}
			(\nr(\fc)^{1/d})^{1-L}
			\ll \sum_{\fn_1|\fn}\sum_{\fc|\fn_1}\nr(\fn)^{-1+\e}\nr(\fn_1)^{1/2+\e/2+(1-L)/d}
			\\
			&\ll \nr(\fn)^{-1+\e} \sum_{\fn_1|\fn}\nr(\fn_1)^{\e+1/2+(1-L)/d}
			\ll \nr(\fn)^{-1+\e} \sum_{\fn_1|\fn}\nr(\fn_1)^{\e+1/2+(1-L)/(2d)}. 
	\end{align*}}
	Thus
	$\Sigma(\fn,N)=\Sigma_{I}+\Sigma_{II} \ll \nr(\fn)^{-1+\e} \sum_{\fn_1|\fn}\nr(\fn_1)^{\e+1/2+(1-L)/(2d)}.$
	
	Suppose ${\underline l}>d+2$. Then for sufficiently small $\e>0$, we have $\e+1/2+(1-L)/(2d)\leq 0 $ for $\Re(z)\in [0,1]$. Then 
	$$
	\Sigma(\fn,N)\ll_{\e} \nr(\fn)^{-1+\e}\sum_{\fn_1|\fn}\nr(\fn_1)^{\e+1/2+(1-L)/(2d)}\leq \nr(\fn)^{-1+\e}\sum_{\fn_1|\fn}1\leq \nr(\fn)^{-1+2\e}.
	$$
	Suppose ${\underline l}\leq d+2$. Then $\e+1/2+(1-L)/(2d)>0$ for $\Re(z)\in [0,1]$. Thus
	\begin{align*}
		\Sigma(\fn,N)\ll \nr(\fn)^{-1+\e}\times\nr(\fn)^{\e+1/2+(1-L)/(2d)} \sum_{\fn_1|\fn}1 \ll\nr(\fn)^{-1/2+3\e+(1-L)/(2d)}.
	\end{align*}
	Note $L={\underline l}-1-2\e$ for $|\Re(z)|\le 1$. This completes the proof.
\end{proof} 

\noindent
{\bf Remark} : With a bit more work, it can be shown that the implied constants in Lemmas~\ref{MVT-L4} and \ref{MVT-L5} are taken to be of the form $\|\alpha\|_{\Acal(\fa)}\nr(\fa)^{N}$ with $\|\,\|_{\Acal(\fa)}$ being a norm on the finite dimensional space $\Acal(\fa)$. 

\smallskip
Theorem~\ref{MVT-thm} (1) for $f\in \Pcal(\Omega_S)$ follows from Corollary~\ref{MAINTHM} combined with Lemmas~\ref{MVT-L4} and \ref{MVT-L5} immediately. The assertion of Theorem~\ref{MVT-thm} (1) for $f \in C(\Omega_S)$ is proved by a standard argement (cf.\ \cite[Proposition 2]{Serre}), where the non-negativity of the measure $\Lambda_{l,\fn}^{(z)}$ is indispensable.
Indeed,
for any $f \in C(\Omega_S)$ and any $\e>0$,
the Weierstrass approximation theorem allows us to take a polynomial function $P_\e \in \Pcal(\Omega_S)$ such that $\sup_{x \in \Omega_S}|f(x)-P_\e(x)|<\e$.
We set $\Lambda_\gn^{(z)}= r(z)^{-1}(C_l^{(z)})^{-1}\Lambda_{l,\gn}^{(z)}$ and $\lambda_S^{(z)}=\otimes_{v \in S}\l_v^{(z)}$ for simplicity.
By the condition ({\bf P}), we have
$|\langle \Lambda_{\gn}^{(z)}, f-P_\e \rangle| \le \langle \Lambda_{\gn}^{(z)}, 1\rangle \sup|f-P_\e|< \langle \Lambda_{\gn}^{(z)}, 1\rangle \e$.
Furthermore, there exists $N_\e>0$ such that $\sup_{z}|\langle\Lambda_{\gn}^{(z)}-\lambda_S^{(z)}, P_\e\rangle| < \e$
and $\sup_{z}|\langle \Lambda_\gn^{(z)}-\lambda_S^{(z)}, 1\rangle| < \langle\lambda_S^{(z)},1\rangle$ for all $\gn$ with $\nr(\gn)>N_\e$.
As a result, we have
\begin{align*}
	\sup_z|\langle\Lambda_{\gn}^{(z)}-\lambda_S^{(z)}, f\rangle|
	\le &\sup_z |\langle\Lambda_{\gn}^{(z)}, f-P_\e \rangle|
	+ \sup_z |\langle\Lambda_{\gn}^{(z)}-\lambda_S^{(z)}, P_\e\rangle|
	+ \sup_z |\langle\lambda_S^{(z)}, P_\e-f\rangle|\\
	<&(\sup_z 3\langle\lambda_S^{(z)},1\rangle + 1)\e.
\end{align*}
This completes the proof as $z$ varies in the compact set $[0,\min(1,\s)]$.

As for the proof of Theorem~\ref{MVT-thm} (2),
we first show that, under the assumption {\bf (P)}, the limit formula in Theorem \ref{MVT-thm} (1) is valid even for $f={\rm ch}_J$ with $J=\prod_{v \in S}[t_v,t'_v]$, which is discontinuous. We set $\Lambda_\gn^{(z)}= r(z)^{-1}(C_l^{(z)})^{-1}\Lambda_{l,\gn}^{(z)}$ and $\lambda_S^{(z)}=\otimes_{v \in S}\l_v^{(z)}$ for simplicity.
Put $J_\delta=\Omega_S\cap \prod_{v\in S}[t_v-\delta,t'_v+\delta]$ for any $\delta \in \RR$.
Then $[0,1] \ni z \mapsto \vol(J_\delta-J_{-\delta};\lambda_S^{(z)})$ is monotonously increasing for each $\delta>0$.
For any fixed $\e>0$, take $\delta>0$ such that $\vol(J_\delta-J_{-\delta};\lambda_S^{(1)})<\e$,
and $f_\e, g_\e \in C(\Omega_S)$ such that ${\rm ch}_{J_{-\delta}} \le g_\e\le {\rm ch}_J \le f_\e \le {\rm ch}_{J_\delta}$.
By Theorem~\ref{MVT-thm} (1), there exists $N_\e>0$ such that both $|\langle \Lambda_{\gn}^{(z)} -\lambda_{S}^{(z)}, f_\e \rangle|<\e$
and $\langle \Lambda_{\fn}^{(z)}, f_\e-g_\e\rangle < \langle \lambda_S^{(z)},  f_\e-g_\e \rangle+\e$ hold for all $\gn$ with $\nr(\gn)>N_\e$ and for all $z\in [0,\min(1,\s)]$. Then, $|\langle  \Lambda_{\gn}^{(z)} -\lambda_{S}^{(z)}, {\rm ch}_J \rangle|$ for $\nr(\gn)>N_\e$ with $z\in [0,\min(1,\s)]$ is majorized by
\begin{align*}
	& |\langle \Lambda_{\gn}^{(z)}, {\rm ch}_J-f_\e \rangle| +|\langle \Lambda_{\gn}^{(z)} -\lambda_{S}^{(z)}, f_\e \rangle| +|\langle \lambda_S^{(z)}, f_\e -{\rm ch}_J\rangle|
	\le \langle \Lambda_{\fn}^{(z)}, f_\e-g_\e \rangle + \e + \vol(J_\delta-J;\lambda_S^{(1)})\\
	\le& (\langle \lambda_S^{(z)}, f_\e-g_\e \rangle +\e) +2\e \le \vol(J_\delta-J_{-\delta}; \lambda_S^{(1)})+3\e<4\e.
\end{align*}
This completes the required convergence for $f={\rm ch}_J$.

Let us return to the proof of Theorem~\ref{MVT-thm} (2).
Let $\Ical$ be the set of all the prime ideals $\fn\subset \cO$ relatively prime to $2\gp_S\prod_{j=1}^{h}\fa_j$.
Set $f={\rm ch}_J$.
As shown above, $\lim_{\nr(\gn)\rightarrow \infty}\Lambda_{l,\gn}^{(z)}(f)>0$ holds with the convergence being uniform in $z\in [0,\min(1,\s)]$. Let us define $\Lambda_{l,\gn}^{(z),*}(f)$ by the same formula as $\Lambda_{l,\gn}^{(z)}(f)$ reducing the summation range from $\Pi_{\rm cus}(l,\fn)$ to $\Pi_{\rm cus}(l,\fn)-\Pi_{\rm cus}(l,\fo)$. By the uniform bound $0<W_{\fn}^{(z)}(\pi)\ll_{\e}\nr(\fn)^{1/2+\e}$ $(z\in [0,1],\,\pi \in \Pi_{\rm cus}(l,\cO)$), we easiy confirm the difference $|\Lambda_{l,\gn}^{(z)}(f)-\Lambda_{l,\gn}^{(z)*}(f)|$ tends to $0$ as $\nr(\fn)\rightarrow \infty$ uniformly in $z\in [0,1]$. Hence there exists $M>0$ such that $\Lambda_{l,\gn}^{(z)*}(f)>0$
for all $z \in [0,\min(1,\s)]$ and for all $\fn \in \Ical$ with $\nr(\fn)>M$.
Hence for any $\fn \in \Ical$ with $\nr(\fn)>M$ and any $z \in [0,\min(1,\s)]$,
there exists $\pi \in \Pi_{\rm cus}(l,\fn)-\Pi_{\rm cus}(l,\cO)$ such that $L(\tfrac{z+1}{2},\pi;\Ad){\rm ch}_{J}(\bfx_S(\pi))\neq 0$. We have $\gf_{\pi}=\fn$ since $\fn$ is prime. This completes the proof. We remark that in the proof above the condition {\bf (P)} is not necessary when $J=\Omega_{S}$, in which case $f$ belongs to $\Pcal(\Omega_S)$.

\section{The cuspidal case}
Let ${\mathbf \Phi}^l(\gn|\bfs,g,h)$ be the automorphic kernel function constructed in \S~\ref{The kernel function}. Let $\xi\cong \otimes_v\xi_v$ be an irreducible cuspidal automorphic representation of $G_\A$ with trivial central character which is everywhere unramified, i.e., $\xi_v^{\bK_v}\not=\{0\}$ for all $v\in \Sigma_F$. Thus we have a set of numbers $\nu_v(\xi) \in i\RR_{\ge0} \cup (0,1)$ $(v\in \Sigma_\infty)$ and $\nu_v(\xi) \in \{iy \ | \ 0\le y \le 2\pi(\log q_v)^{-1}\} \cup \{x+iy \ | \ 0<x<1, \ y\in\{0,2\pi(\log q_v)^{-1}\}\}$ $(v\in \Sigma_\fin)$ such that $\xi_v\cong I(|\,|_v^{\nu_v(\xi)/2})$ for all $v\in \Sigma_F$. Let $\varphi_\xi^{0} =\varphi_\xi^{\rm new}$ be the new vector of $\xi$. Since $\varphi_\xi^0$ is rapidly decreasing on the Siegel set $\fS^1$ of $G_\AA^1$, the integral 
\begin{align}
	\int_{Z_\A G_F\bsl G_\A}\varphi_\xi^{0}(g){\mathbf \Phi}^l(\gn|\bfs,g,g)\,\d g
	\label{CUSPINT}
\end{align}
converges absolutely. By replacing $\Ecal_\b(g)$ with $\varphi_\xi^0$ in every occurrence, almost all proofs to compute \eqref{IIb} so far also work for the integral \eqref{CUSPINT} by a slight modification. We end up with a formula resemble to the one in Corollary~\ref{MAINTHM}. To describe it, we need further notation. For $\pi \in \Pi_{\rm cus}(l,\fn)$, we consider the sum of triple product of cusp forms  
$$
\PP_{\varphi_\xi^0}(l,\fn|\pi)=\sum_{\varphi\in \Bcal_\pi(l,\fn)}\langle \varphi_\xi^0|\varphi\,\bar\varphi \rangle_{L^2}
$$
as \eqref{AverageInnprod}. We modify the definition of
${\bf B}^{(z)}_{\fn}(\bfs|\Delta;\fa)$
in \S~\ref{DMR} by replacing the parameter $z$ with $\nu_v(\xi)$ in the $v$-factor for all $v\in \Sigma_F$. i.e.,
$$
{\bf B}_{\fn}^{\xi}(\bfs|\Delta;\fa)=\prod_{v\in \Sigma_\fin-(S\cup S(\fn))}\Ocal_{0,v}^{\Delta,(\nu_v(\xi))}(a_v)\prod_{v\in S(\fn)}\Ocal_{1,v}^{\Delta,(\nu_v(\xi))}(a_v)\prod_{v\in S}
\Scal_v^{\Delta,(\nu_v(\xi))}(s_v,a_v).
$$
For any quadratic field extension $E=F(\sqrt{\Delta})$ with a prescribed square root of $\Delta\in F^\times$, we define
\begin{align*}
	\Pcal_\Delta(\varphi_\xi^0) = D_F^{-1/2}\nr(\fd_{E/F})^{1/2}
	\{\prod_{ \substack{v\in \Sigma_{\rm dyadic}  \\ \Delta_v^{0}=5}} 2^{-(1+\nu_v(\xi))/2}3(1+2^{-\nu_v(\xi)})^{-1}\}
	\int_{\A^\times E^\times\bsl \A_{E}^\times}\varphi_\xi^0\bigl(\iota_\Delta(\tau)R_\Delta^{-1}\bigr)\,\d^\times \tau. 
\end{align*}

\begin{thm}\label{MAINTHM2} Retain all the assumptions and notation in Theorem \ref{thm0}.
	When $\min_{v\in S}\Re(s_v)>2\min_{v \in \Sigma_{\infty}}l_v-1$, we have the identity
	\begin{align*}
 & (-1)^{\#S} C(l,\fn)\,\sum_{\pi \in \Pi_{\rm cus}(l,\fn)}
		\frac{\PP_{\varphi_\xi^0}(l,\fn|\pi)}{\prod_{v\in S} \{(q_v^{(1+\nu_v(\pi))/2}+q_v^{(1-\nu_v(\pi))/2})-(q_v^{(1+s_v)/2}+q_v^{(1-s_v)/2})\}}\\
		=&\JJ_{\rm hyp}(\bfs,\varphi_\xi^{0})+\JJ_{\rm ell}(\bfs,\varphi_\xi^{0}).\end{align*}
	The right-hand side is given by the absolutely convergent sums 
	\begin{align*}
		\JJ_{\rm hyp}(\bfs,\varphi_\xi^{0})&= \tfrac{1}{2}D_F^{-1}L\left(\tfrac{1}{2},\xi\right)
		\sum_{a\in \cO(S)^\times-\{1\}}{\bf B}_{\fn}^{\xi}(\bfs|1;a(a-1)^{-2}\cO)\prod_{v\in \Sigma_\infty}\Ocal_{v}^{+,(\nu_v(\xi))}((a+1)/(a-1)), \\
		\JJ_{\rm ell}(\bfs,\varphi_\xi^{0})&=\tfrac{1}{2} \sum_{(t:n)_F}
		\Pcal_\Delta(\varphi_\xi^0)\, {\bf B}_{\fn}^{\xi}(\bfs|\Delta;n\ff_{\Delta}^{-2})\,\prod_{v\in \Sigma_\infty}\Ocal_v^{\sgn(\Delta^{(v)}),(\nu_v(\xi))}(t|\Delta|_v^{-1/2}),
	\end{align*}
	where $\Delta=t^2-4n$ and $(t:n)_F$ ranges over the same set as in Corollary~\ref{MAINTHM}.
	
\end{thm}
\begin{proof}
	In accordance with \eqref{kernelftndec}, the integral \eqref{CUSPINT} breaks up to the sum of four terms $\JJ_{\natural}(\bfs,\varphi_\xi^0)=\int_{Z_\A G_F\bsl G_\A}
	\varphi_\xi^0(g)J_{\natural}(\bfs;g)\,\d g$ with $\natural\in \{{\rm id},{\rm unip},{\rm hyp},{\rm ell}\}$. By the cuspidality of $\varphi_\xi^0$, it is seen easily that $\JJ_{\rm id}(\bfs;\varphi_\xi^0)=\JJ_{\rm unip}(\bfs;\varphi_\xi^0)=0$. By the same argument as in \S~\ref{The $F$-hyperbolic term}, the hyperbolic term becomes 
	$$
	\JJ_{\rm hyp}(\bfs,\varphi_\xi^0)=\tfrac{1}{2} \Pcal_1(\varphi_\xi^0) \sum_{a\in F^\times-\{1\}} 
	\prod_{v\in \Sigma_F}\fF_v^{(\nu_v(\xi))}(a),
	$$
	with $\Pcal_1(\varphi_\xi^0)=\int_{F^\times \bsl \A^\times}
	\varphi_{\xi}^{0}\left(\left[\begin{smallmatrix} t & 0 \\ 0 & 1 \end{smallmatrix}\right]\right)\,\d^\times t$ being the Hecke's integral, which is identified with the central value of $L$-function: $\Pcal_1(\varphi_\xi^0)=D_F^{-1/2}L(1/2,\xi)$. Since $\varphi_\xi^0$ has no constant term in the Fourier expansion, the first part of \S~\ref{UnConSh} is irrelevant; the absolute convergence is shown as in \S~\ref{PAbsConv}. In the same way as in \S~\ref{The $F$-elliptic term}, the elliptic term becomes 
	$$
	\JJ_{\rm ell}(\bfs,\varphi_\xi^0)=\tfrac{1}{2}\sum_{(t:n)_F\in \Qcal_{F}^{\rm Irr}\cap \Qcal_{F}^{S}}\{\tint_{\A^\times E^\times\bsl \A_{E}^\times}\varphi_\xi^0\bigl(\iota_\Delta(\tau)R_\Delta^{-1}\bigr)\,\d^\times \tau\}
	\prod_{v\in \Sigma_F}\fE_{v}^{(\nu_v(\xi))}(\hat \gamma_v).
	$$ 
	Since $\lim_{z\rightarrow 1} (z-1)E^{\Delta}(z;1_2)=D_F^{-1/2}\zeta_F(2)^{-1}{\rm Res}_{z=1}\zeta_F(z)\,\vol(\A^\times E^\times \bsl \A_E^\times;\d^\times \tau)$ (cf.\ \cite[Lemma 2.13]{Tsuzuki2015}), from Lemma~\ref{EisPeriodEst}, we obtain the majorization $|\tint_{\A^\times E^\times\bsl \A_{E}^\times}\varphi_\xi^0\bigl(\iota_\Delta(\tau)R_\Delta^{-1}\bigr)\,\d^\times \tau|\ll \vol(\A^\times E^\times \bsl \A_E^\times;\d^\times \tau) \ll_{\e} \{\prod_{v\in \Sigma_F}|m_v|_v^{-1}\}\nr({\frak D}_{\Delta})^{\frac{1}{2}+2\e}$, which should be a substitute of Lemma~\ref{EisPeriodEst}. By Theorem~\ref{OrbIntUnifEx} and the arguments in \S~\ref{ABCelliptic}, we have the absolute convergence and the desired formula. \end{proof}

Theorem \ref{MAINTHM2} will be applied to non-vanishing of $L$-values for $\GL(2)\times \GL(3)$
in our forthcoming paper \cite{SugiyamaTsuzuki2018}.

\section{Explicit formulas of local orbital integrals} \label{Localorbitalintegrals}

In this section, we prove Theorems \ref{explicit hyp} and 
\ref{OrbIntUnifEx} by separating cases as in the following table.
\begin{center}\begin{table}[htb]
		\begin{tabular}{|c|c|c|} \hline
			$v$ & $E_v/F_v$ & Proof \\ \hline
			$\Sigma_\infty$ & split & \S \ref{The proof of Theorem explicit hyp (1)} \\ \hline
			$\Sigma_\infty$ & ramified & \S \ref{EllOrbAcase} \\ \hline
			$\Sigma_{\fin}-(S\cup S(\gn))$ & split & \S\ref{The proof of Theorem explicit hyp (2)} \\ \hline
			$\Sigma_{\fin}-(S\cup S(\gn))$ & non-split & \S\ref{EllOrbIntL1}\\ \hline
			$S(\gn)$ & split & \S \ref{The proof of Theorem explicit hyp (3)}\\ \hline
			$S(\gn)$ & non-split & \S \ref{The proof of Theorem OrbIntUnifEx (4)}\\ \hline
			$S$ & split & \S \ref{The proof of Theorem explicit hyp (4)} \\ \hline
			$S$ & non-split & \S \ref{The proof of Theorem OrbIntUnifEx (5)}\\ \hline
		\end{tabular}
	\end{table}
\end{center}
To ease notation,
in the archimedean cases (\S \ref{The proof of Theorem explicit hyp (1)} and \ref{EllOrbAcase}), we write $l$ and $|\,|$ for $l_v$ and $|\,|_v$, respectively.
In the non-archimedean cases (\S \ref{The proof of Theorem explicit hyp (2)}, \S \ref{The proof of Theorem explicit hyp (3)},
\S \ref{The proof of Theorem explicit hyp (4)},
\S \ref{EllOrbIntL1},
\S \ref{The proof of Theorem OrbIntUnifEx (4)} and
\S \ref{The proof of Theorem OrbIntUnifEx (5)}), we omit the subscript $v$ of $\cO_v$, $\fp_v$, $q_v$, $d_v$ and $|\,|_v$ and write them $\cO$, $\fp$, $q$, $d$ and $|\,|$, respectively.

\subsection{Local hyperbolic orbital integrals}
Let $v \in \Sigma_F$. In this subsection, we compute the hyperbolic local orbital integral
$$
\fF_v^{(z)}(a)=\int_{F_v} \left(\int_{\bK_v}\Phi_v\left(k^{-1}\left[\begin{smallmatrix} a & (a-1)x \\ 0 & 1 \end{smallmatrix} \right] k\right)\,\d k \right)\,\varphi_v^{(0,z)} \left(\left[\begin{smallmatrix} 1 & x \\ 0 & 1 \end{smallmatrix}\right] \right)\,\d x,  \quad v\in \Sigma_F, \quad a \in F_v^\times-\{1\},
$$
where $\Phi_v(g_v)$ is the $v$-th factor of $\Phi(\bfs;g)$.

\subsubsection{The proof of Theorem \ref{explicit hyp} (1)} \label{The proof of Theorem explicit hyp (1)}
Suppose $v\in \Sigma_\infty$, so that $F_v\cong \RR$. Fix $a \in F_v^\times-\{1\}$.
\begin{lem}\label{FHypL8} 
	On the region $|\Re(z)|<1$, we have
	$$
	\fF_v^{(z)}(a)=A_v(0,z)J_v^{(z)}(a)+A_v(0,-z)\,J_{v}^{(-z)}(a),
	$$
	where $J_v^{(z)}(a)$ for $\Re(z)>-1$ is defined as
	\begin{align}
		&\delta(a>0)\,2^{l}\pi^{1/2}\frac{\Gamma\left(\tfrac{z}{2}+1\right)\Gamma\left(\tfrac{l-1}{2}+\tfrac{z+1}{4}\right)
			\Gamma\left( \tfrac{l}{2}+\tfrac{z+1}{4}\right)}{\Gamma\left(\tfrac{z+1}{4}\right)^2\Gamma\left(\tfrac{z+3}{4}\right)\Gamma\left(l+\tfrac{z+1}{4}\right)} \label{FHypL8-f0}\\
&\times		|a|^{l/2}|a+1|^{-l}\int_{0}^{1}(1-y)^{\frac{z+1}{4}-1}{}_2F_1\left(\tfrac{l+1}{2},\tfrac{l}{2};l+\tfrac{z+1}{4};1-
		\left|\tfrac{a-1}{a+1}\right|^2\,y\right)\,\d y.
		\notag
	\end{align}
\end{lem}
\begin{proof}
	Set
	$$
	J_v^{(z)}(a)=\tint_{\bK_v}\tint_{\R}\Phi_v^{l}\left(k^{-1}\left[\begin{smallmatrix} a & (a-1)x \\ 0 & 1 \end{smallmatrix}\right]k\right)\,h_v^{(0,z)}\left(x
	\right)\,\d x\, \d k.
	$$By the Cartan decomposition, we have 
	$$
	\Phi_v^{l}\left(k^{-1}\left[\begin{smallmatrix} a & (a-1)x \\ 0 & 1 \end{smallmatrix}\right]k\right)=\delta(a>0)\,2^{l}a^{l/2}\,\{(a+1)+i(a-1)x\}^{-l}
	$$
	for any $k\in \bK_v$, $a\in \R^\times$ and $x\in \R$. Since $|h^{(0,z)}(x)|\ll |\log x|$ as $x\rightarrow 0$ and $|h^{(0,z)}(x)|\ll (1+x^{2})^{-\frac{\Re(z)+1}{4}}$ as $|x|\rightarrow \infty$, the integral $J_v^{(z)}(a)$ converges absolutely for $\Re(z)>1-2l$ and defines a holomorphic function. From now on, we suppose $a>0$ and $a\neq 1$. Then, 
	{\allowdisplaybreaks\begin{align*}
			J_v^{(z)}(a)&=2^{l}a^{l/2}\tint_{\R}\{(a+1)-i(a-1)x\}^{-l}(1+x^2)^{-\frac{z+1}{4}}{}_2F_1\left(\tfrac{z+1}{4},\tfrac{z+1}{4};\tfrac{z}{2}+1;\tfrac{1}{1+x^2}\right)
			\,\d x
			\\
			&=
			2^{l}a^{l/2}\frac{\Gamma\left(\frac{z}{2}+1\right)}{\Gamma\left(\tfrac{z+1}{4}\right)^2}
			\sum_{m=0}^{\infty}\frac{\Gamma\left(\tfrac{z+1}{4}+m\right)^{2}}{m!\,\Gamma\left(\tfrac{z}{2}+m+1\right)}\tint_{\R}\{(a+1)-i(a-1)x\}^{-l}(1+x^2)^{-\frac{z+1}{4}-m}\,\d x,
	\end{align*}}which is valid for $z\notin\{0,-2,-4,\ldots\}$.
	For any $\Re(\a)>1/2$, 
	{\allowdisplaybreaks\begin{align*}
			&\tint_{\R}\{(a+1)-i(a-1)x\}^{-l}(1+x^2)^{-\a}\,\d x\\
			=&\tint_{\R}{\Gamma(l)^{-1}} \int_{0}^{\infty} e^{-\{(a+1)-i(a-1)x\}t} t^{l-1}\d t\,(1+x^{2})^{-\a}\,\d x
			\\
			=&{\Gamma(l)^{-1}} \tint_{0}^{\infty} e^{-(a+1)t}t^{l-1}\left(\tint_{\R}e^{(a-1)itx}(1+x^{2})^{-\a}\d x\right)\,\d t
			\\
			=&2\pi^{1/2}\left(\tfrac{|a-1|}{2}\right)^{\a-1/2}\Gamma(l)^{-1}\Gamma(\a)^{-1}\tint_{0}^{\infty}t^{l+\a-3/2}e^{-(a+1)t}K_{\a-1/2}(|a-1| t)\,\d t
	\end{align*}}by the formula in the last line of \cite[p.85]{MOS}.
	From now on suppose $\Re(z)>1$. By this, we see that $J_v^{(z)}(a)$ for $\Re(z)>1$ is the product of 
	\begin{align}
		{2^{l+1}a^{l/2}\pi^{1/2}\,\Gamma\left(\tfrac{z}{2}+1\right)}{\Gamma(l)^{-1}\,\Gamma\left(\tfrac{z+1}{4}\right)^{-2}}
		\label{FHypL8-1}
	\end{align}
	and 
	\begin{align}
		\sum_{m=0}^{\infty}\frac{\Gamma\left(\tfrac{z+1}{4}+m\right)}{m!\,\Gamma\left(\tfrac{z}{2}+1+m\right)}
		\left(\tfrac{|a-1|}{2}\right)^{\tfrac{z+1}{4}+m-1/2}\tint_{0}^{\infty}t^{l+\frac{z+1}{4}+m-\frac{3}{2}}e^{-(a+1)t}K_{\frac{z+1}{4}+m-\frac{1}{2}}(|a-1|t)\,\d t.
		\label{FHypL8-2}
	\end{align} 
	By $K_{\a}(z)=\frac{1}{2}\int_{0}^{\infty}\exp\left(-\tfrac{z}{2}(y+y^{-1})\right)\,y^{-\a-1}\d y$, the formula \eqref{FHypL8-2} becomes
	{\allowdisplaybreaks 
		\begin{align*}
			&\left(\tfrac{|a-1|}{2}\right)^{\frac{z+1}{4}-\frac{1}{2}}
			\sum_{m=0}^{\infty}\frac{\Gamma\left(\tfrac{z+1}{4}+m\right)}
			{m!\,\Gamma\left(\tfrac{z}{2}+1+m\right)}
			\left(\tfrac{|a-1|}{2}\right)^{m}\tint_{0}^{\infty}t^{l+\frac{z+1}{4}+m-\frac{3}{2}}e^{-(a+1)t}\,\d t\,\\
			& \times
			\tfrac{1}{2}\tint_{0}^{\infty} \exp\left(-\tfrac{|a-1| t}{2}(y+y^{-1})\right)y^{-\frac{z+1}{4}-m-\frac{1}{2}}\d y
			\\
			=&
			\tfrac{1}{2}\left(\tfrac{|a-1|}{2}\right)^{\frac{z+1}{4}-\frac{1}{2}} 
			{\Gamma\left(\tfrac{z+1}{4}\right)}{\Gamma\left(\tfrac{z}{2}+1\right)^{-1}} \tint_{0}^{\infty}\int_{0}^\infty \d t \,\d y\, t^{l+\frac{z+1}{4}-\frac{3}{2}}e^{-(a+1)t}\\
			&\times \exp\left(-\tfrac{|a-1|t}{2}(y+y^{-1})\right){}_1F_1\left(\tfrac{z+1}{4},\tfrac{z}{2}+1;\tfrac{t}{y}\tfrac{|a-1|}{2}\right)
			y^{-\frac{z+1}{4}-\frac{1}{2}}
			\\
			=&
			\tfrac{1}{2}\left(\tfrac{|a-1|}{2}\right)^{\frac{z+1}{4}-\frac{1}{2}} 
			{\Gamma\left(\tfrac{z+1}{4}\right)}{\Gamma\left(\tfrac{z}{2}+1\right)^{-1}}
			\tint_{0}^{\infty}\int_{0}^\infty \d t \,\d y\,
			t^{l+\frac{z+1}{4}-\frac{3}{2}}e^{-(a+1)yt}\\
			&\times\exp\left(-\tfrac{|a-1| yt}{2}(y+y^{-1})\right){}_1F_1\left(\tfrac{z+1}{4},\tfrac{z}{2}+1;\tfrac{t|a-1|}{2}\right)
			y^{l-1}.
		\end{align*}
	}To have the last equality, we made the variable change $t\rightarrow ty$. The $y$-integral is computed as
	\begin{align*} 
		\tint_{0}^{\infty} \exp\left(\tfrac{-|a-1|}{2}ty^2-(a+1)ty\right)y^{l-1}\,\d y&=\left(\tfrac{1}{|a-1|t}\right)^{l/2} \tint_{0}^{\infty}
		\exp\left(-\tfrac{y^2}{2}-\tfrac{a+1}{|a-1|^{1/2}}t^{1/2}y\right)y^{l-1}\,\d y\\
		&=\left(\tfrac{1}{|a-1|t}\right)^{l/2}\Gamma(l)
		\exp\left(\tfrac{(a+1)^2t}{4|a-1|}\right)\,
		D_{-l}\left(\tfrac{a+1}{|a-1|^{1/2}}t^{1/2}\right)
	\end{align*}
	in terms of the parabolic cylinder function $D_{-l}(z)$ by the first formula on \cite[p.328]{MOS}.
	Hence,
	{\allowdisplaybreaks
		\begin{align*}
			J_v^{(z)}(a)=&{2^{l}a^{l/2} {\pi^{1/2}}}{\Gamma\left(\tfrac{z+1}{4}\right)^{-1}}
			|a-1|^{-l/2}
			\left(\tfrac{|a-1|}{2}\right)^{\frac{z+1}{4}-\frac{1}{2}} 
			\\
			&\times 
			\tint_{0}^{\infty}
			\exp\left(\tfrac{(a+1)^2t}{4|a-1|}-\tfrac{|a-1|t}{2}\right)\,
			D_{-l}\left(\tfrac{a+1}{|a-1|^{1/2}}t^{1/2}\right)\,{}_1F_1\left(\tfrac{z+1}{4},\tfrac{z}{2}+1;\tfrac{t|a-1|}{2}\right)\,t^{\frac{l}{2}+\frac{z+1}{4}-\frac{3}{2}}\,\d t
			\\
			=&{2^{l}a^{l/2} {\pi^{1/2}}}{\Gamma\left(\tfrac{z+1}{4}\right)^{-1}}
			|a-1|^{-l/2}
			\left(\tfrac{|a-1|}{2}\right)^{-l/2} 
			\\
			&\times 
			\tint_{0}^{\infty} 
			\exp\left(\left\{\tfrac{(a+1)^2}{2(a-1)^2}-1\right\}t\right)\,
			D_{-l}\left(\tfrac{a+1}{|a-1|}\sqrt{2t}\right)\,{}_1F_1\left(\tfrac{z+1}{4},\tfrac{z}{2}+1;t\right)\,t^{\frac{l}{2}+\frac{z+1}{4}-\frac{3}{2}}\,\d t
			\\
			=&
			{2^{l}a^{l/2} {\pi^{1/2}}}{\Gamma\left(\tfrac{z+1}{4}\right)^{-1}}
			|a-1|^{-l/2}
			\left(\tfrac{|a-1|}{2}\right)^{-l/2} 
			\\
			&\times 2^{-l/2}
			\tint_{0}^{\infty} 
			e^{-t}\,U\left(\tfrac{l}{2},\tfrac{1}{2},\left(\tfrac{a+1}{a-1}\right)^{2}t\right)\,{}_1F_1\left(\tfrac{z+1}{4},\tfrac{z}{2}+1;t\right)\,t^{\frac{l}{2}+\frac{z+1}{4}-\frac{3}{2}}\,\d t
		\end{align*}
	}by the formula on \cite[p.287]{MOS}. Applying the integral expression
	$$
	e^{-t}{}_1F_1 \left(\tfrac{z+1}{4},\tfrac{z}{2}+1;t\right)
	={\Gamma\left(\tfrac{z}{2}+1\right)}{\Gamma\left(\tfrac{z+1}{4}\right)^{-1}\Gamma\left(\tfrac{z+3}{4}\right)^{-1}}\tint_{0}^{1}e^{-ty}y^{\frac{z+3}{4}-1}(1-y)^{\frac{z+1}{4}-1}\,\d y
	$$
	obtained from the formula on the last line of \cite[p.274]{MOS} by an obvious variable change, and then using the formula
	{\allowdisplaybreaks\begin{align*}
			\tint_{0}^{\infty} t^{\frac{l}{2}+\frac{z+1}{4}-\frac{3}{2}}e^{-ty}
			U\left(\tfrac{l}{2},\tfrac{1}{2},\b t\right)
			\d t
			&=\b^{-\frac{l}{2}-\frac{z+1}{4}+\frac{1}{2}}\,\frac{\Gamma\left(\frac{l}{2}+\frac{z+1}{4}-\frac{1}{2}\right)\Gamma\left(\frac{l}{2}+\frac{z+1}{4}\right)}{\Gamma\left(l+\frac{z+1}{4}\right)}\\
			&\times{}_2F_1\left( \tfrac{l}{2}+\tfrac{z+1}{4}-\tfrac{1}{2},\tfrac{l}{2}+\tfrac{z+1}{4};l+\tfrac{z+1}{4};1-\tfrac{y}{\b}\right),
	\end{align*}}which is deduced from the first formula of \cite[\S 7.5.2]{MOS} by applying $U(a,c;z)=e^{z/2}z^{-c/2}W_{c/2-a,c/2-1/2}(z)$ (\cite[p.304]{MOS}), 
	it turns out that $J_v^{(z)}(a)$ is the product of 
	$$
	2^{l}\pi^{1/2}{\Gamma\left(\tfrac{z}{2}+1\right)\Gamma\left(\tfrac{l-1}{2}+\tfrac{z+1}{4}\right)
		\Gamma\left( \tfrac{l}{2}+\tfrac{z+1}{4}\right)}{\Gamma\left(\tfrac{z+1}{4}\right)^{-2}\Gamma\left(\tfrac{z+3}{4}\right)^{-1}\Gamma\left(l+\tfrac{z+1}{4}\right)^{-1}}
	$$
	and 
	$$
	|a|^{l/2}|a-1|^{-l}\l(a)^{\frac{l}{2}+\frac{z+1}{4}-\frac{1}{2}}
	\tint_{0}^{1}y^{\frac{z+3}{4}-1}(1-y)^{\frac{z+1}{4}-1}{}_2F_1\left(\tfrac{l-1}{2}+\tfrac{z+1}{4}, \tfrac{l}{2}+\tfrac{z+1}{4}; l+\tfrac{z+1}{4};1-\l(a)y
	\right)\,\d y, 
	$$
	where $\l(a)=|(a-1)/(a+1)|^{2}$. We note that the integral is convergent for $\Re(z)>-1$. We apply the formula ${}_2F_1(a,b;c;z)=(1-z)^{c-a-b}{}_2F_1(c-a,c-b;c;z)$ to obtain the desired formula \eqref{FHypL8-f0} for $\Re(z)>1$. Then the equality \eqref{FHypL8-f0} is extended to $\Re(z)>-1$ due to the holomorphy.
\end{proof}

Let us return to prove Theorem \ref{explicit hyp} (1).
Suppose $|\Re(z)|<1$ and set $\l=\l(a)$; we note $0<\l<1$ if $a>0$. By the Gauss connection formula for $(1-z)/4\notin \ZZ$ (on the last line of \cite[p.47]{MOS}), the integral
\begin{align*}
	&\tint_{0}^{1}(1-y)^{\frac{z+1}{4}-1}{}_2F_1\left(\tfrac{l+1}{2},\tfrac{l}{2};l+\tfrac{z+1}{4};1-\l{y}\right)\,\d y
\end{align*}
becomes the sum of 
\begin{align}
	\frac{\Gamma\left(l+\tfrac{z+1}{4}\right)
		\Gamma\left(\tfrac{z-1}{4}\right)
	}{\Gamma\left(\tfrac{z+1}{4}+\tfrac{l-1}{2}\right)
		\Gamma\left(\tfrac{z+1}{4}+\tfrac{l}{2}\right)
	}\tint_{0}^{1}(1-y)^{\frac{z-3}{4}}{}_2F_1\left( \tfrac{l+1}{2},\tfrac{l}{2};\tfrac{5-z}{4};\l y \right)dy
	\label{FHypL10-1}
\end{align}
and 
\begin{align}
	\l^{\frac{z-1}{4}}\,\frac{\Gamma\left(l+\tfrac{z+1}{4}\right)
		\Gamma\left(\tfrac{1-z}{4}\right)}
	{\Gamma\left(\tfrac{l+1}{2}\right)\Gamma\left(\tfrac{l}{2}\right)}
	\tint_{0}^{1}y^{\frac{z-1}{4}}(1-y)^{\frac{z-3}{4}}{}_2F_1\left(\tfrac{z+1}{4}+\tfrac{l}{2},\tfrac{z+1}{4}+\tfrac{l-1}{2};\tfrac{z+3}{4};\l y\right)\d y.
	\label{FHypL10-2}
\end{align}
We see that these are computed as{\allowdisplaybreaks
	\begin{align*}
		\eqref{FHypL10-1}=&\frac{\Gamma\left(l+\tfrac{z+1}{4}\right)
			\Gamma\left(\tfrac{z-1}{4}\right)
		}{\Gamma\left(\tfrac{z+1}{4}+\tfrac{l-1}{2}\right)
			\Gamma\left(\tfrac{z+1}{4}+\tfrac{l}{2}\right)}
		\frac{\Gamma\left(\tfrac{5-z}{4}\right)\Gamma\left(\tfrac{z+1}{4}\right)}{\Gamma\left(\tfrac{l}{2}\right)\Gamma\left(\tfrac{l+1}{2}\right)}
		\sum_{m=0}^{\infty} 
		\frac{\Gamma\left(\tfrac{l}{2}+m\right)\Gamma\left(\tfrac{l+1}{2}+m\right)}{\Gamma\left(\tfrac{z+1}{4}+m+1\right)\Gamma\left(\tfrac{5-z}{4}+m\right)}
		\,\l^{m}\end{align*}
	and
	\begin{align*}
		\eqref{FHypL10-2}=&\l^{\frac{z-1}{4}}\,\frac{\Gamma\left(l+\tfrac{z+1}{4}\right)
			\Gamma\left(\tfrac{1-z}{4}\right)}
		{\Gamma\left(\tfrac{l+1}{2}\right)\Gamma\left(\tfrac{l}{2}\right)}
		\frac{\Gamma\left(\tfrac{z+3}{4}\right)\Gamma\left(\tfrac{z+1}{4}\right)}{\Gamma\left(\tfrac{z+1}{4}+\tfrac{l}{2}\right)\Gamma\left(\tfrac{z+1}{4}+\tfrac{l-1}{2}\right)}
		\sum_{m=0}^{\infty}\frac{\Gamma\left(\tfrac{z+1}{4}+\tfrac{l}{2}+m\right)\Gamma\left(\tfrac{z+1}{4}+\tfrac{l-1}{2}+m\right)}
		{m!\,\Gamma\left(\tfrac{z}{2}+1+m\right)}\l^{m}. 
\end{align*}}
Combining these evaluations with the formula of $J_v^{(z)}(a)$ in Lemma~\ref{FHypL8}, we obtain
{\allowdisplaybreaks
	\begin{align*}
		J_v^{(z)}(a)= \, & \delta(a>0){2^{l}\pi^{1/2}\Gamma\left(\tfrac{z}{2}+1\right)}
		{\Gamma\left(\tfrac{l}{2}\right)^{-1}\Gamma\left(\tfrac{l+1}{2}\right)^{-1}\Gamma\left(\tfrac{z+1}{4}\right)^{-1}\Gamma\left(\tfrac{z+3}{4}\right)^{-1}}|a|^{l/2}|a-1|^{-l}
		\\
		&\times \biggl\{\Gamma\left(\tfrac{z-1}{4}\right)\Gamma\left(\tfrac{5-z}{4}\right) 
		\sum_{m=0}^{\infty} 
		\frac{\Gamma\left(\tfrac{l}{2}+m\right)\Gamma\left(\tfrac{l+1}{2}+m\right)}{\Gamma\left(\tfrac{5+z}{4}+m\right)\Gamma\left(\tfrac{5-z}{4}+m\right)}
		\,\l^{m}\\
		&\qquad +\l^{\frac{z-1}{4}} \Gamma\left(\tfrac{z+3}{4}\right)\Gamma\left(\tfrac{1-z}{4}\right) \sum_{m=0}^{\infty}\frac{\Gamma\left(\tfrac{z+1}{4}+\tfrac{l}{2}+m\right)\Gamma\left(\tfrac{z+1}{4}+\tfrac{l-1}{2}+m\right)}
		{m!\,\Gamma\left(\tfrac{z}{2}+1+m\right)}
		\l^{m}\biggr\}
		\\
		= \, &\delta(a>0){2^{2l-1}\Gamma\left(\tfrac{z}{2}+1\right)\Gamma\left(\tfrac{1-z}{4}\right)}
		{\Gamma\left({l}\right)^{-1}
			\Gamma\left(\tfrac{z+1}{4}\right)^{-1}}|a|^{l/2}|a-1|^{-l}
		\\
		&\times \biggl\{-\sum_{m=0}^{\infty} 
		\frac{\Gamma\left(\tfrac{l}{2}+m\right)\Gamma\left(\tfrac{l+1}{2}+m\right)}{\Gamma\left(\tfrac{5+z}{4}+m\right)\Gamma\left(\tfrac{5-z}{4}+m\right)}
		\,\l^{m}
		+\l^{\frac{z-1}{4}} 
		\sum_{m=0}^{\infty}\frac{\Gamma\left(\tfrac{z+1}{4}+\tfrac{l}{2}+m\right)\Gamma\left(\tfrac{z-1}{4}+\tfrac{l}{2}+m\right)}
		{m!\,\Gamma\left(\tfrac{z}{2}+1+m\right)}\l^{m}\biggr\}
\end{align*}}using the formula $\Gamma\left(\tfrac{z+3}{4}\right)\Gamma\left(\tfrac{1-z}{4}\right)=\frac{\pi}{\sin(\frac{z+3}{4}\pi)}=\frac{-\pi}{\sin(\frac{z-1}{4}\pi)}=-\Gamma\left(\tfrac{z-1}{4}\right)\Gamma\left(\tfrac{5-z}{4}\right)$, and the duplication formula $\Gamma(l)=2^{l-1}\pi^{-1/2}\Gamma(l/2)\Gamma((l+1)/2)$. Pluging this and $A_v(0,z)=\frac{\sqrt{\pi}\Gamma(-z/2)}{\Gamma((1-z)/4)^2}$ to $\fF_v^{(z)}(a)=A_v(0,z)J_v^{(z)}(a)+A_v(0,-z)J_v^{(-z)}(a)$, after a simple computation, we have the formula
{\allowdisplaybreaks \begin{align*}
		\fF_v^{(z)}(a)=&
		2^{2l-1}\sqrt{\pi}\,\Gamma(l)^{-1} 
		\,{2\Gamma\left(1+\tfrac{z}{2}\right)
			\Gamma\left(1+\tfrac{-z}{2}\right)}\,
		{z^{-1}\,\Gamma\left(\tfrac{1+z}{4}\right)^{-1}\Gamma\left(\tfrac{1-z}{4}\right)^{-1}}
		\\
		&\times \delta(a>0)\, |a|^{l/2}|a+1|^{-l} 
		\{-\l(a)^{\frac{z-1}{4}}\Gcal_{l}^{(z)}(\l(a))
		+\l(a)^{\frac{-z-1}{4}}\Gcal_{l}^{(-z)}\left(\l(a)\right)
		\},
\end{align*}}where for $z\in \C$, $l\in 2\N$ and $|\l|<1$, we set 
\begin{align}
	\Gcal_{l}^{(z)}(\l)={\Gamma\left(\tfrac{z+1}{4}+\tfrac{l}{2}\right)\Gamma\left(\tfrac{z+1}{4}+\tfrac{l-1}{2} \right)}{\Gamma\left(\tfrac{z}{2}+1\right)^{-1}}\,{}_2F_1\left(\tfrac{z+1}{4}+\tfrac{l}{2},\tfrac{z+1}{4}+\tfrac{l-1}{2};\tfrac{z}{2}+1;\l\right).
	\label{Gcalftn}
\end{align}
From \cite[line 14, p.153]{MOS} and \cite[line 18, p.164]{MOS},  
\begin{align*}
	\l^{\frac{z-1}{4}}\,\Gcal_{l}^{(z)}(\l)=-2^{2-l}(1-\l)^{\frac{1-l}{2}}\l^{-1/2}\,\fQ_{\frac{z-1}{2}}^{l-1}\left(\l^{-1/2}\right)
\end{align*}
and 
\begin{align*}
	\fQ_{\frac{z-1}{2}}^{l-1}(\l^{-1/2})-\fQ_{\frac{-z-1}{2}}^{l-1}(\l^{-1/2})
	=&\sin\left(\tfrac{\pi z}{2}\right)\,\Gamma\left(l+\tfrac{z-1}{2}\right)\Gamma\left(l+\tfrac{-z-1}{2}\right)\,\fP_{\frac{z-1}{2}}^{1-l}(\l^{-1/2}),
\end{align*}
where $\fP_{\nu}^{\mu}(x)$ and $\fQ_{\nu}^{\mu}(x)$ are the associated Legendre functions of the first kind and of the second kind, respectively. Thus we obtain the desired formula for $|\Re(z)|<1$. By analytic continuation, it remains valid on $|\Re(z)|<2l-1$.  
This completes the proof of Theorem \ref{explicit hyp} (1).

\subsubsection{The proof of Theorem \ref{explicit hyp} (2)} \label{The proof of Theorem explicit hyp (2)}
Let $v \in \Sigma_\fin - (S\cup S(\gn))$.
Recall $\Phi_v={\rm ch}_{Z_v\bK_v}$. For $a\in F_v^\times-\{1\}$ and $x\in F_v$, we have $\Phi_v\left(k^{-1}\left[\begin{smallmatrix} a & (a-1)x \\ 0 & 1 \end{smallmatrix}\right]k\right)\not=0$ for some $k\in \bK_v$ if and only if $a\in \cO^\times$ and $x\in (a-1)^{-1}\cO$. Thus, $\fF_v^{(z)}(a)=0$ unless $a\in \cO^\times$, in which case $\fF_v^{(z)}(a)=A_v(0,z)\,I(z)+A_v(0,-z)I(-z)$, where $I(z)=\int_{x\in (a-1)^{-1}\cO}h_v^{(0,z)}(x)\,\d x$. For $a\in \cO^\times$, we easily have that $I(z)$ is absolutely convergent for all $z$ and
$$I(z)=q^{-d/2}\left\{1-\zeta_{F_v}(1)^{-1}\zeta_{F_v}\left(\tfrac{-z+1}{2}\right)(1-|a-1|^{\frac{z-1}{2}})\right\}$$ 
if $q_v^{z}\not=1$. Hence
\begin{align*}
	\fF_v^{(z)}(a)&=q^{-d/2}\,\delta(a\in\cO^\times)\,\biggl\{A_v(0,z)+A_v(0,-z)\\
	&-{\zeta_{F_v}(-z)}{\zeta_{F_v}\left(\tfrac{-z+1}{2}\right)^{-1}}
	(1-|a-1|^{\frac{z-1}{2}})-{\zeta_{F_v}(z)}{\zeta_{F_v}\left(\tfrac{z+1}{2}\right)^{-1}}(1-|a-1|^{\frac{-z-1}{2}})\biggr\}.
\end{align*}
By \eqref{FFz1}, we are done. The second claim is obvious from the formula of $\fF_v^{(z)}(a)$.

\subsubsection{The proof of Theorem \ref{explicit hyp} (3)}
\label{The proof of Theorem explicit hyp (3)}
Let $v \in S(\gn)$. From \eqref{KBruhat}, we have the equality 
\begin{align*}
	\fF_v^{(z)}(a)&=\int_{F_v} \{\int_{\bK_0(\fp)}\Phi_v\left(k^{-1}\left[\begin{smallmatrix} a & (a-1)x \\ 0 & 1 \end{smallmatrix} \right]k\right)\d k \\
	&+\sum_{\xi\in \cO/\fp}\int_{\bK_0(\fp)}\Phi_v\left(k^{-1}w_0^{-1}
	\left[\begin{smallmatrix} 1 & -\xi \\ 0 & 1 \end{smallmatrix}\right]
	\left[\begin{smallmatrix} a & (a-1)x \\ 0 & 1 \end{smallmatrix}\right]
	\left[\begin{smallmatrix} 1 & \xi \\ 0 & 1 \end{smallmatrix}\right]w_0k\right)
	\,\d k\}\,\varphi_v^{(0,z)}\left(\left[\begin{smallmatrix} 1 & x \\ 0 & 1 \end{smallmatrix}\right]\right)\,\d x,
\end{align*}
from which we see that $\fF_v^{(z)}(a)$ is $\vol(\bK_0(\fp_v))$ times the sum of the following integrals
\begin{align}
	&\tint_{F_v} \Phi_v\left(\left[\begin{smallmatrix} a & (a-1)x \\ 0 & 1 \end{smallmatrix} \right]\right)\,\varphi_v^{(0,z)}\left(\left[\begin{smallmatrix} 1 & x \\ 0 & 1 \end{smallmatrix}\right]\right)\,\d x,\  \label{HOI-1}
	\\
	&\tint_{F_v} \Phi_v\left(
	\left[\begin{smallmatrix} 1 & 0 \\ -(a-1)(x+\xi) & a \end{smallmatrix}\right]
	\right)
	\,\varphi_v^{(0,z)}\left(\left[\begin{smallmatrix} 1 & x \\ 0 & 1 \end{smallmatrix}\right]\right)\,\d x\quad \text{for $\xi \in \cO/\fp$}, \label{HOI-2}
\end{align}
where $\Phi_v={\rm ch}_{Z_v\bK_0(\fp)}$. Since $\Phi_v\left(\left[\begin{smallmatrix} a & (a-1)x \\ 0 & 1 \end{smallmatrix} \right]\right)\not=1$ if and only if $x\in (a-1)^{-1}\cO$, $a\in \cO^\times$, from the proof of Theorem \ref{explicit hyp} (2) in \S \ref{The proof of Theorem explicit hyp (2)}, the integral \eqref{HOI-1} equals $\delta(a\in \cO^\times)q^{-d/2} \Ocal_{0,v}^{1,(z)}((a-1)^{-2})$. The integral \eqref{HOI-2} is computed as $\delta(a\in \cO^\times)\{A_v(0,z)I_\xi(z)+A_v(0,-z)I_\xi(-z)\}$, where $I_\xi(z)=\int_{x\in -\xi+(a-1)^{-1}\fp} h_v^{(0,z)}(x)\d x$. Let $a\in \cO^\times$. Suppose $|a-1|=1$. Then we easily have $I_\xi(z)=\vol(\fp)=q^{-1-d/2}$, and \eqref{HOI-2} becomes $q^{-1-d/2}(A_v(0,z)+A_v(0,-z))=q^{-1-d/2}$. Hence $\fF_v^{(z)}(a)=\vol(\bK_0(\fp))\,(q^{-d/2}+q\times q^{-1-d/2})={2q^{-d/2}}{(1+q)^{-1}}.$ 
Suppose $|a-1|<1$. Noting $-\xi\in \cO\subset (a-1)^{-1}\fp$, we have
\begin{align*}
	I_\xi(z)&=\tint_{x\in (a-1)^{-1}\fp} h_v^{(0,z)}(x)\,\d x
	=I(z)-h_v^{(0,z)}((a-1)^{-1})\,q^{-d/2}(1-q^{-1})|a-1|^{-1},
\end{align*} where $I(z)$ is the same integral as in the proof of Theorem \ref{explicit hyp} (2) in \S \ref{The proof of Theorem explicit hyp (2)}. Thus \eqref{HOI-2} is the sum of $A_v(0,z)I(z)+A_v(0,-z)I(-z)$, which becomes $\delta(a\in \cO^\times)q^{-d/2} \Ocal_{0,v}^{1,(z)}((a-1)^{-2})$ as before and $-q^{-d/2}(1-q^{-1})\,\{A_v(0,z)|a-1|^{\frac{z-1}{2}}+A_v(0,-z)|a-1|^{\frac{-z-1}{2}}\}.$ 
Hence $\fF_v^{(z)}(a)$ equals the following expression multiplied by $\vol(\bK_0(\fp))q^{-d/2}$:
\begin{align*}
	(1+q) \Ocal_{0,v}^{1,(z)}((a-1)^{-2})
	-q \times (1-q_v^{-1})\,\left(A_v(0,z)|a-1|^{\frac{z-1}{2}}+A_v(0,-z)|a-1|^{\frac{-z-1}{2}}\right).
\end{align*}
From this we get the desired formula by a short computation. 

\subsubsection{The proof of Theorem \ref{explicit hyp} (4)}
\label{The proof of Theorem explicit hyp (4)}
Let $v \in S$. Put $s=s_v$. By applying \eqref{nonarchGreenftn} and \eqref{Hsphericalfnt} and then changing the order of integrals, 
\begin{align}
	\fF_v^{(z)}(s;a)&=|a-1|^{-1}\tint_{F_v}\hat\Phi_v\left(s;\left[\begin{smallmatrix} a & x \\ 0 & 1\end{smallmatrix}\right]\right)\,\varphi_{v}^{(0,z)}\left( \left[\begin{smallmatrix} 1 & (a-1)^{-1}x \\ 0 & 1 \end{smallmatrix} \right]\right)\,\d x
	\label{HOIS-0}
	\\
	& = (-q^{-\frac{s+1}{2}})|a-1|^{-1}{|a|^{(s+1)/2}\zeta_{F_v}(s+1)}\,\{A_v(0,z)I(z,s;a)+A_v(0,-z)I(-z,s;a)\},
	\notag
\end{align}
where we set
\begin{align*}
	I(z,s;a)=\tint_{F_v} \max(1,|a|,|x|)^{-(s+1)}\,\max(1,|a-1|^{-1}|x|)^{-\frac{z+1}{2}}\,\d x. 
\end{align*}
As will be shown below, the integral $I(\pm z,s;a)$ is absolutely convergent if $\Re(s)>(|\Re(z)|-1)/2$. Suppose $|a-1|<1$. Then $|a|=1$. Hence
\begin{align*}
	I(z,s;a)
	&=\tint_{|x|\leq |a-1|}\,\d x
	+\tint_{|a-1|<|x| \leq 1}(|a-1|^{-1}|x|)^{-\frac{z+1}{2}}\,\d x
	+\tint_{1<|x|} |x|^{-(s+1)}(\,|a-1|^{-1}|x|)^{-\frac{z+1}{2}}\,\d x
	\\
	&=q^{-d/2}|a-1|
	+|a-1|^{\frac{z+1}{2}}\,q^{-d/2}(1-q^{-1})
	{(1-|a-1|^{-\frac{z-1}{2}})}{(1-q^{\frac{z-1}{2}})^{-1}} \\
	&\qquad +|a-1|^{\frac{z+1}{2}}\,q^{-d/2}(1-q^{-1})\,{q^{-\frac{z-1}{2}-s-1}}{(1-q^{-\frac{z-1}{2}-s-1})^{-1}}
\end{align*}
if $\Re(\frac{z-1}{2}+s+1)>0$. By a computation, we obtain 
\begin{align*}
	I(z,s;a)=
	q^{-d/2}|a-1|\,\left\{-q^{\frac{z-1}{2}}\frac{\zeta_{F_v}\left(\frac{1-z}{2}\right)}{\zeta_{F_v}\left(\frac{1+z}{2}\right)}+\frac{\zeta_{F_v}\left(\frac{1-z}{2}\right)\zeta_{F_v}\left(\frac{z-1}{2}+s+1\right)}{\zeta_{F_v}(1)\zeta_{F_v}(s+1)}|a-1|^{\frac{z-1}{2}}\right\}. 
\end{align*}
From this, combined with \eqref{FAQ1}, we get the formula
\begin{align}
	A_v(0,z)I(z,s;a)+A_v(0,-z)I(-z,s;a)& 
	= (-q^{\frac{s+1}{2}}){|a-1|}{\zeta_{F_v}(s+1)^{-1}}\,q^{-d/2}\Scal_v^{1,(z)}((a-1)^{-2}).
	\label{HOIS-1}
\end{align}
Suppose $|a-1|>1$. Then $|a-1|=|a|$ and 
\begin{align*}
	I(z,s;a)&=\tint_{|x|\leq |a-1|}
	\max(1,|x|,|a|)^{-(s+1)}\,\d x
	+\tint_{|a-1|<|x|}
	|x|^{-(s+1)}(|a-1|^{-1}|x|)^{-\frac{z+1}{2}}\,\d x\\
	&=q^{-d/2}|a-1|^{-s}\,{(1-q^{-\frac{z+1}{2}-s-1})}{(1-q^{-\frac{z-1}{2}-s-1})^{-1}} 
\end{align*}
if $\Re(\tfrac{z-1}{2}+s+1)>0$. Suppose $|a-1|=1$. Then $|a|\leq 1$ and 
\begin{align*}
	I(z,s;a)
	&=\tint_{|x|\leq 1}\d x+\int_{1<|x|}\max(|a|, |x|)^{-(s+1)}|x|^{-\frac{z+1}{2}}\,\d x
	\\
	&=q^{-d/2}\left\{1+{(1-q^{-1})\,q^{-\frac{z-1}{2}-s-1}}{(1-q^{-\frac{z-1}{2}-s-1})^{-1}}\right\}\\
	&=q^{-d/2}{(1-q^{-\frac{z+1}{2}-s-1})}{(1-q^{-\frac{z-1}{2}-s-1})^{-1}}
\end{align*}
if $\Re(\tfrac{z-1}{2}+s+1)>0$. Hence $I(z,s;a)=q^{-d/2}|a-1|^{-s}{\zeta_{F_v}\left(\frac{z-1}{2}+s+1\right)}{\zeta_{F_v}\left(\frac{z+1}{2}+s+1\right)^{-1}}$ for $|a-1|\geq 1$. From this, by a direct computation, we get the formula \eqref{HOIS-1} again.

\subsection{Local elliptic orbital integrals}
Let $v\in \Sigma_F$. In this subsection, we compute the integral \eqref{EllipticorbitalIntegral} with $\Delta_v^{0}\not=1$ to complete the proof of Theorem~\ref{OrbIntUnifEx}. In this section throughout, we fix $(t:n)_F\in \Qcal_F^{{\rm Irr}}$ with the decomposition $(t^2-4n)^{(v)}=\Delta_v^{0}(2m_v)^2$ at a place $v$ as before. Set $a=\tfrac{t}{2m_v}$ and $\tau=\Delta_v^0$ to simplify notation. Recall the construction at the begining of \S~\ref{PEAET}. By multiplying $\varphi_{0, v}(1_2)$, the integral $\fE_v^{(z)}(\hat \gamma_v)$ is transformed to
\begin{align} \label{simple of elliptic}& \varphi_{0,v}(1_2)\fE_v^{(z)}(\hat \gamma_v)
	= 
	\tint_{Z_v\bsl G_v}
	\Phi_v\left(g^{-1} \left[\begin{smallmatrix}a & 1 \\ \tau & a \end{smallmatrix}\right]g\right)\,f_{0,v}(g)\,\d g
	\\
	=& \tint_{F_v^\times} \int_{F_v} \int_{\bK_v} \d k\, \Phi_v\left( k^{-1} 
	\left[\begin{smallmatrix}t & 0 \\ 0 & 1 \end{smallmatrix}\right]^{-1}
	\left[\begin{smallmatrix}1 & -x \\ 0 & 1 \end{smallmatrix}\right]
	\left[\begin{smallmatrix}a & 1 \\ \tau & a \end{smallmatrix}\right]
	\left[\begin{smallmatrix}1 & x \\ 0 & 1 \end{smallmatrix}\right]\left[\begin{smallmatrix}t & 0 \\ 0 & 1 \end{smallmatrix}\right]
	k \right)\,
	f_{0,v}\left(
	\left[\begin{smallmatrix}1 & x \\ 0 & 1 \end{smallmatrix} \right]
	\left[\begin{smallmatrix}t & 0 \\ 0 & 1 \end{smallmatrix}\right]
	\right)\,|t|_v^{-1}\d^\times t\,\d x  \notag
	\\
	= & \tint_{F_v^\times} \tint_{F_v} \tint_{\bK_v} \d k\, \Phi_v\left(k^{-1} 
	\left[\begin{smallmatrix}a-\tau x & t^{-1}(1-\tau x^2)
		\\ \tau t & a+\tau x \end{smallmatrix}\right]
	k \right)\,|t|_v^{\frac{z-1}{2}}\d^\times t\,\d x. \notag
\end{align}

\subsubsection{The proof of Theorem~\ref{OrbIntUnifEx} (1)} \label{EllOrbAcase}

Let $v\in \Sigma_\infty$ and $\Delta_v^{0}=-1$. By $\bfK_v=\bfK_v^0 \cup[\begin{smallmatrix} -1 & 0 \\ 0 & 1 \end{smallmatrix}]\bfK_v^0$ and by (b) in \S~\ref{MCDS}, we have
\begin{align*}\varphi_{0,v}(1_2)\fE_v^{(z)}(\hat\gamma_v)=
	\tint_{F_v^\times}B(t)|t|^{\frac{z-1}{2}}\,d^\times t
\end{align*}
with $$B(t)= \tint_{F_v}\Phi^l(\begin{smallmatrix}
a+x & t^{-1}(1+x^2) \\ -t & a-x
\end{smallmatrix}) dx
= {2^{l}(1+a^2)^{l/2}}\tint_{F_v}\left(2a-i(t^{-1}x^2+t^{-1}+t)\right)^{-l}dx.$$
By the formula $\tint_{\RR}{(1+zx^2)^{-l}}dx= {\Gamma(l-1/2)\sqrt{\pi}}{\Gamma(l)^{-1}}\,\sqrt{z}^{-1}$ ($\Re(z)>0)$ easily proved by the residue theorem,
\begin{align*}
	B(t) =2^l(1+a^2)^{l/2}{\Gamma(l-1/2)\sqrt{\pi}}{\Gamma(l)^{-1}}\times (it)^l(t^2+1+2ati)^{1/2-l} ,
\end{align*}
Since $l$ is even, by decomposing the $t$-integral over $\R^+$ into positive and negative reals, we obtain 
\begin{align*}
	\tint_{F_v^{\times}}B(t)|t|^{\frac{z-1}{2}} d^\times t = 
	2^l(1+a^2)^{l/2}\,i^{l}\,\{I_{l}\left(-4,2a;\tfrac{z+1}{2}\right)+I_{l}\left(-4,-2a;\tfrac{z+1}{2}\right)\},
\end{align*}
where $I_{l}(\Delta,a;s)$ is the Zagier's function defined in \cite[p.110]{Zagier}. We have 
\begin{align*}
	I_{l}\left(-4,2a; \tfrac{z+1}{2} \right)
	&= \Gamma(l-1/2)\sqrt{\pi}\Gamma(l)^{-1}\tint_{0}^{\infty}\tfrac{y^{l+\frac{z-1}{2}}}{(y^2+2iay+1)^{l-1/2}}\d^{\times} y
	\\&= 2^{1-l}\pi \Gamma(l)^{-1}{\Gamma\left(l+\tfrac{z-1}{2}\right)\Gamma\left(l+\tfrac{-z-1}{2}\right)}
	{(-a^2-1)^{-\frac{l-1}{2}}}{\frak P}_{\frac{z-1}{2}}^{1-l}(ai).
\end{align*}
for $|\Re(z)|<2l-1$ by means of the formula \cite[p.961, 8.713, 3]{Gradshteyn}, where the square root of $ai\pm 1$ is chosen so that $\arg(ai\pm 1) \in (-\pi, \pi)$. Consequently, by noting $\sqrt{-1-a^2} = \sgn(a) i \sqrt{a^2+1}$, we have
\begin{align*}
	\varphi_{0,v}(1_2)\fE_v^{(z)}(\hat{\gamma}_v)
	= & \sgn(a) i \sqrt{1+a^2} \times 2\pi{\Gamma\left(l+\tfrac{z-1}{2}\right)\Gamma\left(l+\tfrac{-z-1}{2}\right)}{\Gamma(l)^{-1}}\left({\frak P}_{\frac{z-1}{2}}^{1-l}(ai)-{\frak P}_{\frac{z-1}{2}}^{1-l}(-ai)\right).
\end{align*}
We complete the proof by Lemma~\ref{sphericalftn1value}. 


\subsubsection{The proof of Theorem~\ref{OrbIntUnifEx} (2) and (3)} \label{EllOrbIntL1}
Since $\Phi_v={\rm ch}_{Z_v\bK_v}$, the integral domain of \eqref{simple of elliptic} is restricted to only those $(t,x)\in F_v^\times \times F_v$ such that 
\begin{align}
	&c^{-1}(a-\tau x)\in \cO, \quad 
	c^{-1}(a+\tau x)\in \cO, \quad
	c^{-1}t^{-1}(1-\tau x^2)\in \cO, 
	\label{EllOrbIntL1-1}
	\\
	&c^{-1}\tau t \in \cO,
	\label{EllOrbIntL1-2}
	\\
	&c^{-2}(a^2-\tau )\in \cO^{\times}
	\label{EllOrbIntL1-3}
\end{align}
with some $c\in F_v^\times$. 

(I) {\underline{The case $|a|\leq |\tau|$}}. 

\smallskip
\noindent
(i) Suppose $v\not\in \Sigma_{\rm dyadic}$, $\tau\in \cO^\times-(\cO^\times)^{2}$. From Lemma~\ref{Hensel}, $a^2-\tau \in \cO^\times$. Hence from \eqref{EllOrbIntL1-3} we get $c\in \cO^\times$, which, combined with \eqref{EllOrbIntL1-2}, yields $t\in \cO$. From the last relation in \eqref{EllOrbIntL1-1}, we have the containment $1-\tau x^2\in t\cO$ from which $x\in \cO$ follows. If $t$ were non-unit, then $1-\tau x^2\in t\cO\subset \fp$, which is impossible due to Lemma~\ref{Hensel}. Hence we have that the existence of $c$ satisfying \eqref{EllOrbIntL1-1}, \eqref{EllOrbIntL1-2} and \eqref{EllOrbIntL1-3} is equivalent to $t\in \cO^\times$ and $x\in \cO$.
Thus by Lemm~\ref{sphericalftn1value}, we have $\fE_v^{(z)}(\hat\gamma_v)=\varphi_{0,v}(1_2)^{-1}\tint_{\cO^{\times}}\d^\times t \tint_{\cO}\d x=|m_v| q^{-d/2}$. 

(ii) Suppose $v\in \Sigma_{\rm dyadic}$, $\tau\in \{-5, -1, 5\}$. Let us consider the case $\tau=5$. If $|a|=1$, then \eqref{EllOrbIntL1-3} implies $c\in 2\ZZ_2^{\times}$. Combining this with \eqref{EllOrbIntL1-1} and \eqref{EllOrbIntL1-2}, we obtain $x \in \ZZ_2^\times$ and $t\in 2\ZZ_2^\times$.
Hence by Lemma~\ref{sphericalftn1value},
$$\fE_v^{(z)}(\hat{\gamma}_v) = \varphi_{v,0}(1_2)^{-1}
\tint_{|t|=|2|}\tint_{|x|=1} |t|^{(z-1)/2} d^\times tdx =
\tfrac{3}{2}|m_v|\,2^{-d/2}\tfrac{2^{(-z-1)/2}}{1+2^{-z}}.$$
If $|a|<1$, the condition \eqref{EllOrbIntL1-3} implies $c \in \Z^\times_2$. Then using \eqref{EllOrbIntL1-1} and \eqref{EllOrbIntL1-2}, we get $x \in \Z_2$ and $1-5x^2 \in t\Z_2$. Thus $|4|\le |t|\le 1$ if $|x|=1$ and $|t|=1$ if $|x|<1$, and whence
\begin{align*}\fE_v^{(z)}(\hat{\gamma}_v) =& 
	\varphi_{0,v}(1_2)^{-1} \bigg(\tint_{|4| \le |t| \le 1}d^\times t\tint_{|x|=1}dx|t|^{(z-1)/2} + \tint_{|t|=1}d^\times t \tint_{|x|<1}dx |t|^{(z-1)/2}\bigg) \\
	= & \tfrac{3}{2}|m_v|_v 2^{-d/2}
	\{1+2^{(-z-1)/2}+2^{-z}\}/(1+2^{-z}).
\end{align*}
We note $|\frac{n}{m_v^2}|=1$ from Lemma~\ref{OrbIntUnifExL2} (4) and Proposition~\ref{OrbIntUnifExL4} (1).
Next consider the case $\tau \in \{-5,-1\}$. If $|a|=1$, then $a^2-\tau \in 2\Z_2^\times$, which contradicts to \eqref{EllOrbIntL1-3}. Thus $\fE_v^{(z)}(\hat\gamma_v)=0$. If $|a|<1$, then \eqref{EllOrbIntL1-3} implies $c \in \ZZ_2^\times$. We may set $c=1$.
Then, $x \in \ZZ_2$ and $1-\tau x^2 \in t\ZZ_2$.
Thus, $|2|\le |t| \le 1$ if $|x|=1$ and $|t|=1$ if $|x|<1$. We have
\begin{align*}\fE_v^{(z)}(\hat{\gamma}_v) =&\varphi_{0,v}(1_2)^{-1}
	\bigg(\tint_{|2| \le |t| \le 1}d^\times t\tint_{|x|=1}dx|t|^{(z-1)/2} + \tint_{|t|=1}d^\times t \tint_{|x|<1}dx |t|^{(z-1)/2}\bigg) = 
	|m_v|2^{-d/2}.
\end{align*}
(iii) Suppose $\ord_v(\tau)=1$. Then, $|a|\leq |\tau|$ implies $|a|<1$, where $|c^{-2}(a^{2}-\tau)|=|c|^{-2}|\tau|$ is an odd power of $q$ and \eqref{EllOrbIntL1-3} is never attained. Thus $\fE_v^{(z)}(\hat\gamma_v)=0$.

\smallskip
\noindent
(II) {\underline {The case $|a|>|\tau|$.} 
	
	\smallskip
	\noindent
	Since $|\tau| < |a^2|$, the condition \eqref{EllOrbIntL1-3} yields $|c|=|a|$. Thus the condition \eqref{EllOrbIntL1-1} is equivalent to $ a^{-1}x \in \cO$, $a^{-1}t^{-1}(1-\tau x^2)\in \cO$, and the condition \eqref{EllOrbIntL1-2} to $a^{-1}t\in \cO$. Hence, 
	\begin{align}
		\varphi_{0,v}(1_2)\fE_v^{(z)}(\hat \gamma_v)=
		|a|^{(z+1)/2}\tint_{t \in \tau^{-1} \go-\{0\}}|t|^{(z-1)/2} \vol(X(t))
		\,d^\times t.
		\label{AAAAAAAAAAAA}
	\end{align}
	where we put $X(t)= \{x \in \tau^{-1}\go \, | \, a^{-2}-\tau x^2 \in t\go \}$.
	For $l\in \Z$, we set 
	\begin{align*}
		[q^{-l/2}]=\begin{cases} q^{-l/2}, \quad (l\equiv 0 \pmod{2}), \\
			q^{-(l+1)/2}, \quad (l\equiv 1 \pmod{2}).
		\end{cases}
	\end{align*}
	\begin{lem}\label{formula of X(t)}
		Let $t \in \tau^{-1}\go-\{0\}$.
		If $v \not\in \Sigma_{\rm dyadic}$ and $\tau \in \go^\times-(\go^\times)^2$, 
		we have
		$$\vol(X(t))= \delta(|a^{-2}| \le |t|) \, q^{-d/2} \, [|t|^{1/2}].$$
		If $v\in \Sigma_{\rm dyadic}$ and $\tau =5$,
		$$\vol(X(t))=\delta(|4a^{-2}|\le |t|)2^{-d/2}
		\begin{cases}
		[|t|^{1/2}]  & (|a^{-2}|\le |t|), \\
		2^{-1}|a^{-1}| & (|4a^{-2}|\le |t|<|a^{-2}|).
		\end{cases}
		$$
		If $\tau \in \fp-\fp^2$, we have $\vol(X(t))= \delta(|a^{-2}|\le |t| )[|\varpi^{-1}t|^{1/2}]q_v^{-d/2}$. If $v\in \Sigma_{\rm dyadic}$ and $\tau \in \{-5, -1 \}$, we have
		$$\vol(X(t))=\delta(|2a^{-2}|\le |t|)2^{-d/2}
		\begin{cases}
		[|t|^{1/2}]  & (|a^{-2}|\le |t|), \\
		2^{-1}|a^{-1}| & (|t|=|2a^{-2}|).
		\end{cases}
		$$
	\end{lem}
	\begin{proof}
		In the proof, we write $X$ for $X(t)$ for simplicity. First suppose $\tau \in \go^\times$. Then $x \in X$ is equivalent to $x\in \cO$ and $a^{-2} -\tau x^2 \in t\go$. Suppose $v\not\in \Sigma_{\rm dyadic}$. Due to Lemma~\ref{Hensel}, $X \cap a^{-1}\go = \{ x \in a^{-1}\go
		\ | \ a^{-2}(1-\tau (ax)^2) \in t\go \} = \{ x \in a^{-1}\go  \ | \ |a^{-2}|\le |t|\}$	and $X-a^{-1}\go = \{ x \in \go \ | \ |a^{-1}|<|x|\le |t|^{1/2} \}$; thus
		\begin{align*}
			\vol(X) = & \vol(X\cap a^{-1}\go) + \vol(X-a^{-1}\go)\\
			= & \delta(|a^{-2}|\le |t|) |a^{-1}| q^{-d/2} + \delta(|a^{-1}|^2<
			|t|)([|t|^{1/2}]-|a^{-1}|)q^{-d/2} \\
			=& \delta(|a^{-2}| \le |t|) \, q^{-d/2} \, [|t|^{1/2}].
		\end{align*}
		Suppose $v\in \Sigma_{\rm dyadic}$ and $\tau=5$. Then since $(\Z_2^\times)^2=1+8\Z_2$, we have $1-5(ax)^2\in 4\Z_2^\times$ for $x\in X$. Hence $X\cap a^{-1}\cO^\times$ is empty unless $4a^{-2}\in t\Z_2$, in which case $X\cap a^{-1}\cO^\times= a^{-1}\ZZ_2^\times$. Similarly, $X\cap a^{-1}\varpi\go$ is empty unless $a^{-2}\in t\Z_2$, in which case it is $a^{-1}2\Z_2$. We also have $X-a^{-1}\go= \{x \in \ZZ_2 \ | \ |a^{-1}| <|x|\le |t|^{1/2}\}.$ Thus
		\begin{align*}\vol(X)= &\, q^{-d/2}\{\delta(|4a^{-2}|\le |t|)(1-q^{-1})|a^{-1}|+
			\delta(|a^{-2}|\le |t|)q^{-1}|a^{-1}| \\
			&+ \delta(|a^{-2}|\le |t|)([|t|^{1/2}]-|a^{-1}|)
			\}, 
		\end{align*}
		which is simplified to the desired form. The case $v\in \Sigma_{\rm dyadic}$ with $\tau \in \{-5, -1\}$ is similar. 
		
		Next suppose $\tau \in \fp-\fp^2$. In the case, $1-\tau u^2\in \cO^\times$ for all $u \in \go_v$. Hence $X\cap a^{-1} \go = \{x \in a^{-1} \go \ | \ a^{-2}(1-\tau a^2x^2) \in t \go\} = \{ x \in a^{-1}\go \ | \ |a^{-2}|\le |t| \}$ whose volume is $\delta(|a^{-2}|\leq |t|)|a^{-1}|q^{-d/2}$, and $X-a^{-1}\go =\{ x \in \tau^{-1}\go \ | \ |a^{-1}|< |x| \le |\tau^{-1}t|^{1/2} \}$ whose volume is $\delta(|a|^{-2}\leq |t|)([|\varpi^{-1}t|^{1/2}]-|a^{-1}|)q^{-d/2}$. By $\vol(X) =\vol(X\cap a^{-1}\go)+\vol(X-a^{-1}\cO)$, we are done.
	\end{proof}

	\begin{lem} \label{formula of integral of X(t)} For any $z\in \CC$ such that $q^z\neq q^{\pm1}$ and $a \in F_v$ such that $|a|> |\tau|$, we have
		\begin{align*}& \int_{|a^{-2}| \le |t| \le |\tau^{-1}|}[|\tau^{-1}t|^{1/2}]\, |t|^{(z-1)/2}d^\times t
			= q^{-d/2}\left\{\frac{1+q^{-(-1)^\delta\frac{z+1}{2}}}{1-q^{-z}} + \frac{1+q^{\frac{z-(-1)^\delta}{2}}}{1-q^{z}}|a|^{-z}\right\}
		\end{align*}
		with $\delta=\ord_v(\tau) \in \{0,1\}$.
	\end{lem}
	\begin{proof}
		A direct computation. 
	\end{proof}
	First we consider the case where $\tau \in \go^\times - (\go^\times)^2$. If $v$ is non-dyadic, from \eqref{AAAAAAAAAAAA}, Lemmas \ref{formula of X(t)} and \ref{formula of integral of X(t)}, we get 
	\begin{align*}
		\varphi_{v,0}(1_2)\fE_v^{(z)}(\hat\gamma_v)=
		&=q^{-d}|a|^{(z+1)/2} \left(\tfrac{1+q^{(-z-1)/2}}{1-q^{-z}} + \tfrac{1+q^{(z-1)/2}}{1-q^{z}}|a|^{-z}\right).
	\end{align*}
	To complete the proof, it suffices to use Lemma~\ref{sphericalftn1value} and to note that $|\tfrac{n}{m_v^2}|$ equals $|a|^2$ or $1$ according to $|a|>1$ or $|a|\leq 1$, which follows from Lemma~\ref{OrbIntUnifExL2} (1) and Proposition~\ref{OrbIntUnifExL4} (1). If $v$ is dyadic and $\tau=5$, in the same way, we have
	\begin{align*}
		\varphi_{0,v}(1_2)\fE_v^{(z)}(\hat\gamma_v)&=|a|^{(z+1)/2}2^{-d/2}
		\left\{ \tint_{|4a^{-2}|\leq |t|<|a^{-2}|}2^{-1}|a|^{-1}|t|^{(z-1)/2}\d^\times t+\int_{|a|^{-2}\leq|t|\leq 1}[|t|^{1/2}]|t|^{(z-1)/2}\d^\times t \right\}
		\\
		&=2^{-d}\left(\tfrac{1+2^{(-z-1)/2}}{1-2^{-z}}|a|^{\frac{z+1}{2}}+\tfrac{1+2^{(z-1)/2}}{2^z(1-2^{z})}|a|^{(-z+1)/2}\right).
	\end{align*}
	To complete the proof, it suffices to use Lemma~\ref{sphericalftn1value} and to note that $|\tfrac{n}{m_v^2}|=|a|^{-2}$ if $|a|>1$ from Lemma~\ref{OrbIntUnifExL2} (4) and Proposition~\ref{OrbIntUnifExL4} (1). The remaining cases are similar.

	\subsubsection{The proof of Theorem~\ref{OrbIntUnifEx} (4)}
	\label{The proof of Theorem OrbIntUnifEx (4)}
	Let $v\in S(\fn)$ and $\Phi_v={\rm ch}_{Z_v\bK_0(\fp)}$. By \eqref{KBruhat} and
	\eqref{simple of elliptic}, we have 
	\begin{align} 
		\varphi_{0,v}(1_2)\fE_v^{(z)}(\hat\gamma_v)=\vol(\bfK_0(\gp)) |a|^{(z+1)/2} \int_{F_v^\times}\{B(t)+\sum_{\xi \in \go_v/\gp_v}B_\xi(t)\}|t|^{(z-1)/2}d^\times t,
		\label{expression by Bs :level case}
	\end{align}
	with 
	\begin{align*}
		B(t)&=\int_{F_v} \Phi_v(\left[\begin{smallmatrix}1-\tau x & t^{-1}(a^{-2}-\tau x^2)
			\\ \tau t & 1+ \tau x \end{smallmatrix}\right]) dx, \quad
		B_{\xi}(t)=\int_{F_v} \Phi_v([\begin{smallmatrix}
			1+\tau(x+\xi t)& -\tau t \\ -t^{-1}(a^{-2}-\tau(x+\xi t)^2) & 1-\tau (x+t\xi)
		\end{smallmatrix}]) dx.
	\end{align*}
	We consider the case $|a|\le |\tau|$.
	\begin{lem} \label{some vanishing of Bs}Suppose $|a|\le |\tau|$. Then, we have
		$B(t)=B_{\xi}(t)=0$.
	\end{lem}
	\begin{proof} Suppose $\tau \in \go^\times-(\go^\times)^2$. Then $|a|\leq 1$. We have that $B_\xi(t)\not=0$ if and only if there exists $c\in F_v^\times$ with the properties:\begin{align}
			c(1\pm \tau y)\in \cO, \quad 
			c\tau t \in \cO, \quad 
			c\tau t^{-1}(a^{-2}- \tau y^2) \in \fp, \quad
			c^2(1-\tau a^{-2}) \in \cO^\times,
			\label{ff2}
		\end{align}
		where we set $y=x+\xi t$. By $|a|\leq 1$, then $|1-\tau a^{-2}|=|a|^{-2}$, and hence the last condition of \eqref{ff2} yields $|c|=|a|$. From the first and the second ones, we then have $ay\in \cO$ and $at\in \cO$. The third condition is impossible by Lemma~\ref{Hensel}. Thus $B_\xi(t)=0$ for all $t\in F_v$. In the same way, $B(t)=0$ for all $t\in F_v$. Suppose $\ord_v(\tau) = 1$. Then, $|a|\le|\varpi|=q^{-1}$, and $\tilde B_\xi(t)\not=0$ if and only if \eqref{ff2}; from the last of \eqref{ff2}, $1=|c^2(1-\tau a^{-2})|=|c|^{2}|\tau||a|^{-2}$, which is impossible due to $|\tau|=q^{-1}$. Hence ${\tilde B}_{\xi}(t)=0$ for all $t\in F_v^\times$. Similarly, we obtain $B(t)=0$ for all $t\in F_v^\times$.
	\end{proof}

	Next we consider the case $|a|>|\tau|$.
	For the computation of $B(t)$ and $B_\xi(t)$, we set $X(t)= \{x \in \tau^{-1}\go | a^{-2}-\tau x^2 \in t\go \}$.
	\begin{lem}\label{S(n)-part value of B(t)}
		If $|a|>|\tau|$,
		\begin{align*}
			B(t) = \delta(t \in \tau^{-1}\gp)\vol(X(t)), \quad B_{\xi}(\varpi^{-1}t)= \delta(\varpi^{-1} t \in \tau^{-1}\go)\vol(X(t)).
		\end{align*}
	\end{lem}
	\begin{proof}
		The integrand of $B(t)$ is non-zero if and only if there exists $c \in F_v^\times$ such that $1-\tau x \in c\go$, $a^{-2}-\tau x^2 \in ct \go$, $\tau t \in c\gp$, $1+\tau x \in c\go$, $1-\tau a^{-2} \in c^2 \go^\times$.
		By $|a|>|\tau|$, $1-\tau a^{-2}$ is a unit and whence $|c|=1$.
		From this, $B(t)=\delta(t \in \tau^{-1}\gp)\int_{x \in \tau^{-1}\go, a^{-2}-\tau x^2 \in t\go}dx$.
		
		The integrand of $B_\xi(t)$ is non-zero if and only if $1+\tau(x+\xi t) \in c\go$, $\tau t \in \go$, $a^{-2} - \tau(x+\xi t)^2 \in ct\gp$, $1-\tau(x+\xi t) \in c\go$, $1-\tau a^{-2} \in c^2 \go^\times$. By $|a|>1$, $c$ is a unit and whence
		\begin{align*}B_{\xi}(t) = & \delta(t \in \tau^{-1}\go)\tint_{x \in \tau^{-1}\go,\, a^{-2}-\tau(x+\xi t)^2 \in \varpi t\go}dx = \delta(t \in \tau^{-1}\go)\tint_{x \in \tau^{-1}\go, a^{-2}-\tau x^2 \in \varpi t\go}dx
		\end{align*}
	\end{proof}
	
	Thus, if $\tau \in \go^\times-(\go^\times)^2$, the value \eqref{expression by Bs :level case} for $|a|\leq 1$ vanishes by Lemma \ref{some vanishing of Bs}.
	For $|a|>1$, Lemmas \ref{formula of X(t)}, \ref{formula of integral of X(t)} and \ref{S(n)-part value of B(t)} yield an evaluation of \eqref{expression by Bs :level case} as
	\begin{align*}
		&\vol(\bfK_0(\gp)) |a|^{(z+1)/2} (1+q^{(z+1)/2})\tint_{\tau^{-1}\gp-\{0\}} \vol(X(t))|t|^{(z-1)/2}d^\times t \\
		= & \vol(\bfK_0(\gp)) |a|^{(z+1)/2} (1+q^{(z+1)/2})\left(\tint_{\go-\{0\}} \vol(X(t))|t|^{(z-1)/2}d^\times t -\tint_{\go^\times}q^{-d/2}|t|^{(z-1)/2}d^\times t\right) \\
		= & |a|^{(z+1)/2} \frac{1+q^{(z+1)/2}}{1+q}\left(\tint_{\go-\{0\}} \vol(X(t))|t|^{(z-1)/2}d^\times t -q^{-d} \right) \\
		= & q^{-d} \left( \frac{1+q^{(-z-1)/2}}{1-q^{-z}}\frac{1+q^{(-z+1)/2}}{1+q}|a|^{(z+1)/2} + \frac{1+q^{(z-1)/2}}{1-q^z}\frac{1+q^{(z+1)/2}}{1+q}|a|^{(-z+1)/2} \right).
	\end{align*}
	If $\ord_v(\tau) =1$, the value \eqref{expression by Bs :level case} for $|a|\ge 1$ is evaluated as
	\begin{align*}
		& \vol(\bfK_0(\gp)) |a|^{(z+1)/2} (1+q^{(z+1)/2})\{\tint_{\tau^{-1}\go-\{0\}} \vol(X(t))|t|^{(z-1)/2}d^\times t -\tint_{\tau^{-1}\go^\times}q^{-d/2}q|t|^{(z-1)/2}d^\times t\}\\
		= & (1+q^{(z+1)/2})\times q^{-d} 
		\left( \frac{1}{1-q^{-z}}\frac{1+q^{(-z+1)/2}}{1+q}|a|^{(z+1)/2} + \frac{1}{1-q^z}\frac{1+q^{(z+1)/2}}{1+q}|a|^{(-z+1)/2}\right).
	\end{align*}
	Theorem~\ref{OrbIntUnifEx} (4) follows from these as before by Lemmas~\ref{sphericalftn1value}, Lemma~\ref{OrbIntUnifExL2} and Proposition~\ref{OrbIntUnifExL4}.

	\subsubsection{The proof of Theorem~\ref{OrbIntUnifEx} (5)}
	\label{The proof of Theorem OrbIntUnifEx (5)}
	Let $v\in S$ and $\Phi_v(g_v)=\Phi(s;g_v)$ the Green function on $G_v$ defined in \S \ref{Greenftn}.
	First we consider the case $a\neq0$. The case $a=0$ is treated after Lemma \ref{explicit of J(a):non-dyadic, unramified}.
	From \eqref{nonarchGreenftn}, the integral \eqref{simple of elliptic} is written as
	\begin{align*}
		\varphi_{0,v}(1_2)\fE_v^{(z)}(\hat\gamma_v)= & \int_{{\frak T}_{\Delta,v} \bsl G_v}\Phi(s;g^{-1}\hat{\gamma}_v g)\varphi_{0,v}(g)dg \\
		= & (q^{-\frac{s+1}{2}}-q^{\frac{s+1}{2}})^{-1}\,
		\tint_{F_v^\times}\tint_{F_v}|1-b^2\tau|^{\frac{s+1}{2}} \\
		&\times\{\max(|1-\tau b x|,|1+\tau b x|,|\tau b t|,|t^{-1}b(1-\tau x^2)|)^2 \}^{-\frac{s+1}{2}}
		|t|^{\frac{z-1}{2}}\d^\times t\,\d x,
	\end{align*}
	where $b=a^{-1}$. By the variable change $t\rightarrow t/(\tau b)$, $x\rightarrow x/(\tau b)$, we have
	\begin{align}
		\varphi_{0,v}(1_2)\fE_v^{(z)}(\hat\gamma_v)=|1-b^2\tau|^{\frac{s+1}{2}}|\tau b|^{-\frac{z+1}{2}}(q^{-\frac{s+1}{2}}-q^{\frac{s+1}{2}})^{-1}\tint_{F_v^\times}B(t)|t|^{\frac{z-1}{2}}\,\d^\times t
		\label{EllOrbIntL3-f1}
	\end{align}
	with 
	$$
	B(t)=\tint_{F_v}
	\max(|1-x|,|1+x|,|t|,|t^{-1}(b^2\tau-x^2)|)^{-s-1}
	\,\d x.
	$$
	When $t$ is viewed as a constant, we let $f(x)$ denote the integrand of $B(t)$. First we consider the case when $\tau=\Delta_v^{0}$ is a non-unit. 
	\begin{lem}\label{EllOrbIntL3-L1}
		Suppose $|\tau|=q^{-1}$, $|a|\leq 1$ and set $b=a^{-1}$. Let $\Re(s)>-1/2$. 
		\begin{itemize}
			\item[(i)] When $|\tau b|\geq |t|$, 
			\begin{align*}
				B(t)&=q^{-d/2}|t|^{s+1}|\tau b|^{-2s-1}\left\{q^{-s-1}+(1-q^{-1}){q^{-2s-1}}{(1-q^{-2s-1})^{-1}}\right\}.
			\end{align*}
			\item[(ii)] When $|\tau b|<|t|$, 
			\begin{align*}
				B(t)&=q^{-d/2}|t|^{-s}\left\{1+(1-q^{-1}){q^{-2s-1}}{(1-q^{-2s-1})^{-1}}\right\}.
			\end{align*}
		\end{itemize}
	\end{lem}
	\begin{proof} Suppose $|a|\leq 1$. Then $|b|=|a|^{-1}\geq 1$ and $|\tau|=q^{-1}$. Hence $|\tau b|\geq q^{-1}$. By dividing the integration domain into $D_1=\{x\in F_v||x|\leq \max(|\tau b|,|t|)\}$ and into $D_2=\{x\in F_v | |x|> \max(|\tau b|,|t|)\}$, we write the integral as the sum $B_1(t)+B_2(t)$ with $B_{i}(t)=\int_{D_i}f(x)\d x$. Let $x\in D_2$. Then by $|\tau b|\geq q^{-1}$, we have $|x|>\max(|\tau b|,|t|)\geq \max(q^{-1},|t|)\ge q^{-1}$. Hence $|1\pm x| \leq |x|$, and $|x|^2>|\tau b|^2$ which implies $|x|^2\geq |\tau b|^2q=|\tau b^2|$. Hence $|\tau b^2-x^2|=|x|^2$. Thus $f(x)=(|t|^{-1}|x|^2)^{-s-1}$ for all $x\in D_2$ and 
		\begin{align}
			B_2(t)=\tint_{x\in D_2}(|t|^{-1}|x|^2)^{-s-1}\,\d x=|t|^{s+1} \tint_{D_2}|x|^{-2s-2}\d x.
			\label{EllOrbIntL3-L1-1}
		\end{align}
		(i) Suppose $|t|\leq |\tau b|$. Then $D_1=\{x\in F_v||x|\leq |\tau b|\}$ and $|t|\leq |\tau b|\leq |\tau b^2|$. For $x\in \cO \cap D_1$, we have $|1\pm x|\leq 1\leq |t|^{-1}|\tau b|\leq |t|^{-1}|\tau b^2|=|t^{-1}(\tau b^2-x^2)|$ and $|t|\leq |t|^{-1}|\tau^2 b^2|<|t|^{-1}|\tau b^2|=|t^{-1}(b^2\tau -x^2)|$. Hence $f(x)=(|t|^{-1}|\tau b^2|)^{-s-1}$. For $x\in (F-\cO)\cap D_1$, we have $|1\pm x|=|x|$ and $|x|\leq |\tau b|$. Hence $|xt|\leq |\tau^2 b^2|<|\tau b^2|$. Thus $f(x)=\max(|x|,|t|,|t|^{-1}|\tau b^2|)^{-s-1}=(|t|^{-1}|\tau b^2|)^{-s-1}$. Therefore, 
		\begin{align*}
			B_{1}(t)&=\tint_{|x|\leq|\tau b|}(|t|^{-1}|\tau b^2|)^{-s-1}\d x=q^{-d/2}|\tau b|(|t|^{-1}|\tau b^2|)^{-s-1}.
		\end{align*}
		We have
		\begin{align*}
			B_2(t)& = |t|^{s+1}\tint_{|x|>|\tau b|}|x|^{-2s-2}\d x
			=q^{-d/2}(1-q^{-1})|\tau b|^{-2s-1}|t|^{s+1}\frac{q^{-2s-1}}{1-q^{-2s-1}}. \end{align*}
		(ii) Suppose $|t|>|\tau b|$. Then from $|\tau|=q^{-1}$, we have $|t|\geq |b|\,(\geq 1)$. Thus $|t|^2\geq |b|^2>|\tau b^2|$. Hence $|t|>|t|^{-1}|\tau b^2|$. Moreover, $D_1=\{x\in F_v||x|\leq |t|\,\}$. If $x\in D_1$, we have $|1\pm x|\le \max(1, |x|)\le|t|$ and $|t^{-1}(\tau b^2-x^2)| = \max(|\tau b^2|, |x^2|) \le |t^{-1}|\max(|t|^2, |x|^2) =|t|$;
		thus $f(x)=|t|^{-s-1}$. Hence
		\begin{align*}
			B_{1}(t)& = \tint_{D_1}|t|^{-s-1}dx = q^{-d/2}|t|^{-s}.
		\end{align*}
		Since $D_2=\{x\in F_v ||x|>|t|\}$, from \eqref{EllOrbIntL3-L1-1} we have
		\begin{align*}
			B_2(t)& = |t|^{s+1}\tint_{|x|>|t|}|x|^{-2s-2}\d x =q^{-d/2}(1-q^{-1})\frac{q^{-2s-1}}{1-q^{-2s-1}}|t|^{-s}.\end{align*}
		Therefore, we obtain the desired formula by $B(t)=B_1(t)+B_2(t)$. 
	\end{proof}

	\begin{lem} \label{EllOrbIntL3-L2}
		Suppose $|\tau|=q^{-1}$, $|a|>1$ and set $b=a^{-1}$. Let $\Re(s)>-1/2$. 
		\begin{itemize}
			\item[(i)] When $|t|\leq |\tau b|^2$,
			\begin{align*}
				B(t)&=q^{-d/2}|t|^{s+1}|\tau b|^{-2s-1}\left(-1+q^{-s-1}+{(1-q^{-2s-2})}{(1-q^{-2s-1})^{-1}}\right).
			\end{align*}
			\item[(ii)] When $|\tau b|^2<|t|\leq 1$, 
			\begin{align*}
				B(t)&=q^{-d/2}\biggl( [|t|^{1/2}]+
				|t|^{s+1}[|t|^{1/2}]^{-2s-1}{(1-q^{-1})q^{-2s-1}}{(1-q^{-2s-1})^{-1}} \biggr).
			\end{align*}
			\item[(iii)] When $1<|t|$, 
			\begin{align*}
				B(t)&=q^{-d/2}|t|^{-s}{(1-q^{-2s-2})}{(1-q^{-2s-1})^{-1}}.
			\end{align*}
		\end{itemize}
	\end{lem}
	\begin{proof} Note that $|\tau b|<|\tau| = q^{-1}<1$. As in the proof of Lemma~\ref{EllOrbIntL3-L1}, we write $B(t)=B_1(t)+B_2(t)$, $B_i(t)=\int_{D_i}f(x)\,\d x$ with $D_1=\{x\in F_v|\,|x|\leq \max(|\tau b|,|t|)\,\}$ and $D_2=F_v-D_1$. For $x\in D_2$, we have $|\tau b|<|x|$ and $|x|>|t|$, by which $|t^{-1}(\tau b^2-x^2)|=|t|^{-1}|x|^2>|t|$. Thus $f(x)=\max(|1-x|,|1+x|,|t|^{-1}|x|^2)^{-s-1}$ for $x\in D_2$. We have $f(x)=\max(|1+x|,|1-x|,|t|,|t^{-1}(\tau b^2-x^2)|)^{-s-1}$ for $x\in D_1$. 
		
		\smallskip
		\noindent
		(i) Suppose $|t|\leq |\tau b|$. For $x\in D_1\cap \fp$, we have $|1\pm x|=1$ and $|t|\leq |t|^{-1}|\tau^2 b^2|<|t|^{-1}|\tau b^2|=|t^{-1}(\tau b^2-x^2)|$; thus $f(x)=\max(1,|t|^{-1}|\tau b^2|)^{-s-1}$. The set $D_1\cap (F_v-\fp)$ is empty due to $|\tau b|<1$. Hence
		\begin{align*}
			B_1(t)&=\tint_{\substack{|x|\leq |\tau b| \\ |x|<1}}\max(1,|t|^{-1}|\tau b^2|)^{-s-1}\d x=q^{-d/2}|\tau b|\max(1,|t|^{-1}|\tau b^2|)^{-s-1}. 
		\end{align*}
		For $x\in D_2\cap \fp$, we have $|1\pm x|=1$ and $f(x)=\max(1,|t|^{-1}|x|^2)^{-s-1}$. For $x\in D_2\cap \cO^\times$, we have $|1\pm x|\leq 1<|\tau b|^{-1}\leq |t|^{-1}=|t|^{-1}|x|^2$; thus $f(x)=|t|^{s+1}$. For $x\in D_2\cap(F_v-\cO)$, we have $|t|\leq |\tau b|<1<|x|$, which implies $|x|<|t|^{-1}|x|^2$; thus $f(x)=(|t|^{-1}|x|^{2})^{-s-1}$. Therefore, 
		{\allowdisplaybreaks
			\begin{align*}
				B_2(t)=& \biggl\{\tint_{\substack{x\in D_2\cap \fp \\ |t|^{1/2}<|x|}}(|t|^{-1}|x|^2)^{-s-1} \d x+\tint_{\substack{x\in D_2\cap \fp \\ |t|^{1/2}\geq |x|}}\d x
				+\tint_{D_2\cap \cO^\times} |t|^{s+1} \d x
				+\tint_{D_2\cap (F_v-\cO)}(|t|^{-1}|x|^{2})^{-s-1}\d x
				\biggr\}
				\\
				=& |\tau b|^{z+1}\biggl\{|t|^{s+1} \tint_{\max(|\tau b|,[|t|^{1/2}])<|x|<1}
				|x|^{-2s-2}\d x
				+\int_{|\tau b|<|x|\leq [|t|^{1/2}]}\d x\\
				&+|t|^{s+1}\tint_{|x|=1} \d x
				+|t|^{s+1} \tint_{1<|x|}|{x}|^{-2s-2}\d x
				\biggr\}
				\\
				=&q^{-d/2}(1-q^{-1})\biggl\{\frac{q^{-2s-1}\max(|\tau b|,[|t|^{1/2}])^{-2s-1}-q^{-2s-1}}{1-q^{-2s-1}}\times |t|^{s+1}\\
				&+\delta(|\tau b|<|t|^{1/2})\frac{[|t|^{1/2}]-|\tau b|}{1-q^{-1}} +\frac{q^{-2s-1}}{1-q^{-2s-1}}\times |t|^{s+1}
				\biggr\}.
		\end{align*}}From this, 
		\begin{align*}
			B_2(t)=&q^{-d/2}(1-q^{-1})
			\begin{cases}
				|\tau b|^{-2s-1}|t|^{s+1}\frac{q^{-2s-1}}{1-q^{-2s-1}} \quad \text{if $|\tau b|^{2}\geq |t|$}, \\
				|t|^{s+1}[|t|^{1/2}]^{-2s-1}\frac{q^{-2s-1}}{1-q^{-2s-1}}+\frac{[|t|^{1/2}]-|\tau b|}{1-q^{-1}} \quad \text{if $|\tau b|^2<|t|\leq |\tau b|$}.
			\end{cases}
		\end{align*} 
		(ii) Suppose $|t|>|\tau b|$. Then $D_1=\{x\in F_v|\,|x|\leq |t|\}$. For $x\in D_1\cap \gp$, we have $|1\pm x|= 1$ and $|t^{-1}(\tau b^2-x^2)|\leq |t|^{-1}\max(|\tau b^2|,|x|^2) \leq |t|$; hence $f(x)=\max(1,|t|)^{-s-1}$. For $x\in D_1\cap (F_v-\gp)$, $|1\pm x|\leq |x|\leq |t|$ and $|t^{-1}(\tau b^2-x^2)|\leq |t|$; thus $f(x)=|t|^{-s-1}$. Hence
		\begin{align*}
			B_1(t)&=\tint_{|x|\leq \min(q^{-1},|t|)}\max(1,|t|)^{-s-1}\,\d x+\tint_{1\leq |x|\leq |t|}|t|^{-s-1}\d x
			\\
			&=q^{-d/2}\{\max(1,|t|)^{-s-1}\,\min(q^{-1},|t|)
			+|t|^{-s}\delta(|t|\geq 1)(|t|-1)\}\\
			&=q^{-d/2}
			\begin{cases}
				|t|\quad (|\tau b|<|t|\leq 1), \\
				|t|^{-s}\quad (1<|t|).
			\end{cases}
		\end{align*} 
		We have $D_2=\{x\in F_v||x|>|t|\}$. For $x\in D_2$, we have $|\tau b|<|t|<|x|$ which in turn yields $|\tau b^2|<|x|^2$ and $|t|<|t|^{-1}|x|^2$. Hence $f(x)=\max(|1+x|,|1-x|,|t|^{-1}|x|^2)|^{-s-1}$ for $x\in D_2$. We have
		{\allowdisplaybreaks
			\begin{align*}
				B_2(t)& = \tint_{|t|<|x|<1} \max(1,|t|^{-1}|x|^2)^{-s-1}\d x+\tint_{|x|=1}\delta(|t|<1)|t|^{s+1}\d x \\
				&+\tint_{\max(1,|t|)<|x|}(|t|^{-1}|x|^2)^{-s-1}\d x
				\\
				&=q^{-d/2}(1-q^{-1})|\tau b|^{z+1}\biggl\{\delta(|t|<1)|t|^{s+1}\frac{q^{-2s-1}([|t|^{1/2}]^{-2s-1}-q^{2s+1})}{1-q^{-2s-1}}
				\\
				&\quad+\delta(|t|\leq 1)\frac{[|t|^{1/2}]-|t|}{1-q^{-1}}
				+\delta(|t|<1)|t|^{s+1}
				+|t|^{s+1}\frac{q^{-2s-1}\max(1,|t|)^{-2s-1}}{1-q^{-2s-1}} \biggr\}
				\\
				&=q^{-d/2}(1-q^{-1})
				\begin{cases}
					|t|^{-s}\frac{q^{-2s-1}}{1-q^{-2s-1}} \quad (1<|t|), \\
					|t|^{s+1}[|t|^{1/2}]^{-2s-1}\frac{q^{-2s-1}}{1-q^{-2s-1}}+\frac{[|t|^{1/2}]-|t|}{1-q^{-1}} \quad (|\tau b|<|t|\leq 1).
				\end{cases}
		\end{align*}}Here to have the second equality, we split the integral over $|t|<|x|<1$ to those over $|t|^{1/2}<|x|<1 $ and over $|t|<|x|\leq |t|^{1/2}$.
		
		From the evaluations of $B_1(t)$ and $B_2(t)$, by $B(t)=B_1(t)+B_2(t)$ we are done. \end{proof}
	The following is given by a direct computation.
	\begin{lem}\label{fundamental integral formula} If $\a, \b \in \CC$, then
		$$\int_{|\tau b|^2 < |t| \le 1}[|t|^{1/2}]^{\a}|t|^{\b}\d^\times t=q^{-d/2}\left(\frac{1+q_v^{-\a-\b}}{1-q^{-\a-2\b}} + \frac{-1-q^{-\a-\b}}{1-q^{-\a-2\b}}|\tau b|^{\a+2\b} \right).$$
	\end{lem}
	
	\begin{lem}\label{explicit of J(a):ramified}
		Let $\ord_v(\tau)=1$. Set $b=a^{-1}$. Suppose $\Re(s)>(|\Re(z)|-1)/2$.
		For $|a|\leq 1$,
		$$\int_{F_v^\times}B(t)|t|^{\frac{z-1}{2}} d^\times t
		= q^{-d}|b|^{-s+\frac{z-1}{2}}(1+q^{\frac{-z-1}{2}})
		\frac{\zeta_{F_v}(s+1+\frac{z-1}{2})\zeta_{F_v}(s+1+\frac{-z-1}{2})}{\zeta_{F_v}(s+1)}.
		$$
		For $|a|>1$,
		$$\int_{F_v^\times}B(t)|t|^{(z-1)/2}d^\times t = q^{-d}\frac{1+q^{-\frac{z+1}{2}}}{\zeta_{F_v}(s+1)}\left( \zeta_{F_v}(z)\zeta_{F_v}\left(s+1+\tfrac{-z-1}{2}\right) +\zeta_{F_v}(-z)\zeta_{F_v}\left(s+1+\tfrac{z-1}{2}\right)|b|^z \right).$$
	\end{lem}
	\begin{proof} This is proved by a straightforward computation by means of Lemmas~\ref{EllOrbIntL3-L1}, \ref{EllOrbIntL3-L2} and \ref{fundamental integral formula}. \end{proof}
	
	We obtain the formula for $a=\frac{t}{2m_v}\neq 0$ in Theorem~\ref{OrbIntUnifEx} (5) for the case $\tau\in \fp-\fp^2$ from \eqref{EllOrbIntL3-f1} applying Lemma \ref{explicit of J(a):ramified} combined with Lemma~\ref{sphericalftn1value}, Lemma~\ref{OrbIntUnifExL2} (2) and Proposition~\ref{OrbIntUnifExL4} (2). We note that $L(s,\varepsilon_{\Delta, v})=1$.

	\smallskip
	\noindent
	
	Next we consider the case when $\tau=\Delta_v^{0}$ is a non-square unit.
	\begin{lem} \label{EllOrbIntL3-L4}
		Suppose $\tau \in \cO^\times-(\cO^\times)^2$ and $|a|>1$. Set $b=a^{-1}$. Suppose $\Re(s)>-1/2$. 
		\begin{itemize}
			\item[(i)] When $|t|\leq |b|^2$, 
			\begin{align*}
				B(t)&=q^{-d/2}|t|^{s+1}|b|^{-2s-1}{(1-q^{-2s-2})}{(1-q^{-2s-1})^{-1}}.
			\end{align*}
			\item[(ii)] When $|b|^2<|t|\leq 1$, 
			\begin{align*}
				B(t)&=q^{-d/2}\left\{(1-q^{-1})|t|^{s+1}[|t|^{1/2}]^{-2s-1}{q^{-2s-1}}{(1-q^{-2s-1})^{-1}} +[|t|^{1/2}] \right\}.
			\end{align*}
			\item[(iii)] When $1<|t|$, 
			\begin{align*}
				B(t)&=q^{-d/2}|t|^{-s}{(1-q^{-2s-2})}{(1-q^{-2s-1})^{-1}}.
			\end{align*}
		\end{itemize}
	\end{lem}
	\begin{proof} Let $D_1=\{x\in F_v|\,|x|\leq \max(|b|,|t|)\}$ and $D_2=\{x\in F_v|\,|x|>\max(|b|,|t|)\}$, and set $B_i(t)=\int_{D_i}f(x)\d x$ for $i=1,2$ to write $B(t)$ as the sum $B_1(t)+B_2(t)$, where
		$$
		f(x)=\max(|1+x|,|1-x|,|t|,|t^{-1}(\tau b^2-x^2)|)^{-s-1}, x\in F_v.
		$$
		(i) Suppose $|t|\leq |b|$. Then $D_1=\{x\in F_v||x|\leq |b|\}$. Since $|b|<1$, the set $D_1\cap (F-\fp)$ is empty. If $x\in D_1\cap \fp$, then $|1\pm x|=1$ and $|t|\leq |t|^{-1}|b|^{2}=|t^{-1}(\tau b^2-x^2)|$ by Lemma~\ref{Hensel}; hence $f(x)=\max(1,|t|^{-1}|b|^2)^{-s-1}$. Thus
		\begin{align*}
			B_1(t)&=\max(1,|t|^{-1}|b|^2)^{-s-1} \tint_{|x|\leq |b|}\d x=q^{-d/2}|b|\max(1,|t|^{-1}|b|^{2})^{-s-1}.
		\end{align*}
		For $x\in D_2$, we have $|x|>|b|\geq |t|$; hence $|t^{-1}(\tau b^2-x^2)|=|t|^{-1}|x|^{2}>|t|$ and $f(x)=\max(|1-x|,|1+x|,|t|^{-1}|x|^2)^{-s-1}$. Thus $B_2(t)$ equals
		{\allowdisplaybreaks
			\begin{align*}
				&\tint_{|b|<|x|<1} \max(1,|t|^{-1}|x|^2)^{-s-1}\d x+\tint_{|x|=1}|t|^{s+1}\d x	+\tint_{|x|>1} \max(|x|,|t|^{-1}|x|^2)^{-s-1}\d x
				\\
				=& |t|^{s+1}\tint_{\max(|b|,|t|^{1/2})<|x|<1} |x|^{-2s-2}\d x +\tint_{|b|<|x|\leq |t|^{1/2}}\d x\\
				&+|t|^{s+1}\,q^{-d/2}(1-q^{-1})
				+|t|^{s+1}\tint_{|x|>1}|x|^{-2s-2}\d x
				\\
				= & |t|^{s+1} \tint_{\max(|b|,[|t|^{1/2}) <|x|\leq q^{-1}}|x|^{-2s-2}\d x 
				+\tint_{|b|<|x|\leq [|t|^{1/2}]} \d x\\
				&+|t|^{s+1}\times q^{-d/2}(1-q^{-1})
				+|t|^{s+1}\times q^{-d/2}(1-q^{-1})\tfrac{q^{-2s-1}}{1-q^{-2s-1}}
				\\
				= &\,q^{-d/2}(1-q^{-1})\biggl\{
				|t|^{s+1}\tfrac{q^{-2s-1}(\max(|b|,[|t|^{1/2}])^{-2s-1}-q^{2s+1})}{1-q^{-2s-1}}
				+\delta(|b|^2\leq |t|)\tfrac{[|t|^{1/2}]-|b|}{1-q^{-1}}
				+|t|^{s+1}\tfrac{1}{1-q^{-2s-1}}
				\biggr\}.
		\end{align*}}
		(ii) Suppose $|t|>|b|$. Then $D_1=\{x\in F_v|\,|x|\leq |t|\}$. If $x\in D_1\cap \fp$, then $|1\pm x|=1$, $|t^{-1}(\tau b^2-x^2)|\leq |t|^{-1}\max(|b|^2,|x|^2)\leq |t|$; hence $f(x)=\max(1,|t|)^{-s-1}$. If $x\in D_1\cap (F_v-\fp)$, then $|1\pm x|\leq |x|\leq |t|$ and $|t^{-1}(\tau b^2-x^2)|\leq |t|$ as above; hence $f(x)=|t|^{-s-1}$. Thus
		\begin{align*}
			B_1(t)& = 
			\max(1,|t|)^{-s-1}\tint_{x\in D_1\cap \fp} \d x
			+|t|^{-s-1} \tint_{1\leq |x|\leq |t|} \d x
			=q^{-d/2}
			\begin{cases}
				|t| \quad (|t|<1), \\
				|t|^{-s}\quad (|t|\geq 1).
			\end{cases}
		\end{align*} 
		For all $x\in D_2$, we have $|t|< |t|^{-1}|x|^{2}=|t^{-1}(\tau b^2-x^2)|$. If $x\in D_2\cap \fp$, we have $|1\pm x|=1$; thus $f(x)=\max(1,|t|^{-1}|x|^2)^{-s-1}$. If $x\in D_2\cap \cO^\times$, we have $|1\pm x|\leq 1<|t|^{-1}=|t|^{-1}|x|^2$; hence $f(x)=|t|^{s+1}$. Note that $D_2\cap \cO^\times$ is empty unless $|t|<1$. If $x\in D_2\cap (F_v-\cO)$, then $|t|<|x|$ and $|1\pm x|=|x|<|t|^{-1}|x|^2$; hence $f(x)=(|t|^{-1}|x|^2)^{-s-1}$. Thus, 
		{\allowdisplaybreaks\begin{align*}
				B_2(t) =&
				\delta(|t|<1)\tint_{|t|<|x|<1}\max(1,|t|^{-1}|x|^2)^{-s-1}\d x
				+\delta(|t|<1)|t|^{s+1}\tint_{|x|=1}\d x\\
				&+|t|^{s+1}\tint_{|x|>\max(1,|t|)} |x|^{-2s-2}\d x
				\\
				=&
				\delta(|t|<1)|t|^{s+1} \tint_{|t|^{1/2}<|x|<1}|x|^{-2s-2}\d x
				+\delta(|t|<1)\tint_{|t|<|x|\leq |t|^{1/2}}\d x \\
				&\quad +\delta(|t|<1)|t|^{s+1}\,q^{-d/2}(1-q^{-1}) +|t|^{s+1}\,q^{-d/2}(1-q^{-1})\tfrac{q^{-2s-1}\max(1,|t|)^{-2s-1}}{1-q^{-2s-1}}
				\\
				=&
				q^{-d/2}(1-q^{-1}) \biggl\{
				\delta(|t|<1)|t|^{s+1} \tfrac{q^{-2s-1}[|t|^{1/2}]^{-2s-1}-1}{1-q^{-2s-1}}
				+\delta(|t|<1) \tfrac{[|t|^{1/2}]-|t|}{1-q^{-1}}\\
				&+\delta(|t|<1)|t|^{s+1}
				+|t|^{s+1}\tfrac{q^{-2s-1}\max(1,|t|)^{-2s-1}}{1-q^{-2s-1}}
				\biggr\}
				\\
				=&q^{-d/2}(1-q^{-1})
				\begin{cases}
					|t|^{-s}\tfrac{q^{-2s-1}}{1-q^{-2s-1}}\quad (|t|>1), \\
					|t|^{s+1}[|t|^{1/2}]^{-2s-1}\tfrac{q^{-2s-1}}{1-q^{-2s-1}}+\tfrac{[|t|^{1/2}]-|t|}{1-q^{-1}} \quad (|b|<|t|\leq 1).
				\end{cases}
		\end{align*}} From  the evaluations of $B_1(t)$ and $B_2(t)$, we are done.
	\end{proof}

	\begin{lem} \label{EllOrbIntL3-L5}
		Suppose $\tau \in \cO^\times-(\cO^\times)^2$ and $|a|\leq 1$. Set $b=a^{-1}$. If $\Re(s)>-1/2$, 
		\begin{align*}
			B(t)&=q^{-d/2}{(1-q^{-2s-2})}{(1-q^{-2s-1})^{-1}}\begin{cases} |t|^{s+1}|b|^{-2s-1} , \quad (|t|\leq |b|), \\
				|t|^{-s}, \quad (|b|<|t|).
			\end{cases}
		\end{align*}
	\end{lem}
	\begin{proof}
		Let $D_i$ and $B_i(t)$ $(i=1,2)$ be as in the proof of Lemma~\ref{EllOrbIntL3-L4}. 
		For $x\in D_2$, we have $|t|^{-1}|x|^2>|x|>|b|\ge 1$, $|x|>|t|$, $|1\pm x|=|x|>|t|$ and $|t^{-1}(\tau b^2-x^2)|=|t|^{-1}|x|^2$. Hence $f(x)=(|t|^{-1}|x|^2)^{-s-1}$. Thus
		\begin{align*}
			B_2(t)&=|t|^{s+1} \tint_{\max(|b|,|t|)<|x|} |x|^{-2s-2}\d x
			=q^{-d/2}(1-q^{-1})|t|^{s+1} \frac{q^{-2s-1}\max(|b|,|t|)^{-2s-1}}{1-q^{-2s-1}}.
		\end{align*}
		(i) Suppose $|t|\leq |b|$. Then $D_1=\{x\in F_v||x|\leq |b|\}$. If $x\in D_1$, then $|1\pm x|\leq \max(1,|x|) \leq |b|\leq |t|^{-1}|b|^2$ and $|t^{-1}(\tau b^2-x^2)|=|t|^{-1}|b|^2|\tau-b^{-2}x^{2}|=|t|^{-1}|b|^2$ because $\tau$ is not congruent to a square modulo $\fp$ by assumption (Lemma \ref{Hensel}); hence $f(x)=(|t|^{-1}|b|^2)^{-s-1}$. Thus
		\begin{align*}
			B_1(t)&=\tint_{|x|\leq |b|}(|t|^{-1}|b|^2)^{-s-1}\d x
			=q^{-d/2}|t|^{s+1}|b|^{-2s-1}.
		\end{align*}
		(ii) Suppose $|t|>|b|$. Then $|t|>1$ and $D_1=\{x\in F_v||x|\leq |t|\}$. If $x\in D_1$, then we have $|1\pm x|\leq \max(1,|x|)\leq |t|$ and $|t^{-1}(\tau b^2-x^2)|\leq |t|^{-1}\max(|b|^2,|x|^2)\leq |t|$; hence $f(x)=|t|^{-s-1}$. Thus, 
		\begin{align*}
			B_1(t)& = |t|^{-s-1}\tint_{|x|\leq |t|}\d x 
			=q^{-d/2}|t|^{-s}. 
		\end{align*}
		Having evaluations of $B_i(t)$ and $B(t)=B_1(t)+B_2(t)$, we are done. 
	\end{proof}
	
	\begin{lem}\label{explicit of J(a):non-dyadic, unramified}Let $\tau \in \go^\times-(\go^\times)^2$. Suppose $\Re(s)>(|\Re(z)|-1)/2$.
		For $|a|\leq 1$,
		$$\int_{F_v^\times}B(t)|t|^{\frac{z-1}{2}} d^\times t
		= q^{-d}|b|^{-s+\frac{z-1}{2}}(1-q^{-s-1})
		\frac{\zeta_{F_v}(s+1+\frac{z-1}{2})\zeta_{F_v}(s+1+\frac{-z-1}{2})}{L_{F_v}(s+1, \varepsilon_{\Delta,v})}.
		$$
		For $|a|>1$,
		$$\int_{F_v^\times}B(t)|t|^{(z-1)/2}d^\times t = q^{-d}(1-q^{-s-1})\left( \frac{\zeta_{F_v}(z)\zeta_{F_v}(s+1+\frac{-z-1}{2})}{L_{F_v}(\frac{z+1}{2}, \varepsilon_{\Delta,v})}
		+ \frac{\zeta_{F_v}(-z)\zeta_{F_v}(s+1+\frac{z-1}{2})}{L_{F_v}(\frac{-z+1}{2} , \varepsilon_{\Delta,v})}|b|^z \right).$$
	\end{lem}
	\begin{proof} This is proved by a straightforward calculation by means of Lemma~\ref{EllOrbIntL3-L4}, \ref{EllOrbIntL3-L5}, and \ref{fundamental integral formula}. 
	\end{proof}
	
	We obtain the formulafor $a=\frac{t}{2m_v}\neq 0$ in Theorem~\ref{OrbIntUnifEx} (5) for the case $\tau\in \cO^\times-(\cO^\times)^2$ from \eqref{EllOrbIntL3-f1} applying Lemma \ref{explicit of J(a):non-dyadic, unramified} combined with Lemma~\ref{sphericalftn1value}, Lemma~\ref{OrbIntUnifExL2} (1) and Proposition~\ref{OrbIntUnifExL4} (1).
	
	The case $a=0$ is included in the case $|a|\le1$.
	Indeed,
	we can evaluate $\fE_v^{(z)}(\hat{\gamma}_v)$ when $0<|a|\le 1$ as
	$$\varphi_{0,v}(1_2)\fE_v(\hat{\gamma}_v) = |a^2-\tau|^{\frac{s+1}{2}}(q^{-\frac{s+1}{2}}-q^{\frac{s+1}{2}})
	\int_{F_v^\times}{\tilde B}(a,t)|t|^{\frac{z-1}{2}}d^\times t,$$
	where we set ${\tilde B}(a,t)=|a|^{-s}|\tau|^{-1}B(a,\tau a^{-1}t)$
	and we write $B(a,t)$ for $B(t)$ to make dependence on $a$ explicit.
	The function $\tilde B(a,t)$ is independent of $a$ by Lemma \ref{EllOrbIntL3-L1} (i) and Lemma \ref{EllOrbIntL3-L5} (i).
	Since $\tilde B(a,t)$ is continuous at $a=0$ by its integral representation,
	we obtain the formula of $\tilde B(0,t)$ by $\lim_{a\rightarrow a}\tilde B(a,t)$.
	Hence the formula in Theorem~\ref{OrbIntUnifEx} (5) is valid when $a=\frac{t}{2m_v}=0$.

	\section*{Acknowledgements}
	The second author was supported by Grant-in-Aid for Scientific research (C) 15K04795.

\medskip
\noindent
{Shingo SUGIYAMA\\
	Department of Mathematics, College of Science and Technology, Nihon University, Suruga-Dai, Kanda, Chiyoda, Tokyo 101-8308, Japan} \\
	{\it E-mail} : {\tt s-sugiyama@math.cst.nihon-u.ac.jp}

\medskip
\noindent
{Masao TSUZUKI \\ Department of Science and Technology, Sophia University, Kioi-cho 7-1 Chiyoda-ku Tokyo, 102-8554, Japan} \\
{\it E-mail} : {\tt m-tsuduk@sophia.ac.jp}

\end{document}